\documentclass[11pt]{amsart}
\usepackage[leqno]{amsmath}
\usepackage{amssymb}
\usepackage{enumerate}

\usepackage{subfigure}
\usepackage{graphicx}
\usepackage[all]{xy}
\usepackage{mathtools}
\usepackage{mathrsfs}
\usepackage{bm}

\usepackage[usenames,dvipsnames]{xcolor}
\usepackage[pdftitle={Regularity of the SLE(4) uniformizing map and the SLE(8) trace},
  pdfauthor={Konstantinos Kavvadias, Jason Miller, and Lukas Schoug},
colorlinks=true,linkcolor=NavyBlue,urlcolor=RoyalBlue,citecolor=PineGreen,bookmarks=true,bookmarksopen=true,bookmarksopenlevel=2,unicode=true,linktocpage]{hyperref}

\usepackage{microtype}
\usepackage{centernot}
\usepackage{stmaryrd}

\numberwithin{equation}{section}

\usepackage{comment}

\newtheorem{theorem}{Theorem}[section]
\newtheorem{remark}[theorem]{Remark}
\newtheorem{lemma}[theorem]{Lemma}
\newtheorem{proposition}[theorem]{Proposition}

\newcommand{\A}{\mathbf{A}}

\newcommand{\C}{\mathbf{C}}
\newcommand{\D}{\mathbf{D}}
\newcommand{\E}{\mathbf{E}}

\newcommand{\h}{\mathbf{H}}
\newcommand{\N}{\mathbf{N}}
\newcommand{\Z}{\mathbf{Z}}
\newcommand{\p}{\mathbf{P}}

\newcommand{\R}{\mathbf{R}}
\newcommand{\s}{\mathbf{S}}

\newcommand{\Fh}{\mathfrak {h}}

\newcommand{\CA}{\mathcal {A}}
\newcommand{\CB}{\mathcal {B}}
\newcommand{\CC}{\mathcal {C}}

\newcommand{\CE}{\mathcal {E}}
\newcommand{\CF}{\mathcal {F}}
\newcommand{\CI}{\mathcal {I}}

\newcommand{\CL}{\mathcal {L}}

\newcommand{\CT}{\mathcal {T}}

\newcommand{\CW}{\mathcal {W}}

\newcommand{\CZ}{\mathcal {Z}}

\newcommand{\CG}{\mathcal {G}}

\newcommand{\SLE}{{\rm SLE}}

\newcommand{\dist}{\mathrm{dist}}

\newcommand{\diam}{\mathrm{diam}}
\newcommand{\var}{\mathrm{var}}
\newcommand{\im}{\mathrm{Im}}
\newcommand{\re}{\mathrm{Re}}
\newcommand{\cov}{\mathrm{cov}}

\newcommand{\loc}{\mathrm{loc}}
\newcommand{\supp}{\mathrm{supp }}

\newcommand{\distqh}{\mathrm{dist}_{\mathrm{qh}}}

\newcommand{\hcap}{\mathrm{hcap}}

\newcommand{\confrad}{{\rm CR}}

\newcommand{\one}{{\bf 1}}

\newcommand{\wt}{\widetilde}
\newcommand{\ol}{\overline}

\newcommand{\giv}{\,|\,}

\newcommand{\BES}{\mathrm{BES}}

\newcommand{\strip} {\mathscr{S}}
\newcommand{\cyl}{\mathscr{C}}

\newcommand{\qhd}{\dist_{\mathrm{qh}}}
\newcommand{\shadow}{\mathrm{SH}}

\renewcommand{\a}{\alpha}

\newcommand{\sol}[1]{{}}

\newcommand{\wh}{\widehat}

\newcommand{\qwedgeW}[2]{\mathsf{QWedge}_{\bm{\gamma}=#1}^{\mathbf{W}=#2}}
\newcommand{\qconeW}[2]{\mathsf{QCone}_{\bm{\gamma}=#1}^{\mathbf{W}=#2}}
\newcommand{\qwedgeA}[2]{\mathsf{QWedge}_{\bm{\gamma}=#1}^{\bm{\alpha}=#2}}
\newcommand{\qconeA}[2]{\mathsf{QCone}_{\bm{\gamma}=#1}^{\bm{\alpha}=#2}}

\newcommand{\greenD}[1]{G_{#1}^{\mathrm{D}}}
\newcommand{\greenN}[1]{G_{#1}^{\mathrm{N}}}

\begin{document}

\title{Regularity of the $\SLE_4$ uniformizing map and the $\SLE_8$ trace}

\author{Konstantinos Kavvadias, Jason Miller, and Lukas Schoug}

\begin{abstract}
We show that the modulus of continuity of the $\SLE_4$ uniformizing map is given by $(\log \delta^{-1})^{-1/3+o(1)}$ as $\delta \to 0$.  As a consequence of our analysis, we show that the Jones-Smirnov conditions for conformal removability (with quasihyperbolic geodesics) do not hold for $\SLE_4$.  We also show that the modulus of continuity for $\SLE_8$ with the capacity time parameterization is given by $(\log \delta^{-1})^{-1/4+o(1)}$ as $\delta \to 0$, proving a conjecture of Alvisio and Lawler.
\end{abstract}

\date{\today}
\maketitle

\setcounter{tocdepth}{1}

\tableofcontents

\parindent 0 pt
\setlength{\parskip}{0.20cm plus1mm minus1mm}

\section{Introduction}

\subsection{Overview}

The Schramm-Loewner evolution ($\SLE$) was introduced by Schramm in 1999 \cite{schramm2000sle} as a candidate to describe the scaling limits of the interfaces in statistical mechanics models on two-dimensional lattices at criticality.  It has since been proved to arise as a scaling limit in several cases \cite{smirnov2001cardy,lsw2004lerw,ss2009dgff,smirnov2010ising}.  $\SLE$ has also been the subject of intensive study as it has deep connections to the Gaussian free field (GFF) \cite{ss2013continuum,dub2009partition,ms2016imag1} and Liouville quantum gravity (LQG) \cite{she2016zipper,dms2014mating}.

Recall that (chordal) $\SLE$ is a family of probability measures indexed by a parameter $\kappa > 0$ ($\SLE_\kappa)$ on curves which connect two boundary points $x,y$ in a simply connected domain $D$.  In the case that $D = \h$, $x = 0$, $y = \infty$, $\SLE_\kappa$ is defined by considering the random family of conformal maps $(g_t)$ which solve the chordal Loewner equation
\begin{equation}
\label{eqn:chordal_loewner}
\partial_t g_t(z) = \frac{2}{g_t(z) - W_t},\quad g_0(z) = z
\end{equation}
where $W = \sqrt{\kappa} B$ and $B$ is a standard Brownian motion.  For each $t \geq 0$, we let $\h_t$ be the domain of $g_t$.  For $\kappa > 0$ with $\kappa \neq 8$ it was shown by Rohde and Schramm \cite{rs2005basic} that there exists a curve~$\eta$ in~$\h$ from~$0$ to~$\infty$, the so-called $\SLE$ trace, so that $\h_t$ is the unbounded component of $\h \setminus \eta([0,t])$.  The result in the case that $\kappa=8$ was proved by Lawler, Schramm, and Werner in \cite{lsw2004lerw}.  The time parameterization for $\eta$ which is induced by~\eqref{eqn:chordal_loewner} is the so-called half-plane capacity parameterization since the half-plane capacity of the hull $K_t = \h \setminus \h_t$ associated with $\eta([0,t])$ is equal to $2t$ for each $t \geq 0$. $\SLE_\kappa$ in a simply connected domain $D$ connecting boundary points $x,y$ is defined as the image of the $\SLE_\kappa$ in $\h$ from $0$ to $\infty$ under a conformal map $\h \to D$ which takes $0$ to $x$ and $\infty$ to $y$.  The parameter $\kappa > 0$ determines the roughness of the curve.  $\SLE_\kappa$ curves are simple for $\kappa \in (0,4]$, self-intersecting but not space-filling for $\kappa \in (4,8)$, and space-filling for $\kappa \geq 8$ \cite{rs2005basic}. We sometimes omit the word trace and instead refer to the curve as an $\SLE_\kappa$.

A number of works have been focused on the fractal properties of $\SLE_\kappa$.  Let us mention two examples: the regularity of the associated uniformizing conformal map and the regularity of the trace.  Suppose that $D$ is a simply connected domain and $x,y \in \partial D$ are distinct.  Let $\eta$ be an $\SLE_\kappa$ in $D$ from $x$ to $y$.
\begin{itemize}
\item If $D$ is bounded and has smooth boundary and $\kappa \neq 4$ then the components of $D \setminus \eta$ are H\"older domains \cite{rs2005basic}.  This means that a conformal map from $\D$ to any component of $D \setminus \eta$ is H\"older continuous up to $\partial \D$.  Moreover, the optimal H\"older exponent was determined in \cite{gms2018multifractal}.  When $\kappa = 4$, a conformal map from $\D$ to a component of $D \setminus \eta$ is \emph{not} H\"older continuous and there was previously no bound on its modulus of continuity.  As we will explain in more detail just below, this is important for various applications.
\item The continuity of $t \mapsto \eta(t)$ was proved in \cite{rs2005basic} for $\kappa \neq 8$ and for $\kappa = 8$ in \cite{lsw2004lerw}.  If $D = \h$, $x = 0$, $y = \infty$, and $\eta$ is given the standard (half-plane capacity) time parameterization, then $t \mapsto \eta(t)$ is locally H\"older continuous if $\kappa \neq 8$ and is not locally H\"older continuous when $\kappa = 8$.  Moreover, the optimal H\"older exponent was derived in \cite{lind2008holder,vl2011holder} and the related tip multifractal spectrum in \cite{vl2012tip}.  The exact modulus of continuity of $\SLE_8$ with the half-plane capacity parameterization was previously unknown, though it was conjectured by Alvisio and Lawler \cite{al2014sle8} that there should exist $\beta > 0$ so that on any compact time interval there a.s.\ exists a constant $c > 0$ so that it is at most $c (\log \delta^{-1})^{-\beta}$ for $\delta > 0$ sufficiently small.
\end{itemize}

As mentioned above, the value $\kappa=4$ is special because it is the critical value at or below which $\SLE_\kappa$ curves are simple and above which they are not.  In particular, $\SLE_4$ curves are \emph{almost} self-intersecting in the sense that the harmonic measure of a ball of radius $\epsilon$ centered at a point on the curve can decay to $0$ as $\epsilon \to 0$ faster than any power of $\epsilon$.  This is what leads to the uniformizing map not being H\"older continuous.  Also, the value $\kappa=8$ is special because it is the critical value at or above which $\SLE_\kappa$ is space-filling while for $\kappa < 8$ it is not.  This is reflected by the fact that the left and right sides of the outer boundary of $\eta([0,t])$ (i.e., the parts of $\partial \h_t$ to the left and right of $\eta(t)$) are \emph{almost} intersecting in the sense that the harmonic measure of a ball of radius $\epsilon$ centered at $\eta(t)$ can decay to $0$ as $\epsilon \to 0$ faster than any power of $\epsilon$.  Our main results on the modulus of continuity of the $\SLE_4$ uniformizing map and the $\SLE_8$ trace with the capacity time parameterization will in a sense amount to estimating precisely the decay of the harmonic measure as $\epsilon \to 0$ in these two cases.  In particular, we will show in the former case that the harmonic measure in the complement of $\eta$ can decay as quickly as $\exp(-\epsilon^{-3+o(1)})$ and in the latter case can decay as quickly as $\exp(-\epsilon^{-4+o(1)})$ in $\h_t$ as $\epsilon \to 0$.  Although we will not carry this out here, we more generally believe that the techniques of this paper can be used to show that the following are true.
\begin{itemize}
\item The dimension of the set of points $z$ on an $\SLE_4$ curve $\eta$ with the property that the harmonic measure of $B(z,\epsilon)$ in the complement of $\eta$ decays like $\exp(-\epsilon^{-a+o(1)})$ as $\epsilon \to 0$ is $3/2-a/2$ for $a \in [0,3]$.
\item The dimension of the set of points $z \in \h$ so that the harmonic measure of $B(z,\epsilon)$ in $\h_{\tau_z}$, where $\tau_z$ the first time that $\eta \sim \SLE_8$ hits $z$, decays like $\exp(-\epsilon^{-a+o(1)})$ as $\epsilon \to 0$ is $2-a/2$ for $a \in [0,4]$.
\end{itemize}

The strategy taken in the works \cite{rs2005basic,lind2008holder, vl2011holder,vl2012tip,gms2018multifractal,bs2009harmonic,s2020boundaryspectrum,abv2016boundary} is based on estimating the derivatives of either the solution to~\eqref{eqn:chordal_loewner} or its time-reversal.  The main results we will establish in this article will be based on a completely different method, in particular making use of the relationship between certain types of $\SLE_\kappa$ curves and LQG \cite{she2016zipper,dms2014mating}.

\subsection{Main results}

Our first main result is focused on the modulus of continuity for the uniformizing map for $\SLE_4$.

\begin{theorem}
\label{thm:sle4_thm}
Suppose that $\eta$ is an $\SLE_4$ in $\D$ from $-i$ to $i$ and let $\D_L$ be the component of $\D \setminus \eta$ which is to the left of $\eta$.  Let $\varphi$ be the unique conformal transformation from $\D$ to $\D_L$ which fixes $-i$, $-1$, and $i$.  For every $\zeta > 0$ and $\xi>0$ there a.s.\ exists a constant $c > 0$ so that
\begin{align}\label{eq:mod4}
	|\varphi(z)- \varphi(w)| \leq c \!\left(\log\left( 1 + \frac{1}{|z-w|} \right) \right)^{-1/3+\zeta}
\end{align}
for all $z,w \in \D \setminus (B(-i,\xi) \cup B(i,\xi))$. Moreover, for every $\zeta > 0$ we a.s.\ have that
\[ \sup\!\left\{ |\varphi(z) - \varphi(w)| \!\left(\log \frac{1}{|z-w|} \right)^{1/3+\zeta} : z,w \in \D,\ z \neq w \right\}  = \infty.\]
\end{theorem}

One important application for estimating the modulus of continuity of the $\SLE_4$ uniformizing map is whether existing methods can be used to determine if $\SLE_4$ is conformally removable.  Recall that a compact set $K \subseteq \C$ is \emph{conformally removable} if every homeomorphism $\varphi \colon \C \to \C$ which is conformal on $\C \setminus K$ is conformal on $\C$.  The article \cite{js2000remove} by Jones and Smirnov gives various sufficient conditions for a set to be conformally removable (collectively referred to as the Jones-Smirnov conditions).  One commonly used condition from \cite{js2000remove} is that if $K$ is the boundary of a H\"older domain then it is conformally removable.  This condition can be used to show that $\SLE_\kappa$ curves for $\kappa < 4$ are conformally removable as it was shown in \cite{rs2005basic} that they arise as boundaries of H\"older domains.  \cite[Corollary~4]{js2000remove} shows that in fact a modulus of continuity of $\exp(-\sqrt{(\log \delta^{-1}) (\log \log \delta^{-1})}/o(1))$ as $\delta \to 0$ for the uniformizing map suffices for the boundary to be conformally removable (see also the refinement \cite{kn2005remove}), which is a much weaker condition than H\"older continuity.  Theorem~\ref{thm:sle4_thm}, however, shows that the modulus of continuity of the $\SLE_4$ uniformizing map decays to $0$ as $\delta \to 0$ much more slowly.

All of the conditions for conformal removability for boundaries of domains given in \cite{js2000remove} are a consequence of a general condition which we will briefly describe.  Suppose that $D \subseteq \C$ is a domain and $z_0 \in D$.  Let $\Gamma$ be a family of paths in $D$ from $z_0$ to $\partial D$ whose accumulation sets cover $\partial D$.  Let $\CW$ be a Whitney cube decomposition of $D$.  For each $Q \in \CW$, we call the set $\shadow(Q)$ of those points in $\partial D$ which are the endpoint of a path in $\Gamma$ which passes through $Q$ the \emph{shadow} of $Q$.  We let $s(Q) = \diam(\shadow(Q))$.  Then the sufficient condition given in \cite{js2000remove} is
\begin{align}
\label{eqn:js_condition}
	\sum_{Q \in \CW} s(Q)^2 < \infty.
\end{align}
In practice, one takes $\Gamma$ to be the family of quasihyperbolic geodesics from $z_0$ to $\partial D$.  Recall that these are geodesics with respect to the length metric where the length of a path $\gamma \colon [0,1] \to D$ is given by $\int_0^1 \dist(\gamma(t),\partial D)^{-1} |\gamma'(t)| dt$.  Let $\qhd(\cdot,z_0)$ be the quasihyperbolic distance to $z_0$.  It is further explained in \cite{js2000remove} that a necessary condition for~\eqref{eqn:js_condition} to hold (with quasihyperbolic geodesics) is that $\qhd(\cdot,z_0)$ is an $L^1$ function.  In particular, if $\int_D \qhd(w,z_0) dw = \infty$ (where $dw$ denotes Lebesgue measure on $D$) then~\eqref{eqn:js_condition} does not hold.

\begin{theorem}
\label{thm:jones_smirnov_sle4}
Suppose that $\eta$ is an $\SLE_4$ in $\D$ from $-i$ to $i$ and let $\D_L$ be the component of $\D \setminus \eta$ which is to the left of $\eta$.  Almost surely, for any fixed $z_0 \in \D_L$ we have that
\[ \int_{\D_L} \qhd(w ,z_0) dw = \infty.\]
In particular,~\eqref{eqn:js_condition} (with quasihyperbolic geodesics) does not hold for $\SLE_4$.
\end{theorem}

We now provide some additional context related to the importance of the question of whether $\SLE_4$ is conformally removable.  Sheffield proved in \cite{she2016zipper} that the $\SLE_\kappa$ curves for $\kappa < 4$ arise as conformal weldings of certain types of LQG surfaces which have boundary.  Recall that if $\D_1$, $\D_2$ are two copies of the unit disk and $\phi \colon \partial \D_1 \to \partial \D_2$ is a homeomorphism, then a conformal welding with welding homeomorphism $\phi$ consists of a simple curve $\eta$ on $\s^2$ and a pair of conformal transformations $\psi_i$, $i \in \{1,2\}$, from $\D_i$ to the two components of $\s^2 \setminus \eta$ so that $\phi = \psi_2^{-1} \circ \psi_1$.  If $\eta$ is removable, then it is the \emph{only} curve which can arise from the welding homeomorphism $\phi$.  Since the domains which form the complement of an $\SLE_\kappa$ curve with $\kappa < 4$ are H\"older domains \cite{rs2005basic}, it follows from \cite{js2000remove} that one has uniqueness of the welding when the interface is an $\SLE_\kappa$ curve.  This is not known, however, for $\kappa=4$ and Theorem~\ref{thm:jones_smirnov_sle4} implies that one cannot use the results of \cite{js2000remove} (with quasihyperbolic geodesics) to establish this uniqueness.  (See for a weaker form of uniqueness of the welding \cite{mmq2018uniqueness} in the case $\kappa=4$.)

\textbf{Update:} In \cite{kms2022sle4} the present authors proved that $\SLE_4$ is, in fact, conformally removable. Moreover, in \cite{kms2023nonsimple} we further proved that non-simple $\SLE_\kappa$ curves, where $\kappa$ is chosen so that the adjacency graph of $\SLE_\kappa$ bubbles is almost surely connected (which holds at least for $\kappa \in (4,8)$ sufficiently close to $4$, see \cite{gp2020connectivity}), are conformally removable as well.

Our next main result is focused on the modulus of continuity for $\SLE_8$ with the capacity parameterization.

\begin{theorem}
\label{thm:sle8_thm}
Suppose that $\eta$ is an $\SLE_8$ in $\h$ from $0$ to $\infty$ parameterized by half-plane capacity.  For every $\zeta > 0$ and $T>0$ there a.s.\ exists a constant $c > 0$ so that
\begin{align*}
	|\eta(s) - \eta(t)| \leq c \!\left(\log\left(1+ \frac{1}{|s-t|}\right) \right)^{-1/4+\zeta} \quad\text{for all}\quad 0 \leq s,t \leq T.
\end{align*}
Moreover, for every $\zeta > 0$ and $T > 0$ we a.s.\ have that
\begin{align*}
	\sup\!\left\{ |\eta(s) - \eta(t)| \!\left(\log \frac{1}{|s-t|} \right)^{1/4+\zeta}  : 0 \leq s < t \leq T\right\} = \infty.
\end{align*}
\end{theorem}

As we mentioned above, Theorem~\ref{thm:sle8_thm} proves a more precise form of \cite[Conjecture~1]{al2014sle8}.

We remark that the half-plane capacity parameterization is only one of a number of ways of parameterizing $\SLE_\kappa$.  Another possibility is the so-called \emph{natural parameterization} \cite{ls2011natural,lr2015minkowski}, which is the parameterization which conjecturally describes the scaling limit of discrete lattice models where the interfaces are parameterized according to the number of edges they traverse (see \cite{lv2018natural} for an example of this).  For $\kappa < 8$, it is shown in \cite{zhan2019optimal} that the optimal H\"older exponent in this case is equal to $1/d_\kappa$ where $d_\kappa=1+\kappa/8$ is the dimension of $\SLE_\kappa$ \cite{rs2005basic,bef2008dimension}.  For $\kappa \geq 8$, the natural parameterization is equivalent to parameterizing the curve according to Lebesgue measure and it is known \cite{ghm2020kpz} that the optimal H\"older exponent in this case is $1/2$.

\subsection{Outline and proof strategy}

\begin{figure}[ht!]
\begin{center}
\includegraphics[scale=0.85]{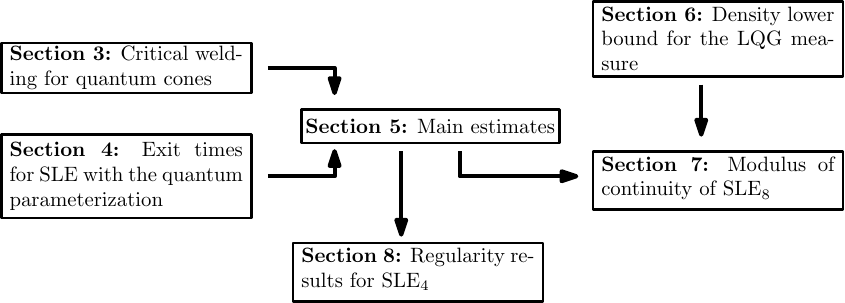}	
\end{center}
\caption{\label{fig:paper_scheme} Schematic illustration of how the sections of the paper fit together.}
\end{figure}

The remainder of this article is structured as follows.  We will collect a number of preliminaries in Section~\ref{sec:preliminaries}.  In Section~\ref{sec:critical_welding} we establish a version of one of the welding results from \cite{dms2014mating} in the critical case that $\kappa = 4$ and in Section~\ref{sec:exit_times} we prove bounds for exit times for $\SLE$ when it is parameterized by quantum length. We will prove the main estimates which will be used to prove our results in Section~\ref{sec:main_lemma}.  In Section~\ref{sec:density}, we will establish a lower bound for the mean density of the LQG area measure in a certain setting.   We then deduce our main result on $\SLE_8$ (Theorem~\ref{thm:sle8_thm}) in Section~\ref{sec:sle8}. Finally, we will prove our results on $\SLE_4$ (Theorems~\ref{thm:sle4_thm} and~\ref{thm:jones_smirnov_sle4}) in Section~\ref{sec:sle4}.

We will now give an overview of the strategy to prove our main theorems.  The description which follows assumes the reader has some basic familiarity with the welding results of \cite{she2016zipper,dms2014mating}; we describe these results in more detail in Section~\ref{sec:preliminaries}.

\begin{figure}[ht!]
\begin{center}
\includegraphics[scale=0.85]{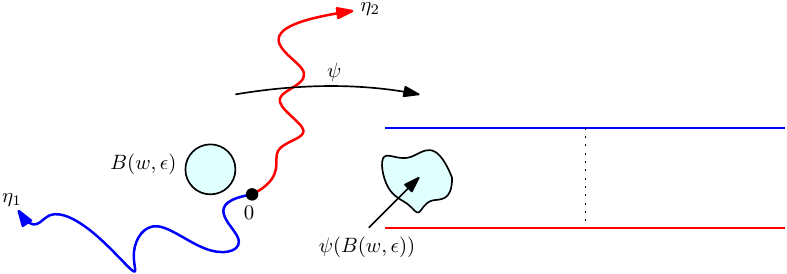}	
\end{center}
\caption{\label{fig:sle4_setup} Illustration of the setup to prove Theorems~\ref{thm:sle4_thm} and~\ref{thm:jones_smirnov_sle4}.}
\end{figure}

We begin with the case $\kappa=4$.  Suppose that $\eta$ is an $\SLE_4$ in $\D$ from $-i$ to $i$ and $\D_L$ is the component of $\D \setminus \eta$ which is to the left of $\eta$.  Fix $\epsilon > 0$.  Suppose that $\eta$ passes through a point $z$ and $w \in \D_L$ has distance of order $\epsilon$ to $z$.  Let $\phi$ be the conformal transformation $\D_L \to \D$ which fixes $-i$, $-1$, and $i$.  Then
\begin{itemize}
\item the derivative of $\phi$ at $w$ is up to constants given by $\epsilon^{-1} \diam(\phi(B(w,\epsilon)))$ and
\item $\diam(\phi(B(w,\epsilon)))$ is proportional to the probability that a Brownian motion starting from $\phi(w)$ first exits $\partial \D$ in the clockwise arc from $-i$ to $i$.
\end{itemize}
By the conformal invariance of Brownian motion, this probability is equal to the probability that a Brownian motion starting from $w$ first exits $\D_L$ in the clockwise arc of $\partial \D_L$ from $-i$ to $i$.  The local behavior of $\eta$ near $z$ is described by a two-sided whole-plane $\SLE_4$.  This is a pair of paths $\eta_1$, $\eta_2$ in $\C$ where the marginal law of $\eta_1$ is a whole-plane $\SLE_4(2)$ from $0$ to $\infty$ and the conditional law of $\eta_2$ given $\eta_1$ is an $\SLE_4$ in $\C \setminus \eta_1$ from $0$ to $\infty$.  It therefore suffices to bound (for each $a  >0$) the conditional probability given $\eta_1$, $\eta_2$ of the event that
\begin{enumerate}
\item[($A$)] There exists $w$ in the component of $\C \setminus (\eta_1 \cup \eta_2)$ which is to the left of $\eta_2$ with distance proportional to $\epsilon$ to $0$ and $\eta_1 \cup \eta_2$, the harmonic measure of each of $\eta_1$ and $\eta_2$ is at least $1/4$, and so that the probability that a Brownian motion starting from $w$ makes it macroscopically far away from $\eta_1$, $\eta_2$ before hitting $\eta_1 \cup \eta_2$ is at most $\exp(-\epsilon^{-a})$.
\end{enumerate}
We will show that $\p[A] = O(\epsilon^{a/2+o(1)})$ as $\epsilon \to 0$.  As the dimension of $\SLE_4$ is $3/2$ \cite{bef2008dimension,rs2005basic}, it takes $O(\epsilon^{-3/2+o(1)})$ balls of radius $\epsilon$ to cover it, which will lead to the choice $a=3$.  (We will also prove a lower bound for the probability of a variant of $A$ with some additional conditions.)

We will determine the asymptotic behavior for $\p[A]$ using the relation between $\SLE$ and LQG.  See Figure~\ref{fig:sle4_setup} for an illustration of the setup.  Namely, if one draws the pair of paths $(\eta_1,\eta_2)$ on top of an independent LQG surface called a weight-$4$ quantum cone, then the quantum surfaces parameterized by the two components of $\C \setminus (\eta_1 \cup \eta_2)$ are independent quantum wedges of weight $2$.  This result was proved for $\kappa < 4$ previously in \cite{dms2014mating}; the purpose of Section~\ref{sec:critical_welding} is to establish the case $\kappa=4$.  If we let $\psi$ be a conformal transformation which takes the component of $\C \setminus (\eta_1 \cup \eta_2)$ which is to the left of $\eta_2$ to the strip $\strip = \R \times (0,\pi)$ sending $0$ to $-\infty$, $\infty$ to $+\infty$, then there is an explicit description for the field which describes the surface parameterized by $\strip$.  Moment estimates for the LQG area measure will imply that $B(w,\epsilon)$ must contain a certain amount of LQG mass which, combined with the explicit description of the field on $\strip$, allows us to bound the probability that $\re(\psi(B(w,\epsilon)))$ is too small.  The event $E$ is roughly equivalent to the event that a Brownian motion starting from $\psi(w)$ first exits $\strip$ to the right of the $y$-axis.

 \begin{figure}[ht!]
\begin{center}
\includegraphics[scale=0.85]{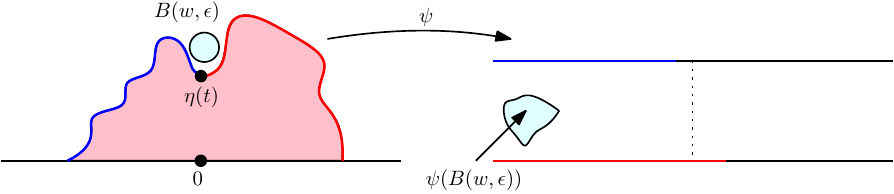}	
\end{center}
\caption{\label{fig:sle8_setup} Illustration of the setup to prove Theorem~\ref{thm:sle8_thm}.}
\end{figure}
 
We now describe the strategy in the case $\kappa=8$.  Suppose that $\eta'$ is an $\SLE_8$ in $\h$ from $0$ to $\infty$ and let $(g_t)$ be its associated Loewner flow.  Let also $f_t = g_t - W_t$ be the centered Loewner flow.  Fix $t \in [0,1]$, $\zeta > 1$, and $\epsilon > 0$.  It follows from \cite{ghm2020kpz} that if we let $\tau_\epsilon = \inf\{s \geq t : |\eta'(s) - \eta'(t)| = \epsilon\}$ then the event that there exists $w$ with $B(w,\epsilon^\zeta) \subseteq \eta'([t,\tau_\epsilon])$ occurs with overwhelming probability as $\epsilon \to 0$ (with $\zeta > 1$ fixed).  To prove the upper bound in Theorem~\ref{thm:sle8_thm}, we want to get a lower bound on $\diam(f_t(B(w,\epsilon^\zeta)))$ as the half-plane capacity of $f_t(\eta'([t,\tau_\epsilon]))$ (roughly speaking) behaves like $\diam(f_t(\eta'([t,\tau_\epsilon)))^2 \geq \diam(f_t(B(w,\epsilon^\zeta)))^2$.  If $t$ is before $\eta'$ first leaves $\D_+ = \D \cap \h$ then giving a lower bound to $\diam(f_t(\eta'([t,\tau_\epsilon])))$ is equivalent to bounding the probability that a Brownian motion starting from $f_t(w)$ first exits $\h$ in $f_t(\partial \h \setminus [-2,2])$.  This probability is comparable to the probability that a Brownian motion starting from $f_t(w)$ first exits $\h$ in $\partial \h \setminus I$ for some fixed compact interval~$I$.  This, in turn, is equivalent to bounding the probability that a Brownian motion starting from $w$ first exits $\h \setminus \eta'([0,t])$ in $\partial \h \setminus [-2,2]$.  We will estimate this probability in a manner similar to in the argument for $\SLE_4$ described above.  See Figure~\ref{fig:sle8_setup} for an illustration of the setup.  Namely, suppose that $\eta'$ is drawn on top of a certain type of LQG surface called a weight-$1$ quantum wedge and that we have reparameterized $\eta'$ by quantum area.  Then for each time $t$, the quantum surface parameterized by $\h \setminus \eta'([0,t])$ has the law of a weight-$1$ quantum wedge.  Fix $\xi > 1$.  We will then estimate (for each $a > 0$) the probability of the event that
 \begin{enumerate}
\item[($A$)] There exists $w$ so that $B(w,\epsilon^\xi)$ is contained in the unbounded component of $\h \setminus \eta'([0,t])$ with distance proportional to $\epsilon$ to $\eta'(t)$ so that the probability that a Brownian motion starting from $w$ exits $\h \setminus \eta'([0,t])$ in $\partial \h \setminus [-2,2]$ is at most $\exp(-\epsilon^{-a})$.
\end{enumerate}
Using the explicit description of the field which describes the surface parameterized by $\h \setminus \eta'([0,t])$ (after mapping to $\strip$), we will show that $\p[A] = O(\epsilon^{a/2+o(1)})$ as $\epsilon \to 0$ (this will be proved simultaneously with the statement we described for $\SLE_4$ above).  Since $\SLE_8$ is space-filling, a segment of it is covered by of order $\epsilon^{-2}$ balls of radius $\epsilon$, which leads us to consider $a=4$. (We will also prove a lower bound for the probability of a variant of $A$ with some additional conditions.)

As the event $A$ is phrased in terms of $\SLE_8$ when parameterized by quantum area and the quantum measure behaves very differently from Lebesgue measure, it takes some care to deduce Theorem~\ref{thm:sle8_thm} from the bound for $\p[A]$ mentioned just above.  We will now explain how this works in more detail.  In order to emphasize the dependence on $t$, we now write $A_t$ for $A$ as above and make the choice $a = 4+2\beta$ for $\beta > 0$ fixed.  Let $\eta_c'$ be $\eta'$ parameterized by capacity, for $z \in \h$ we let $\tau_z = \inf\{t \geq 0 : \eta_c'(t) = z\}$, and let $B_z$ be the event defined in the same way as $A_t$ but with $\eta_c'|_{[0,\tau_z]}$ in place of $\eta'|_{[0,t]}$. Fix $\delta \in (0,1)$ small and $M > 0$ large.  On the event $E= E_1 \cap E_2$ where $E_1 = \{ \eta_c'([0,\delta]) \subseteq e^{-1}\D_+\}$ and $E_2 = \{\mu_h(\D_+) \leq M\}$ (where $h$ is the field which describes the quantum surface on which $\eta_c'$ is drawn and $\mu_h$ is the associated LQG measure) we have that
\[ \int_0^M \one_{A_t} dt \geq \int_{e^{-1}\D_+} \one_{B_z}\one_{z \in \eta_c'([0,\delta])} d\mu_h(z).\]
By Fubini's theorem, the expectation of the left hand side is $O(\epsilon^{2+\beta})$ and therefore so is
\[ \E\!\left[ \int_{e^{-1}\D_+} \one_{B_z} \one_E \one_{z \in \eta_c'([0,\delta])} d\mu_h(z) \right] = \int_{e^{-1}\D_+} \p[B_z,\ E_1,\ z \in \eta_c'([0,\delta])] f(z) dz\]
where $f(z) = \E[ \one_{E_2} \mu_h(dz)]$ (i.e., the density with respect to Lebesgue measure of the measure $X \mapsto \E[ \one_{E_2}\mu_h(X)]$) and $dz$ denotes Lebesgue measure.  We will prove in Section~\ref{sec:density} that $f(z)$ is bounded from below by a constant $c > 0$ in $e^{-1}\D_+$.  (The argument in Section~\ref{sec:density} is somewhat involved because although it is not difficult to control $\E[\mu_h(dz)]$, lower bounding $\E[ \one_{E_2} \mu_h(dz)]$ is much less straightforward.)  In the end, we obtain that the above expectation is at least $c$ times
\[ \E\!\left[ \int_{e^{-1}\D_+} \one_{B_z} \one_{E_1} \one_{z \in \eta_c([0,\delta])} dz \right].\]
It was shown in \cite{ghm2020kpz} that for $\zeta > 1$ fixed in each interval of time that an $\SLE_8$ travels distance $\epsilon$ it with overwhelming probability fills in a ball of radius $\epsilon^\zeta$.  Therefore if $B_z$ occurs \emph{for some} $z \in \D_+$ then the Lebesgue measure for $z \in \D_+$ for which $B_z$ occurs is very likely to be at least $\epsilon^{2 \zeta}$.  Combining the various bounds gives Theorem~\ref{thm:sle8_thm}.

We remark that the same strategy can be used for $\eta \sim \SLE_\kappa$ with $\kappa > 8$, but it does not give the optimal regularity, as in this case, $\eta$ is more regular.

\subsection{Notation}

Let $a,b$ be some quantities. We write $a \lesssim b$ if there is a constant, $C$, independent of any parameters of interest, such that $a \leq Cb$ and $a \gtrsim b$ if $b \lesssim a$. Moreover, we write $a \asymp b$ if $a \lesssim b$ and $a \gtrsim b$. Often we will write explicitly what the implicit constants may depend on. Furthermore, if $a$ and $b$ depend on some parameter $x$, then we write $a = o_x(b)$ if $a/b \rightarrow 0$, as $x \rightarrow 0$ or $\infty$ (which will be clear from context).  Lastly, we write $a = o_x^\infty(b)$ if $a = o_x(b^t)$ for each $t>0$ as $x \rightarrow 0$ (or $t<0$ as $x \rightarrow \infty$).

We will also allow the constants to vary between different occasions. More precisely, we may denote two different constants in the same way, between lines, or even between inequalities on the same line. We shall often write the dependence of the parameter as a subscript, i.e., we often write, say, $C_p$ to emphasize that the constant depends on $p$. We allow the constants to depend on $\gamma$ without explicitly writing it out. In Section~\ref{sec:critical_welding}, however, we will be explicit with the dependence on $\gamma$ as well, as we will consider limits in the parameter $\gamma$.

We let $\Z$ denote the set of integers, $\N$ the set of positive integers, $\R$ the real numbers, $\C$ the complex plane, $\h$ the upper half-plane, $\D$ the unit disk, $\D_+ = \D \cap \h$, $\s^1 = \partial \D$, $\N_0 = \N \cup \{0\}$ and $\strip = \R \times (0,\pi)$. Moreover, we let $\cyl$ denote the infinite cylinder $\R \times [0,2\pi]$, where for each $t \in \R$, we identify the points $\{t \} \times \{0\}$ and $\{ t \} \times \{ 2\pi \}$. For a set of points $A$ in the complex plane, we denote by $A+b$, sometimes $(A+b)$ for clarity, the set $\{z+b: z \in A\}$. In the same way we denote by $bA$ the set of points $\{bz: z \in A \}$.

When we write $\p^z$, we mean the probability measure given by the probability measure $\p$ together with a Brownian motion, independent of everything else, started at $z$.  The exact definition of $\p$ will be clear from the context.

\subsection*{Acknowledgements} K.K.\ was supported by the EPSRC grant EP/L016516/1 for the University of Cambridge CDT (CCA).  J.M.\ and L.S.\ were supported by ERC starting grant 804166 (SPRS). We also thank two anonymous referees for very good comments that helped improve this paper.

\section{Preliminaries}\label{sec:preliminaries}

\subsection{Bessel processes}

Recall that a $d$-dimensional Bessel process, denoted by $\BES^d$, is a solution to the SDE
\begin{align}
\label{eq:Bessel}
	dX_t = \frac{a}{X_t} dt + dB_t,\quad X_0 = x \geq 0,
\end{align}
where $a = (d-1)/2$ and $B$ is a standard Brownian motion  If $d \in (0,2)$, then $X_t$ will a.s.\ hit and be instantaneously reflected at $0$; if $d = 2$, then $X_t$ will a.s.\ not hit $0$, but $\inf_{t \geq 0} X_t^x = 0$ for all $x \geq 0$ (where the superscript $x$ denotes the starting value); and if $d>2$, then a.s.\ $\lim_{t \to \infty} X_t = \infty$.  Moreover, $X_t$ satisfies Brownian scaling: if $(X_t)$ is a $\BES^d$ then so is the process $(r^{-1} X_{r^2 t})$.

Fix $d$, let $a = (d-1)/2$, and let $X$ be a $\BES^d$.  Then the law of $X$ can be obtained as follows.  Start with a Brownian motion $B$ and denote by $N_t$ the local martingale
\begin{align}
\label{eq:BESmtg}
	N_t = \!\left( \frac{B_t}{B_0} \right)^a \exp\!\left( -\frac{a(a-1)}{2} \int_0^t \frac{ds}{B_s^2} \right).
\end{align}

Let $(\CF_t)$ be the $\sigma$-algebra generated by $(B_t)$ and for $\epsilon > 0$ we let $\tau_\epsilon$ be the first time $t$ that $B_t$ hits $\epsilon$. We define the probability measure $\wh{\p}_{a,\epsilon,t}$ on $\CF_{t \wedge \tau_\epsilon}$ by 
\begin{align*}
\frac{d\p_{a,\epsilon,t}}{d\p} = N_{t \wedge \tau_\epsilon}.
\end{align*}
Note that $(N_t)$ solves the SDE
\begin{align*}
dN_{t \wedge \tau_\epsilon} = \frac{a}{B_{t \wedge \tau_\epsilon}} \one_{\{t < \tau_\epsilon \}}N_{t \wedge \tau_\epsilon}dB_t.
\end{align*}
We can easily see that if $s < t$, then $\wh{\p}_{a,\epsilon,t}$ restricted to $\CF_{s \wedge \tau_\epsilon}$ is $\wh{\p}_{a,\epsilon,s}$. Hence we can write $\wh{\p}_{a,\epsilon}$ and so Girsanov's theorem implies that under the measure $\wh{\p}_{a,\epsilon}$, $(B_t)$ solves \eqref{eq:Bessel} up until time $\tau_\epsilon$. Note that this equation does not depend on $\epsilon$ except in the specification of the values of $t$. Therefore we can write $\wh{\p}_{a}$, and let $\epsilon \to 0$ and state that $(B_t)$ solves \eqref{eq:Bessel} up until time $\tau_0$ where $\tau_0 = \lim_{\epsilon \to 0}\tau_\epsilon$ is the first time $t$ that $B_t$ hits $0$.

Another related process is the radial Bessel process with parameter $a \in \R$, which is defined to be the solution $(Y_t)_{t \geq 0}$ to the SDE
\begin{align}
\label{eq:radBES}
	dY_t = a \cot(Y_t) + dB_t,\quad Y_0 = y \in (0,\pi).
\end{align}
In the case that $a > -1/2$, this process has an invariant measure with density $\psi_a(y) = c_a \sin^{2a}(y)$ with respect to Lebesgue measure on $(0,\pi)$, where $c_a$ is a normalizing constant.  We note that $(\pi - Y_t)$ is also a radial Bessel process with parameter $a$ and locally, around $0$, both $(Y_t)$ and $(\pi-Y_t)$ behave like a $\BES^d$ (where $d = 2a+1$).  We shall make this a bit more precise.

For a Brownian motion $B$, we define
\begin{align}\label{eq:radBESmtg}
	M_t = \!\left( \frac{\sin B_t}{\sin B_0} \right)^a \exp\!\left( -\frac{a(a-1)}{2} \int_0^t \frac{ds}{\sin^2 (B_s)} + \frac{a^2}{2} t \right).
\end{align}

Set $\tau_\epsilon = \inf\{t : B_t \leq \epsilon \, \text{or} \, B_t \geq \pi - \epsilon\}$. Then $M_t$ is a local martingale for $t < \tau_\epsilon$ satisfying
\begin{align*}
dM_t = a \cot(B_t) M_t dB_t,\quad M_0 = 1.
\end{align*}
We define a measure $\wh{\p}_\epsilon$ such that if $V$ is a random variable depending only on $B_s, 0\leq s \leq t \wedge \tau_\epsilon$, then 
\begin{align*}
\E_{\wh{\p}_\epsilon}[ V ] = \E[ M_{t \wedge \tau_\epsilon} V ].
\end{align*}
Then, the Girsanov theorem implies that
\begin{align*}
dB_t = a \cot(B_t)dt + d\wt{B}_t,\quad t < \tau_\epsilon,
\end{align*}
where $\wt{B}_t$ is a standard Brownian motion with respect to $\wh{\p}_\epsilon$. We also note that if $Y_t$ solves \eqref{eq:radBES}, then $Y_t$ hits $\{0,\pi\}$ a.s.\ when $a \in (-1/2,1/2)$ while it does not hit $\{0,\pi\}$ a.s.\ when $a \geq 1/2$.

Finally, we note the following about the relationship between ordinary and radial Bessel processes. Let $(X)$ and $(Y)$ be solutions to~\eqref{eq:Bessel} and~\eqref{eq:radBES}, respectively, with $X_0 = Y_0 =x \in (0,\frac{\pi}{2})$. Let $\wt{\mu}_{1,x}$ (resp.\ $\wt{\mu}_{2,x}$) denote the law of $(X_t)_{0 \leq t \leq 1 \wedge T_1}$ (resp.\ $(Y_t)_{0 \leq t \leq 1 \wedge T_2}$), where $T_1 = \inf\{ t \geq 0: X_t = \frac{\pi}{2}\}$ and $T_2 = \inf\{ t \geq 0: Y_t = \frac{\pi}{2}\}$. Then, letting $B$ be a Brownian motion and $T = \inf\{ t \geq 0: B_t = \frac{\pi}{2}\}$, there is a constant $c \geq 1$, depending only on $a$, such that $c^{-1} \leq M_t/N_t \leq c$ for $0 \leq t \leq 1 \wedge T$ and consequently, $\wt{\mu}_{1,x}$ and $\wt{\mu}_{2,x}$ are mutually absolutely continuous.

\subsection{Chordal $\SLE$ and $\SLE_\kappa(\underline{\rho})$ processes}
Chordal $\SLE_\kappa$ processes, or more generally, $\SLE_\kappa(\underline{\rho})$ processes are random fractal curves growing in a simply connected domain, between two marked boundary points.  We recall their definitions. Fix some $\kappa > 0$, let $\underline{x}_L = (x_{l,L}, \dots, x_{1,L})$ and $\underline{x}_R = (x_{1,R},\dots, x_{r,R})$, where $x_{l,L}<\dots<x_{1,L} \leq 0 \leq x_{1,R}<\dots<x_{r,R}$, and let $\underline{\rho}_L = (\rho_{1,L},\dots,\rho_{l,L})$ and $\underline{\rho}_R = (\rho_{1,R},\dots,\rho_{r,R})$, where $\rho_{j,q} \in \R$ for $j \in \N$, $q \in \{L,R\}$. Let $(g_t)$ denote the solution to~\eqref{eqn:chordal_loewner} where $W$ solves the following system of SDEs,
\begin{align}\label{eq:driving_function}
    dW_t &= \sum_{j=1}^l \frac{\rho_{j,L}}{W_t - V_t^{j,L}} dt + \sum_{j=1}^r \frac{\rho_{j,R}}{W_t-V_t^{j,R}} dt + \sqrt{\kappa}dB_t, \\
    dV_t^{j,q} &= \frac{2}{V_t^{j,q}-W_t} dt, \quad V_0^{j,q} = x_{j,q}, \quad j =1,\dots,N_q, \quad q \in \{L,R\}, \nonumber
\end{align}
where $B$ is a Brownian motion and $N_L = l$ and $N_R = r$. For each $t \geq 0$,  $g_t$ defines a conformal map from the simply connected domain $\h_t = \h \setminus K_t$ onto $\h$, where $K_t = \{z \in \h: T_z \leq t\}$ and $T_z = \inf\{ t \geq 0: g_t(z) - W_t = 0 \}$. The solution to~\eqref{eq:driving_function} exists until the continuation threshold, that is, the first time $t>0$ such that $\sum_{j \leq k} \rho_{j,q} \leq -2$ for some $k > 0$, $q \in \{L,R\}$. Geometrically this means that the solution exists until the first time $t > 0$ such that the hull $K_t$ has swallowed force points whose total weight is at most $-2$. Almost surely, there is a continuous curve $\eta: [0,\infty) \rightarrow \overline{\h}$ such that $\h_t$ is the unbounded connected of $\h \setminus \eta([0,t])$ \cite{rs2005basic,ms2016imag1} and the curve $\eta$ is an $\SLE_\kappa(\underline{\rho}_L;\underline{\rho}_R)$ process with force points $(\underline{x}_L;\underline{x}_R)$.  The number $\rho_{j,q}$ is called the weight of the force point $x_{j,q}$ and it is often convenient to write $\underline{\rho} = (\underline{\rho}_L;\underline{\rho}_R)$. The family of conformal maps $(g_t)$ is called the $\SLE_\kappa(\underline{\rho})$ Loewner chain and $(K_t)$ its hulls. An $\SLE_\kappa(\underline{\rho})$ process in a general simply connected domain $D \subsetneq \C$ is defined as the conformal image of an $\SLE_\kappa(\underline{\rho})$ taking $\h$ to $D$ and the start and endpoints, as well as the force points, to their corresponding points in $\partial D$. An $\SLE_\kappa$ process is an $\SLE_\kappa(\underline{0})$ process (where $\underline{0}$ is the zero vector), that is, the curve generating $K_t$ when $W_t = \sqrt{\kappa} B_t$. The laws of $\SLE_\kappa$ and $\SLE_\kappa(\underline{\rho})$ processes are mutually absolutely continuous away from the boundary of the domain.

Let $\eta \sim \SLE_\kappa$.  We now recall the phases of $\SLE_\kappa$ \cite{rs2005basic}.  If $0 < \kappa \leq 4$, then $\eta$ is a.s.\ a simple curve which does not intersect the boundary away from its start and endpoints; if $4 < \kappa < 8$, then $\eta$ a.s.\ intersects; but never traverses, itself as well as the boundary and if $\kappa \geq 8$, then $\eta$ is a.s.\ space-filling. Moreover, the Hausdorff dimension of $\eta$ is $d_\kappa = \min(2,1+\kappa/8)$ \cite{rs2005basic,bef2008dimension}, and the $d_\kappa$-dimensional Minkowski content exists \cite{lr2015minkowski}.

\subsection{Whole-plane $\SLE_\kappa(\rho)$}

Whole-plane $\SLE_{\kappa}(\rho)$ is a variant of $\SLE$ which describes a random growth process $K_t$ where, for each $t \in \R$, $K_t \subseteq \C$ is compact with $\C_t = \C \setminus K_t$ simply connected (viewed as a subset of the Riemann sphere). For each $t$, we let $g_t: \C_t \mapsto \C \setminus \ol{\D}$ be the unique conformal transformation with $g_t(\infty) = \infty$ and $g_t'(\infty) > 0$. Then $g_t$ solves the whole-plane Loewner equation 
\begin{align}\label{eq:wp_Loewner_equation}
    \partial_t g_t(z) = g_t(z)\frac{W_t + g_t(z)}{W_t - g_t(z)},\quad g_0(z) = z
\end{align}
where the pair $(O_t,W_t)_{t \in \R}$ is the unique stationary solution taking values in $\s^1 \times \s^1$ which solves the system of equations 
\begin{align}\label{eq:wpSLEdriv}
	\begin{split}
	dW_t &= -\frac{\kappa}{2}W_t dt + i\sqrt{\kappa}W_t dB_t + \frac{\rho}{2}\Psi(O_t,W_t)dt \\
	dO_t &= \Psi(W_t,O_t)dt
	\end{split}
\end{align}
where $\Psi(w,z) = -z\frac{z+w}{z-w}$ and 
 $B_t$ is a two-sided standard Brownian motion with $B_{0} = 0$. We will be interested in the case when $\rho = 2$ and $\kappa \in (0,4]$.  The following proposition is a special case of Proposition~2.1 of \cite{ms2017imag4} and guarantees that whole-plane $\SLE_\kappa(\rho)$ processes can be constructed.
\begin{proposition}\label{prop:stationary_uniqueness}
For all $\kappa > 0$ and $\rho > -2$, there exists a unique stationary solution to~\eqref{eq:wpSLEdriv} indexed by $\R$.
\end{proposition}

\subsection{Radial $\SLE_\kappa(\rho)$}
A radial $\SLE_\kappa(\rho)$ in $\D$, targeted at $0$ is a random growth process which locally looks like an $\SLE_\kappa$ and grows from some point on  $\partial \D$ towards $0$. We denote the hull of the process at time $t$ by $K_t$ and let $g_t: \D \setminus K_t \to \D$ be the unique conformal map fixing $0$ with $g_t'(0) > 0$.  If we parameterize time so that $\log g_t'(0) = t$ for all $t \geq 0$, then $(g_t)$ solves the SDE~\eqref{eq:wp_Loewner_equation} with $W$ given by the solution to~\eqref{eq:wpSLEdriv}. Whole-plane $\SLE_\kappa(\rho)$ can be seen as a bi-infinite time version of the radial $\SLE_\kappa(\rho)$. 

\subsection{Gaussian free fields}\label{sec:GFF}
Let $D \subseteq \C$ be a simply connected domain with harmonically non-trivial boundary, let $C_0^\infty(D)$ denote the set of smooth functions, compactly supported in $D$, and let $H_0(D)$ denote the Hilbert space closure of $C_0^\infty(D)$ with respect to the Dirichlet inner product $(f,g)_\nabla = \frac{1}{2\pi} \int_D \nabla f(z) \cdot \nabla g(z) dz$. The zero-boundary Gaussian free field (GFF) $h$ is the random distribution defined by $h = \sum_{n \geq 1} \alpha_n \phi_n$ where $(\phi_n)_{n \geq 1}$ is a $(\cdot,\cdot)_\nabla$-orthogonal basis of $H_0(D)$ and $(\alpha_n)_{n \geq 1}$ is a sequence of independent $N(0,1)$-distributed random variables.  Equivalently, one can define the zero-boundary GFF as the centered Gaussian process $h: H_0(D) \rightarrow L^2(\p)$ with covariance kernel given by the Green's function $\greenD{D}$ with Dirichlet boundary conditions on $D$, i.e., $\E[ (h,f) (h,g) ] = \int_{D \times D} f(z) g(w) \greenD{D}(z,w) dz dw$. By the conformal invariance of the Green's function, it is clear that $h$ is conformally invariant as well.

An important property of the GFF is the domain Markov property, that is, if $h$ is a zero-boundary GFF and $U \subseteq D$ is open, then the law of $h$ restricted to $U$, given the values of $h$ at $\partial U$ is that of a zero-boundary GFF $h_U$ on $U$ plus the harmonic extension of its values on $\partial U$ to $U$. Moreover, the zero-boundary part and the harmonic part are independent.  In the same way, a GFF with boundary data $f$ is defined as the sum of a zero-boundary GFF on $D$ and the harmonic extension of $f$ to $D$.

A GFF on $D$ is not a function but rather a random variable in the space of distributions $H^{-1}(D)$, the dual space of $H_0(D)$. Recall that there exists an orthonormal basis $(f_n)_{n \geq 1}$ of $L^2(D)$ consisting of eigenfunctions of the operator $-\tfrac{1}{2\pi} \Delta$ with Dirichlet boundary conditions, which we can (and choose) to order so that the eigenvalues $(\lambda_n)_{n \geq 1}$ form a nondecreasing sequence of positive numbers. Then, $H^{-1}$ is the space of distributions $h$ on $D$ such that $\sum_{n \geq 1} \lambda_n^{-1} (h,f_n)^2$ is finite. Moreover, the functions $(\phi_n)_{n \geq 1}$, defined by $\phi_n = \lambda_n^{-1/2} f_n$ form an orthonormal basis of $H_0(D)$.

A free boundary GFF is defined in the same way, replacing $H_0(D)$ by the closure $H(D)$ with respect to $(\cdot,\cdot)_\nabla$ of the space of functions $f \in C^\infty(D)$ such that $\int_D f dz = 0$. As the free boundary GFF is only defined on test functions with mean zero, it is not canonically defined in a space of distributions, but rather in a space of distributions modulo additive constant. This can be remedied by fixing the value of the field acting on a specific test function hence fixing the value of the additive constant. Equivalently one can define the free boundary GFF in terms of its correlations.  We denote by $\greenN{D}$ the Green's function with Neumann boundary data on $D$ and note that as $\greenN{\h}(z,w) = -\log|z-w| - \log|z-\overline{w}|$. Then the free boundary GFF in $\h$ is the centered Gaussian process with covariance kernel given by $\greenN{\h}$. We define the free boundary GFF in any other simply connected domain $D$ as the conformal image of the one in $\h$.

\begin{remark}[Radial/lateral decomposition of a free boundary GFF]
Let $H_1(\h)$ (resp.\ $H_2(\h)$) be the subspace of functions in $H(\h)$ which are constant (resp.\ have mean zero) on each semicircle centered at zero. Then $H(\h) = H_1(\h) \oplus H_2(\h)$. The projection of a free boundary GFF $h$ onto $H_1(\h)$ is given by the function $h_1(z) = h_{|z|}(0)$ which takes the average value of $h$ on the semicircle of radius $|z|$, centered at zero, and is called the radial part of $h$. The function $h_2 = h-h_1 = h - h_{|\cdot|}(0)$ is the projection of $h$ onto $H_2(\h)$ and is called the lateral part of $h$. It holds that $h_1$ and $h_2$ are independent and while $h_1$ is only defined modulo additive constant, $h_2$ has well-defined values. 
\end{remark}

\begin{remark}\label{rmk:stripdecomposition}
It will often be convenient to consider a distribution on the strip $\strip = \{z: 0 < \im(z) < \pi\}$. In this case, we let $H_1(\strip)$ (resp.\ $H_2(\strip)$) denote the subspace of functions in $H(\strip)$ which are constant (resp.\ have zero average) on each vertical line $\{x \} \times (0,\pi)$. Then $H(\strip) = H_1(\strip) \oplus H_2(\strip)$. Note also that if $f \in H_j(\h)$, then $\wt f(z) = f(e^z)$ belongs to $H_j(\strip)$, $j=1,2$.
\end{remark}

Finally, we introduce the whole-plane GFF.  Let $H_0(\C)$ denote the Hilbert space closure with respect to $(\cdot,\cdot)_\nabla$ of the set of $f \in C_0^\infty(\C)$ satisfying $\int f dz = 0$.  The whole-plane GFF is defined in the same way as the zero-boundary or the free boundary GFF, but with the orthonormal basis in the sum being that of $H_0(\C)$. This defines a random distribution on $\C$ and as in the case of the free boundary GFF, it is only defined modulo additive constant, until we fix a normalization.

\begin{remark}[Radial/lateral decomposition of a whole-plane GFF]
Just like in the case of a free boundary GFF, we can naturally decompose a whole-plane GFF $h$ into its radial and lateral parts. Let $H_1(\C)$ (resp.\ $H_2(\C)$) be the subspace of functions in $H_0(\C)$ which are constant (resp.\ have mean zero) on circles centered at zero. Then $H_0(\C) = H_1(\C) \oplus H_2(\C)$ and the projection of $h$ on $H_1(\C)$ is given by the function $h_1(z) = h_{|z|}(0)$ which takes the average value of $h$ on the circle of radius $|z|$ centered at zero. The projection $h_2 = h-h_{|\cdot|}(0)$ of $h$ onto $H_2(\C)$ is independent of $h_{|\cdot|}(0)$. As in the free boundary case, $h_1$ and $h_2$ are called the radial and lateral parts of $h$, respectively, and while $h_1$ is only defined modulo additive constant, $h_2$ is well-defined.
\end{remark}

If $h$ is a free boundary or whole-plane GFF and $U$ is an open subset of its domain of definition, then the restriction of $h$ to $U$,  conditional on the values of $h$ on $\partial U$, can be decomposed as the sum of a zero-boundary GFF on $U$ and the harmonic extension to $U$ of the values of $h$ on $\partial U$. This is the domain Markov property of the free boundary and whole-plane GFFs. Consequently, we have that in the interior of the domain of definition, the laws of zero-boundary, free boundary and whole-plane GFF are absolutely continuous with respect to each other.

\subsection{Liouville quantum gravity}
\label{sec:LQG}

We now introduce Liouville quantum gravity (LQG) surfaces and their related measures. For more details, we refer the reader to \cite{ds2011lqgkpz,rv2010gmcrevisit,dms2014mating,hp2018critical}. Fix $\gamma \in (0,2]$. A $\gamma$-LQG surface is an equivalence class of pairs $(D,h)$, consisting of a domain $D \subseteq \C$ and a distribution $h \in H_{\loc}^{-1}(D)$\footnote{Recall that $h \in H_{\loc}^{-1}(D)$ if and only if $h|_U \in H^{-1}(U)$ for each $U \subset \subset D$ and that $h_n \to h$ in $H_{\loc}^{-1}(D)$ if and only if $h_n|_U \to h|_U$ in $H^{-1}(D)$ for each $U \subset \subset D$.} on $D$, where two pairs $(D,h)$ and $(\wt{D},\wt{h})$ are equivalent if there is a conformal map $\psi: \wt{D} \rightarrow D$ such that
\begin{align}\label{eq:LQG_coordinate_change}
	\wt{h} = h \circ \psi + Q \log |\psi'|, 
\end{align}
where $Q = Q_\gamma = \gamma/2 + 2/\gamma$.

One can also consider LQG surfaces with marked points $(D,h,x_1,\dots,x_n)$, $(\wt{D},\wt{h},\wt{x}_1,\dots,\wt{x_n})$ and consider them equivalent if they satisfy~\eqref{eq:LQG_coordinate_change} and $\psi(x_j) = \wt{x}_j$ for all $j = 1,\dots,n$.

The surfaces of interest are those where the distribution $h$ locally looks like a GFF. Consider $\gamma \in (0,2)$. For $z \in D$ and $0 < \epsilon < \dist(z,\partial D)$ we denote by $h_\epsilon(z)$ the average value of $h$ on the circle of radius $\epsilon$. The $\gamma$-LQG area (or quantum area) measure with respect to $h$ on $D$ is defined as the weak limit
\begin{align}\label{eq:quantum_area}
	\mu_h(dz) = \lim_{\epsilon \rightarrow 0} \epsilon^{\gamma^2/2} e^{\gamma h_\epsilon(z)} dz,
\end{align}
where $dz$ denotes the two-dimensional Lebesgue measure. Similarly, if $x \in \partial D$ and $h_\epsilon(x)$ denotes the average value of $h$ on $\partial B(x,\epsilon) \cap D$, then on a linear segment of $\partial D$ we define the $\gamma$-LQG length (or quantum length) measure with respect to $h$ as the weak limit
\begin{align}\label{eq:quantum_length}
	\nu_h(dx) = \lim_{\epsilon \rightarrow 0} \epsilon^{\gamma^2/4} e^{\frac{\gamma}{2} h_\epsilon(x)} dx,
\end{align}
where $dx$ denotes Lebesgue measure on $\partial D$.  (One can also associate with $\gamma$-LQG a canonical metric \cite{ms2020metric,ms2016qle2,dddf2020tightness,gm2021metric}, but we will not need this in the present paper.)

If $(D,h)$ and $(\wt{D},\wt{h})$ are related by $\psi$ as in~\eqref{eq:LQG_coordinate_change}, then for all $A \subseteq \wt{D}$, $\mu_{\wt{h}}(A) = \mu_h(\psi(A))$, that is, $\mu_h$ is the push-forward $\mu_{\wt{h}}$ by $\psi$. Similarly, $\nu_h$ is the push-forward of $\nu_{\wt{h}}$ by $\psi$. This justifies the definition of a quantum surface as an equivalence class. Note that this gives a way of measuring the quantum length of boundary segments which are not linear by mapping the domain to, say, $\h$ and measuring the quantum length of the image set. More generally, one can define the boundary length of a curve in $D$ by mapping the complement of the curve to, say, $\h$ and measuring the quantum length of the image of the curve.  It turns out that chordal $\SLE_\kappa(\underline{\rho})$ with $\kappa = \gamma^2$ curves have well-defined quantum length \cite{she2016zipper} and by absolute continuity, so do $\SLE_\kappa$-type curves started in the interior of the domain, such as whole-plane $\SLE_\kappa$ and interior flow lines of Gaussian free fields.

The case $\gamma = 2$ is critical, in the sense that the limiting measures produced from the above renormalization procedure are trivial.  Instead, one can for example, define the critical LQG area measure $\mu^2$ as the limit $(2-\gamma)^{-1} \mu^\gamma$ where $(\mu^\gamma)$ are the subcritical LQG measures; as the derivative of a certain martingale, subject to a limit; or as the limit of the exponential of field approximations with the additional factor $\sqrt{\log(1/\epsilon)}$. Each of the above examples give the same measure, modulo multiplicative constant (see \cite{drsv2014criticalconvergence,drsv2014criticalgmckpz,js2017unique,hrv2018lqgdisk,powell2018critical,aps2019critical}). Analogously to the last example of critical LQG measure, we define the critical LQG boundary measure as the weak limit (see \cite{hp2018critical})
\begin{align}
	\nu_h^2(dx) = \lim_{\epsilon \rightarrow 0} \epsilon\!\left(-\frac{h_\epsilon(x)}{2} + \log(1/\epsilon) \right) e^{h_\epsilon(x)} dx.
\end{align}
As in the case of area measures, we have that $(2-\gamma)^{-1} \nu_h^\gamma \rightarrow 2 \nu_h^2$ weakly in probability, for suitable~$h$.

\begin{remark}\label{rmk:different_fields}
Throughout the paper we will estimate moments of $\mu_h(A)$ for various fields $h$ and sets $A$ and we will repeatedly use \cite[Propositions~3.5,~3.6 and~3.7]{rv2010gmcrevisit} when $h$ is a zero-boundary GFF. However, these results concern fields with covariance kernel of the form $K(z,w) = f(z-w)$ for some positive definite function $f$, which is not the case for the zero-boundary GFF. A zero-boundary GFF in a domain $D$ has covariance kernel given by $G_D(z,w) = - \log|z-w| + g_D(z,w)$, where $g_D(\cdot,w)$ is the harmonic extension of the function $\zeta \mapsto \log| \zeta - w|$ from $\partial D$ to $D$. This presents no issue, as each set which will be considered when using the mentioned results will be such that its distance to the boundary is at least a constant times its diameter, so that $g_D$ is bounded from above and below on said set. Then, by the Kahane convexity inequality (see \cite[Lemma~1]{kahane1985chaos} or \cite[Proposition~6.1]{aru2020gmclens}), the analogous moment bounds follow.
\end{remark}

Next, we define two classes of random surfaces which will be crucial in our analysis. Here we consider $\gamma \in (0,2]$.

\textbf{Quantum cones:} Fix some $\alpha \in (-\infty,Q)$ and define the process $A_t: \R \rightarrow \R$ as $A_t = B_t + \alpha t$, where $B_t$ is a standard two-sided Brownian motion, conditioned so that $B_t - (Q-\alpha)t > 0$ for all $t<0$ and unconditioned for $t>0$ (so that $B|_{[0,\infty)}$ is a standard Brownian motion). An $\alpha$-quantum cone is a doubly marked $\gamma$-LQG surface $(\C,h,0,\infty)$ such that if $h_s(0)$ denotes the average value of $h$ on $\partial B(0,s)$, then the radial part, i.e., the process $t \mapsto h_{e^{-t}}(0)$ has the same law as $A_t$ and the lateral part, $h_2 = h - h_{|\cdot|}(0)$, has the law of the lateral part of a whole-plane GFF. We denote the law of an $\alpha$-quantum cone by $\qconeA{\gamma}{\alpha}$.  We note that if the marked points for a quantum cone parameterized by $\C$ are taken to be $0$ and $\infty$ then the law of the field which defines the surface is specified up to a global rescaling.  The particular embedding that we have just defined is the so-called circle-average embedding.

Sometimes it is convenient to parameterize a quantum cone by the infinite cylinder $\cyl$. An $\alpha$-quantum cone $h$ with $\alpha < Q$ can be defined such that if $X_t$ denotes the average value of $h$ on the vertical line $\{t+iy:y \in (0,2\pi) \}$, then $(X_{-t})_{t \geq 0}$ has the law of $(\wt{B}_{t}-(Q-\alpha)t)_{t \geq 0}$ conditional on $\wt{B}_{t}-(Q-\alpha)t < 0$ for all $t > 0$ for some standard  Brownian motion $\wt{B}$, $(X_t)_{t \geq 0}$ has the law of $(B_{t}+(Q-\alpha)t)_{t \geq 0}$ for some standard Brownian motion $B$ independent of $\wt{B}$, and the law of $h_2 = h - X_{\re(\cdot)}$ is that of the projection of a free boundary GFF onto the space of functions in $H(\cyl)$ which have mean zero on vertical lines.  When we parameterize a quantum cone by $\cyl$ with the marked points at $\mp \infty$, the law of the field which defines the surface is specified up to a horizontal translation of $\cyl$.  The particular choice of horizontal translation we have just described is called the first exit parameterization for a quantum cone.

One might also embed the $\alpha$-quantum cone so that the last time its projection onto $H_1(\cyl)$ hits the value $0$ is at $t = 0$. This is the so-called circle-average embedding or last exit parameterization. One obtains it by performing the coordinate change $z \mapsto \log z$ to an $\alpha$-quantum cone parameterized by $\C$ and it can be sampled as follows. As above, let $X_t$ denote the average value of the field on the line $\{ t+iy: y \in [0,2\pi] \}$. Then for $t>0$, $(X_t)$ has the law of $(B_t + (Q-\alpha)t)_{t \geq 0}$ where $B$ is a standard Brownian motion with $B_0 = 0$, conditioned so that $B_t + (Q-\alpha)t > 0$ for all $t>0$. Moreover, $(X_{-t})_{t \geq 0}$ has the law of $(\wh B_t-(Q-\alpha)t)_{t \geq 0}$ where $\wh B$ is a standard Brownian motion, independent of $B$, with $\wh B_0 = 0$. Furthermore, the projection of $h$ onto $H_2(\cyl)$ is sampled independently of $X$ from the law of the projection of a free boundary GFF on $H_2(\cyl)$. The additive constant is then chosen so that the average on $[0,2\pi i]$ is $0$.

\textbf{Quantum wedges:} Fix $\alpha \in (-\infty,Q)$, and let $A_t$ be as above, but with $B_t$ replaced by $B_{2t}$. An $\alpha$-quantum wedge is the doubly marked $\gamma$-LQG surface $(\h,h,0,\infty)$ such that if $h_s(0)$ denotes the average value of $h$ on $\partial B(0,s) \cap \h$, then the radial part $t \mapsto h_{e^{-t}}(0)$ has the same law as $A_t$ and the lateral part $h_2 = h-h_{|\cdot|}(0)$ has the same law as the lateral part of a free boundary GFF on $\h$. We denote the law of an $\alpha$-quantum wedge by $\qwedgeA{\gamma}{\alpha}$.  We note that if the marked points for a quantum wedge parameterized by $\h$ are taken to be $0$ and $\infty$ then the law of the field which defines the surface is specified up to a global rescaling.  The particular embedding that we have just defined is the so-called circle-average embedding.

Again, an $\alpha$-quantum wedge $h$ with $\alpha < Q$ can be defined on the strip $\strip$ such that if $X_t$ denotes the average of $h$ on the vertical line $\{t+iy:y \in (0,\pi)\}$, then $X_t$ is defined as in the case of quantum cones but with $B_t$ and $\wt{B}_t$ replaced by $B_{2t}$ and $\wt{B}_{2t}$, and $h_2 = h - X_{\re(\cdot)}$ has the law of the projection of a free boundary GFF onto the space of functions in $H(\strip)$ which have mean zero on vertical lines. When we parameterize a quantum wedge by $\strip$ with the marked points at $\pm \infty$, the law of the field which defines the surface is specified up to a horizontal translation of $\strip$.  The particular choice of horizontal translation we have just described is called the first exit parameterization for a quantum wedge.

Just as in the case of a quantum cone, we may parameterize the quantum wedge so that its average on vertical lines process hits the value $0$ for the last time at time $t = 0$.  In this case, the projection onto $H_1(\strip)$ is defined as in the case of the quantum cone, but with $B_t$ and $\wh B_t$ replaced by $B_{2t}$ and $\wh B_{2t}$ and the projection onto $H_2(\strip)$. Again, this is called the circle-average embedding or the last exit parameterization of the quantum wedge.

We note that it is also possible to define $\alpha$-quantum wedges for $\alpha \in [Q,Q+\gamma/2)$.  The case $\alpha \in (Q,Q+\gamma/2)$ was explained in \cite[Definition~4.15]{dms2014mating}. The case $\alpha = Q$ is a bit special and will be used extensively and it will be convenient to parameterize it by the strip $\strip = \{ z: 0 < \im(z) < \pi \}$.

A $Q$-quantum wedge is the doubly marked $\gamma$-LQG surface $(\strip,h,-\infty,+\infty)$, the first exit parameterization of which is such that if $X_t$ denotes the average value of $h$ on the vertical line $\{t+iy: y \in (0,\pi)\}$, then $(-X_{-t/2})_{t \geq 0}$ has the law of a $\BES^3$ with $X_0 = 0$, $(X_{t/2})_{t \geq 0}$ is a standard Brownian motion, and the law of $h_2 = h-X_{\re(\cdot)}$ is that of the projection of a free boundary GFF onto the space of functions in $H(\strip)$ which have mean zero on vertical lines. We remark that in~\cite{gp2019dla}, the above parameterization of a $Q$-quantum wedge is called the circle-average embedding.

Another way to parameterize the space of quantum wedges or cones is by a number called its weight. The weight of an $\alpha$-quantum wedge is defined as $W = \gamma(Q+\gamma/2-\alpha)$ and the weight of an $\alpha$-quantum cone is defined as $W = 2\gamma(Q-\alpha)$. This becomes convenient when cutting or gluing quantum surfaces. The following was proven in \cite{dms2014mating} and states that one can cut a quantum wedge $\CW$ with an independent $\SLE_\kappa(\rho_L;\rho_R)$ process into two quantum wedges $\CW_L$ and $\CW_R$ which are independent and such that the sum of their weights is that of the weight of $\CW$.  The law of a quantum cone (resp.\ wedge) with weight $W$ is denoted by $\qconeW{\gamma}{W}$ (resp.\ $\qwedgeW{\gamma}{W}$). In particular, we note that $\qwedgeA{\gamma}{Q} = \qwedgeW{\gamma}{\gamma^2/2}$.

\begin{theorem}[Theorems~1.2 and~1.4 of \cite{dms2014mating}]
\label{thm:wedge_welding}
Let $\gamma \in (0,2)$. Fix $\rho_L,\rho_R > -2$ and write $W_q = 2+\rho_q$, $q \in \{L,R\}$, and $W = W_L + W_R$. Assume that $W \geq \gamma^2/2$ and let $(\h,h,0,\infty) \sim \qwedgeW{\gamma}{W}$ have the circle-average embedding. Let $\eta$ be an $\SLE_\kappa(\rho_L;\rho_R)$ process from $0$ to $\infty$ in $\h$ with force points at $0^-$, $0^+$ and where $\kappa = \gamma^2$, sampled independently of $h$. Let $\h_L$ (resp.\ $\h_R$) denote the union of the components of $\h \setminus \eta$ which lie to the left (resp.\ right) of $\eta$ and let $\CW_q$ be the quantum surface defined by $(\h_q, h|_{\h_q},0,\infty)$ for $q \in \{L,R\}$. Then $\CW_L \sim \qwedgeW{\gamma}{W_L}$, $\CW_R \sim \qwedgeW{\gamma}{W_R}$, they are independent of each other, and for each $t>0$, the quantum length of the left side of $\eta([0,t])$ coincides with that of the right side. Moreover, the pair $(h,\eta)$ is a.s.\ determined by $\CW_L$ and $\CW_R$.
\end{theorem}

Similarly, the following theorem on cutting a quantum cone with a whole-plane $\SLE_\kappa(\rho)$ was proved in \cite{dms2014mating}.

\begin{theorem}[Theorem~1.5 of \cite{dms2014mating}]
\label{thm:cone_welding}
Let $\gamma \in (0,2)$. Fix $\rho>-2$ and write $W = 2+\rho$. Let $(\C,h,0,\infty) \sim \qconeW{\gamma}{W}$ have the circle-average embedding and let $\eta$ be a whole-plane $\SLE_\kappa(\rho)$ process from $0$ to $\infty$ in $\C$, with $\kappa = \gamma^2$, sampled independently of $h$. Then the quantum surface $\CW = ( \C \setminus \eta, h|_{\C \setminus \eta},0,\infty) \sim \qwedgeW{\gamma}{W}$ and for each $t>0$, the quantum lengths of the two sides of $\eta([0,t])$ coincide. Moreover, $\CW$ a.s.\ determines $(h,\eta)$, modulo rotation about $0$.
\end{theorem}

Next, we mention a result on cutting a critical LQG surface with an $\SLE_4$ process.

\begin{theorem}[Theorem~1.2 of \cite{hp2018critical}]
\label{thm:critical_wedge_welding}
Let $(\h,h,0,\infty)\sim \qwedgeW{2}{4}$ and let $\eta$ be an $\SLE_4$ process from $0$ to $\infty$ in $\h$ and independent of $h$. Let $\h_L$ (resp.\ $\h_R$) denote the component of $\h \setminus \eta$ to the left (resp.\ right) of $\eta$ and let $\CW_q$ be the quantum surface defined by $(\h,h|_{\h_q},0,\infty)$ for $q \in \{L,R\}$. Then $\CW_L$ and $\CW_R$ are independent, have the law $\qwedgeW{2}{2}$ and their boundary lengths along $\eta$ agree.
\end{theorem}

Finally, we record the following result on exploring a certain quantum wedge by an $\SLE_8$ process.

\begin{theorem}[Theorems~1.5,~1.9 and~1.11 of \cite{dms2014mating}]
\label{thm:sle8thm}
Let $(\h,h,0,\infty) \sim \qwedgeW{\sqrt{2}}{1}$ and let $\eta'$ be an $\SLE_8$ process from $0$ to $\infty$ in $\h$ sampled independently of $h$ and then reparameterized by quantum area (so that $\mu_h(\eta'([0,t])) = t$ for all $t \geq 0$).  For each $t \geq 0$, the quantum surface parameterized by $\h \setminus \eta'([0,t])$ and marked by $\eta'(t)$ and $\infty$ has law $\qwedgeW{\sqrt{2}}{1}$.  Moreover, if we let $L_t$ (resp.\ $R_t$) denote the difference of the quantum length of the part of $\eta'([0,t]) \cap \h$ to the left (resp.\ right) of $\eta'(t)$ and the part of $\eta'([0,t]) \cap \partial \h$ to the left (resp.\ right) of $0$, then $L$ and $R$ are independent standard Brownian motions.  Finally, $(L,R)$ a.s.\ determine $(h, \eta')$.
\end{theorem}

\subsection{Imaginary geometry}
We will often use the so-called imaginary geometry coupling of $\SLE$ and a GFF. In doing this, we consider the formal vector field $e^{i h/\chi}$, where $h$ is some Gaussian free field with deterministic boundary data and $\chi = 2/\sqrt{\kappa} - \sqrt{\kappa}/2$ and view $\SLE_\kappa(\underline{\rho})$ processes, for $\kappa \in (0,4)$, as its flow lines, where the weights and locations of the force points depend on the boundary data of $h$. In this coupling, the flow lines are a.s.\ determined by $h$ and they satisfy some convenient rules of interaction. 

Similarly, one can couple $\SLE_4(\underline{\rho})$ curves with $h$, in which case we call them level lines. They follow similar interaction rules, but this case is easier to handle, since there is no winding term (since $\chi \rightarrow 0$ as $\kappa \rightarrow 4$).

The imaginary geometry coupling was developed mainly in \cite{she2016zipper, ms2016imag1, ms2017imag4} and the level line coupling in \cite{ss2013continuum} as well as in \cite{ww2017levellines}. We will not review the imaginary geometry or level line coupling in any detail, but rather refer to \cite[Section~2.2]{ms2016imag2} or \cite[Section~2.2]{mw2017slepaths} for an introduction.

We will often refer to $\SLE_\kappa(\underline{\rho})$ processes, $\kappa \in (0,4)$ (resp.\ $\kappa =4$) as flow lines (resp.\ level lines) and define flow lines (resp.\ level lines) of other angles (resp.\ heights) with the previous ones in mind, without explicitly stating the boundary data, as it will be clear from the context.

\section{Critical welding for quantum cones}
\label{sec:critical_welding}

In this section we prove the following theorem.  In order to read the remainder of the article, one only needs the statement of the theorem so the proof can be skipped on a first reading.

\begin{theorem}\label{thm:critical_cone_welding}
Suppose that $(\cyl,h,-\infty,+\infty) \sim \qconeW{2}{4}$ has the first exit parameterization.  Let $\eta_1,\eta_2$ be a pair of curves, independent of $h$, such that $\eta_1$ is a whole-plane $\SLE_{4}(2)$ process in $\cyl$, from $-\infty$ to $+\infty$, and the conditional law of~$\eta_2$ given~$\eta_1$ is that of an $\SLE_4$ process in $\cyl \setminus \eta_1$, from $-\infty$ to $+\infty$. Furthermore, let $D_1,D_2$ denote the two connected components of $\cyl \setminus (\eta_1 \cup \eta_2)$. Then the quantum surfaces $(D_1,h|_{D_1},-\infty,+\infty)$ and $(D_2,h|_{D_2},-\infty,+\infty)$ are independent and have law $\qwedgeW{2}{2}$.
\end{theorem}

We are going to deduce Theorem~\ref{thm:critical_cone_welding} from Theorem~\ref{thm:cone_welding} by taking a limit as $\kappa \uparrow 4$.  There are two sources of subtlety involved in taking this limit.  First, while it is obvious that one has the Carath\'{e}odory convergence of a whole-plane $\SLE_\kappa(2)$ run up to a \emph{fixed} and finite time to a whole-plane $\SLE_4(2)$ as $\kappa \uparrow 4$ due to the local uniform convergence of the corresponding Loewner driving function, there is some work involved in ruling out pathological behavior of the curve near $\infty$ in order to take a limit when the curve has been run for the full amount of time.  Addressing this issue is the main focus of Section~\ref{subsec:cara}.  Second, while it is not difficult to see that the field which describes the first exit parameterization using $\strip$ of the two quantum surfaces to the left and right of $\eta_1 \cup \eta_2$ converges to that of $\qwedgeW{2}{2}$, one has to rule out degenerate behavior for the horizontal translation which could come as one takes a limit of the conformal maps as above.  Addressing this issue is the main focus of Section~\ref{subsec:quantum_conv}.
\subsection{Carath\'{e}odory Convergence}
\label{subsec:cara}

This subsection is dedicated to proving the following proposition.
\begin{proposition}\label{prop:Cara}
For $\kappa \in (0,4]$, let $\eta^\kappa$ be a whole-plane $\SLE_\kappa(2)$ process in $\C$ from $0$ to $\infty$. Let $\phi^\kappa$ be the unique conformal transformation mapping $\C \setminus \eta^\kappa$ to $\h$ fixing the origin and $\infty$ and such that $\im(\phi^\kappa(i)) = 1$. Then the law of $(\phi^\kappa)^{-1}$ converges weakly to the law of $(\phi^4)^{-1}$ with respect to the topology of local uniform convergence of conformal maps on $\h$ as $\kappa \uparrow 4$.
\end{proposition}

As we mentioned above, the challenge in proving Proposition~\ref{prop:Cara} is that $\phi^\kappa$ is the uniformizing map for the \emph{whole} curve rather than just the curve drawn up to a finite time.  In particular, it does not immediately follow from the local uniform convergence of the driving function.  In order to circumvent this issue, we will use the fact that a radial $\SLE_\kappa(2)$ curve (which is defined on an infinite time interval using the capacity parameterization) can be realized as a chordal $\SLE_\kappa(\kappa-8)$ curve with an interior force point run up to a finite time which, in turn, can be realized as a reverse $\SLE_\kappa(\kappa)$ with force point located at $0$ (see Lemma~\ref{lem:SLEcomp}).  We emphasize that the reason for doing this is to transfer the problem into a matter analyzing an $\SLE$-type curve at a finite time.

We now recall the definition of a reverse $\SLE_\kappa(\kappa)$ process with force point at $0$. Consider the reverse Loewner differential equation
\begin{align}\label{eq:rLDE}
\partial_t \wh g_t(z) = -\frac{2}{\wh g_t(z)-\wh W_t}, \quad \wh g_0(z) = z.
\end{align}
For each $t \geq 0$, $\wh g_t$ is the unique conformal map from $\h$ to $\h \setminus K_t$, satisfying $\wh g_t(z) - z \to 0$ as $z \to \infty$, for some family of compact $\h$-hulls $(K_t)$. A reverse $\SLE_\kappa(\kappa)$ process with force point at $0$ is the random curve generating the growth process $K_t$ when we solve~\eqref{eq:rLDE} with $(\wh W_t,\wh V_t)$ given by
\begin{align*}
d \wh W_t &= \sqrt{\kappa} dB_t - \re \!\left(\frac{\kappa}{\wh V_t-\wh W_t}\right)\! dt, \\
d \wh V_t &= -\frac{2}{\wh V_t - \wh W_t} dt, \quad \wh V_0 = 0,
\end{align*}
where $B$ is a standard Brownian motion. The centered reverse Loewner flow of the point $z \in \overline{\h}$ is given by $\wh f_t(z) = \wh g_t(z) - \wh W_t$. Much of the following discussion is carried out in the proof of \cite[Proposition~3.8]{dms2014mating}. We let $\wh Q_t = \wh V_t - \wh W_t$ denote the reverse $\SLE_\kappa(\kappa)$ flow of the force point and set $\wh \theta_t = \arg(\wh Q_t)$ and $\wh J_t = \log \im(\wh Q_t)$. Note that $\wh \theta_t$ and $\wh J_t$ determine $\wh Q_t$. Then, making the random time change $ds(t) = |\wh Q_t|^{-2} dt$, we have that $d J_s = d \wh J_{t(s)} = 2 ds$ and $\theta_s = \wh \theta_{t(s)}$ satisfies
\begin{align}\label{eq:SDEX}
	d \theta_s = \sqrt{\kappa} \sin( \theta_s) dB_s + 2 \sin(2 \theta_s) ds.
\end{align}
That is, under this time change, the only randomness of $\wh Q$ which remains is that of $\theta$ which we shall hence study. Moreover, as described in the proof of \cite[Proposition~3.8]{dms2014mating},~\eqref{eq:SDEX} has a reversible, invariant measure $\mu(du) = \wt c_\kappa \sin^{\alpha_\kappa}(u) du$, where $\alpha_\kappa = 8/\kappa - 2 > -1$ and $\wt c_\kappa$ is a normalizing constant. If we define the time change $d\beta(s) = \kappa \sin^2(\theta_s) ds$, then $( \theta_{\beta^{-1}(s)})_{s \geq 0}$ solves~\eqref{eq:radBES} with $a = 4/\kappa$, that is, $(\theta_{\beta^{-1}(s)})_{s \geq 0}$ is a radial Bessel process. Consequently, in order to understand $\wh Q$, it is natural to study the convergence of radial Bessel processes. This is done in Lemma~\ref{lem:radBESkappato4}. Moreover, we shall also relate the driving pair of whole-plane $\SLE_\kappa(2)$ processes to radial Bessel processes, and deduce its convergence as $\kappa \to 4$ from that of the radial Bessel process. This is the content of Proposition~\ref{prop:statconv}.

After proving Lemma~\ref{lem:radBESkappato4} and Proposition~\ref{prop:statconv} we turn to analyzing the stationary solution to~\eqref{eq:SDEX}, in Lemma~\ref{lem:thetainint}. This is then used together with Lemma~\ref{lem:radBESkappato4} to deduce the convergence of reverse $\SLE_\kappa(\kappa)$ as $\kappa \to 4$, in Lemma~\ref{lem:sle_converges}. Next, we recall \cite[Proposition~3.10]{dms2014mating}, which describes the relationship between reverse $\SLE_\kappa(\kappa)$ and $\SLE_\kappa(\kappa-8)$. With this at hand, we describe how to relate the convergence results from reverse $\SLE_\kappa(\kappa)$ (and hence $\SLE_\kappa(\kappa-8)$) to radial $\SLE_\kappa(2)$, before finally proving Proposition~\ref{prop:Cara}.

\begin{lemma}\label{lem:radBESkappato4}
For each $\kappa \in (0,4]$, let $Y^\kappa$ be a solution to~\eqref{eq:radBES} with $a = 4/\kappa$ and denote by $\mu^\kappa$ the law of $Y^\kappa$, when $Y_0^\kappa$ is sampled from the invariant distribution of $Y^\kappa$. Then, as $\kappa \uparrow 4$, $\mu^\kappa$ converges to $\mu^4$ weakly with respect to the topology of local uniform convergence on $C([0,\infty))$.
\end{lemma}
\begin{proof}
Fix $T>0$. We let $\mu_T^\kappa$ (resp.\ $\mu_{y,T}^\kappa$) denote the law of $(Y_t^\kappa)_{0 \leq t \leq T}$, where $Y_0^\kappa$ is sampled from the invariant distribution of $Y^\kappa$ (resp.\ $Y_0^\kappa = y \in (0,\pi)$). We shall begin by proving the weak convergence of the measures $(\mu_{y,T}^\kappa)$ and then deduce the result for the measures $(\mu_T^\kappa)$.

Assume for now that $Y_0^\kappa = y$ for all $\kappa$. We couple the processes $Y^\kappa$ as follows. Let $B$ be a Brownian motion with $B_0 = y$ and let $M^\kappa$ be given by~\eqref{eq:radBESmtg} with $a = 4/\kappa$.  Let also $\tau$ be the first time that $B$ exits the interval $(0,\pi)$. Then, under the measure $\p_\kappa^\dagger$ defined by $\p_\kappa^\dagger[A] = \E[\one_A M_{T \wedge \tau}^\kappa]$, the law of $(B_t)_{0 \leq t \leq T}$ is $\mu_{y,T}^\kappa$.  Note that $M_{T\wedge \tau}^\kappa \rightarrow M_{T\wedge \tau}^4$ a.s.\ as $\kappa \rightarrow 4$ and that $\E[M_{T\wedge \tau}^\kappa] = 1$ for all $\kappa$. By Fatou's lemma, 
\begin{align*}
	2 = \E\!\left[ \liminf_{\kappa \rightarrow 4} M_{T\wedge \tau}^\kappa + M_{T\wedge \tau}^4 - |M_{T\wedge \tau}^\kappa-M_{T\wedge \tau}^4|\right] &\leq \liminf_{\kappa \rightarrow 4} \E[ M_{T\wedge \tau}^\kappa + M_{T\wedge \tau}^4 - |M_{T\wedge \tau}^\kappa-M_{T\wedge \tau}| ] \\
	&= 2 - \limsup_{\kappa \rightarrow 4} \E[|M_{T\wedge \tau}^\kappa-M_{T\wedge \tau}^4|],
\end{align*}
that is, $\limsup_{\kappa \rightarrow 4} \E[|M_{T\wedge \tau}^\kappa - M_{T\wedge \tau}^4|] \leq 0$ and hence $M_{T\wedge \tau}^\kappa \rightarrow M_{T\wedge \tau}^4$ in $L^1(\p)$.

Let $F: C([0,T]) \rightarrow [0,\infty)$ be continuous and bounded. Then,
\begin{align*}
	\int_{C([0,T])} F(f) d\mu_{y,T}^\kappa(f) = \E[F((B_t)_{0 \leq t \leq T}) M_{T\wedge \tau}^\kappa] \rightarrow \E[F((B_t)_{0 \leq t \leq T}) M_{T\wedge \tau}^4] = \int_{C([0,T])} F(f) d\mu_{y,T}^4(f),
\end{align*}
proving the weak convergence of the measures $(\mu_{y,T}^\kappa)$.

Recalling that the invariant density of the process $Y^\kappa$ is given by $\psi_{4/\kappa}(y) = c_{4/\kappa} \sin^{8/\kappa}(y)$, we have
\begin{align*}
	\int_{C([0,T])} F(f) d\mu_T^\kappa(f) = \int_0^\pi \!\left( \int_{C([0,T])} F(f) d\mu_{y,T}^\kappa(f) \right) \psi_{4/\kappa}(y) dy
\end{align*}
and noting that $c_{4/\kappa} \leq c_4$ for all $\kappa \in [1,4]$, we have that
\begin{align*}
\left| \!\left(\int_{C([0,T])} F(f) d\mu_{y,T}^\kappa(f) \right) \psi_{4/\kappa}(y) \right| \leq c_4 \| F \|_\infty,
\end{align*}
uniformly in $\kappa \in [1,4]$.  By the weak convergence of $\mu_{y,T}^\kappa$ to $\mu_{y,T}^4$,  and the dominated convergence theorem, we have that $\mu_T^\kappa$ converges weakly to $\mu_T^4$. Since this holds for each $T>0$, the result follows.
\end{proof}

The next result is concerned with the convergence of the driving pair of a whole-plane $\SLE_\kappa(2)$ as $\kappa \uparrow 4$ and relies on Lemma~\ref{lem:radBESkappato4}. This is one of the main ingredients in the proof of Proposition~\ref{prop:Cara}.

\begin{proposition}\label{prop:statconv}
For $\kappa \in (0,4]$, let $(O^\kappa,W^\kappa)$ be the unique stationary solution to~\eqref{eq:wpSLEdriv} with $\rho = 2$ (recall Proposition~\ref{prop:stationary_uniqueness}). Then as $\kappa \uparrow 4$, we have that the law of  $(O^\kappa,W^\kappa)$ converges to the law of $(O^4,W^4)$ weakly with respect to the topology of local uniform convergence on continuous functions $\R \to \s^1 \times \s^1$.
\end{proposition}

Before proving Proposition~\ref{prop:statconv}, we note the following. Assume that $\kappa \in (0,4]$, $\rho = 2$ and let $ \vartheta_t = \arg(W_t) - \arg(O_t)$, where $(O_t,W_t)$ is the driving pair of a whole-plane $\SLE_\kappa(2)$ process. Then $\vartheta$ solves the SDE
\begin{align}\label{eq:thetaSDE}
	d\vartheta_{t} = \sqrt{\kappa}dB_{t} + 2 \cot \!\left (\frac{\vartheta_{t}}{2} \right )dt,\quad \vartheta_0 \in [0,2\pi).
\end{align}
We note that a solution to~\eqref{eq:thetaSDE} can be obtained by starting with a solution $Y$ to~\eqref{eq:radBES} with $a = 4/\kappa$ and then setting $\vartheta_t = 2Y_{\kappa t/4}$. Consequently, Lemma~\ref{lem:radBESkappato4} implies that if $\vartheta^\kappa$ is a solution to~\eqref{eq:thetaSDE}, started from its invariant distribution (which has density $\psi_{4/\kappa}(y/2)$, $y \in (0,2\pi)$), then as $\kappa \uparrow 4$, the law of $\vartheta^\kappa$ converges to that of $\vartheta^4$ weakly with respect to the topology of local uniform convergence on $C([0,\infty))$.

We are now ready to prove Proposition~\ref{prop:statconv}.
\begin{proof}[Proof of Proposition~\ref{prop:statconv}]
Let $(O^\kappa,W^\kappa)$ be the unique stationary solution to~\eqref{eq:wpSLEdriv} indexed by $\R$ with $\rho = 2$ and $\kappa \in (0,4]$. Let also $\nu_\kappa$ be its unique invariant measure.  One can find $\nu_\kappa$ explicitly as follows.  Let $(O_t,W_t)_{t \geq 0}$ be a solution to~\eqref{eq:wpSLEdriv}, with $(O_0,W_0)$ sampled from $\nu_\kappa$. For any fixed $u \in [0,2\pi)$, the process $(O_t e^{iu},W_t e^{iu})_{t \geq 0}$ is also a stationary solution to~\eqref{eq:wpSLEdriv} and by uniqueness, $(O_0,W_0) \overset{d}{=} (O_0 e^{iu},W_0 e^{iu})$.  In particular, the laws of $W_0$ and $W_0 e^{iu}$ are the same for all $u \in [0,2\pi)$ and consequently $W_0$ is uniformly distributed on the circle $\s^1$. Noting also that $\vartheta_t = \arg(W_t) - \arg(O_t)$ is the unique stationary solution to~\eqref{eq:thetaSDE}, we can sample a pair $(O_0,W_0)$ from $\nu_\kappa$ by first sampling $W_0$ from the uniform distribution on $\partial \D$ and then sampling $\vartheta_0$ independently of $W_0$ from the invariant measure of~\eqref{eq:thetaSDE} and setting $O_0 = W_0 e^{-i\vartheta_0}$.

By the above construction, together with Lemma~\ref{lem:radBESkappato4}, it is clear that $\nu_\kappa$ converges weakly to $\nu_4$ as $\kappa \uparrow 4$. Note that if $(O_t^\kappa,W_t^\kappa)_{t \geq 0}$ is started from $(e^{iy},e^{ix})$ for some fixed $x,y \in \R$ with $x-y \in [0,2\pi) $, then the law of $(O^{\kappa}_{t},W^{\kappa}_{t})_{t \geq 0}$ on $C([0,\infty) \rightarrow \s^1 \times \s^1)$ converges to the law of $(O_t^4,W_t^4)_{t \geq 0}$ as $\kappa \uparrow 4$. This can be seen by combining the weak convergence of the law of $\vartheta_t^\kappa = \arg(W_t^\kappa) - \arg(O_t^\kappa)$ as $\kappa \uparrow 4$ with the way that $(O_t^\kappa,W_t^\kappa)_{t \geq 0}$ can be obtained from $\vartheta^\kappa$ as already explained. Thus the weak convergence of $\nu_\kappa$ to $\nu_4$ as $\kappa \uparrow 4$ implies that $(O_t^\kappa,W_t^\kappa)_{t \geq 0}$ converges weakly to $(O_t^\kappa,W_t^\kappa)_{t \geq 0}$ as $\kappa \uparrow 4$ when $(O_0^\kappa,W_0^\kappa)$ has law  given by $\nu_\kappa$. Since the processes $(O_{-t}^\kappa,W_{-t}^\kappa)_{t \geq 0}$ and $(O_t^\kappa, W_t^\kappa)_{t \geq 0}$ have the same law on $C([0,\infty) \rightarrow \s^1 \times \s^1)$, we obtain that the law of $(O_t^\kappa,W_t^\kappa)_{t \in \R}$ on $C(\R \rightarrow \s^1 \times \s^1)$ converges weakly to the law of $(O_t^4,W_t^4)_{t \in \R}$ as $\kappa \uparrow 4$.
\end{proof}

Next, we analyze the stationary solution to~\eqref{eq:SDEX}.

\begin{lemma}\label{lem:thetainint}
Let $(\wt{\theta}_s)_{s \in \R}$ be a solution to~\eqref{eq:SDEX} for $\kappa \in [1,4]$, started from its invariant distribution. Then there exists a universal constant $c_1 \in (0,1)$, uniform in $\kappa \in [1,4]$, such that for all $c \in (0,1]$ with probability at least $1 - c_1 c$ it holds that 
\begin{align*}
	\wt{\theta}_s \in [ce^{-\frac{3|s|}{2}},\pi - ce^{-\frac{3|s|}{2}}],\quad \text{for all} \quad s \in \R.
\end{align*}
\end{lemma}
\begin{proof}
The strategy of this proof is rather standard: we note that a time-changed version of $\theta$ behaves like a Bessel process of dimension $1+8/\kappa$ when close to the endpoints $0$ and $\pi$ and then deduce the necessary inequalities using the Brownian scaling of the Bessel process.

Fix $c \in (0,1)$. We let $(X_s) = (X_s^\kappa)$ be a $\BES^{1+8/\kappa}$ with $X_0 = 1$ and define the random times $T_1 = \inf \{t \geq 0 : X_t \leq e^{-1} \}$ and $T_2 = \inf \{t \geq 0 : X_t \geq e \}$ and $T = T_1 \wedge T_2$. We recall that $(X_s)$ can be obtained by weighting the law of a Brownian motion started at $1$ with the local martingale $N$, defined in~\eqref{eq:BESmtg} and that this induces a coupling of Bessel processes with different dimensions corresponding to different $\kappa$.  We will omit $\kappa$ in the notation throughout, but keep in mind that it is a central parameter for the processes. For a Brownian motion $B$ with $B_0 = 1$, and analogously defined times $T_1^B$, $T_2^B$ and $T^B$, the event $\{8e^2 \leq T_2^B < T_1^B\}$ has positive probability and since $N_{8e^2 \wedge T^B} \geq e^{-4-48e^4}$, uniformly in $\kappa \in [1,4]$ and hence there is a $0 < p\leq 1$ such that $\p[8e^2 \leq T_2 < T_1] \geq p$ for all $\kappa \in [1,4]$. Since $(X_s)$ and $(r^{-1} X_{r^2 s})$ have the same law,  this implies that if $X$ starts at $c e^{-n}$, then with probability at least $p$, $X$ hits $c e^{-n+1}$ before $c e^{-n-1}$ and after at least $8c^2 e^{-2(n-1)}$ units of time.

Let $\beta$ be the time-change above and let $\wh{\theta}_u = \wt{\theta}_{\beta^{-1}(u)}$.  Fix some $n \in \N$ and let $n_0 = n + \lceil \log(\pi/2c) \rceil$.  By the absolute continuity of $\wh{\theta}$ and $X$, there is a constant $c_2$, uniform in $\kappa \in [1,4]$, such that if $\wh{\theta}$ starts at $c e^{-n+j}$, then with probability at least $c_2 p$, it hits $c e^{-n+j+1}$ before $c e^{-n+j-1}$ and does so after at least $8c^2 e^{2(1+j-n)}$ units of time. Assume that we are working on this event and let $\wt{\theta}_0 = c e^{-n+j}$, $T_j = \inf \{t \geq 0: \wt{\theta}_t \notin [c e^{-n+j-1}, ce^{-n+j+1}] \}$ and $r_j = \beta(T_j)$. Then $r_j \leq T_j \kappa \sin^2(c e^{-n+j+1}) \leq T_j 2 \kappa c^2 e^{2(j+1-n)}$ and $r_j \geq 8c^2 e^{2(j+1-n)}$, which together imply that $T_j \geq 1$. Thus, there is some constant $q > 0$, uniform in $\kappa \in [1,4]$, such that if $\wt{\theta}_0 = c e^{-n+j}$, then with probability at least $q$, $\wt{\theta}$ hits $c e^{-n+j+1}$ before $c e^{-n+j-1}$ and after at least $1$ unit of time and by the Markov property of $\wt{\theta}$, with probability at least $1-(1-q)^n$, there is at least one $j \in \{0,1,\dots,n_0\}$ such that this occurs. Moreover, by absolute continuity, there is a constant $c_3 > 0$, independent of $\kappa$ and $n$, such that $\p[ \wt{\theta} \, \text{hits} \, ce^{-\frac{3n}{2}} \, \text{before it hits} \, \frac{\pi}{2} \,|\, \wt{\theta}_0 = ce^{-n}] \leq c_3 \p[ X \, \text{hits} \, ce^{-\frac{3n}{2}} \, \text{before it hits} \, \frac{\pi}{2} \,|\, X_0 = ce^{-n}] \lesssim e^{-\frac{n}{2}}$, where the implicit constant can be taken independent of $n,\kappa,c$.

From the previous paragraph and Markov property of $\theta$,  if $\sigma$ is a stopping time such that $\wt{\theta}_{\sigma} = e^{-n+j}$ and $\sigma^{\pm} = \inf\{t \geq \sigma : \wt{\theta}_t = e^{-n+j\pm1}\}$,  then $\p[\sigma+1 \leq \sigma^+ < \sigma^- \giv \,\,  \CF_{\sigma}] \geq q$.  Applying this result to $\sigma = \sigma_s = \inf\{t \geq 0 : \wt{\theta}_t \in \{e^{-n+s},e^{-n-s}\}\}$,  $0 \leq s \leq \lfloor \frac{n}{2}\rfloor$,  we obtain that there is some constant $\wt{p} \in (0,1)$, independent of $c$ and $\kappa$, such that for each $n \in \N$,
\begin{align}\label{eq:thetainint}
	\p[ \wt{\theta}_s \in [c e^{-\frac{3n}{2}},\pi-ce^{-\frac{3n}{2}}] \  \text{for all} \ s \in [0,1] \,|\, \wt{\theta}_0 = c e^{-n}] \geq 1-\wt{p}^n.
\end{align}
so by the Markov property and the fact that $(\theta_t)$ and $(\pi-\wt{\theta}_t)$ have the same law, ~\eqref{eq:thetainint} holds uniformly in $\wt{\theta}_0 \in [ce^{-n},\pi-ce^{-n}]$. Comparing with the invariant density of~\eqref{eq:SDEX}, we have for all $s \geq 0$ that $\p[ \wt{\theta}_s \in [ce^{-s},\pi-ce^{-s}]] \geq 1-c_1 ce^{-(\frac{8}{\kappa}-1)s}$, where $c_1$ is independent of $c$ and $\kappa$ and by the above, this implies that the probability that $\wt{\theta}_n \in [ce^{-n},\pi-ce^{-n}]$ and $\wt{\theta}_{n+s} \in [ce^{-\frac{3n}{2}},\pi-e^{-\frac{3n}{2}}]$ for all $s \in [0,1]$ is at least $1-c\wh{p}^n$ for some $\wh{p}\in (0,1)$. Summing over $n$ gives the result for $s \geq 0$. For $s \leq 0$ the claim follows since $(\wt{\theta}_s)_{s \geq 0}$ and $(\wt{\theta}_s)_{s \leq 0}$ have the same law.
\end{proof}

We now turn to proving the convergence of reverse $\SLE_\kappa(\kappa)$.  The discussion in the beginning of the section gives that if $\wt{\theta}_s$ is a stationary solution to~\eqref{eq:SDEX}, then letting $Q_s^{\kappa} = e^{2s}(\cot(\wt{\theta}_s) + i)$ and making the random time change $\wh{s}(t) = \inf\{r \in \R: \int_{-\infty}^r |Q_u^{\kappa}|^2 du \geq t \}$, the process $\wh Q_t^{\kappa} = Q_{\wh s(t)}^{\kappa}$ has the law of the flow of the force point of a centered reverse $\SLE_\kappa(\kappa)$ with force point $0$.  Let also $\wh{W}_t^{\kappa}$ be the driving function of the latter reverse $\SLE_{\kappa}(\kappa)$ process. (Note that on the event that $\wt{\theta}_s \in \!\big[c e^{-\frac{3|s|}{2}},\pi-c e^{-\frac{3|s|}{2}} \!\big]$ for all $s \in \R$, which has probability at least $1-c_1c$, we have that $\int_{-\infty}^0 |Q_u^{\kappa}|^2 du < \infty$ and $\int_{-\infty}^\infty |Q_u^{\kappa}|^2 du = \infty$. Consequently, sending $c$ to $0$, the process $\wh{Q}_t^{\kappa} = Q_{\wh{s}(t)}^{\kappa}$ is well-defined.)

\begin{lemma}\label{lem:sle_converges}
For $\kappa \in (0,4]$, let $\wt{\theta}^\kappa$ be the stationary solution to~\eqref{eq:SDEX}, let $Q^\kappa$, $\wh{Q}^\kappa$ and $\wh{W}^\kappa$ be as above and define $\wh{T}^\kappa = \int_{-\infty}^0 |Q_u^\kappa|^2 du = \inf\{ t \geq 0: \im(\wh{Q}_t^\kappa)=1\}$. Then the law of $(\wh{T}^\kappa,\wh{Q}^\kappa,\wh{W}^\kappa)$ converges to the law of $(\wh{T}^4,\wh{Q}^4,\wh{W}^4)$, weakly on $\R \times C([0,\infty)) \times C([0,\infty))$ as $\kappa \uparrow 4$. 
\end{lemma}
\begin{proof}
Recall that under the time-change $\beta^\kappa(r) = \kappa \int_0^r \sin^2(\wt{\theta}_u) du$, $(\wt{\theta}_{(\beta^\kappa)^{-1}(r)})_{r \geq 0}$ solves~\eqref{eq:radBES} and is started from the invariant distribution of ~\eqref{eq:SDEX}.  By the argument at the end of the proof of Lemma~\ref{lem:radBESkappato4}, we obtain that the law of $\wt{\theta}_{(\beta^\kappa)^{-1}(r)}^\kappa$ converges to the law of $\wt{\theta}_{(\beta^4)^{-1}(r)}^4$ weakly with respect to the topology of local uniform convergence as $\kappa \uparrow 4$ and by the form of $\beta^\kappa$,  the law of $\wt{\theta}^\kappa$ converges to the law of $\wt{\theta}^4$ in the same way.

For $n \in \N$, $\kappa \in [1,4]$ we let $F_{n,\kappa}$ be the event that $\wt\theta_{-n}^\kappa \in [e^{-n},\pi - e^{-n}]$ and that $\wt\theta_{-n-t}^\kappa \in [e^{-\frac{3n}{2}},\pi - e^{-\frac{3n}{2}}]$, for all $t \in [0,1]$. The proof of Lemma~\ref{lem:thetainint} implies that there exists a universal constant $p \in (0,1)$ independent of $n$ and $\kappa$ such that $\p\!\left[ F_{n,\kappa}^c \right] \leq p^n$ and so $\p\!\left[ \cup_{n=N}^{\infty}F_{n,\kappa}^c \right] \leq \frac{p^N}{1-p}$, for all $N \in \N$. We observe that there exists a finite universal constant $A$, independent of $N$ and $\kappa$, such that if $\cap_{n=N}^{\infty}F_{n,\kappa}$ holds then we have that $\int_{-\infty}^{-N}|Q_s^\kappa|^2 ds \leq Ae^{-N}$ and $|Q_s^\kappa|^2 \leq Ae^s$,  for all $s\leq -N$. Together with the weak convergence of the laws of $\wt{\theta}^{\kappa}$, we obtain that the law of  $(\wh{T}^\kappa,\wh{s}^\kappa,Q^\kappa,\wt{\theta}^\kappa)$ on $\R \times C([0,\infty)) \times M \times C(\R)$ converges weakly to the law of $(\wh{T}^4,\wh{s}^4,Q^4,\wt{\theta}^4)$ as $\kappa \uparrow 4$, where $M$ is the space of functions $f \in C(\R)$ such that $\lim_{x \to -\infty}f(x) = 0$, endowed with the norm $||f||_{M} = \sup_{t \leq 0}|f(t)| + \sum_{n=0}^{\infty} 2^{-n} \min(1, \sup_{t \in [n,n+1]}|f(t)| )$, making $M$ a complete metric space. By combining the above observations, we obtain that the law of  $(\wh{T}^\kappa,\wh{Q}^\kappa,\wh{W}^\kappa)$ on $\R \times C([0,\infty)) \times C([0,\infty))$ converges weakly to the law of $(\wh{T}^4,\wh{Q}^4,\wh{W}^4)$ as $\kappa \uparrow 4$. 
\end{proof}

The next lemma  was proven in \cite{dms2014mating} and gives us a way to compare $\SLE_\kappa(\kappa-8)$ and reverse $\SLE_\kappa(\kappa)$ processes.

\begin{lemma}[Proposition~3.10 of \cite{dms2014mating}]\label{lem:SLEcomp}
Fix $\kappa \in (0,4]$, $r > 0$ and sample $z_r = x + ir$ so that $\arg(z_r)$ is given by the stationary solution to~\eqref{eq:SDEX} and then sample a forward $\SLE_\kappa(\kappa - 8)$ process with force point located at $z_r$ and let $(f_t)$ denote its centered Loewner chain. Then the evolution of $f_t(z_r)$ considered in the time-interval from $0$ to $\inf \{t \geq 0 : \im(f_t(z_r)) = 0 \}$ has the same law as $\wh{Q}_{\wh{T}_r - t}$ for $t \in [0,\wh{T}_r]$, where $\wh{Q}_t$ is the evolution of the force point of a reverse $\SLE_\kappa(\kappa)$ with force point located at $0$ and $\wh{T}_r = \inf \{t \geq 0 : \im(\wh{Q}_t) = r \}$. 
\end{lemma}

Now, fixing $\kappa \in (0,4]$, Lemma~\ref{lem:SLEcomp} gives us another way of sampling a quadruple $(z_\kappa,T^\kappa,Z^\kappa,W^\kappa)$, consisting of a point $z_\kappa=\cot(\theta^\kappa) + i$, where $\theta^\kappa$ is sampled from the invariant measure of~\eqref{eq:SDEX}, the centered flow $Z^\kappa$ of the force point of an $\SLE_\kappa(\kappa-8)$ with force point $z_\kappa$, the driving function of said $\SLE_\kappa(\kappa-8)$ and the time $T^\kappa$ at which the flow of force point hits the real line.  The new way to do it is by sampling a pair $(\wh{Q}^\kappa,\wh{W}^\kappa)$ consisting of the evolution of $0$ and the driving function of a reverse $\SLE_\kappa(\kappa)$ with force point at $0$, letting $\wh{T}^\kappa = \inf\{t \geq 0: \im(\wh{Q}_t^\kappa) = 1 \}$ and then writing $\wh{Q}_t^{\kappa,\circ} = \wh{Q}_{\wh{T}^\kappa-t}^\kappa$, $\wh{W}_t^{\kappa,\circ} = \wh{W}_{\wh{T}^\kappa-t}^\kappa - \wh{W}_{\wh{T}^\kappa}^\kappa$ for $0 \leq t \leq \wh{T}^\kappa$.  Then $(\wh{Q}_0^{\kappa,\circ},\wh{T}^\kappa,\wh{Q}^{\kappa,\circ},\wh{W}^{\kappa,\circ})$ has the same law as $(z_\kappa,T^\kappa,Z^\kappa,W^\kappa)$, and since the former converges weakly to $(\wh{Q}_0^{4,\circ},\wh{T}^4,\wh{Q}^{4,\circ},\wh{W}^{4,\circ})$, as $\kappa \uparrow 4$, we have that the latter converges weakly to $(z_4,T^4,Z^4,W^4)$.

We now turn to the convergence of radial $\SLE_\kappa(2)$. Suppose that we have the above setup and fix a sequence $(\kappa_n)_{n \in \N}$ in $(0,4)$ such that $\kappa_n \uparrow 4$ as $n \to \infty$. By the Skorokhod representation theorem, we can find a coupling such that $(z_{\kappa_n},T^{\kappa_n},Z^{\kappa_n},W^{\kappa_n})$ converges to $(z_4,T^4,Z^4,W^4)$ a.s. Let $(K_t^{\kappa_n})_{0 \leq t \leq T^{\kappa_n}}$ and $(g_t^{\kappa_n})_{0 \leq t \leq T^{\kappa_n}}$ be the family of compact $\h$-hulls and the Loewner chain, respectively, corresponding to the driving function $W_t^{\kappa_n}$. We also consider the conformal transformation $\psi^{\kappa_n} : \h \rightarrow \D$ such that $\psi^{\kappa_n}(0) = -i$ and $\psi^{\kappa_n}(z_{\kappa_n}) = 0$. In particular, we have that $\psi^{\kappa_n}(z) = \frac{c_{\kappa_n}z +1 }{d_{\kappa_n}z + i}$ where $c_{\kappa_n} = i\sin^{2}(\theta^{\kappa_n}) - \sin(\theta^{\kappa_n})\cos(\theta^{\kappa_n})$ and $d_{\kappa_n} = \sin^{2}(\theta^{\kappa_n}) - i\sin(\theta^{\kappa_n})\cos(\theta^{\kappa_n})$. For $0 \leq t \leq T^{\kappa_n}$, we let $\phi_t^{\kappa_n} : \h \rightarrow \D$ be the conformal transformation such that $F_t^{\kappa_n} = \phi_t^{\kappa_n} \circ g_t^{\kappa_n} \circ (\psi^{\kappa_n})^{-1} : \D \setminus \psi^{\kappa_n}(K_t^{\kappa_n}) \rightarrow \D$ satisfies $(F_t^{\kappa_n})'(0) > 0$ and $F_t^{\kappa_n}(0) = 0$. Set $z_t^{\kappa_n} = g_t^{\kappa_n}(z_{\kappa_n})$, $y_t^{\kappa_n} = \im(z_t^{\kappa_n})$, and note that there exists a unique $\lambda_t^{\kappa_n} \in \partial \D$ such that $\phi_t^{\kappa_n}(z) = \lambda_t^{\kappa_n}\frac{z-z_t^{\kappa_n}}{z-\bar{z}_t^{\kappa_n}}$. We consider the time-change, $t^{\kappa_n}$, as the solution to
\begin{align*}
	s = \int_{0}^{t^{\kappa_n}(s)}\frac{4(y^{\kappa_n}_{s})^2}{|z_s^{\kappa_n} - W_s^{\kappa_n}|^{4}}ds,
\end{align*}
and we set $\wt{g}_s^{\kappa_n} = F_{t^{\kappa_n}(s)}^{\kappa_n}$, $\wt{W}_s^{\kappa_n} = \phi_{t^{\kappa_n}(s)}^{\kappa_n} \!\left (W_{t^{\kappa_n}(s)}^{\kappa_n}\right )$ and $\wt{K}_s^{\kappa_n} = \psi^{\kappa_n}\!\left (K_{t^{\kappa_n}(s)}^{\kappa_n}\right )$. Then by \cite{sw2005coordinate} we know that $(\wt K_s^{\kappa_n})_{s \geq 0}$ has the law of the hulls of a radial $\SLE_{\kappa_n}(2)$ process, $\wt{\eta}^{\kappa_n}$, in $\D$ starting from $-i$ and targeted at $0$ with the force point located at $\psi^{\kappa_n}(\infty)$ and we parameterize $\wt{\eta}^{\kappa_n}$ by $\log$ of the conformal radius as seen from $0$. The driving pair is given by $(\wt{\lambda}_s^{\kappa_n},\wt{W}_s^{\kappa_n})_{s \geq 0}$ where $\wt{\lambda}_s^{\kappa_n} = \lambda_{t^{\kappa_n}(s)}^{\kappa_n}$.

With probability one, $W^{\kappa_n}$ converges to $W^4$ locally uniformly as $n \to \infty$, and thus, by the argument used to prove  \cite[Proposition~4.47]{law2005slebook}, we obtain that a.s.,
\begin{align*}
	\sup_{0 \leq s \leq t,\ z \in K}\!\left | (g_s^{\kappa_n})^{-1}(z) - (g_s^4)^{-1}(z) \right | \rightarrow 0,
\end{align*}
as $n \to \infty$, for all $0 \leq t < T^4$ and every compact set $K \subseteq \h$. By the choice of the coupling and the above transformation formulas, we also a.s.\  have that
\begin{align*}
	\sup_{0 \leq s \leq t,\,\,z \in K} \!\left |(\wt{g}_s^{\kappa_n})^{-1}(z) - (\wt{g}_s^4)^{-1}(z) \right | \to 0,
\end{align*}
as $n \to \infty$ for each compact subset $K \subseteq \D$  and all $t \geq 0$. For all $s \geq 0$, let $\wt{f}_s^{\kappa_n} : \D \setminus \wt{\eta}^{\kappa_n}([0,s]) \rightarrow \D$ be the conformal transformation such that $\wt{f}_s^{\kappa_n}(0) = 0$ and $\wt{f}_s^{\kappa_n}(\wt{\eta}^{\kappa_n}(s)) = -i$. We fix $e^{i\vartheta}\ \in \partial{\D} \setminus \{-i\}$ and set
\begin{align*}
	\wt{S}^{\kappa_n} = 1 \wedge \inf \{ s \geq 0 :   \wt{\lambda}_s^{\kappa_n}  (\wt{W}_s^{\kappa_n} )^{-1} = e^{i(\vartheta+\pi/2)} \}.
\end{align*}
Note that $\p[\wt{S}^{\kappa_n} < 1] > 0$, for each $n \in \N$ and that conditional on the event $\{ \wt{S}^{\kappa_n} < 1 \}$, $\wh{\eta}^{\kappa_n} = \wt{f}_{\wt{S}^{\kappa_n}}^{\kappa_n}\!\left ( \wt{\eta}^{\kappa_n}|_{[\wt{S}^{\kappa_n},\infty)} \right)$ has the law of a radial $\SLE_{\kappa_n}(2)$ in $\D$ starting from $-i$ and targeted at $0$ with the force point located at $e^{i\vartheta}$.  Combining with the above convergences,  we obtain that for each compact set $K \subseteq \D$, it a.s.\ holds that
\begin{align*}
	\sup_{z \in K}\!\left |(\wt{f}_{\wt{S}^{\kappa_n}}^{\kappa_n})^{-1}(z) - (\wt{f}_{\wt{S}^4}^4)^{-1}(z)\right | \to 0 \quad \text{as} \quad n \rightarrow \infty.
\end{align*}

For $ \kappa \in (0,4]$, let
$\alpha^\kappa$ be the square root in $\h$ of $-i ({\lambda^\kappa})^{-1} \in \partial \D$ with $\lambda^\kappa = \psi^\kappa(\infty)$ and consider the conformal transformations $B^\kappa : \h \rightarrow \D$ with $B^\kappa(z) = \lambda^\kappa \frac{z - \alpha^\kappa}{z - \Bar{\alpha}^\kappa}$ and $\phi^\kappa = B^\kappa \circ f_{T^\kappa}^\kappa \circ (\psi^\kappa)^{-1}$,  where $f_{T^\kappa}^\kappa(z) = g_{T^\kappa}^\kappa(z) - W_{T^\kappa}^\kappa$. Note that $\phi^\kappa$ maps $\D \setminus \wt{\eta}^\kappa$ bijectively onto $\D$ with $\phi^\kappa(0) = -i$, $\phi^\kappa(\psi^\kappa(\infty)) = \psi^\kappa(\infty)$. We also consider the conformal transformation $F^\kappa : \D \setminus \wh{\eta}^\kappa \rightarrow \D$ with $F^\kappa = \phi^\kappa \circ (\wt{f}_{\wt{S}^\kappa}^\kappa)^{-1}$.  Then $F^\kappa(0) = -i$ and $F^\kappa(e^{i\vartheta}) = \psi^\kappa(\infty)$ on the event $\{\wt{S}^{\kappa} < 1\}$. The above also implies that $(F^{\kappa_n})^{-1}$ converges locally uniformly to $(F^4)^{-1}$ as $n \rightarrow \infty$ a.s.

With the above in mind, we prove Proposition~\ref{prop:Cara}.

\begin{proof}[Proof of Proposition~\ref{prop:Cara}]
Suppose that we have the above setup and let $(\kappa_n)_{n \in \N}$ be a sequence as above. Fix $t \in \R$, let $(O^\kappa,W^\kappa)$ be the driving pair and $(g^\kappa)$ be the Loewner chain corresponding to $\eta^\kappa$. We construct a coupling of the curves $\eta^{\kappa_n}$ as follows.  By Proposition~\ref{prop:statconv} and the Skorokhod representation theorem, we can find a coupling such that a.s.\ $(O^{\kappa_n},W^{\kappa_n})$ converges to $(O^4,W^4)$ uniformly on compact subsets of $(-\infty,t]$.

Let $\eta^{\kappa_n}|_{(-\infty,t]}$ be the curves with driving pair $(O^{\kappa_n}|_{(-\infty,t]},W^{\kappa_n}|_{(-\infty,t]})$. Let $\wt{S}^\kappa$ be defined as above, but with $e^{i\vartheta}$ replaced by $-i(O_t^\kappa)^{-1}W_t^\kappa$. Note that since $(O_t^{\kappa_n},W_t^{\kappa_n}) \rightarrow (O_t^4,W_t^4)$ as $n \to \infty$ a.s., each convergence above remains valid in the case when restricting to $(-\infty,t]$. Then we set $\eta^{\kappa_n}(t + s) = (g_t^{\kappa_n})^{-1}\!\left( -\frac{iW_t^{\kappa_n}}{\wh{\eta}^{\kappa_n}(s)} \right)$ for $s \geq 0$ where $\wh{\eta}^{\kappa_n}$ is independent of $\eta^{\kappa_n}|_{(-\infty,t]}$. The conformal Markov property of whole-plane $\SLE_{\kappa}(2)$ implies that $\eta^{\kappa_n}$, conditioned on the event $\{ \wt{S}^{\kappa_n} < 1\}$, has the law of a whole-plane  $\SLE_{\kappa_n}(2)$ from $0$ to $\infty$ conditioned on the event $\{\wt{S}^{\kappa_n} < 1 \}$.

Let $G_t^\kappa : \D \rightarrow \D$ be the unique conformal transformation such that $G_t^\kappa(\psi^\kappa(\infty)) = -iW_t^\kappa(O_t^\kappa)^{-1}$, $G_t^\kappa(i) = i$ and $G_t^\kappa(-i) = -i$. The above observations imply that $(F^{\kappa_n})^{-1} \circ (G_t^{\kappa_n})^{-1}$ converges locally uniformly to $(F^{4})^{-1} \circ (G^{4}_{t})^{-1}$ as $n \to \infty$ a.s. We also consider the conformal transformation $\psi_t^\kappa : \C \setminus ( \overline{B(0,e^t)} \cup f_t^\kappa(\eta^{\kappa}|_{[t,\infty)} ) ) \rightarrow \C \setminus \overline{B(0,e^t)}$, where 
\begin{align*}
	\psi_t^\kappa(z) = \frac{-ie^t W_t^\kappa}{(G_t^\kappa \circ F^\kappa)\!\left( -\frac{ie^t W_t^\kappa}{z}\right)},
\end{align*}
and $f_t^\kappa: \C \setminus \eta^\kappa((-\infty,t]) \rightarrow \C \setminus \overline{B(0,e^t)}$ is the conformal transformation given by $f_t^\kappa(z) = e^t g_t^\kappa(z)$. Moreover we set $\theta^\kappa = -\arg(O_t^\kappa(W_t^\kappa)^{-1})/2$ and we consider the conformal map $\omega_t^\kappa : \C \setminus \overline{B(0,e^t)} \rightarrow \h$ given by 
\begin{align*}
	\omega_t^\kappa(z) = \frac{\frac{e^t e^{-i\theta_t^\kappa}}{z} - \frac{e^{i\theta_t^\kappa}}{W_t^\kappa}}{\frac{e^t}{z} - \frac{1}{W_t^\kappa}}.
\end{align*}
Since $(O^{\kappa_n},W^{\kappa_n})$ converges locally uniformly on $(-\infty,t]$ to $(O^4,W^4)$ as $n \to \infty$ a.s., deterministic estimates for the whole-plane Loewner equation imply that $(f_t^{\kappa_n})^{-1}$ converges to $(f_t^4)^{-1}$ locally uniformly as $n \to \infty$. Therefore we obtain that the inverse of $\omega_t^{\kappa_n} \circ \psi_t^{\kappa_n} \circ f_t^{\kappa_n}$ converges to the inverse of $\omega_t^4 \circ \psi_t^4 \circ f_t^4$ locally uniformly as $n \to \infty$ a.s. 

Since $\eta^\kappa((-\infty,t])$ has conformal radius $e^t$, by applying \cite[Proposition~9.11]{dms2014mating}, we obtain that $\eta^\kappa((-\infty,t]) \subseteq B(0,4e^t)$. Also, $\eta^{\kappa}|_{[t,\infty)}$ has the law of a radial $\SLE_\kappa(2)$ in $\C \setminus \eta^\kappa((-\infty,t])$ and since radial $\SLE_\kappa(2)$ does not hit fixed points for $\kappa \in (0,4]$ , we obtain that $i \notin (\bigcup_{n \in \N} \eta^{\kappa_n}((-\infty,t])) \cup \eta^4((-\infty,t])$, a.s.\ for all $t$ sufficiently small. Therefore the maps $\phi^\kappa$ are well-defined and $\im((\omega_t^{\kappa_n}\circ \psi_t^{\kappa_n}\circ f_t^{\kappa_n})(i)) \rightarrow \im((\omega_t^4\circ \psi_t^4 \circ f_t^4)(i))$ as $n \to \infty$. Finally, we observe that 
\begin{align*}
	\phi^\kappa(z) = \frac{(\omega_t^\kappa \circ \psi_t^\kappa \circ f_t^\kappa)(z)}{\im((\omega_t^\kappa\circ \psi_t^\kappa \circ f_t^\kappa)(i))}
\end{align*}
and so $(\phi^{\kappa_n})^{-1}$ converges to $(\phi^4)^{-1}$ locally uniformly as $n \to \infty$ a.s. Since the sequence $(\kappa_n)$ was arbitrary, the proof is done.
\end{proof}

\subsection{Convergence of quantum surfaces}
\label{subsec:quantum_conv}
Most of this subsection is dedicated to proving the following.
\begin{proposition}\label{prop:cutcone}
Suppose that $(\cyl,h,-\infty,+\infty) \sim \qconeW{2}{4}$ has the first exit parameterization and let $\eta$ be an independent whole-plane $\SLE_4(2)$ in $\cyl$, from $-\infty$ to $+\infty$. Then the quantum surface parameterized by $\cyl \setminus \eta$ has law $\qwedgeW{2}{4}$.
\end{proposition}
Proposition~\ref{prop:cutcone} is the key result needed in proving Theorem~\ref{thm:critical_cone_welding} and before proving it, we show how to use it to deduce Theorem~\ref{thm:critical_cone_welding}.
\begin{proof}[Proof of Theorem~\ref{thm:critical_cone_welding}]
Let $\phi_*: \cyl \setminus \eta_1 \rightarrow \strip$ be a conformal map with $\phi_*(-\infty) = -\infty$, $\phi_*(+\infty) = +\infty$ and set $h^* = h \circ \phi_*^{-1} + 2 \log |(\phi_*^{-1})'|$ and $T^* = \inf\{ t \in \R: h_1^*(t) = 0 \}$, where $h_1^*$ is the projection of $h^*$ on $H_1(\strip)$. Then, writing $\wh{h}^* = h^*(\cdot + T^*)$, Proposition~\ref{prop:cutcone} implies that $(\strip,\wh{h}^*,-\infty,+\infty) \sim \qwedgeW{2}{4}$. We note that conditional on $\eta_1$,  $\wh{\eta} = \phi_*(\eta_2)$ has the law of an $\SLE_4$ from $-\infty$ to $+\infty$ in $\strip$. Moreover, by the scale invariance of $\SLE_4$, the law of $\wh{\eta} + T^*$ is that of an $\SLE_4$ process as well.

For $j=1,2$, we set $\wh{D}_j = \phi_*(D_j)+T^*$ and let $\psi_j: \wh{D}_j \rightarrow \strip$ be a conformal transformation such that $\psi_j(-\infty) = -\infty$ and $\psi_j(+\infty) = +\infty$. By Theorem~\ref{thm:critical_wedge_welding}, $(\wh{D}_1, \wh h^*|_{\wh D_1},-\infty,+\infty)$ and $(\wh{D}_2, \wh h^*|_{\wh D_2},-\infty,+\infty)$ are independent quantum surfaces, and hence the same is true for $(D_1,h|_{D_1},-\infty,+\infty)$ and $(D_2,h|_{D_2},-\infty,+\infty)$ as well.  Furthermore, if we for $j=1,2$ parameterize $\wh{h}^* \circ \psi_j^{-1} + 2\log|(\psi_j^{-1})'|$ to have the first exit parameterization, then $(\strip,\wh{h}^* \circ \psi_j^{-1} + 2\log|(\psi_j^{-1})'|,-\infty,+\infty) \sim \qwedgeW{2}{2}$.  Finally, letting $\varphi_j:D_j \rightarrow \strip$ be the conformal map defined by $\varphi_j(z) = \psi_j(\phi_*(z) - T^*)$, we have that $\wh{h}^* \circ \psi_j^{-1} + 2\log|(\psi_j^{-1})'| = h \circ \varphi_j^{-1} + 2\log|( \varphi_j^{-1})'|$. Thus, the proof is complete.
\end{proof}

We now turn to the proof of Proposition~\ref{prop:cutcone}.  For $\gamma \in (0,2]$, we let $(\cyl,h^\gamma,-\infty,+\infty) \sim \qconeW{\gamma}{4}$ and $\eta^\gamma$ be a whole-plane $\SLE_\kappa(2)$ from $-\infty$ to $+\infty$ in $\cyl$, where $\kappa = \gamma^2$.  For each $\gamma \in (0,2)$, we know that the surface parameterized by $\cyl \setminus \eta^\gamma$ has law $\qwedgeW{\gamma}{4}$. In proving Proposition~\ref{prop:cutcone}, we shall take a limit as $\gamma \to 2$ and show that the limiting object has law $\qwedgeW{\gamma}{4}$.
The main hurdle in proving it is to make sure that the conformal maps embedding the quantum surfaces into $\strip$ do not degenerate as $\gamma \to 2$. We shall be more precise.  Let $\phi_\gamma$ be the unique conformal map from $\cyl \setminus \eta^\gamma$ to $\strip$, fixing $\pm \infty$ and such that $\re(\phi_\gamma(\frac{i\pi}{2}))=0$.  Then, by the Skorokhod representation theorem and Proposition~\ref{prop:Cara} (as well as composing and precomposing with $\log z$), one can find a coupling in which $\phi_\gamma \to \phi_2$ locally uniformly as $\gamma \to 2$. However, the maps $\phi_\gamma$ do not specify a certain embedding of the resulting quantum surfaces. So if we want each of them to have the first exit parameterization, we need to translate them properly. Here is where the problem can arise: we must prove that the first hitting time of $0$ for the average on vertical lines process does not go to infinity as $\gamma \to 2$. 

Throughout this section we fix some $\phi \in C_0^\infty(\strip)$ with $\int_{\strip}\phi(z) dz = 1$.  One of the key ingredients in the proof of Proposition~\ref{prop:cutcone} is that for fixed $c_0$, we can find $c_1>0$ such that with sufficiently high probability, $\{x \in \R: |(h^\gamma(\cdot+x),\phi)| \in [-c_0,c_0]\} \subseteq [-c_1,c_1]$, see Lemma~\ref{lem:translbound}. This gives the tightness of the law of the first hitting time of $0$ for the average on vertical lines processes corresponding to the surfaces embedded into $\strip$ by the maps $\phi_\gamma$ above and ensures that the conformal maps embedding the surfaces according to the first exit parameterization do not degenerate.

We begin by proving a variance bound for a free boundary GFF $h$ (Lemma~\ref{lem:varbound}), which we will use to prove that $|(h_2,\phi(\cdot-x))|$ and $|(h_1^\gamma,\phi(\cdot-x))-\alpha_\gamma x|$ (where $\alpha_\gamma = 4/\gamma-\gamma/2$) do not grow too quickly in $x$ (Lemmas~\ref{lem:latbound} and~\ref{lem:wedgebound}). This immediately gives Lemma~\ref{lem:translbound}.
\begin{lemma}\label{lem:varbound}
Let $h$ be a free boundary GFF on the strip $\strip$ with the additive constant fixed so that its average over $\{0\}  \times [0,\pi]$ is equal to $0$. Then there exists a finite constant $C_\phi > 0$, depending only on $\phi$, such that 
\begin{align*}
	\var[(h,\phi(\cdot + x)) - (h,\phi(\cdot + y))] \leq C_\phi|x-y|, \quad \text{for all} \quad x,y \in \R.
\end{align*}
\end{lemma}
\begin{proof}
Set $K = \supp(\phi)$, $r = \dist(K,\partial\strip)/2 > 0$ and $K_1 = \{z \in \strip: \dist(z,K) \leq r \}$. Fix $x,y \in \R$ and suppose first that $|x-y|\geq r$. Let $h_1$ (resp.\ $h_2$) be the projection of $h$ to $H_1(\strip)$ (resp.\ $H_2(\strip)$) and note that $\var[(h,\phi)] < \infty$. Note also that $(h_{1}(\frac{u}{2}))_{u \in \R}$ has the law of a two-sided standard Brownian motion with $h_1(0) = 0$. Then we have that 
$\var[(h,\phi(\cdot + x )) - (h,\phi(\cdot + y )) ]$ is at most 
\begin{align*}
	3(\var [(h,\phi(\cdot + x)) - h_1(x)] + \var[h_1(x) - h_1(y)] + \var[(h,\phi(\cdot + y )) - h_1(y)])
\end{align*}
Note that, since $h_2$ is translation invariant, the first and third terms of the above sum are both equal to $\var[(h,\phi)]$. Moreover, $\var[h_1(x) - h_1(y)] = 2|x-y|$ and thus, setting $C = 2+ 2\var[(h,\phi)]/r$, we have 
\begin{align*}
	\var[(h,\phi(\cdot + x )) - (h,\phi(\cdot + y ))] \leq C|x-y|.
\end{align*}
Now suppose that $|x-y| \leq r$. Then, by doing a change of variables and using that for $u \in \R$, $\greenN{\h}(e^{z-u},e^{w-u}) = 2u + \greenN{\h}(e^z,e^w)$, we have that
\begin{align*}
	&\var[(h,\phi(\cdot + x )) - (h,\phi(\cdot + y ))] \\
	&=\int_{\strip}\int_{\strip}\greenN{\h}(e^z,e^w)(\phi(z+x) - \phi(z+y))(\phi(w+x) - \phi(w+y))dzdw \\
	&= 2(x-y)\int_{K_{1}}\int_{K_{1}}\phi(z)(\phi(w) - \phi(w+y-x))dzdw \\
	&+\int_{K_{1}}\int_{K_{1}}\greenN{\h}(e^{z},e^{w})\phi(z)(2\phi(w) - \phi(w+y-x) - \phi(w+x-y))dzdw
\end{align*}
Then, the result follows by noting that $\int_{K_{1}}\int_{K_{1}}|\greenN{\h}(e^{z},e^{w})|dzdw < \infty$ and that $|\phi(z) - \phi(w)|\leq \|\phi'\|_\infty |z-w|$ for any $z,w \in \strip$.
\end{proof}

We now prove the bound on the growth rate of $|(h_2,\phi(\cdot+x))|$.
\begin{lemma}\label{lem:latbound}
Let $h$ be a free boundary GFF on $\strip$ with additive constant fixed as in the previous lemma and let $h = h_1 + h_2$ be its decomposition as above. Fix $\delta > 0$. Then there a.s.\ exists a (random) constant $C > 0$ such that 
\begin{align*}
	|(h_2,\phi(\cdot + x))| \leq C(|x|^{\frac{1}{2}+\delta}+1), \quad \text{for all} \quad x \in \R.
	\end{align*}
\end{lemma}
\begin{proof}
We begin by noting that if we let $Y_x = (h,\phi(\cdot + x))$, then by Lemma~\ref{lem:varbound} and the Sudakov-Fernique inequality,
\begin{align*}
	\E\!\left[ \sup_{0 \leq t \leq T} Y_t \right] \leq C \E\!\left[ \sup_{0 \leq t \leq T} B_t \right] \leq \wt C \sqrt{T},
\end{align*}
where $B_t$ is a Brownian motion, where the last inequality follows by Brownian scaling and since $\E[\sup_{0 \leq t \leq 1} B_t]$ is finite.  By applying the same argument to the Gaussian process $(-Y_t)_{0 \leq t \leq T}$,  we obtain the same upper bound but with $Y$ replaced by $-Y$.
 Consequently, by Markov's inequality,
\begin{align}\label{eq:sup_bound}
	\p\!\left[ \sup_{0 \leq t \leq T} |Y_t| > T^{\frac{1}{2} + \delta} \right] \leq \wt C T^{-\delta}.
\end{align}
By the Borel-Cantelli lemma, applied to the sequence of events $\{ \sup_{0 \leq t \leq 2^k} |Y_t| > 2^{k(\frac{1}{2} + \delta}) \}$ (since~\eqref{eq:sup_bound} implies that the sum of the probabilities of those events is finite), we have that there a.s.\ exists a random $n_0 \in \N$ such that for all $n \geq n_0$, $\sup_{0 \leq t \leq 2^n} |Y_t |\leq 2^{n(\frac{1}{2}+\delta)}$. Thus, 
we see that there exists some random, a.s.\ finite constant $\wh C > 0$ such that
\begin{align*}
	|Y_t|\leq \wh C(|t|^{\frac{1}{2}+\delta}+1),\quad \text{for all} \quad t \in \R.
\end{align*}
Finally the law of the iterated logarithm implies  that there exists a finite random constant $C_1 > 0$ such that 
\begin{align*}
	|h_1(t)| \leq C_1 (|t|^{\frac{1}{2}+\delta}+1),\quad \text{for all} \quad t \in \R,
\end{align*}
and since $|(h_2,\phi(\cdot+x))|\leq |(h,\phi(\cdot+x))|+|(h_1,\phi(\cdot+x))|$ for each $x\in \R$, the claim of the lemma follows.
\end{proof}
Next is the corresponding bound for $h_1^\gamma$ rather than $h_2^\gamma$.
\begin{lemma}\label{lem:wedgebound}
Fix $\gamma \in [1,2]$ and let $(\strip,h^\gamma,-\infty,+\infty) \sim \qwedgeW{\gamma}{4}$ have the first exit parameterization and set  $\alpha_\gamma = \frac{4}{\gamma} - \frac{\gamma}{2}$. Let $h_1^\gamma$ be the projection of $h^\gamma$ to $H_1(\strip)$. Then for all $\epsilon \in (0,1)$ and $\delta \in (0,1)$, there exists a finite deterministic constant $C > 0$ independent of $\gamma$ such that with probability at least $1 - \epsilon$ we have that 
\begin{align*}
|(h_1^\gamma,\phi(\cdot-x)) - \alpha_\gamma x| \leq C(|x|^{\frac{1}{2}+\delta}+1),\quad \text{for all} \quad x \in \R.
\end{align*}
\end{lemma}
\begin{proof}
Note that $h_1^\gamma$ can be sampled as follows: Let $B,\wt{B}$ be two independent standard Brownian motions starting from $0$ and set
\begin{align*}
	T^\gamma = \sup\{ t \geq 0 : \wt{B}_{2t} - \alpha_\gamma t = 0 \}.
\end{align*}
$T^{\gamma}$ is well-defined since $\alpha_\gamma > 0$ and hence $\wt{B}_{2t} - \alpha_\gamma t \rightarrow -\infty$ as $t \rightarrow +\infty$, a.s. We consider the process $Y_t^\gamma = \wt{B}_{2(T^\gamma+t)}-\alpha_\gamma (t+T^\gamma)$ for $t \geq 0$. Then we set $h_1^\gamma(t)$ to be equal to $B_{2t} + \alpha_\gamma t$ for $t \geq 0$ and equal to $Y_{-t}^\gamma$ for $t \leq 0$. We also let $h_2$ be given by the projection to $H_2(\strip)$ of a free boundary GFF on $\strip$ which is independent of $B$ and $\wt{B}$.

We have thus constructed a coupling of $(h^\gamma)_{\gamma \in [1,2]}$ by setting $h^\gamma = h_1^\gamma + h_2$. We claim that under this coupling, a.s.\ there exists a random finite constant $C > 0$ such that 
\begin{align*}
\sup_{\gamma \in [1,2]}|h_1^\gamma (u) - \alpha_\gamma u| \leq C(|u|^{\frac{1}{2} + \delta}+1),\quad \text{for all} \quad u \in \R.
\end{align*}
Indeed, we note that the law of the iterated logarithm implies that a.s.\ we can find a random finite constant $C_1 > 0$ such that 
\begin{align*}
        \max(|B_{2t}|,|\wt{B}_{2t}|) \leq C_1(|t|^{\frac{1}{2}+\delta}+1),\quad \text{for all} \quad t \geq 0.
\end{align*}
Note also that $1 = \alpha_{2} \leq \alpha_\gamma \leq \alpha_1 = 7/2$ for all $\gamma \in [1,2]$ and so $T^{2}\geq T^{\gamma} \geq T^1$. By enlarging $C_1$ if necessary, we can assume that $\alpha_\gamma T^\gamma < 4 T^2 \leq C_1$, for all $\gamma \in [1,2]$ and since $h_1^\gamma(u) - \alpha_\gamma u = B_{2u}$ for $u \geq 0$ and $h_1^\gamma(u) - \alpha_\gamma u = \wt{B}_{2(T^\gamma-u)}-\alpha_\gamma T^\gamma$ for $u \leq 0$, there exists a finite random constant $C_2 > 0$ such that 
\begin{align*}
	|h_1^\gamma(u) - \alpha_\gamma u|\leq C_2 (|u|^{\frac{1}{2}+\delta}+1),\quad \text{for all} \quad u \in \R.
\end{align*}
Since $\alpha_\gamma u = \int_{\strip} \alpha_\gamma u\phi(z-u)dz$ for $u \in \R$, we obtain that we can find $C>0$ independent of $\gamma \in [1,2]$ such that with probability at least $1 - \epsilon$ we have that 
\begin{align*}
	\sup_{\gamma \in [1,2]}|(h^{\gamma}_{1},\phi(\cdot-u))-\alpha_{\gamma}u| \leq C(|u|^{\frac{1}{2}+\delta}+1),\quad \text{for all} \quad u \in \R,
\end{align*}
under the coupling we have constructed. Thus, the proof is done.
\end{proof}

We now deduce the key to the tightness needed in the proof of Proposition~\ref{prop:cutcone}.
\begin{lemma}\label{lem:translbound}
Fix $\epsilon \in (0,1)$ and $c_0 >0$. Then there exists a constant $c_1 >0$ such that the following is true for all $\gamma \in [1,2]$. Suppose that $(\strip,h^\gamma,-\infty,+\infty) \sim \qwedgeW{\gamma}{4}$ has the first exit parameterization, where $\alpha_\gamma$ is as in Lemma~\ref{lem:wedgebound}. Then, 
\begin{align*}
	\p[\{x \in \R: (h^\gamma(\cdot+x),\phi) \in [-c_0,c_0] \} \subseteq [-c_1,c_1]] \geq 1-\epsilon.
\end{align*}
\end{lemma}
\begin{proof}
Fix $\epsilon \in (0,1)$ and $\delta \in (0,\frac{1}{2})$. Then by Lemma~\ref{lem:latbound} and Lemma~\ref{lem:wedgebound}, we obtain that there exists a finite constant $C_1 > 0$ independent of $\gamma \in [1,2]$ such that with probability at least $1 - \epsilon$,
\begin{align*}
	|(h^\gamma(\cdot+u),\phi) - \alpha_\gamma u|\leq |(h_1^\gamma(\cdot+u),\phi) - \alpha_\gamma u| + |(h_2^\gamma(\cdot+u),\phi)|\leq C_1 (|u|^{\frac{1}{2}+\delta}+1).
\end{align*}
Hence, by the reverse triangle inequality
\begin{align*}
	|u|^{\frac{1}{2}-\delta}\leq 4C_1 + \frac{|(h^{\gamma}(\cdot+u),\phi)|}{|u|^{\frac{1}{2}+\delta}} 	\leq 4C_1 + |(h^{\gamma}(\cdot+u),\phi)|,\quad \text{for all} \quad |u|\geq 1
\end{align*}
By setting $c_1 = \max \!\Big\{1,(4C_1+c_0)^{\frac{1}{\frac{1}{2}-\delta}} \!\Big\}$, the result follows.
\end{proof}
Before finally proving Proposition~\ref{prop:cutcone}, we prove that the distributions of the corresponding quantum surfaces converge weakly in law.
\begin{lemma}\label{lem:coneconv}
Let $\cyl$ be the infinite cylinder and for $\gamma \in (0,2]$ let $(\cyl,h^\gamma,-\infty,+\infty) \sim \qconeW{\gamma}{4}$ have the first exit parameterization. Then the law of $h^\gamma$ on $H_{\textup{loc}}^{-1}(\cyl)$ converges weakly as $\gamma \uparrow 2$ to the law of $h^2$. The same holds if we instead consider $(\strip,h^\gamma,-\infty,+\infty) \sim \qwedgeW{\gamma}{4}$ with the first exit parameterization.
\end{lemma}
\begin{proof}
We prove the claim only in the case of quantum cones, the proof of the other case is similar. We will construct a coupling of $(h^{\gamma})_{\gamma \in (0,2]}$ such that $h^\gamma \rightarrow h^2$ as $\gamma \uparrow 2$ in $H_{\textup{loc}}^{-1}(\cyl)$ a.s.\ and that will complete the proof of the lemma.

Let $h_2$ be the projection of a free boundary GFF on $\cyl$ (with the additive constant fixed so that its average over $\{0\} \times [0,2\pi]$ is equal to zero) to the space of distributions which have mean zero on vertical lines. Let also $B,\wt{B}$ be two independent standard Brownian motions with $B_0 = \wt{B}_0 = 0$ and such that they are both independent of $h_2$. Set $\alpha_\gamma = 2 / \gamma$ and $T^\gamma = \sup\{ t \geq 0 : \wt{B}_t - \alpha_\gamma t = 0\} < \infty$ and consider the process $Y_u^\gamma = \wt{B}_{T^\gamma +u} - \alpha_\gamma (u+T^\gamma)$ for $u \geq 0$. We also consider the process $X^\gamma$ indexed by $\R$ given by $X_u^\gamma = B_u + \alpha_\gamma u$ for $u \geq 0$ and $X_u^\gamma = Y_{-u}^\gamma$ for $u \leq 0$.

We sample $h^\gamma$ by setting its projection to $H_{1}(\cyl)$ to be equal to $X^\gamma$ and its projection to $H_2(\cyl)$ to be equal to $h_2$. Note that in the coupling we have constructed, $T^\gamma \rightarrow T^2$ a.s., as $\gamma \uparrow 2$ and thus a.s.\ $X^\gamma \to X^2$ uniformly on compact subsets of $\R$, as $\gamma \uparrow 2$. Therefore we obtain that a.s., $h^\gamma \rightarrow h^2$ in $H_{\textup{loc}}^{-1}(\cyl)$ as $\gamma \uparrow 2$.
\end{proof}

Finally, we can prove Proposition~\ref{prop:cutcone}, which will be used to deduce Theorem~\ref{thm:critical_cone_welding}. Here we will need Proposition~\ref{prop:Cara}.

\begin{proof}[Proof of Proposition~\ref{prop:cutcone}]
\noindent{\it Step 1. Setup.} For $\gamma \in (0,2]$, we let $(\cyl,h^\gamma,-\infty,+\infty) \sim \qconeW{\gamma}{4}$ have the first exit parameterization and let $\eta^\gamma$ be a whole-plane $\SLE_\kappa(2)$ from $-\infty$ to $+\infty$ where $\kappa = \gamma^2 \in (0,4]$. Let $\phi_\gamma$ be the unique conformal map from $\cyl\setminus \eta^{\gamma}$ to the strip $\strip$ fixing $-\infty$ and $+\infty$ and such that $\re(\phi_\gamma(i\pi/2)) = 0$. We set 
\begin{align*}
	\wt{h}^\gamma = h^\gamma \circ \phi_\gamma^{-1} + Q_\gamma \log |(\phi_\gamma^{-1})'|,
\end{align*}
and we let $\wt{h}_1^\gamma$ (resp.\ $\wt{h}_2^\gamma$) be the projection of $\wt{h}^{\gamma}$ to $H_1(\strip)$ (resp.\ $H_2(\strip)$). We also set 
\begin{align*}
	\wt{X}^\gamma = \inf \{ x \in \R : (\wt{h}^\gamma(\cdot+x),\phi) = 0\}, \quad \wt{Y}^\gamma = \inf \{ y \in \R : \wt{h}_1^\gamma(y) = 0 \}.
\end{align*}
We note that $(\eta^\gamma,\phi_\gamma)$ can be sampled as follows. Let $\wt{\eta}^\gamma$ be a whole-plane $\SLE_\kappa(2)$ in $\C$ from $0$ to $\infty$ and let $\wt{\phi}_\gamma$ be the unique conformal map from $\C \setminus \wt{\eta}^\gamma$ to $\h$ which fixes $0$ and $\infty$ and satisfies $\im(\wt{\phi}_\gamma(i)) = 1$. We also consider the conformal map $\psi (z) = \log z$ from $\C$ to the cylinder $\cyl$. Then $\eta^\gamma = \psi(\wt{\eta}^\gamma)$ has the law of a whole-plane $\SLE_\kappa(2)$ in $\cyl$ from $-\infty$ to $+\infty$ and 
\begin{align*}
	\phi_\gamma(z) = (\psi \circ \wt{\phi}_\gamma \circ \psi^{-1})(z) - \re((\psi \circ \wt{\phi}_\gamma \circ \psi^{-1})(i\pi/2))
\end{align*}
is the unique conformal map from $\cyl \setminus \eta^\gamma$ to $\strip$ fixing $-\infty$ and $+\infty$ and such that $\re(\phi_\gamma(i\pi/2)) = 0$.

\noindent{\it Step 2. Convergence of $h^\gamma$ and $\phi_\gamma$.} Fix a sequence $(\gamma_n)$ in $(0,2)$ such that $\gamma_n \to 2$ as $n \to \infty$. Then the Skorokhod representation theorem combined with Proposition~\ref{prop:Cara} implies that we can find a coupling of $(\wt{\eta}^{\gamma_n}) $ such that $\wt{\phi}_{\gamma_n} \to \wt{\phi}_2$ as $n \to \infty$ locally uniformly a.s. This implies that $\psi \circ \wt{\phi}_{\gamma_n} \circ \psi^{-1} \to \psi \circ \wt{\phi}_2 \circ \psi^{-1}$ locally uniformly and $\wt{\phi}_{\gamma_n}(i) \to \wt{\phi}_2(i)$ as $n \to \infty$ a.s. Hence $\phi_{\gamma_n} \to \phi_2$ locally uniformly as $n \to \infty$ a.s.\ under this coupling. By Lemma~\ref{lem:coneconv} we can find a coupling of $(h^{\gamma_n})$ and $h^2$ such that $h^{\gamma_n} \to h^2$ as $n \to \infty$ in $H_{\textup{loc}}^{-1}(\cyl)$ a.s. By the Skorokhod representation theorem, we can find a coupling of $(h^{\gamma_n},\eta^{\gamma_n})$ and $(h^2,\eta^2)$ such that a.s., $h^{\gamma_n} \to h^2$ in $H_{\textup{loc}}^{-1}(\cyl)$ and $\phi_{\gamma_n} \to \phi_2$ locally uniformly as $n \to \infty$.

\noindent{\it Step 3. Tightness of $(\wt X^{\gamma_n})$ and $(\wt Y^{\gamma_n})$.} Next we claim that $(\wt{X}^{\gamma_n})$ and $(\wt{Y}^{\gamma_n})$ are tight sequences. Indeed, we first observe that under the above coupling, we have that $\wt{h}^{\gamma_n} \to \wt{h}^2$ as distributions as $n \to \infty$, i.e., $(\wt{h}^{\gamma_n},g) \to (\wt{h}^2,g)$ as $n \to \infty$ for all $g \in C_0^\infty(\strip)$ a.s. Consider the random distribution $\wh{h}^\gamma = \wt{h}^\gamma(\cdot+\wt{Y}^\gamma)$ with $\gamma \in (0,2)$. By Theorem~\ref{thm:cone_welding}, $(\strip,\wh{h}^\gamma,-\infty,+\infty) \sim \qwedgeW{\gamma}{4}$ with the first exit parameterization. 

For $M \in (0,\infty)$, $\gamma \in (0,2)$ we set 
\begin{align*}
	d_{\gamma,M} = \diam \{x \in \R : (\wt{h}^\gamma(\cdot+x),\phi) \in [-M,M] \} = \diam \{x \in \R : (\wh{h}^\gamma(\cdot+x),\phi) \in [-M,M] \}.
\end{align*}
Fix $\epsilon \in (0,1)$. The a.s.\ convergence of $(\wt{h}^{\gamma_n},\phi)$ as $n \to \infty$ under the specific coupling implies that its law is tight and so there exists $M \in (0,\infty)$ such that $\p[(\wt{h}^{\gamma_n},\phi) \in [-M,M]] \geq 1 - \epsilon$ for all $n \in \N$. Lemma~\ref{lem:translbound} implies that there exists $N \in (0,\infty)$ such that 
\begin{align*}
	\p[\{ x \in \R : (\wh{h}^{\gamma_n}(\cdot+x),\phi) \in [-M,M] \} \subseteq{[-N,N]} ] \geq 1 - \epsilon,
\end{align*}
for all $n \in \N$. Note that if the above event holds, we have that $d_{\gamma_n,M} \leq 2N$ and so $\p[d_{\gamma_n,M} \leq 2N] \geq 1 - \epsilon$, for all $n \in \N$. Moreover, we observe that if $d_{\gamma_n,M} \leq 2N$ and $(\wt{h}^{\gamma_n},\phi) \in [-M,M]$, then $\wt{X}^{\gamma_n} \in [-2N,2N]$ since $(\wt{h}^{\gamma_n}(\cdot+\wt{X}^{\gamma_n}),\phi) = 0$. Hence $\wt{X}^{\gamma_n} \in [-2N,2N]$ with probability at least $1 - 2\epsilon$ for all $n \in \N$. This shows that $(\wt{X}^{\gamma_n})$ is tight. Also, if $\wh{X}^\gamma = \inf \{ x \in \R : (\wh{h}^\gamma(\cdot+x),\phi) = 0 \}$ and $\{ x \in \R : (\wh{h}^\gamma(\cdot+x),\phi) \in [-M,M] \} \subseteq{[-N,N]}$ then $\wh{X}^\gamma \in [-N,N]$, so $\p[\wh{X}^{\gamma_n} \in [-N,N]] \geq 1 - \epsilon$, for all $n \in \N$. Since $\wh{X}^\gamma = \wt{X}^\gamma - \wt{Y}^\gamma$, we have that $(\wt{X}^{\gamma_n}-\wt{Y}^{\gamma_n})$ is tight and so $(\wt{Y}^{\gamma_n})$ is also tight.

\noindent{\it Step 4. Tightness of $(\wt h^{\gamma_n})$.} Next we show that $(\wt{h}^{\gamma_n})$ is tight in $H_{\textup{loc}}^{-1}(\strip)$. Suppose that we have the above coupling. Lemma~\ref{lem:coneconv} and the fact that $H_{\textup{loc}}^{-1}(\strip)$ is separable and complete imply that $(\wh{h}^{\gamma_n})$ is tight and so for fixed $\epsilon \in (0,1)$ we can find a compact subset $K$ of $H_{\textup{loc}}^{-1}(\strip)$ such that $\p[\wh{h}^{\gamma_n} \in K] \geq 1 - \epsilon$, for all $n \in \N$. We can also find $M > 0$ such that $\p[\wt{Y}^{\gamma_n} \in [-M,M]] \geq 1 - \epsilon$. We consider $\wt{K}$ to be the set of $h \in H_{\textup{loc}}^{-1}(\strip)$ such that $h = \psi(\cdot+x)$ for some $\psi \in K$ and $x \in [-M,M]$. We show that $\wt{K}$ is sequentially compact and hence a compact set. Fix $(h_n)$ in $\wt{K}$ with $h_n = \psi_n(\cdot+x_n)$ for $\psi_n \in K$, $x_n \in [-M,M]$. By passing to subsequences if necessary, we can assume that $\psi_n \to \psi$ in $H_{\textup{loc}}^{-1}(\strip)$ and $x_n \to x$ as $n \to \infty$ for some $\psi \in K$ and $x \in [-M,M]$. Let $D$ be a simply connected set which is compactly contained in $\strip$. Let also $(f_n)$ and $(\lambda_n)$ be the eigenfunctions and eigenvalues of the Laplace operator on $D$ as in Section~\ref{sec:GFF},  and let $g_n = \lambda_n^{-1/2} f_n$ be the orthonormal basis of $H_0(D)$.  Then,  it is easy to see that $\| h_n-h \|_{H^{-1}(D)} \leq \| \psi_n-\psi \|_{H^{-1}(D)} + \| \psi(\cdot+x_n)-\psi(\cdot+x) \|_{H^{-1}(D)}$,  where $h = \psi(\cdot+x)$.  Note that we have that $\psi = (g,\cdot)_{\nabla}$,  where $g = \sum_{m \geq 1} (h,g_m) g_m \in H_0(D)$ and the convergence is taken with respect to $H_0(D)$.  This follows from the fact that $\sum_{m \geq 1} (h,f_m)^2 \lambda_m^{-1} < \infty$.  Hence,  we have that $\psi(\cdot+x_n) = (g(\cdot+x_n),\cdot)_{\nabla}$ and $\psi(\cdot+x) = (g(\cdot+x),\cdot)_{\nabla}$ for all $n$.  It follows that $\| \psi(\cdot+x_n) - \psi(\cdot+x) \|_{\nabla} = \| g(\cdot+x_n)-g(\cdot+x) \|_{\nabla}$ for all $n$.  It is not hard to see that $\| g(\cdot+x_n)-g(\cdot+x) \|_{\nabla}$ converges to $0$ as $n \to \infty$ by the continuity of the Dirichlet inner product with respect to translations.

Hence we obtain that $h_n \to h$ in $H_{\textup{loc}}^{-1}(\strip)$ as $n \to \infty$ and clearly $h \in \wt{K}$, so $\wt{K}$ is compact. We also observe that if $\wh{h}^{\gamma_n} \in K$ and $\wt{Y}^{\gamma_n} \in [-M,M]$, then $\wt{h}^{\gamma_n} \in \wt{K}$ and so this proves the claim of tightness. 

\noindent{\it Step 5. Convergence of quantum wedges.} The above imply that $(\wt{h}^{\gamma_n},\wt{X}^{\gamma_n},\wt{Y}^{\gamma_n})$ is tight, so by passing to a subsequence if necessary, we can assume that we can find a coupling such that $(\wt{h}^{\gamma_n},\wt{X}^{\gamma_n},\wt{Y}^{\gamma_n}) \to (\wt{h},\wt{X},\wt{Y})$ a.s.\ in $H_{\textup{loc}}^{-1}(\strip) \times \R \times \R$. By applying similar arguments we obtain that $\wt{h}^{\gamma_n}(\cdot+\wt{Y}^{\gamma_n}) \to \wt{h}(\cdot+\wt{Y})$ in $H_{\textup{loc}}^{-1}(\strip)$ as $n \to \infty$ a.s.\ and so $(\strip,\wt{h}(\cdot+\wt{Y}),-\infty,+\infty) \sim \qwedgeW{2}{4}$ with the first exit parameterization. Note also that $\wt{Y} = \inf \{y \in \R : \wt{h}_1(y) = 0 \}$, where $\wt{h}_1$ is the projection of $\wt{h}$ on $H_1(\strip)$. It is also easy to see that $\wt{h}$ has the same law as $\wt{h}^2$ in $H_{\textup{loc}}^{-1}(\strip)$. Therefore by combining everything, we obtain that $(\wt{h}^{\gamma_n})$ converges in law to $\wt{h}^2$ as $n \to \infty$ and that $(\strip,\wt{h}^2(\cdot+\wt{Y}^2),-\infty,+\infty) \sim \qwedgeW{2}{4}$ with the first exit parameterization. This completes the proof of the proposition.
\end{proof}

\section{Exit times for $\SLE$ with the quantum parameterization}
\label{sec:exit_times}

This section is dedicated to proving bounds on exit times for $\SLE$ curves parameterized by quantum length. They will be important for comparing different normalizations of quantum cones and will be used in the proof of the main estimate in Section~\ref{sec:main_lemma} that we need to prove the regularity results.

Before moving on to proving the first bound we recall the following about hitting times of Brownian motion with drift. Fix $b,\alpha > 0$, let $(B_t)$ be a standard Brownian motion with $B_0 = 0$ and set $\tau_\alpha^b = \sup \{t \geq 0: B_t + \alpha t = b\}$. Then the probability density function of $\tau_\alpha^b$ is given by
\begin{align*}
    f_{\tau_\alpha^b}(t) = \frac{\alpha}{\sqrt{2\pi t}}e^{-\frac{(b-\alpha t)^2}{2 t}}\quad\text{for}\quad t>0,
\end{align*}
see \cite[IV.31]{bs2002bmbook}.
Moreover, if $X_t = B_{2t} + \alpha t$ is conditioned so that $B_{2t} + \alpha t > 0$ for all $t > 0$ then a sample from the law of $X_t$ can be produced by first sampling a standard Brownian motion $B_t$ and then setting $X_t = B_{2(t + \tau)} + \alpha (t + \tau)$ where $\tau = \sup\{t \geq 0 : B_{2t} + \alpha t = 0\}$.

We begin by providing an upper bound on the exit times from $\D$ and $B(0,\epsilon)$ for a whole-plane $\SLE_\kappa(\rho)$.
\begin{proposition}\label{prop:UBexittime}
Fix $\gamma \in (0,2]$, $W > \gamma^2/2$, and let $\CC = (\C,h,0,\infty) \sim \qconeW{\gamma}{W}$ have the circle-average embedding. Let $\eta$ be a whole-plane $\SLE_\kappa(W - 2)$ in $\C$ from $0$ to $\infty$ sampled independently of $h$ and then parameterized by quantum length with respect to $h$. If we set $T = \inf\{t \geq 0 : \eta(t) \in \partial\D\}$, then there exists some $p > 0$ such that
\begin{align}\label{eq:conetimeLB}
	\p[T \geq R] = O(R^{-p}) \quad\text{as}\quad R \to \infty.
\end{align}
Moreover, if we let $T_\epsilon = \inf \{ t \geq 0: \eta(t) \in \partial B(0,\epsilon) \}$ for $\epsilon \in (0,1)$, then there exist constants $\zeta \in (0,1),c > 0$ such that
\begin{align}\label{eq:conetimeLBeps}
	\p[ T_\epsilon \geq \epsilon^\zeta] = O(\epsilon^c) \quad\text{as}\quad \epsilon \rightarrow 0.
\end{align}
\end{proposition}
\begin{proof}
The proofs of~\eqref{eq:conetimeLB} and~\eqref{eq:conetimeLBeps} are similar, so we will only prove the former. 

Let $\varphi : \C \setminus \eta \rightarrow \strip$ be the unique conformal transformation with $\varphi(0) = - \infty$, $\varphi(\infty) = + \infty$ and such that if $h^w = h \circ \varphi^{-1} + Q\log |(\varphi^{-1})'|$, then $(\strip,h^w,-\infty,+\infty) \sim \qwedgeW{\gamma}{W}$ has the circle-average embedding. 

Let $h_1(z) = X_{\re(z)}$ (resp.\ $h_2$) be the projection of $h^w$ to $H_1(\strip)$ (resp.\ $H_2(\strip)$). Then $X_t = B_{2t} + \alpha t$ for $t > 0$, where $a = \frac{W}{\gamma} - \frac{\gamma}{2} > 0$ and $B_t$ is a standard Brownian motion with $B_0 = 0$ and conditioned so that $B_{2t} + a t > 0$, for all $t > 0$. We fix $\beta \in \!\left(1/\gamma,2/\gamma\right)$ and for $R > 2$ we set 
\begin{align*}
    S_R = \inf\{t \geq 0 : X_t = \beta \log R \}.
\end{align*}
Note that $S_R < \infty$ a.s.\ since $a > 0$ and that $\!\left(X_t / \sqrt{2}\right)_{t \geq 0}$ has the law of a standard Brownian motion with drift $\alpha/\sqrt{2} > 0$ (conditioned to stay positive). Moreover, $S_R$ is stochastically dominated from above by $\tau_{a/\sqrt{2}}^b$ (defined above) with $b = \frac{\beta \log R}{\sqrt 2}$. Fix $\delta > 0$. By comparing $S_R$ to $\tau_{\alpha/\sqrt{2}}^b$ we obtain that for $c > 0$ sufficiently large
\begin{align*}
    \p[S_R \geq c\log R ] = O(R^{-\delta}) \quad \text{as} \quad R \to \infty.
\end{align*}
Set $z_R = S_R + i\pi/2$ and note that $(X_{t + S_R})_{t \geq 0}$ has the same law as $(\wt B_{2t} + a t)_{t \geq 0}$, where $\wt B$ is a standard Brownian motion starting from $\beta \log R$ and conditioned on the event $E_R$ that $\wt B_{2t} + a t > 0$ for all $t > 0$. Note also that $\p[E_R] \to 1$ and so $\p[E_R] \geq 1/2$ for all $R$ sufficiently large. We also have that 
\begin{align*}
	\exp\!\left( \gamma \min_{S_R  \leq t \leq S_R + \pi/4} X_t \right)\mu_{h_2}(B^+(z_R,\pi/4)) \leq \mu_{h^w}(B^+(z_R,\pi/4)).
\end{align*}
 where $B^+(w,\pi/4) = \{z \in B(w,\pi/4):\re(z) \geq \re(w) \}$. Let $q > 0$ be such that $q(\beta \gamma - 1) = \delta$. Then we have that 
\begin{align*}
	\p[\mu_{h^w}(B^+(z_R,\pi/4))\leq R] \leq 2R^{-\delta}\E \!\left[\exp\!\left(-\gamma q \min_{0 \leq t \leq \pi /4}B_{2t}\right)\right] \E \!\left[\mu_{h_2}(B^+(i\pi/2,\pi/4))^{-q}\right] \lesssim R^{-\delta},
\end{align*}
since $X$ and $h_2$ are independent, $h_2$ is translation invariant and the expectations are finite. 

Next we set 
\begin{align*}
	\CT_R = \inf\{t \in \R : \nu_{h^w}((-\infty,t]\times \{0\}) = R\ \textup{or} \ \nu_{h^w}((-\infty,t]\times \{\pi\}) = R \}
\end{align*}
and we note that there exists $p \in (0,1)$ such that both of the expectations $\E \!\left[ \nu_{h^w}((-\infty,0]\times \{0,\pi\})^p\right]$ and $\E \!\left[ \nu_{h_2}([-1,0]\times \{0,\pi\})^p\right]$ are finite. We also observe that if $S_R \leq c\log R$ and $\CT_R \leq S_R$, then 
\begin{align*}
    R^p \leq \nu_{h^w}((-\infty,S_R]\times \{0,\pi\})^p \leq \nu_{h^w}((-\infty,0]\times \{0,\pi\})^p + R^{\frac{\beta \gamma p}{2}}\sum_{n=0}^{c\log R}\nu_{h_2}([n,n+1]\times \{0,\pi\})^p
\end{align*}
and since $1 - \beta \gamma /2 > 0$ we obtain that for some $q > 0$ we have that 
\begin{align*}
    \p[\CT_R \leq S_R,\,S_R \leq c\log R] = O(R^{-q}) \quad \text{as} \quad R \to \infty.
\end{align*}
Therefore we have that $S_R \leq c\log R$, $\CT_R \geq S_R$ and $\mu_{h^w}(B(z_R,\pi/4)) \geq \mu_{h^w}(B^+(z_R,\pi/4)) \geq R$ off an event with probability $O(R^{-q})$ as $R \to \infty$ for some fixed $q > 0$. Let $A_R$ be the event such that the above occur. From now on, we assume that we work on that event.

We note that there exists a universal constant $d > 0$ such that for every $z \in B(z_R,\pi/4)$, the probability that a Brownian motion starting from $z$ exits $\strip$ on $(-\infty,S_R]\times \{0,\pi\}$ is at least $d$. We fix $M > 0$ sufficiently large to be chosen later and suppose that there exists $w \in \varphi^{-1}(B(z_R,\pi/4)) \setminus B(0,M)$. Since $\eta([0,T]) \subseteq{\D}$, the Beurling estimate implies that the probability that a Brownian motion starting from $w$ exits $\C \setminus \eta$ on $\eta([0,T])$ is at most $CM^{-\frac{1}{2}}$ for some finite universal constant $C > 0$. But if $T \geq R$ then $\varphi^{-1}((-\infty,S_R]\times \{0,\pi\}) \subseteq{\eta([0,T])}$ and so the latter probability is at least $d$, hence $M \leq C^2 / d^2$. 

Thus for $M > C^2 / d^2$, we have that $\varphi^{-1}(B(z_R,\pi/4))\subseteq{B(0,M)}$ and so on $A_R \cap \{T \geq R\}$ we have that 
\begin{align*}
	\mu_h(B(0,M)) \geq \mu_{h}(\varphi^{-1}(B(z_R,\pi/4))) \geq R.
\end{align*}
By Lemma~\ref{lem:LBmassball}, we obtain that the latter occurs with probability $O(R^{-p})$ for some fixed $p > 0$ and so this completes the proof of the proposition.
\end{proof}

The next lemma is the analogue of Proposition~\ref{prop:UBexittime}, replacing the first exit time of $B(0,\epsilon)$ with the last exit time. More precisely, we prove that it is highly unlikely that the last exit time of $B(0,\epsilon)$ is greater than a certain positive power of $\epsilon$.
\begin{lemma}\label{lem:qtimeUB}
Fix $\gamma \in (0,2]$ and let $\CC = (\C,h,0,\infty) \sim \qconeW{\gamma}{\gamma^2}$ have the circle-average embedding. Let $\eta$ be a whole-plane $\SLE_\kappa(\kappa - 2)$ with $\kappa = \gamma^2$ from $0$ to $\infty$ sampled independently of $h$ and then parameterized by quantum length with respect to $h$ and set $S_\epsilon = \sup \{t \geq 0 : \eta(t) \in B(0,\epsilon)\}$. Then there exist $\zeta \in (0,1)$ and $b > 0$ such that
\begin{align*}
    \p[S_{\epsilon} \geq \epsilon^{\zeta}]= O(\epsilon^b) \quad \text{as} \quad \epsilon \rightarrow 0.
\end{align*}
\end{lemma}
\begin{proof}
Fix $a \in (0,1)$ and note that by~\eqref{eq:conetimeLBeps} there exist $\zeta \in (0,a)$ and $c > 0$ such that $\p[T_{\epsilon^a}\geq \epsilon^\zeta] = O(\epsilon^c)$ as $\epsilon \rightarrow 0$. Let $A(\epsilon)$ be the event that $\eta$ intersects $B(0,\epsilon)$ after it has left $B(0,\epsilon^{a})$ for the first time and $A^*(\epsilon)$ be the event that $\eta$ intersects $B(0,\epsilon)$ after it has left $\D$ for the first time. The scale invariance of the law of $\eta$ implies that $\p[A(\epsilon)] = \p[A^*(\epsilon^{1-a})]$. Finally,  noting that $\p[A^*(\epsilon)] = O(\epsilon^d)$ for some $d > 0$ by Lemma~\ref{lem:Breturn_time}, we have that
\begin{align*}
    \p[S_\epsilon \geq \epsilon^\zeta] \leq \p[T_{\epsilon^a} \geq \epsilon^{\zeta}] + \p[A(\epsilon) \cap \{T_{\epsilon^a} < \epsilon^\zeta\}] = O(\epsilon^b),
\end{align*}
where $b = \min\{c,d(1-a )\} > 0$.
\end{proof}

Finally we conclude this section with a lower bound on the quantum mass of small ball with respect to a quantum cone with the circle-average embedding.

\begin{lemma}\label{lem:LBcone}
Fix $\gamma \in (0,2]$. Let $\CC = (\C,h,0,\infty) \sim \qconeW{\gamma}{\gamma^2}$ have the circle-average embedding and fix $\sigma > 0$. There exists a constant $a>0$ such that for each $w \in B(0,1/2)$,
\begin{align*}
    \p[\mu_h(B(w,\epsilon)) \geq \epsilon^a] = 1 - O(\epsilon^{\sigma/2}).
\end{align*}
\end{lemma}
\begin{proof}
Consider first a zero-boundary GFF $h_{\D}^0$ on $\D$. By Markov's inequality and \cite[Proposition~3.7]{rv2010gmcrevisit} (recall Remark~\ref{rmk:different_fields}),
\begin{align*}
    \p[ \mu_{h_{\D}^0}(B(w,\epsilon)) \leq \epsilon^a] = \p[\mu_{h_{\D}^0}(B(w,\epsilon))^{-1} \geq \epsilon^{-a}] \leq \epsilon^a \E[\mu_{h_{\D}^0}(B(w,\epsilon))^{-1}] \lesssim \epsilon^{a-2-\gamma^2}.
\end{align*}
Next, we shall transfer this into the result for the quantum cone. In order to do so, we consider a whole-plane GFF $h^w$. We choose the normalization so that the average of $h^w$ on $\partial \D$ is $0$. Then we may write $h^w = h_{\D}^0 + \Fh_{\D}$ where conditionally on $h_{|\partial\D}$, $\Fh_{\D}$ is the harmonic extension of the values of $h^w$ on $\partial \D$ to $\D$. Borrowing the terminology of \cite{mq2020geodesics}, we say that $\D$ is $M$-good if $\sup_{z \in B(0,15/16)} |\Fh_{\D}(z)| \leq M$. Denote the event that $\D$ is $M$-good by $E_M$. By \cite[Lemma~4.4]{mq2020geodesics}, there exist constants $c_1,c_2>0$ such that for each $M>0$,
\begin{align*}
	\p[E_M^c] \leq c_1 e^{-c_2 M^2}.
\end{align*}
Hence, there exists a constant $c_0 > 0$ such that if we let $M^2 = -c_0 \log \epsilon$, then $\p[E_M] = 1-O(\epsilon^{\sigma/2})$. Henceforth, assume that $M$ is such. By \cite[Remark~4.2 and equation~(4.3)]{mq2020geodesics} we have that for $p,q>1$ such that $\frac{1}{p} + \frac{1}{q} = 1$,
\begin{align}
	\p[\mu_{h^w}(B(w,\epsilon)) \leq \epsilon^a] \leq \p[E_M^c]+ e^{(p-1)\wh{c}M^2} \p[\mu_{h_{\D}^0}(B(w,\epsilon)) \leq \epsilon^a]^{1/q} \lesssim \epsilon^{\sigma/2} + \epsilon^{(p-1)(\frac{a-2-\gamma^2}{p}-\wt{c})},
\end{align}
for some constants $\wh c,\wt c > 0$. Thus choosing $a>0$ large enough we have that $\p[\mu_{h^w}(B(w,\epsilon)) \geq \epsilon^a] = 1 - O(\epsilon^{\sigma/2})$.  Since the field which describes the quantum cone restricted to $\D$ is obtained by adding a positive function to the field $h^w$, it follows that
\begin{align*}
	\p[\mu_h(B(w,\epsilon))\geq \epsilon^a] \geq \p[\mu_{h^w}(B(w,\epsilon)) \geq \epsilon^a] = 1 - O(\epsilon^{\sigma/2}).
\end{align*}
\end{proof}

\section{Main estimates}
\label{sec:main_lemma}
This section is dedicated to the proofs of the main estimates that we will need in order to prove Theorems~\ref{thm:sle4_thm},~\ref{thm:jones_smirnov_sle4} and~\ref{thm:sle8_thm}. They are contained in Lemmas~\ref{lem:main_lemma_ubd}, \ref{lem:main_lemma_lbd}, \ref{lem:mainlemma8}, and~\ref{lem:mainlemma4}. Lemmas~\ref{lem:main_lemma_ubd} and ~\ref{lem:main_lemma_lbd} contain the estimate needed to prove the second part of Theorem~\ref{thm:sle8_thm}, as well as an estimate related to whole-plane $\SLE$ which serves as the model on which we base the main estimates, Lemmas~\ref{lem:mainlemma8} and~\ref{lem:mainlemma4}, which will give us the modulus of continuity for $\SLE_8$ and the uniformizing map of $\SLE_4$, respectively. How these are used is explained in Sections~\ref{sec:sle8} and~\ref{sec:sle4}.

\subsection{Upper bound}
\label{subsec:upper_bound}

\begin{lemma}
\label{lem:main_lemma_ubd}
Fix $\gamma \in (0,2)$ and let $\kappa = \gamma^2$. Let $\CC = (\C,h,0,\infty) \sim \qconeW{\gamma}{\gamma^2}$ have the circle-average embedding. Let $\eta_1,\eta_2$ be a pair of paths independent of $h$ such that $\eta_1$ is a whole-plane $\SLE_\kappa(\kappa-2)$ process in $\C$ from $0$ to $\infty$ and the conditional law of $\eta_2$ given $\eta_1$ is that of a chordal $\SLE_\kappa(\frac{\kappa}{2}-2;\frac{\kappa}{2}-2)$ from $0$ to $\infty$ in $\C \setminus \eta_1$.  We then parameterize $\eta_1,\eta_2$ by quantum length.  
 Let $U_1$ (resp.\ $U_2$) be the component of $\C \setminus (\eta_1 \cup \eta_2)$ that lies to the left (resp.\ right) of $\eta_1$. For each $R \geq 0$ we let $A_R = \eta_1([R,\infty)) \cup \eta_2([R,\infty))$ (so that $A_0 = \eta_1 \cup \eta_2$).  Then for fixed $R, \sigma > 0$
\begin{align}
\label{eqn:mainlemma1}
    \p\!\left[ \sup_{z \in B(0,\epsilon) \cap U_1} \p^z[B \text{ first hits } A_0 \text{ in } A_R \,|\, \eta_1, \eta_2] \leq \exp(-\epsilon^{-\sigma}) \right] = O(\epsilon^{\sigma/2 + o(1)}),
\end{align}
where $B$ is a Brownian motion.
\end{lemma}
The main idea in the proof of Lemma~\ref{lem:main_lemma_ubd} is to prove that we can find a small ball contained in $U_1 \cap B(0,\epsilon)$, which when mapped to $\strip$ (with a conformal map which embeds the quantum wedge $\CW_1$ parameterized by $U_1$ in a certain way) is not too close to $\partial \strip$ and has real part comparable to where the vertical line average process $X_t^1$ for the embedding of $\CW_1$ into $\strip$ first hits $b \log \epsilon$ for some $b > 0$.  Controlling the location of the image of the ball is accomplished by lower bounding its quantum mass (using Lemma~\ref{lem:LBcone}) and upper bounding the quantum mass of the region of $\strip$ which is very far to the left (with Lemma~\ref{lem:pthmoment} and a comparison with $X_t^1$) to see that the image of the ball can not be contained in that region.  By the conformal invariance of Brownian motion, the probability that $B$ first exits $A_0$ in $A_R$ is then comparable to the probability that a Brownian motion in $\strip$ starting from where $X_t^1$ first hits $b \log \epsilon$ first exits $\strip$ to the right of where $X_t^1$ first hits $0$.  

\noindent{\it General setup for the proof of Lemma~\ref{lem:main_lemma_ubd}.}  We begin by fixing some notation.  Let $\CC = (\C,h,0,\infty)$ have law $\qconeW{\gamma}{\gamma^2}$ with the circle-average embedding and we let $\eta_1$, $\eta_2$ be as in the statement of Lemma~\ref{lem:main_lemma_ubd}.  We note that by \cite{ms2017imag4} we can view $\eta_1$, $\eta_2$ as arising as the flow lines of a whole-plane GFF independent of $h$ with angle difference $\pi\kappa/(4-\kappa)$.  The pair $\eta_1$, $\eta_{2}$ divides $\CC$ into two independent wedges $\CW_1$ and $\CW_2$ with law $\qwedgeW{\gamma}{\gamma^2/2}$ (Theorems~\ref{thm:wedge_welding} and~\ref{thm:cone_welding}) and respectively parameterized by $U_1$ and $U_2$. We assume that $\eta_1$ and $\eta_2$ are parameterized by quantum length. Next, for $j = 1,2$, we let $(\strip,h^j,-\infty,+\infty)$ be the first exit parameterization of $\CW_j$ by $\strip$, and let $X_t^j$ denote the average on vertical lines process and $h_2^j$ the projection of $h^j$ on $H_2(\strip)$. Moreover, we let $\varphi_j:U_j \rightarrow \strip$ be the conformal map embedding $\CW_j$ into $\strip$ as above and define the random time $u_{\alpha,\epsilon}^j = \inf\{ t \in \R: X_t^j = \alpha \log \epsilon \}$.

Fix $\xi > 1$.  We emphasize that $\CC$ has the circle-average embedding. Let $w$ be the leftmost point in $U_{1,\epsilon} = U_1 \cap B(0,\epsilon)$ such that $B(w,\epsilon^\xi) \subseteq U_{1,\epsilon}$ breaking ties by taking the point with the smallest imaginary part.  If there is no such point, we take $w = 0$.  By Lemma~\ref{lem:fitball} we have that $w \neq 0$ with probability $1-o_\epsilon^\infty(\epsilon)$.

\begin{proof}[Proof of Lemma~\ref{lem:main_lemma_ubd}]
\noindent{\it Step 1. $\varphi_1(B(w,\epsilon^\xi))$ is not too far to the left.}  We will prove that with probability $1-O(\epsilon^{\sigma/2})$ at least part of $\varphi_1(B(w,\epsilon^\xi))$ lies to the right of the line $\{ \re(z) = -\epsilon^{-\sigma} \}$ but not too far to the right and not too close to the boundary, then prove that this is sufficient for a Brownian motion started from some point in $\varphi_1(B(w,\epsilon^\xi))$ to hit $\varphi_1(A_R)$ with sufficiently high probability. The result then follows by the conformal invariance of Brownian motion.

We note that $w$ is independent of $\CC$ and therefore we can apply Lemma~\ref{lem:LBcone} to $w$. Note that the event $B(w,\epsilon^\xi) \subseteq U_{1,\epsilon}$ is equivalent to the event $w  \neq 0$.  Thus by Lemma~\ref{lem:LBcone} there exists an $a>1$ such that
\begin{align}\label{eq:mass_ball_close_to_origin}
	\p[\mu_h(B(w,\epsilon^{\xi})) \geq \epsilon^{a\xi} \,|\, B(w,\epsilon^\xi) \subseteq U_{1,\epsilon} ] = 1-O(\epsilon^{\sigma/2}).
\end{align}
Moreover, by Lemma~\ref{lem:pthmoment},
\begin{align*}
	\p[\mu_{h^1}(\strip_-+u_{\alpha,\epsilon}^1) \geq \epsilon^{a \xi}] \leq \epsilon^{-a \xi p} \E[\mu_{h^1}(\strip_-+u_{\alpha,\epsilon}^1)^p] = c_p \epsilon^{p(\alpha \gamma - a\xi)}
\end{align*}
and hence if $\alpha \geq (\sigma/(2p) + a\xi)/\gamma$ we have that
\begin{align}\label{eq:image_not_left}
	\p[\varphi_1(B(w,\epsilon^\xi)) \not\subseteq (\strip_- + u_{\alpha,\epsilon}^1)\,|\, B(w,\epsilon^\xi) \subseteq U_{1,\epsilon}] = 1 - O(\epsilon^{\sigma/2}).
\end{align}
Thus, outside of an event of probability $O(\epsilon^{\sigma/2})$, $\varphi_1(B(w,\epsilon^\xi))$ will be no farther to the left than the line $\{ \re(z) = u_{\alpha,\epsilon}^1 \}$. Next we shall argue that the line can not be too far to the left.

Recall that the process $(-X_{-t/2}^1)_{t \geq 0}$ is a $\BES^3$ with $X_0^1 = 0$.  Therefore if $-s > 0$ is the last time that $X_{-t}^1$ hits some fixed negative value then the process $(X_{t+s}^1)_{t \geq 0}$ has the law of a one-dimensional Brownian motion run at twice the speed \cite[Chapter~VI, Proposition~3.10 and Theorem~3.11]{ry1999book}.  In particular, $(X_{t+u_{\alpha,\epsilon}^1}^1)_{t \geq 0}$ has the law of a Brownian motion started from $X_{u_{\alpha,\epsilon}^1}^1=\alpha\log\epsilon$ run at twice the speed.  We know that $t=0$ is the first hitting time of $0$ for $X^1$ and hence to bound the probability that $u_{\alpha,\epsilon}^1$ lies far to the left we just have to bound the probability that the first-passage time of $0$ of a Brownian motion started at $\alpha\log\epsilon$ is large. Recall that the first-passage time $\tau$ of level $0$ of a Brownian motion started at the point $-b$ ($b>0$) has density \cite[Section~3.6]{bass2011stochastic},
\begin{align*}
	f_\tau(t) = \frac{b}{\sqrt{2\pi t^3}} \exp\!\left(-\frac{b^2}{2t}\right).
\end{align*}
Thus if $b = -\alpha\log\epsilon$ we have that
\begin{align*}
	\p[\tau > \epsilon^{-q}] = \int_{\epsilon^{-q}}^\infty \frac{\alpha\log(1/\epsilon)}{\sqrt{2\pi} t^{3/2}} \exp\!\left( -\frac{\alpha^2 (\log \epsilon)^2}{2t} \right) dt = \epsilon^{\frac{q}{2} + o(1)}.
\end{align*}
Hence, $\p[ u_{\alpha,\epsilon}^1 < -\epsilon^{-q}] = \p[\tau > 2 \epsilon^{-q}] =  O(\epsilon^{\frac{q}{2} + o(1)})$. Consequently, by~\eqref{eq:image_not_left} with $q = \sigma$, we have that
\begin{align}\label{eq:totheright}
	\p[\varphi_1(B(w,\epsilon^\xi)) \not\subseteq (\strip_- - \epsilon^{-\sigma})\,|\,B(w,\epsilon^\xi) \subseteq U_{1,\epsilon}] = 1 - O(\epsilon^{\sigma/2 + o(1)}).
\end{align}

\noindent{\it Step 2. $\varphi_1(B(w,\epsilon^\xi))$ is not too close to the boundary.} By Lemma~\ref{lem:mass_close_to_boundary}, we have that
\begin{align*}
	&\p[ \mu_{h^1}(\{ z \in \strip_-: \re(z) \geq -\epsilon^{-\sigma}, \dist(z,\partial \strip) < \epsilon^c\}) > \epsilon^{a \xi}] \\
	&\leq \epsilon^{-p a \xi} \E[\mu_{h^1}(\{ z \in \strip_-: \re(z) \geq -\epsilon^{-\sigma}, \dist(z,\partial \strip)<\epsilon^c \})^p] \\
	&\lesssim \epsilon^{M c - \sigma - p a \xi}.
\end{align*}
Thus, choosing $c$ sufficiently large, we have it follows from~\eqref{eq:mass_ball_close_to_origin} that with probability at least $1-O(\epsilon^{\sigma/2})$, $\varphi_1(B(w,\epsilon^\xi)) \not\subseteq (\strip_- - \epsilon^{-\sigma}) \cup \{ z \in \strip_-: \re(z) \geq -\epsilon^{-\sigma}, \dist(z,\partial \strip) < \epsilon^c\}$.

\noindent{\it Step 3. The images of $\eta_j(R)$ under $\varphi_1$ are not too far to the right.}
Now we prove that with high probability, the images of $\eta_j(R)$ under $\varphi_1$ are not mapped too far to the right. Set $T_{\epsilon,d} = \inf\{t \geq 0 : X_t^1 = -d\log \epsilon\}$ for some large fixed $d > 0$ to be chosen later and recall that $(X_t^1)_{t \geq 0}$ has the law of a one-dimensional Brownian motion $B$ run at twice the speed, with $B_0 = 0$. We note that 
\begin{align*}
	\nu_{h^1}([T_{\epsilon,d},T_{\epsilon,d}+1] \times \{0\}) \geq \epsilon^{-\frac{\gamma d}{2}} \exp\left( \frac{\gamma}{2}\min_{T_{\epsilon,d}\leq t \leq T_{\epsilon,d}+1}B_{2t}-B_{2T_{\epsilon,d}} \right) \nu_{h_2^1}([T_{\epsilon,d},T_{\epsilon,d}+1] \times \{0\}),
\end{align*}
and that $\min_{t \in [T_{\epsilon,d},T_{\epsilon,d}+1]} B_{2t} - B_{2 T_{\epsilon,d}}$ has the same law as $\min_{t \in [0,1]} B_{2t}$,  which has finite exponential moments. Moreover, since $h_2^1$ is translation invariant and $\nu_{h_2^1}([0,1] \times \{ 0 \})$ has finite negative moments, the above together with Markov's inequality imply that off an event with probability $O(\epsilon^{\sigma / 2})$, we have that $ \nu_{h^1}((-\infty,T_{\epsilon,d}+1] \times \{0\}) \geq R$, and similarly $\nu_{h^1}((-\infty,T_{\epsilon,d}+1] \times \{\pi\}) \geq R$, by picking $d>0$ large enough. In particular, off an event with probability $O(\epsilon^{\sigma /2 })$ we have that $\re(\varphi_1(\eta_1(R))),\re(\varphi_1(\eta_2(R))) \in (-\infty,T_{\epsilon,d}+1]$. Also for $k \in \Z$ we set $A_{\epsilon,k} = [k-1,k] \times ([0,\epsilon^c] \cup [\pi - \epsilon^c,\pi])$ and $A_{\epsilon} = [0,\epsilon^{-\sigma}] \times ([0,\epsilon^c] \cup [\pi - \epsilon^c,\pi])$. Then if $p \in (0,1\wedge \tfrac{2}{\gamma^2})$,
\begin{align*}
	\E[ \mu_{h_2^1}(A_{\epsilon})^p ] &\leq \sum_{n=1}^{\lceil\epsilon^{-\sigma}\rceil}\E[ \mu_{h_2^1}(A_{\epsilon,n})^p ] = \lceil\epsilon^{-\sigma}\rceil\E[ \mu_{h_2^1}(A_{\epsilon,1})^p]
\end{align*}
by the invariance of $h_2^1$ under horizontal translations. From the proof of Lemma~\ref{lem:mass_close_to_boundary}, we have that there is a constant $M=M_p>0$ such that 
\begin{align}\label{eq:moment_bound_close_to_boundary}
	\E[ \mu_{h_2^1}(A_{\epsilon,1})^p ] \lesssim \epsilon^{M c},
\end{align}
where the implicit constant is independent of $\epsilon$ and $c$, and thus 
\begin{align}\label{eq:upper_bound1}
\E[ \mu_{h_2^1}(A_{\epsilon})^p ] \lesssim \epsilon^{Mc-\sigma}.
\end{align}
Also we have that 
\begin{align}\label{eq:upper_bound2}
\mu_{h^1}(\wt{A}_{\epsilon}) \leq \epsilon^{-\gamma d}\mu_{h_2^1}(\wt{A}_{\epsilon}) \quad\text{where}\quad \wt{A}_{\epsilon} = [0,T_{\epsilon,d}] \times ([0,\epsilon^c] \cup [\pi - \epsilon^c,\pi]).
\end{align}
Moreover, since $\mu_{h^1}([T_{\epsilon,d},T_{\epsilon,d}+2]) \times [0,\epsilon^c] \cup [\pi-\epsilon^c,\pi]))$ has the same law as $\epsilon^{-\gamma d} \mu_{h^1}([0,2]) \times [0,\epsilon^c] \cup [\pi-\epsilon^c,\pi]))$ and the exponential moments of $\sup_{t \in [0,4]} B_t$ are finite, it follows from \eqref{eq:moment_bound_close_to_boundary} that for $p \in (0, 1 \wedge \tfrac{2}{\gamma^2})$,
\begin{align}\label{eq:moment_bound_last_part}
	\E[ \mu_{h^1}([T_{\epsilon,d},T_{\epsilon,d}+2]) \times [0,\epsilon^c] \cup [\pi-\epsilon^c,\pi]))^p ] \leq \epsilon^{Mc - p\gamma d}.
\end{align}

By combining~\eqref{eq:upper_bound1},~\eqref{eq:upper_bound2},~\eqref{eq:moment_bound_last_part} and Markov's inequality applied to $\mu_{h_2^{1}}$ together with the fact that $\p[ T_{\epsilon,d} \geq \epsilon^{-\sigma} / 2] = O(\epsilon^{\sigma / 2})$, we obtain that the quantum mass of $[0,T_{\epsilon,d}+2] \times ([0,\epsilon^c] \cup [\pi -\epsilon^c,\pi])$ is at most $\epsilon^{\alpha \xi}$ with probability $1 - O(\epsilon^{\sigma / 2})$ for $c > 0$ sufficiently large. Therefore off an event with probability $O(\epsilon^{\sigma / 2})$ we have that $\re(\varphi_1(\eta_j(R))) \in (-\infty,T_{\epsilon,d}+1]$ for $j=1,2$, $\varphi_1(B(w,\epsilon^{\xi})) \not\subseteq (\strip_- - \epsilon^{-\sigma}) \cup \{z \in \strip : -\epsilon^{-\sigma} \leq \re(z) \leq T_{\epsilon,d}+2, \dist(z,\partial \strip) < \epsilon^c \}$ and $T_{\epsilon,d} \leq \epsilon^{-\sigma} / 2$.

\noindent{\it Step 4.  Conclusion of the proof.}
Suppose now that we work on the above event. If $\varphi_1(B(w,\epsilon^\xi)) \subseteq{\strip_+ + T_{\epsilon,d}+2}$, then there exists a universal constant $\delta \in (0,1)$ such that with probability at least $\delta$ a Brownian motion starting from $z \in \varphi_1(B(w,\epsilon^\xi))$ exits $\strip$ on $[T_{\epsilon,d}+1,\infty) \times \{0,\pi\}$. If $\varphi_1(B(w,\epsilon^{\xi})) \not\subseteq \strip_+ + T_{\epsilon,d}+2$, then there exists $z \in \varphi_1(B(w,\epsilon^{\xi}))$ such that $-\epsilon^{-\sigma} \leq \re(z) \leq 2+\epsilon^{-\sigma}/2$ and $\im(z) \in [\epsilon^c,\pi-\epsilon^c]$. We also note that if $\wt{B}$ is  a planar Brownian motion, started from a point $z_0 \in \{ z \in \strip: -\epsilon^{-\sigma} \leq \re(z) \leq  2+ \epsilon^{-\sigma}/2, \dist(z,\partial \strip) \geq \epsilon^c  \}$, then  by Gambler's ruin, we know that
\begin{align}\label{eq:gambler}
	\p[ \wt{B} \, \text{hits} \, \{z: \re(z) \geq -\epsilon^{-\sigma}, \, \im(z) =\pi/2 \} \, \text{before} \, \partial \strip] \geq \frac{\epsilon^c}{\pi}.
\end{align}
By the Markov property of Brownian motion, the probability that $\wt B$ hits $\{ \re(z) = \epsilon^{-\sigma} \}$ before exiting $\strip$ is bounded from below by the probability that $\wt B$ hits $\{z: \re(z) \geq -\epsilon^{-\sigma}, \, \im(z) =\pi/2 \}$ before exiting $\strip$ times the probability that a Brownian motion started from $-\epsilon^{-\sigma} + i \pi/2$ hits $\{ \re(z) = \epsilon^{-\sigma} \}$ before exiting $\strip$. Thus, \eqref{eq:gambler} and \cite[Lemma~IV.5.1]{gm05harmonic} imply that
\begin{align}
	\p[ \wt{B} \, \text{hits} \, \{ \re(z) = \epsilon^{-\sigma} \} \, \text{before} \, \partial \strip ] \geq \frac{\epsilon^c}{2\pi} e^{-2\epsilon^{-\sigma}} \geq e^{-3\epsilon^{-\sigma}},
\end{align}
for sufficiently small $\epsilon$. Hence with probability at least $ e^{-3\epsilon^{-\sigma}}$ a Brownian motion starting from $z$ exits $\strip$ in $[\epsilon^{-\sigma},\infty) \times \{0,\pi\}$. The result is deduced by observing that $[T_{\epsilon,d}+2,\infty) \times \{0,\pi\} \subseteq \varphi_1(A_R)$. 
\end{proof}

\subsection{Lower bound}

\begin{lemma}
\label{lem:main_lemma_lbd}
Fix $\kappa \in (0,4]$.  Let $\eta_1,\eta_2$ be a pair of paths where $\eta_1$ is a whole-plane $\SLE_\kappa(\kappa-2)$ process in $\C$ from $0$ to $\infty$ and the conditional law of $\eta_2$ given $\eta_1$ is that of a chordal $\SLE_\kappa(\frac{\kappa}{2}-2;\frac{\kappa}{2}-2)$ from $0$ to $\infty$ in $\C \setminus \eta_1$.  Fix $\zeta \in [0,1)$ and $\epsilon > 0$. Let $\tau_i = \inf\{t \geq 0 : \eta_i(t) \notin B(0,\epsilon^{\zeta})\}$, $\wt{\tau}_i =\inf\{t \geq 0 : \eta_i(t) \notin B(0,\epsilon^{\zeta} / 2)\}$. Fix $C>1$ large and $D>0$ small. Let $E$ be the event that
\begin{enumerate}[(i)]
\item\label{it:lbd_cond1} $\sup_{z \in B(0,\epsilon) \cap U_1}\p^z[ B\, \text{hits}\, \partial B(0,\epsilon^{\zeta})\, \text{before} \,\cup_{i=1}^2 \eta_i([0,\tau_i])\,|\,\eta_{1}|_{[0,\tau_1]},\eta_{2}|_{[0,\tau_2]}] \leq \exp(-\epsilon^{-\sigma})$
\item\label{it:lbd_cond2} $\eta_1$ and $\eta_2$ do not return to $B(0,C^{n-1}\epsilon^{\zeta})$ after leaving $B(0,C^{n} \epsilon^{\zeta})$ for $n = 1,2,3$.
\item\label{it:lbd_cond3} Let $\tau_{i}^C = \inf\{ t \geq 0 : \eta_i(t) \notin B(0,C^{3} \epsilon^{\zeta}\}$. Then $\dist(\eta_1([\wt{\tau}_1, \tau_{1}^C]), \eta_2([\wt{\tau}_2,\tau_{2}^C])) \geq D \epsilon^{\zeta}$.
\end{enumerate}
Then we can find $C > 1$ and $D > 0$ such that 
\begin{align}
\label{eqn:mainlemma3}
\p[ E ] \geq \epsilon^{\sigma / 2 + o(1)} \quad \text{as} \quad \epsilon \to 0
\end{align}	
\end{lemma}

Conditions~\eqref{it:lbd_cond2}, \eqref{it:lbd_cond3} are technical conditions in the definition of $E$ that we will need for the proof of the lower bound in Theorem~\ref{thm:sle4_thm} as well as in Theorem~\ref{thm:jones_smirnov_sle4}.  Condition~\eqref{it:lbd_cond1} determines the exponent in~\eqref{eqn:mainlemma3}.  The proof of Lemma~\ref{lem:main_lemma_lbd} is much more involved than that of Lemma~\ref{lem:main_lemma_ubd}.  The reason for this is that it is a statement purely about $\eta_1, \eta_2$ and does not involve LQG.  In particular, we have to argue that conditioning the process which gives the average on vertical lines for the quantum wedge parameterized by $U_1$ (in the context of Lemma~\ref{lem:main_lemma_ubd}) taking a very long time to hit $0$ after first hitting $b \log \epsilon$ (an event with polynomial probability in $\epsilon$ as $\epsilon \to 0$) does not lead to degenerate behavior in $\eta_1$, $\eta_2$ away from $0$.

We note that Lemma~\ref{lem:main_lemma_ubd} does not hold for $\gamma = 2$ because in this case the amount of mass in $\strip_-$ which is close to $\partial \strip$ does not decay to $0$ as $\epsilon \to 0$ as a power of $\epsilon$.  This property for $\gamma < 2$ is important for ruling out the possibility that the aforementioned small ball is mapped too close to $\partial \strip$.  On the other hand, we emphasize that Lemma~\ref{lem:main_lemma_lbd} does hold for $\gamma = 2$.

Recall the notation in the proof of Lemma~\ref{lem:main_lemma_ubd}. In order to prove Lemma~\ref{lem:main_lemma_lbd}, we introduce a different embedding  of $\CC$ which is similar to the circle-average embedding but is more amenable to ``cutting'' and ``gluing'' operations.  We let $\phi$ be a radially symmetric $C_0^\infty$ function which is supported in $B(0,1) \setminus B(0,99/100)$ such that $\phi \geq 0$ and $\int \phi(x) dx = 1$.  Consider the process $R_s = (h(\cdot s) + Q \log s, \phi)$.  We say that $h$ scaled so that $\sup\{s \geq 0 : R_s = 0\} = 1$ is a smooth canonical description of the surface $\CC$.  We note that the embedding associated with the smooth canonical description is defined a way which is analogous to the circle-average embedding except instead of integrating the field against the uniform measure on the boundary of a circle centered at the origin we are considering the field integrated a radially symmetric smooth bump function.

Now, let $\CC$ be embedded as the smooth canonical description and (recalling the notation in Section~\ref{subsec:upper_bound}) let $\varphi_j : U_j \to \strip$ be the conformal maps which for which $h^j$ are embedded with the first exit parameterization. Let $h_R^1$ be the restriction of $h^1$ to $[1,\infty) \times (0,\pi)$, $A$ be the quantum length of $(-\infty,1] \times \{0\}$ and  $B$ be the quantum length of $(-\infty,1] \times \{\pi\}$, both with respect to $h^1$.  Let $x_2$ (resp.\ $y_2$) be the point on $\R$ (resp.\ $\R + i \pi$) such that the quantum length of $(-\infty,x_2]$ (resp.\ $(-\infty,y_2]$) is equal to $B$ (resp.\ $A$).  Let $\CA$ be the surface which is formed by gluing according to quantum length the quantum surface described by $h_R^1$ with the quantum surface described by the restriction of $h^2$ to the part of $\strip$ which is to the right of the line $L$ which connects $x_2$ to $y_2$.  The gluing is defined by looking at the part of $\CC$ which corresponds to the parts of $\CW_1$, $\CW_2$ described just above.  In particular, an embedding of $\CA$ is given by $(\A,h|_{\A})$, where $\A = \C \setminus \wt{A}$ and $\wt{A}$ is the domain bounded by $\eta_1([0,A])$, $\eta_2([0,B])$, $\varphi_1^{-1}(\{1\} \times (0,\pi)$ and $\varphi_2^{-1}(L)$.  This embedding will be considered when working with $\CA$ parameterized by $\A$. We note that $\A$ is homeomorphic to $\C \setminus \ol{\D}$.

Denote by~$\wt \eta_1$ and~$\wt \eta_2$ the curves formed when gluing the boundaries of the part of~$\CW_1$ corresponding to~$h_R^1$ and the part of~$\CW_2$ to the right of the straight line from~$x_2$ to~$y_2$ when parameterized by $\strip$ as above, together according to quantum length. We assume that $\wt \eta_1$ and $\wt \eta_2$ are parameterized by quantum length with the normalization that $\wt \eta_1(0)$ (resp.\ $\wt \eta_2(0)$) corresponds to the image of $(1,0)$ (resp.\ $(1,\pi)$) under the welding homeomorphism, where they are considered as boundary points of $\CW_1$.

Fix $0 < r < s$, $0<u<v$, and some large constant $C>1$ and small constant $D > 0$.  We say that $\CA$ is $(r,s,u,v,C,D)$-stable if the following are true.
\begin{enumerate}[(i)]
\item\label{it:surface_good} $\C \setminus \A$ intersects $\C \setminus B(0,r)$ and is contained in $B(0,s / 2)$.
\item\label{it:curves_good} After time $C$, neither $\wt \eta_1$ nor $\wt \eta_2$ enters $B(0,u)$ and moreover, $\wt \eta_1$ and $\wt \eta_2$ do not enter $\C \setminus B(0,v)$ until after time $C$.
\item\label{it:not_return} $\wt{\eta}_1$ and $\wt{\eta}_2$ do not return to $B(0,sC^{n-1})$ after leaving $B(0,sC^{n})$ for $n = 1,2,3$.
\item\label{it:positive_distance} Let $\sigma_{i}^C = \inf\{ t \geq 0 : \wt{\eta}_i(t) \notin B(0,sC^{3})\}$ and $\wt{\sigma}_i = \inf\{ t \geq 0 : \wt{\eta}_i(t) \notin B(0,s/2)\}$ for $i=1,2$. Then $\dist(\wt{\eta}_1([\wt{\sigma}_1,\sigma_{1}^C]),\wt{\eta}_2([\wt{\sigma}_2,\sigma_{2}^C])) \geq Ds$.
\item\label{it:cut_glue_good} The above items remain true if we consider the following situation.  Let $K$ be a compact hull scaled so that if we let $\psi$ be the unique conformal map from $\A$ to $\C \setminus K$ which fixes and has positive real derivative at $\infty$ then the field $h \circ \psi^{-1} + Q \log | (\psi^{-1})'|$ is a smooth canonical description.
\end{enumerate}
Let $G$ be the event that $\CA$ is $(r,s,u,v,C,D)$-stable.  We expand a bit on what we mean by~\eqref{it:cut_glue_good}. Let $K$ be a compact hull, $\psi$ the conformal map as in~\eqref{it:cut_glue_good}, $\wh{h} = h \circ \psi^{-1} + Q \log|(\psi')^{-1}|$ and we write $\wh{R}_s = (\wh{h}(\cdot s)+Q \log s,\phi)$, where $\phi$ is the radially symmetric text function in the definition of the smooth canonical embedding above. Then,  $K \subseteq B(0,99/100)$ and $\sup \{ s \geq 0: \wh{R}_s = 0 \} = 1$ and condition~\eqref{it:surface_good} is satisfied with $\C \setminus K$ in place of $\C \setminus \A$. Moreover, setting $\wh{\eta}_j = \psi(\wt{\eta}_j)$, conditions~\eqref{it:curves_good}-\eqref{it:positive_distance} are satisfied with $\wh{\eta}_1$ and $\wh{\eta}_2$ in place of $\wt{\eta}_1$ and $\wt{\eta}_2$. We note that $G$ is measurable with respect to~$(A,B,h_R^1)$ and~$h^2$, as $(A,B,h_R^1)$ and $h^2$ determine $\CA$ (see Lemma~\ref{lem:unique_welding} below), which is the reason that we have defined the event in this way. 

\begin{lemma}
\label{lem:good}
For each $\delta \in (0,1)$ there exists $c_1, c_2 > 0$ so that $\p[G] \geq 1-\delta$ provided $r,u ,D\in (0,c_1)$ and $s,v ,C\in (c_2,\infty)$.
\end{lemma} 

The proof of Lemma~\ref{lem:good} is essentially the same as that of \cite[Proposition~9.17]{dms2014mating}.  We will thus not give all of the details but instead describe the main ideas together with precise references to \cite{dms2014mating}.  We begin by noting that it is clear that we can make the part of~$G$ which corresponds to~\eqref{it:surface_good},~\eqref{it:curves_good}, ~\eqref{it:not_return} and ~\eqref{it:positive_distance} occur with probability as close to~$1$ as we want by adjusting the parameters in the definition of~$G$. This is because $\C \setminus \A$ is a bounded set whose interior is homeomorphic to $\D$ and $\wt \eta_1$, $\wt \eta_2$ are simple curves which both tend to $\infty$ as $t \to \infty$ and have positive distance from~$0$. It should be noted that the event~\eqref{it:not_return} is not increasing in $C$ on its own, but rather that the conditional probability of~\eqref{it:not_return}, given the event~\eqref{it:surface_good}, is increasing in $C$. The reason being that the corresponding event is increasing if we replace $\wt{\eta}_1$, $\wt{\eta}_2$ with $\eta_1$, $\eta_2$ and that on the event~\eqref{it:surface_good}, $\wt{\eta}_1$, $\wt{\eta}_2$ consist only of the parts of $\eta_1$, $\eta_2$ which are concerned with the event~\eqref{it:not_return}. The non-trivial part of Lemma~\ref{lem:good} is showing that we can make part~\eqref{it:cut_glue_good} occur with probability as close to~$1$ as we want by adjusting the parameters in the definition, and this part follows from the same argument used to prove \cite[Proposition~9.20]{dms2014mating}.  The essential point is that one can consider the family of test functions which can arise to construct the smooth canonical description after applying the function $\psi$ as in~\eqref{it:cut_glue_good} and see that it is a compact subset of the space of test functions.  This, in particular, implies that the value of~$h$ integrated against any such test function is comparable to the integral of~$h$ against the test function which defines the smooth canonical description for the original surface $h$ (see \cite[Proposition~9.18, 9.19]{dms2014mating} and the discussion immediately after).

Fix $\alpha > 0$.  We now want to show that the statement of Lemma~\ref{lem:good} occurs even if we condition on the event that the average process takes a long time to hit $0$ for the first time after first hitting $\alpha \log \delta$.  Let
\[ Z = \{ \p[ G \giv A, B, h_R^1] \geq 1/4 \}.\]
Then Lemma~\ref{lem:good} implies that by adjusting the parameters in the definition of $G$ we can make $\p[Z]$ as close to $1$ as we want.  Fix $\sigma > 0$.  For each $\delta, \epsilon > 0$ we let $E_{\epsilon,\delta}$ be the event that the average process for $h^1$ takes time at least $\epsilon^{-\sigma}$ to first hit $0$ after first hitting $\alpha \log \delta$.  Our aim is to obtain a lower bound on $\p[Z \giv E_{\epsilon,\delta}]$ which is uniform in $\delta \in (0,1)$ and  $\epsilon$ as $\epsilon \to 0$.  The basic idea of the proof is to obtain uniform (in $\epsilon, \delta$) control on the Radon-Nikodym derivative between the joint law of $(A,B,h_R^1)$ and the joint law of $(A,B,h_R^1)$ conditioned on $E_{\epsilon,\delta}$.  This is carried out in Lemmas~\ref{lem:length_bound}--\ref{lem:q_prob_from_below}.  In particular, in Lemma~\ref{lem:length_bound} we will show that the law of $(A,B)$ given $E_{\epsilon,\delta}$ is tight uniformly in $\delta \in (0,1)$ as $\epsilon \to 0$, in Lemma~\ref{lem:ab_density} we will obtain a lower bound on the conditional density of $(A,B)$ given $h_R^1$, and in Lemma~\ref{lem:q_prob_from_below} we will combine everything to obtain the desired lower bound on $\p[Z \giv E_{\epsilon,\delta}]$.

We begin by proving that $(\CA,\wt \eta_1,\wt \eta_2)$ is a.s.\ determined by $A$, $B$, $h_R^1$ and $h^2$.

\begin{lemma}\label{lem:unique_welding}
The surface $\CA$ decorated by $\wt{\eta}_1$ and $\wt{\eta}_2$, is a.s.\ determined by $A,B,h_R^1$ and $h^2$.
\end{lemma}
\begin{proof}
 Let $h_L^1$ be the part of $h^1$ which is to the left of the vertical line with real part $1$. Then $h_L^1,h_R^1$ together determine $h^1$ so it follows that $h_L^1,h_R^1,A,B$ and $h^2$ determine $\CA$ and $\wt{\eta}_1,\wt{\eta}_2$. In particular, the results of \cite{dms2014mating} imply that there exists a measurable function $F$ defined on the product space corresponding to $(h_L^1,h_R^1,A,B,h^2)$ such that $F(h_L^1,h_R^1,A,B,h^2) = (\CA,\wt{\eta}_1,\wt{\eta}_2)$.
 
Let $\wh{h}^1$ be sampled from the conditional law of $h^1$ given $(h_R^1,A,B)$ and let $(\wh{\CA},\wh{\eta}_1,\wh{\eta}_2)$ be the path decorated surface which corresponds to $\wh{h}_R^1$ (part of $\wh{h}^1$ to the right of the vertical line with real part $1$) and $h^2$. Note that the marginal law of $\wh{h}^1$ is also a $\gamma$-quantum wedge of weight $\frac{\gamma^2}{2}$ and so $(\wh{\CA},\wh{\eta}_1,\wh{\eta}_2)$ can be defined by $(\wh{\CA},\wh{\eta}_1,\wh{\eta}_2) = F(\wh{h}_L^1,\wh{h}_R^1,A,B,h^2)$. Let $U_1$ (resp.\ $U_2$) be the part of $\CA$ which corresponds to $h_R^1$ (resp.\ the part of $h^2$ to the right of the line from $x_2$ to $y_2$) and define $\wh{U}_1$ (resp.\ $\wh{U}_2$) in the same way but in terms of $\wh{\CA}$, $\wh{h}^1$ and $h^2$. Let also $\varphi_j : U_j \rightarrow \strip$ be the conformal transformation such that $\varphi_j(\infty) = +\infty$, $\varphi_j(\wt{\eta}_1(0)) = (1,0)$ and $\varphi_j(\wt{\eta}_2(0)) = (1,\pi)$, for $j=1,2$. We define $\wh{\varphi}_j : \wh{U}_j \rightarrow \strip$ analogously. Set $\psi_j : U_j \rightarrow \wh{U}_j$ with $\psi_j = \wh{\varphi}_j^{-1} \circ \varphi_j$ for $j = 1,2$ and note that $\psi_j(\infty) = \infty$, $\psi(\wt{\eta}_1(0)) = \wh{\eta}_1(0)$ and $\psi_j(\wt{\eta}_2(0)) = \wh{\eta}_2(0)$. Moreover we have that $\nu_{h^1}((-\infty,\varphi_1(x)] \times \{0\}) = \nu_{h^2}((-\infty,\varphi_2(x)] \times \{0\})$ for each $x \in \wt{\eta}_1$ and $\nu_{h^1}((-\infty,\varphi_1(y)] \times \{\pi\}) = \nu_{h^2}((-\infty,\varphi_2(y)] \times \{\pi\})$ for each $y \in \wt{\eta}_2$. Similarly, $\nu_{\wh{h}^1}((-\infty,\wh{\varphi}_1(x)] \times \{0\}) = \nu_{h^2}((-\infty,\wh{\varphi}_2(x)] \times \{0\})$ for each $x \in \wh{\eta}_1$ and $\nu_{\wh{h}^1}((-\infty,\wh{\varphi}_1(y)] \times \{\pi\}) = \nu_{h^2}((-\infty,\wh{\varphi}_2(y)] \times \{\pi\})$ for each $y \in \wh{\eta}_2$. Hence the mapping $\psi : \CA \rightarrow \wh{\CA}$ defined by $\psi(z) = \psi_1(z)$ for $z \in \overline{U}_1 \setminus (\wt{\eta}_1 \cup \wt{\eta}_2)$, $\psi(z) = \psi_2(z)$ for $z \in \overline{U}_2 \setminus (\wt{\eta}_1 \cup \wt{\eta}_2)$ and $\psi(z) = \wh{\varphi}_j^{-1} \circ \varphi_j(z)$ for $z \in \wt{\eta}_1 \cup \wt{\eta}_2$, is a well-defined homeomorphism mapping $\CA$ onto $\wh{\CA}$ which is conformal in $\CA \setminus (\wt{\eta}_1 \cup \wt{\eta}_2)$. We claim that $\psi$ is conformal in $\CA$. Indeed, note that $\eta_1$ has the law of a chordal $\SLE_{\kappa} (\tfrac{\kappa}{2};\tfrac{\kappa}{2})$ curve in $\C \setminus \eta_2$ conditioned on $\eta_2$. Similarly, $\eta_2$ has the law of a chordal $\SLE_{\kappa}(\tfrac{\kappa}{2} ;\tfrac{\kappa}{2})$ curve in $\C \setminus \eta_1$ conditioned on $\eta_1$. By \cite[Lemma~4.1,~Proposition~4.2 and~Lemma~4.3]{mmq2018uniqueness} we have that a chordal $\SLE_{\kappa}$ $\eta$ satisfies the following two conditions if $\kappa \in (0,8)$.
\begin{enumerate}[(I)]
	\item For any compact rectangle $K \subseteq \h$ and any $\beta \in (0,1)$, there exist $M>0$ and $\epsilon_0 > 0$ such that for all $\epsilon \in (0,\epsilon_0)$ and for all $z \in K$, the number of excursions of $\eta$ between $\partial B(z,\epsilon^\beta)$ and $\partial B(z,\epsilon)$ is at most $M$. \label{cond1}
	\item For any compact rectangle $K \subseteq \h$ and any $\alpha > \xi > 1$, there exists $\delta_0 > 0$ such that for any $\delta \in (0,\delta_0)$, for any $t > 0$ such that $\eta(t) \in K$, one can find a point $y$ such that \label{cond2}
	\begin{enumerate}[(a)]
		\item $B(y,\delta^\alpha) \subseteq B(\eta(t),\delta) \setminus \eta$ and $B(y,2\delta^\alpha) \cap \eta \neq \emptyset$.
		\item Let $O$ be the connected component of $B(\eta(t),\delta) \setminus \eta$ that contains $y$ and denote by $\eta(t;\delta)$ the excursion of $\eta$ in $B(\eta(t),\delta)$ which contains $\eta(t)$. For any point $a \in \partial O \setminus \eta(t;\delta)$, any path contained in $O \cup \{a\}$ which connects $y$ to $a$ must exit the ball $B(y,\delta^\xi)$.
	\end{enumerate}
\end{enumerate}
Moreover, chordal $\SLE_{\kappa}(\tfrac{\kappa}{2};\tfrac{\kappa}{2})$ restricted to any interval of time $[s,t]$ for $0 < s < t < \infty$ is absolutely continuous with respect to a chordal $\SLE_{\kappa}$ on the same interval. Furthermore, such a process is a.s.\ transient, so it is easy to see that~\eqref{cond1} and~\eqref{cond2} are satisfied for $\eta_1$ given $\eta_2$ and vice-versa. Thus both of $\eta_1$ and $\eta_2$ satisfy~\eqref{cond1} and~\eqref{cond2}. Let $\phi : \CA \setminus \wt{\eta}_2  \rightarrow \h$ be the conformal transformation mapping $\CA \setminus \wt{\eta}_2$ onto $\h$ and such that $\phi(\wt{\eta}_1(0)) = 0$, $\phi(\infty) = \infty$. Similarly, let $\wh{\phi} : \wh{\CA} \setminus \wh{\eta}_2 \rightarrow \h$ be the conformal transformation mapping $\wh{\CA} \setminus \wh{\eta}_2$ onto $\h$ and such that $\wh{\phi}(\wh{\eta}_1(0)) =0$, $\wh{\phi}(\infty) = \infty$. Then $g = \wh{\phi} \circ \psi \circ \phi^{-1}$ is a homeomorphism mapping $\h$ onto $\h$ which is conformal in $\h \setminus \phi(\wt{\eta}_1)$ and fixes $0$ and $\infty$. Clearly, both of $\wt{\eta}_1$ and $\wh{\eta}_1$  satisfy~\eqref{cond1} and~\eqref{cond2} when viewed as paths in $\CA \setminus \wt{\eta}_2$ and $\wh{\CA} \setminus \wh{\eta}_2$ respectively, since they are both parts of chordal $\SLE_{\kappa}(\tfrac{\kappa}{2};\tfrac{\kappa}{2})$ in certain domains. Since the property of satisfying~\eqref{cond1} and~\eqref{cond2} is preserved under conformal transformations, we obtain that they are both satisfied by $\phi(\wt{\eta}_1)$ and $g(\phi(\wt{\eta}_1)) = \wh{\phi}(\wh{\eta}_1)$. Therefore \cite[Theorem~1.2]{mmq2018uniqueness} implies that $g$ is conformal in $\h$. Since $\psi = \wh{\phi}^{-1} \circ g \circ \phi$ when restricted to $\CA \setminus \wt{\eta}_2$, it follows that $\psi$ is conformal in $\CA \setminus \wt{\eta}_2$ and a similar argument shows that $\psi$ is conformal in $\CA$. Therefore, $\CA$ and $\wh{\CA}$ are equivalent as path-decorated quantum surfaces. In particular, we have $F(h_L^1,h_R^1,A,B,h^2) = F(\wh{h}_L^1,h_R^1,A,B,h^2)$ a.s.\ and so $f(F(h_L^1,h_R^1,A,B,h^2)) = f(F(\wh{h}_L^1,h_R^1,A,B,h^2))$ for each bounded and measurable function $f$. This implies that the conditional variance of $f(F(h_L^1,h_R^1,A,B,h^2))$ given $(h_R^1,A,B,h^2)$ is zero. Therefore, 
\begin{align*}
f(F(h_L^1,h_R^1,A,B,h^2)) = \E[ f(F(h_L^1,h_R^1,A,B,h^2)) \,|\,h_R^1,A,B,h^2]\quad\text{a.s.}
\end{align*}
Since $f$ was arbitrary, it follows that $F(h_L^1,h_R^1,A,B,h^2)$ does not depend on $h_L^1$. Hence $(\CA,\wt{\eta}_1,\wt{\eta}_2)$ is determined by $(h_R^1,A,B,h^2)$ a.s. This completes the proof.
\end{proof}

We now turn to the tightness of law of $(A,B)$ given $E_{\epsilon,\delta}$.
\begin{lemma}\label{lem:length_bound}
For every $q \in (0,1)$ there exists $c > 0$ such that for every $\delta \in (0,1)$ we have that
\begin{align*}
	\p[ \max( \nu_{h^1}((-\infty,0]\times\{0\}),\nu_{h^1}((-\infty,0]\times\{\pi\}) ) \leq c \, | \, E_{\epsilon,\delta}] \geq 1-q.
\end{align*}
for all $\epsilon$ sufficiently small.
\end{lemma}
\begin{proof}
For each $t \in \R$ we let $X_t^1$ be the average of $h^1$ on the vertical line $t + (0, i\pi)$.  Let $T_{\epsilon,\delta} = \inf\{ t \geq u_{\alpha,\delta}^1: X_t^1 = - \wt c \epsilon^{-\sigma/2} \}$ and $A_{\epsilon,\delta} = \{ T_{\epsilon,\delta} < 0 \}$, where $\wt c > 0$ is a fixed constant to be determined later and depending only on $q$. Then, by Gambler's ruin, $\p[A_{\epsilon,\delta}] = \alpha \epsilon^{\sigma/2} \log(1/\delta)/ \wt c$. By the form of the density of the first passage time of a Brownian motion, we have that $\p[E_{\epsilon,\delta}] \asymp \alpha \epsilon^{\sigma/2} \log(1/\delta)$ and $\p[E_{\epsilon,\delta} \, | \, A_{\epsilon,\delta}] \asymp 1$ uniformly in $\delta$ and $\epsilon$ small enough and the implicit constants are universal. In particular, there exists $d > 0$ depending only on $\wt c$ such that $\p[ E_{\epsilon,\delta} \,|\, A_{\epsilon,\delta} ] \geq d$ for all $\epsilon$ sufficiently small. Furthermore, we have that
\begin{align*}
\E[ \nu_{h^1}((-\infty,u_{\alpha,\delta}^1] \times \{0\})^p \, | \, A_{\epsilon,\delta} \cap E_{\epsilon,\delta} ] = \E[ \nu_{h^1}((-\infty,u_{\alpha,\delta}^1] \times \{0\})^p]
\end{align*}
since $A_{\epsilon,\delta} \cap E_{\epsilon,\delta}$ is determined by $(X_{t}^1)_{t > u_{\alpha,\delta}^1}$ and independent of $(X_t^1)_{t \leq u_{\alpha,\delta}^1}$ and $h_2^1$, and the latter two determine $\nu_{h^1}((-\infty,u_{\alpha,\delta}^1] \times \{0\})$. Consequently, for $0<p<1$, 
\begin{align*}
	&\E[ \nu_{h^1}((-\infty,0] \times \{0\})^p \, | \, A_{\epsilon,\delta} \cap E_{\epsilon,\delta}] \\
	&\leq \E[ \nu_{h^1}((-\infty,u_{\alpha,\delta}^1] \times \{0\})^p] + \E[ \nu_{h^1}((u_{\alpha,\delta}^1,0] \times \{0\})^p \, | \, A_{\epsilon,\delta} \cap E_{\epsilon,\delta}] \\
		&= \wt c_p \delta^{\alpha p \gamma/2} + \E[ \nu_{h^1}((u_{\alpha,\delta}^1,0] \times \{0\})^p \, | \, A_{\epsilon,\delta} \cap E_{\epsilon,\delta}].
\end{align*}
Moreover, the conditional law of $(-X_t^1)|_{[u_{\alpha,\delta}^1,T_{\epsilon,\delta}]}$ given $A_{\epsilon,\delta}$ is that of a $\BES^3$ process started from the point $\alpha \log(1/\delta)$ run at twice the speed until hitting the value $\wt c \epsilon^{-\sigma/2}$. Thus, the law of $\nu_{h^1}([u_{\alpha,\delta}^1,T_{\epsilon,\delta}] \times \{0\})$ conditioned on $A_{\epsilon,\delta}$ is the same as the (unconditional) law of $\nu_{h^1}([-T_{\epsilon,\delta}^2 , -T_{\delta}^1] \times \{0\})$ where
\begin{align*}
&T_{\delta}^1 = \inf\{ t \geq 0 : X_{-t}^1 = \alpha \log (\delta) \}, \\
&T_{\epsilon,\delta}^2 = \inf\{ t \geq T_{\delta}^1 : X_{-t}^1 = -\wt{c} \epsilon^{-\sigma / 2} \}.
\end{align*}
Furthermore, since $0 < T_\delta^1 < T_{\epsilon,\delta}^2 < \infty$ a.s., it follows that $\nu_{h^1}([-T_{\epsilon,\delta}^2,-T_\delta^1] \times \{0\})$ is stochastically dominated from above by $\nu_{h^1}((-\infty,0] \times \{0\})$. Hence, since $\p[ E_{\epsilon,\delta} \, | \, A_{\epsilon,\delta} ] \asymp 1$, we obtain that
\begin{align*}
\E[ \nu_{h^1}((u_{\alpha,\epsilon}^1 , T_{\epsilon,\delta}] \times \{0\})^p \, | \, A_{\epsilon,\delta} \cap E_{\epsilon,\delta} ] \lesssim \E[ \nu_{h^1}((-\infty,0] \times \{0\})^p]
\end{align*}
where the implicit constant is independent of $\epsilon$ and $\delta$. Thus, we obtain that
\begin{align*}
	 &\E[ \nu_{h^1}((u_{\alpha,\delta}^1,0] \times \{0\})^p \, | \, A_{\epsilon,\delta} \cap E_{\epsilon,\delta}] \\
	 &\leq \E[ \nu_{h^1}((u_{\alpha,\delta}^1,T_{\epsilon,\delta}] \times \{0\})^p \, | \, A_{\epsilon,\delta} \cap E_{\epsilon,\delta}] + \E[ \nu_{h^1}((T_{\epsilon,\delta},0] \times \{0\})^p \, | \, A_{\epsilon,\delta} \cap E_{\epsilon,\delta}] \\
	 &\lesssim \E[ \nu_{h^1}((-\infty,0] \times \{0\})^p] + \E[ \nu_{h^1}((T_{\epsilon,\delta},0] \times \{0\})^p \, | \, A_{\epsilon,\delta} \cap E_{\epsilon,\delta}] \\
	 &\leq \wt c_p + \E[ \nu_{h^1}((T_{\epsilon,\delta},0] \times \{0\})^p \, | \, A_{\epsilon,\delta} \cap E_{\epsilon,\delta}].
\end{align*}
Note that the conditional law of $\nu_{h^1}((T_{\epsilon,\delta},0] \times \{0\})$ given $A_{\epsilon,\delta}$ is the same as that of $\nu_{h^1}((\tau_{\epsilon},0]\times\{0\})$, where $\tau_\epsilon = \inf\{ t \in \R: X_t^1 = -\wt c \epsilon^{-\sigma/2} \}$. Consequently, since $\p[ E_{\epsilon,\delta} \, | \, A_{\epsilon,\delta}] \asymp 1$, we have that
\begin{align*}
	\E[ \nu_{h^1}((T_{\epsilon,\delta},0] \times \{0\})^p \, | \, A_{\epsilon,\delta} \cap E_{\epsilon,\delta}] \lesssim \E[ \nu_{h^1}( (-\infty,0] \times \{0\})^p] \lesssim 1,
\end{align*}
where the implicit constant is uniform in $\epsilon$ small enough. Therefore,
\begin{align*}
	\E[ \nu_{h^1}((-\infty,0] \times \{0\})^p \, | \, A_{\epsilon,\delta} \cap E_{\epsilon,\delta}] \leq \wt C,
\end{align*}
where $\wt C$ is independent of $\epsilon$, $\delta$ and depends only on $\wt c$. The same holds with $\nu_{h^1}((-\infty,0] \times \{0\})$ replaced by $\nu_{h^1}((-\infty,0] \times \{\pi\})$.  Note that
\begin{align*}
&\p[ \max(\nu_{h^1}((-\infty,0] \times \{0\}), \nu_{h^1}((-\infty,0] \times \{\pi\})) \leq c \,|\,E_{\epsilon,\delta} ]\\
& \geq \p[ \max(\nu_{h^1}((-\infty,0] \times \{0\}), \nu_{h^1}((-\infty,0] \times \{\pi\}))\leq c \,|\, E_{\epsilon,\delta} \cap A_{\epsilon,\delta} ] \p[ A_{\epsilon,\delta} \,|\, E_{\epsilon,\delta}]
\end{align*}
and so Markov's inequality implies that
\begin{align}\label{eq:1}
\p[ \max(\nu_{h^1}((-\infty,0] \times \{0\}), \nu_{h^1}((-\infty,0] \times \{\pi\})) \leq c \,|\, E_{\epsilon,\delta} ] \geq (1 - 2\wt C c^{-p}) \p[ A_{\epsilon,\delta} \,|\, E_{\epsilon,\delta}],
\end{align}
for sufficiently small $\epsilon > 0$. Hence we need to find a lower bound on $\p[ A_{\epsilon,\delta} \,|\, E_{\epsilon,\delta}]$. To do this, we start the average process at $u_{\alpha,\delta}^1$ and run it until the first time it hits $0$. Then it is a Brownian motion run at twice the speed and stopped the first time it hits $0$. Let $E_{\epsilon,\delta}^1$ be the event that $X^1$ takes time at least $\wt c^2 \epsilon^{-\sigma}$ to hit $0$ for the first time. Then $\p[ E_{\epsilon,\delta}^1]$ is of order $\wt c ^{-1} \epsilon^{\sigma / 2} \log \delta^{-1}$ and $E_{\epsilon,\delta} \subseteq E_{\epsilon,\delta}^1$ when $\wt c \in (0,1)$. Then $\p[ A_{\epsilon,\delta}^c , E_{\epsilon,\delta} ] = \p[ A_{\epsilon,\delta}^c, E_{\epsilon,\delta} \,|\,E_{\epsilon,\delta}^1 ] \p[E_{\epsilon,\delta}^1]$. Also we note that if we start time at $\wt c^2 \epsilon^{-\sigma}$ (so we again have a Brownian motion), then for $E_{\epsilon,\delta}$ we have another $(1-\wt c^2) \epsilon^{-\sigma}$ units of time to finish. We can break this into of order $\wt c^{-2}$ rounds of length $\wt c^2 \epsilon^{-\sigma}$. In each such round, the Brownian motion has a positive probability of exiting the interval $[-\wt c \epsilon^{-\sigma / 2} , 0]$. Therefore
\begin{align*}
\p[ A_{\epsilon,\delta}^c,E_{\epsilon,\delta} \,|\,E_{\epsilon,\delta}^1] \leq e^{-c_1/ \wt c^2}
\end{align*}
for some constant $c_1 > 0$. So altogether
\begin{align}\label{eq:2}
\p[ A_{\epsilon,\delta}^c \,|\,E_{\epsilon,\delta} ] &= \p[ A_{\epsilon,\delta}^c,E_{\epsilon,\delta}] / \p[ E_{\epsilon,\delta}]\leq e^{-c_2/ \wt c^2}
\end{align}
for some universal constant $c_2 >0$. Therefore by~\eqref{eq:1} and~\eqref{eq:2}, if we first pick $\wt c \in (0,1)$ such that $1 - e^{-c_2/ \wt c^2} \geq \sqrt{1-q}$ and then pick $c>0$ such that $(1-2\wt C c^{-p}) \geq \sqrt{1-q}$, we complete the proof.
\end{proof}

With Lemma~\ref{lem:length_bound} at hand, we can lower bound the conditional density of $(A,B)$ given $E_{\epsilon,\delta}$ and $h_R^1$.
\begin{lemma}
\label{lem:ab_density}
Let $f_{\epsilon,\delta}(a,b \,| \,h_R^1)$ be the conditional density of $(A,B)$ given $E_{\epsilon,\delta}, h_R^1$. Then for each $\xi \in (0,1)$ there exist finite constants $c > 0$ and $0 < a_1 < a_2$, $0 < b_1 < b_2$ depending only on $\xi$ and $\gamma$ such that for all $\delta \in (0,1)$ and all $\epsilon > 0$ sufficiently small, with probability at least $1 - \xi$ we have that
\begin{align*}
f_{\epsilon,\delta}(a,b \,|\, h_{R}^{1}) \geq c \quad \text{for all} \quad (a,b) \in [a_1,a_2] \times [b_1,b_2]
\end{align*}
\end{lemma}

We remark that the conditional law of $h_R^1$ given $E_{\epsilon,\delta}$ is equal to its unconditioned law.  The statement of Lemma~\ref{lem:ab_density} is a statement about the conditional density $f_{\epsilon,\delta}(a,b \,|\, h_R^1)$ of $(A,B)$ given $E_{\epsilon,\delta}$ and $h_R^1$, which is a function of $h_R^1$.  The probability that we have in mind for $h_R^1$ is the unconditioned law of $h_R^1$, however there is actually no distinction.

\begin{proof}[Proof of Lemma~\ref{lem:ab_density}]
Let $\phi_1$, $\phi_2$ be $C^\infty$ with $\| \phi_1 \|_\nabla = \| \phi_2 \|_\nabla = 1$, which have disjoint support contained in $[0,1] \times [0,\pi]$.  We further assume that $\phi_1 \geq 0$ on $[0,1] \times \{0\}$, $\phi_1 > 0$ on $[1/4,3/4] \times \{0\}$, and $\phi_1$ has mean-zero on vertical lines, i.e., $\phi_1 \in H_2(\strip)$.  We assume that the same holds for $\phi_2$ but with $[0,1] \times \{0\}$ and $[1/4,3/4] \times \{0\}$ replaced by $[0,1] \times \{\pi\}$ and $[1/4,3/4] \times \{\pi\}$.  Note that we can write $h^1 = \alpha_1 \phi_1 + \alpha_2 \phi_2 + \wt{h}$ where $\alpha_1,\alpha_2$ are independent $N(0,1)$ (since we have $(\phi_1,\phi_2)_\nabla = 0$) and $\wt{h}$ is independent of $\alpha_1,\alpha_2$.  Moreover, $h_R^1$ and $E_{\epsilon,\delta}$ are determined by $\wt{h}$.  Then we have that
\[ A = \nu_{h^1}((-\infty,1] \times \{0\}) = \nu_{h^1}((-\infty,0] \times \{0\}) + \int_0^1 e^{\alpha_1 \gamma \phi_1(x)/2} d \nu_{\wt{h}}(x).\]
In particular, the conditional law of $A$ given $\wt{h}$ is given by applying the function
\[ F_1(\alpha) = \nu_{h^1}((-\infty,0] \times \{0\}) + \int_0^1 e^{\alpha \gamma \phi_1(x)/2} d \nu_{\wt{h}}(x)\]
to a $N(0,1)$ random variable (independent of $\wt{h}$).  We note that
\[ F_1'(\alpha) = \int_0^1 \frac{\gamma}{2} \phi_1(x) e^{\alpha \gamma \phi_1(x)/2} d\nu_{\wt{h}}(x).\]
The law of $B$ admits a similar expression, say in terms of the function $F_2$.  Note that given $\wt{h}$ we have that $A,B$ are independent with conditional densities given by $\rho_1(\cdot \,|\, \wt h)$ and $\rho_2(\cdot \,|\, \wt h)$ respectively, where $\rho_1(x \,|\, \wt h) = 0$ if $x < \nu_{h^1}((-\infty,0] \times \{0\})$ and
\[ \rho_1(x \,|\, \wt h) = (F_1^{-1})'(x) \frac{\exp(-F_1^{-1}(x)^2 / 2)}{\sqrt{2\pi}} \quad\text{for}\quad x \geq \nu_{h^{1}}((-\infty,0] \times \{0\}).\]
We define $\rho_2(\cdot \,|\, \wt h)$ similarly but with $F_1$ and $\nu_{h^1}((-\infty,0] \times \{0\})$ replaced by $F_2$ and $\nu_{h^1}((-\infty,0] \times \{\pi\})$.  By the conditional independence of $A,B$ given $\wt{h}$, we have that the conditional density of $(A,B)$ given $\wt{h}$ is $\rho(x,y \,|\, \wt h) = \rho_1(x \,|\, \wt h) \rho_2(y \,|\, \wt h)$.

Next we fix $\xi \in (0,1)$ and for $c,d,M > 0$ (to be determined) depending only on $\xi$ we consider the events $V = V_1 \cap V_2$  and $F = \{ \E[ \one_V \,|\,E_{\epsilon,\delta},h_{R}^1 ] \geq 1 - \xi \}$, where 
\begin{align*}
V_1 &= \{ \max(\nu_{h^1}((-\infty,0] \times \{0\}),\ \nu_{h^1}((-\infty,0] \times \{\pi\}) ) \leq d \} \quad\text{and}\\
V_2 &= \{ [d,2d] \subseteq F_j([-M,M]),\ F_j'|_{[-M,M]} \leq c \  \text{for} \  j=1,2\}.
\end{align*}
Note that since $\max(\nu_{h^1}((-\infty,0] \times \{0\}) , \nu_{h^1}((-\infty,0] \times \{\pi\})) < \infty$, $F_j$ is continuous and $\lim_{x \to -\infty}F_1(x) = \nu_{h^1}((-\infty,0] \times \{ 0 \})$, $\lim_{x \to -\infty}F_2(x) = \nu_{h^1}((-\infty,0] \times \{ \pi \})$ and $\lim_{x \to +\infty}F_j(x) = +\infty$ for $j = 1,2$ a.s., we obtain that $\p[V]$ can be made to be sufficiently close to $1$ by choosing $c,d,M$ appropriately. Therefore
\begin{align*}
   \p[ F^c ]
&= \p[ \E[ \one_V \,|\, E_{\epsilon,\delta}, h_R^1] \leq 1-\xi ]
 = \p[ \E[ \one_{V^c} \,|\, E_{\epsilon,\delta}, h_R^1] \geq \xi ]\\
&\leq \frac{\E[\E[\one_{V^c} \,|\, E_{\epsilon,\delta}, h_R^1]]}{\xi}
  = \frac{\p[V^c \, | \, E_{\epsilon,\delta}]}{\xi}.
\end{align*}
In the final equality above, we have used that law of $h_R^1$ is equal to the law of $h_R^1$ conditioned on $E_{\epsilon,\delta}$.  Lemma~\ref{lem:length_bound} then implies that the right hand side can be made to be at most $\xi$ by choosing $c,d,M > 0$ sufficiently large.  Altogether, for this choice we obtain that $\p[ F ] \geq 1 - \xi$ uniformly in $\delta \in (0,1)$ and as $\epsilon \to 0$.  Observe that on $V$ we have that $\rho_j(x \,|\, \wt h) \geq e^{-M^2 / 2} / ( c\sqrt{2\pi} )$ for all $x \in [d,2d]$ and $j=1,2$, hence $\rho(x,y \,|\,\wt h) \geq e^{-M^2} / ( c^2 2\pi )$ for all $x,y \in [d,2d]$. Note also that $f_{\epsilon,\delta}(a,b \,|\, h_{R}^1) = \E[ \rho(a,b \,|\, \wt h)\,|\,E_{\epsilon,\delta},h_{R}^1]$ for all $a,b > 0$ a.s. Thus we obtain that on $F$,
\begin{align*}
f_{\epsilon,\delta}(a,b\,|\, h_{R}^1) \geq \E[ \rho(a,b \,|\,\wt h)\one_V \,|\,E_{\epsilon,\delta},h_{R}^1] \geq \frac{e^{-M^2}}{c^2 2\pi} \E[ \one_V \,|\,E_{\epsilon,\delta},h_R^1] \geq \frac{(1-\xi) e^{-M^2}}{c^2 2\pi}
\end{align*}
for all $a,b \in [d,2d]$. This completes the proof.
\end{proof}

\begin{lemma}
\label{lem:q_prob_from_below}
There exists $p_0 \in (0,1)$ and $0 < r < s$, $0 < u < v$ and $C > 1>D>0$ so that $\p[ Z \giv E_{\epsilon,\delta}] \geq p_0$ for all $\delta \in (0,1)$ and all $\epsilon > 0$ sufficiently small.
\end{lemma}
\begin{proof}
Let $\CZ$ be the Radon-Nikodym derivative between the joint law of $(A,B,h_R^1)$ given $E_{\epsilon,\delta}$ with respect to the unconditioned law of $(A,B,h_R^1)$.  Then we have that
\begin{align*}
      \p[ Z \giv E_{\epsilon,\delta}]
&=  \E[ \one_{\{ \p[ G \giv A,B,h_R^1] \geq 1/4 \}} \CZ(A,B,h_R^1)].
\end{align*}
Moreover,
\begin{align*}
\CZ (A,B,h_{R}^1) = \frac{f_{\epsilon,\delta}(A,B \,|\, h_{R}^1)}{g(A,B \,|\, h_{R}^1)},
\end{align*}
where $f_{\epsilon,\delta}(\cdot \,|\, h_{R}^1)$ is as in Lemma~\ref{lem:ab_density} and $g(\cdot \,|\,h_{R}^1)$ is the conditional density of the law of $(A,B)$ given~$h_{R}^1$. Let $H$ be the event that 
\begin{align*}
f_{\epsilon,\delta}(a,b \,|\, h_{R}^1) \geq c \quad \text{for all} \quad (a,b) \in [a_1,a_2] \times [b_1,b_2]
\end{align*}
where $a_1,a_2,b_1,b_2,c$ are as in Lemma~\ref{lem:ab_density} for $\xi= \frac{1}{2}$. Note that $H$ is determined by $h_{R}^1$. Fix $M > 0$ large such that $M^{-1} \leq (a_2-a_1)(b_2-b_1)/2$ and set 
\begin{align*}
X = \{ (a,b) \in [a_1,a_2] \times [b_1,b_2]: g(a,b \,|\,h_{R}^1) \leq M\}.
\end{align*}
Then
\begin{align*}
\p[ Z \,|\,E_{\epsilon,\delta}] &\geq \frac{1}{M}\E[ \one_Z \one_H \one_{\{(A,B) \in X\}} f_{\epsilon,\delta}(A,B \,|\,  h_{R}^1)] \geq \frac{c}{M} \E[ \one_Z \one_H \one_{\{(A,B) \in X\}} ], 
\end{align*}
so it suffices to give a lower bound on $\E[ \one_Z \one_H \one_{\{(A,B) \in X \}}]$. Fix $\xi \in (0,1)$ independent of $\delta,\epsilon$ (to be chosen later) and assume that we have chosen the parameters for $G$ so that, by Lemma~\ref{lem:good}, $\p[ Z ] \geq 1 - \xi$. Then we have
\begin{align*}
\E[ \one_Z \one_H \one_{\{(A,B) \in X\}}] \geq \E[ \one_H \one_{\{(A,B) \in X \}}] - \xi.
\end{align*}
Since $X$ is determined by $h_R^1$, we have
\begin{align*}
\E[ \one_H \one_{\{(A,B) \in X \}}] = \E\!\left[ \one_H \int_{X} g(a,b \,|\,h_{R}^1) da db \right].
\end{align*}
Note that $Y = \int_{X} g(a,b \,|\,h_{R}^1) da db$ does not depend on $\epsilon$ and $Y > 0$ a.s.\ since $\text{Leb}(X) \geq (a_2-a_1)(b_2-b_1) / 2$ and $g(a,b \, |\, h_R^1) > 0$ for a.e.\ $a,b$ a.s.\ on $H$. Fix $q > 0$ and let $Z_q = \{ Y > q \}$. By making $q > 0$ sufficiently small we can assume that $\p[ Z_q ] \geq 3/4$.  Hence,
\begin{align*}
\E[ Y \one_H ] - \xi &\geq q \p[ Z_q \cap H] - \xi
 \geq q / 4 -\xi \geq q/8
\end{align*}
for $\xi > 0$ sufficiently small. Note that $q > 0$ is independent of $\delta,\epsilon$ and so this completes the proof.
\end{proof}

\begin{remark}
We note that Lemmas~\ref{lem:qtimeUB} and~\ref{lem:LBcone}  hold with the smooth canonical description in place of the circle-average embedding. This will be used in what follows.
\end{remark}

\begin{lemma}
\label{lem:lbd}
For any $C > 1$ (as in the definition of the $(r,s,u,v,C,D)$-stability of the surface $\CA$), there exists $b > 0$ such that the following holds. Fix $\xi \in (0,1)$. Then there exist $\delta_0 \in (0,1)$ and $C_\xi > 0$ depending only on $\xi$ such that the following holds. Let $\wh \CW$ denote the surface parameterized by the part of $h^1$ which is to the left of the line $\{z: \re(z) = u_{b,\delta}^1 \}$ and by $\CW_2$. For $\delta \in (0,\delta_0)$ let $H_{C,C_\xi,b,\delta}$ be the event that with probability at least $\exp(-C_\xi b^2 (\log\delta)^2)$, a Brownian motion starting from $0$ exits $\wh \CW$ in the part of the boundary of $\CW_2$ which has boundary length distance at least $C$ from either $\wt \eta_1(0)$ or $\wt \eta_2(0)$. Then $\p[ H_{C,C_\xi,b,\delta} ] \geq 1 - \xi$.
\end{lemma}
\begin{proof}
The claim will be deduced in three main steps. In \noindent{\it Step 1,} we show that with high probability a Brownian motion starting from $0$ exits the part contained in $\wh \CW$ of a ball of radius some fixed power of $\delta$ with positive probability which is uniform in $\delta$. In \noindent{\it Step 2,} we show that the ball of \noindent{\it Step 1} is contained in $\wh \CW$ and that its image under $\varphi_j$ is not mapped too far to the right in the strip. In \noindent{\it Step 3,} we conclude the proof by showing that conditioned on the above, the Brownian motion hits a point in $\wh \CW$ with positive probability which is uniform in $\delta$ whose image under $\varphi_2$ is not mapped too far either to the left or right in the strip. Then the claim is deduced since the images of $\wt \eta_1([C,\infty))$ and $\wt \eta_2([C,\infty))$ under $\varphi_2$ are not mapped too far to the right in the strip.

\noindent{\it Step 1.} Fix $\xi \in (0,1)$ and $\alpha > 0$. For $\delta \in (0,1]$ we let $I_\delta$ be the unique arc of $\partial B(0,\delta)$ contained in the boundary of the unbounded connected component of $U_2 \setminus B(0,\delta)$ which separates $0$ from $\infty$ in $U_2$ and let $x_\delta$ be the center of $I_\delta$. Let also $J_\delta$ be the subarc of $I_\delta$ centered at $x_\delta$ and having half of the length of $I_\delta$. Then we can find constants $q \in (0,1),l,d > 0$ such that with probability at least $1-\xi / 3$, we have that $\dist(J_1,\partial U_2) \geq d$, the length of $I_1$ is at least $l$ and for each $z \in J_1$, the harmonic measure of both of $\eta_1$ and $\eta_2$ as seen from $z$ is at least $q$. Hence by scale invariance of the joint law of $(\eta_1,\eta_2)$ we obtain that for each $\delta \in (0,1)$ and for some $c_1,c_2 > 0$, with probability at least $1-\xi / 3$  the following hold.
\begin{itemize}
\item The length of $I_\delta$ is at least $\delta l$.
\item $\dist(J_\delta, \partial U_2) \geq d \delta$.
\item For each $z \in J_\delta$ the harmonic measure of both of $\eta_1$ and $\eta_2$ as seen from $z$ is at least $q$.
\item $\re(\varphi_2(\wt \eta_j(C))) \in [-c_2,c_2]$ for $j = 1,2$.
\item $|u_{\alpha,\delta}^j| \leq c_1 (\log\delta)^2$ for $j = 1,2$.
\end{itemize}
Let $F_1$ be the event that the above hold. 

\noindent{\it Step 2.} We let $\beta > 0$, $\zeta_0 \in (0,1)$ be such that $\beta \in (0,\alpha / \gamma)$ and $\frac{\beta}{\zeta_0} > \wt \alpha$, where $\wt \alpha > 2 $ is such that 
\begin{align}\label{eq:lower_bound_mass}
\p[ \mu_h(B(w,\delta)) \geq \delta^{\wt \alpha} \,\text{for each}\,w \in B(0,1/2)] = 1 - O(\delta^2)\quad \text{as}\quad \delta \to 0,
\end{align}
where the implicit constant depends only on $\gamma$. (The existence of such an $\wt{\alpha}$ follows by applying Lemma~\ref{lem:LBcone} to each point in a grid in $B(0,1/2)$ and taking a union bound.) Fix $0<\zeta_1 < \zeta < \wt{\zeta}  < \zeta_0 < 1$ and set $S_{\delta,j} = \sup \{t \geq 0 : \eta_j(t) \in B(0,\delta) \}$ and $S_{\delta} = S_{\delta,1} \vee S_{\delta,2}$. Then Lemma~\ref{lem:qtimeUB}
 implies that $\p[S_{\delta^{\zeta_1}}\geq \delta^{\zeta_1 p}] = O(\delta^{\zeta_1 q})$ as $\delta \to 0$, for some constants $p,q > 0$. We fix $\zeta_2 \in (0,\zeta_1 p)$. Next we let $\wh T^{\delta,j} = \inf\{t \in \R : \nu_{h^j}((-\infty,t] \times \{0\}) = \delta^{\zeta_2} \ \text{or} \ \nu_{h^j}((-\infty,t] \times \{\pi\}) = \delta^{\zeta_2} \}$ and note that since $\nu_{h^j}((-\infty,0] \times \{0,\pi\})$ has negative moments of all orders, Markov's inequality implies that $\p[\wh T^{\delta,j}\geq u^j_{b_1,\delta}] = O(\delta^{p(\zeta_2 - b_1)})$ for some fixed $b_1 \in (0,\zeta_2)$. Moreover $2(u^j_{b_1/2,\delta}-u^j_{b_1,\delta})$ has the law of the first hitting time of $(b_1/2)\log(1/\delta)$ for a Brownian motion starting from 0. Thus, a short calculation yields that $\p[2(u^j_{b_1/2,\delta}-u^j_{b_1,\delta}) \leq 2 \log(1/\delta)] = O(\delta^c)$ for some fixed $c > 0$. Combining everything we obtain that for some $c > 0$, off an event with probability $O(\delta^c)$, the following hold:
\begin{align*}
        S_{\delta^{\zeta_1}} < \delta^{\zeta_1 p}, \quad
        \wh T^{\delta,j} < u^j_{b_1,\delta}, \quad
        u^j_{b,\delta} - u^j_{b_1,\delta} \geq 2\log(1/\delta),
\end{align*}
with $b = b_1 /2$. Note that on the above event, $\varphi_j(\eta_1([0,S_{\delta^{\zeta_1}}]) \cup \eta_2([0,S_{\delta^{\zeta_1}}]))\subseteq{(-\infty,\wh T^{\delta,j}] \times \{0,\pi\}}$. Suppose that there exists $z \in  \varphi_j(U_j \cap B(0,\delta^{\zeta})) \cap (\strip_++u_{b,\delta}^j)$. Then there exists a universal constant $r > 0$ such that with probability at least $r > 0$ a Brownian motion starting from $z$ exits $\strip$ on $[u^j_{b_1,\delta},\infty) \times \{0,\pi\}$. Hence a Brownian motion starting from $\varphi_j^{-1}(z)$ exits $U_j$ on $\eta_1([S_{\delta^{\zeta_1}},\infty)) \cup \eta_2([S_{\delta^{\zeta_1}},\infty))$ with probability at least $r$. The Beurling estimate implies that the latter occurs with probability at most $O(\delta^{\frac{\zeta - \zeta_1}{2}})$ and so we get a contradiction for $\delta$ sufficiently small. Hence
\begin{align*}
\varphi_j(U_j \cap B(0,\delta^{\zeta})) \subseteq \strip_- + u_{b,\delta}^j.
\end{align*}
Also, we have that $B(0,\delta^{\zeta_1}) \subseteq \wh \CW$. Moreover, by~\eqref{eq:lower_bound_mass},
\begin{align*}
\p[ \mu_h(B(w,\delta^{\wt{\zeta}})) \leq \delta^{\wt{\zeta} \wt{\alpha}}\,\text{for some}\,w \in B(0,1/2)] = O(\delta^{2\wt{\zeta}})\quad \text{as}
\quad \delta \to 0
\end{align*}
and by Proposition~\ref{lem:pthmoment} and Markov's inequality it follows that
\begin{align*}
\p[ \mu_{h^j}((-\infty,u_{\alpha,\delta}^j] \times (0,\pi)) \geq \delta^{\beta}] \lesssim \delta^{p(\alpha \gamma - \beta)}.
\end{align*}
By combining everything, we obtain that we can find constants $b,c > 0$ such that off an event with probability $O(\delta^c)$, the following hold:
\begin{itemize}
\item $\varphi_j(U_j \cap B(0,\delta^{\zeta})) \subseteq \strip _- + u_{b,\delta}^j$,
\item $B(0,\delta^{\zeta_1}) \subseteq \wh \CW$,
\item $\mu_h(B(w,\delta^{\wt{\zeta}})) \geq \delta^{\wt{\zeta} \wt{\alpha}}\,\text{for all}\,w \in B(0,1/2)$,
\item $\mu_{h^j}((-\infty,u_{\alpha,\delta}^j] \times (0,\pi)) \leq \delta^{\beta}$.
\end{itemize}
Let $F_2$ be the event that the above hold and set $F = F_1 \cap F_2$. 

\noindent{\it Step 3.} By choosing $\delta$ sufficiently small we have that $\p[F] \geq 1 - \xi$. Note that if $F$ occurs, then a Brownian motion which is independent of $(h,\eta_1,\eta_2)$ and starts from $0$ has probability at least $p^* > 0$ to exit $B(0,\delta^{\zeta})$ on some point $z \in J_{\delta^\zeta}$ and then make a loop around $B(z,d\delta^\zeta / ( 2 \Lambda ))$ before exiting $B(z,d\delta^\zeta/ \Lambda)$. Here $p^*$ depends only on the implicit constants and $\Lambda >0$ is large but fixed (to be chosen). Since the harmonic measure of both of $\eta_1$ and $\eta_2$ as seen from $z$ is at least $q$, there exists a constant $x \in (0,\pi/2)$ such that $\dist(\varphi_2(z) , \partial \strip) \geq x$. Hence for $\Lambda > 0$ sufficiently large we have $\dist(\varphi_2(B(z,d\delta^{\zeta} / \Lambda)),\partial \strip) \geq x / 2$. Since $\delta^{\wt{\zeta}} \leq d \delta^\zeta / (2\Lambda)$ and $z \in B(0,1/2)$ for sufficiently small $\delta$, we have that 
\begin{align*}
\mu_h(B(z,d\delta^{\zeta} / ( 2 \Lambda ))) \geq \mu_h(B(z,\delta^{\wt{\zeta}})) \geq \delta^{\wt{\zeta}\wt{\alpha}} > \delta^{\beta},
\end{align*}
for sufficiently small $\delta$ (recall that $\beta > \wt{\zeta} \wt{\alpha}$) and together with the fact that $\mu_{h^2}((-\infty,u_{\alpha,\delta}^2] \times (0,\pi)) < \delta^{\beta}$, this implies that $\varphi_2(B(z,d\delta^{\zeta} / ( 2 \Lambda ))) \not \subseteq \strip_- + u_{\alpha,\delta}^2$. Note that by making the loop around $B(z,d\delta^{\zeta} / ( 2 \Lambda ))$ the Brownian motion hits a point whose image under $\varphi_2$ lies in $(\strip_+ + u_{\alpha,\delta}^2) \cap (\strip_- + u_{b,\delta}^2)$ and has distance from $\partial \strip$ at least $x / 2$. Note also that the above point $w$ lies in $[-c_1\log \delta^{-1} , 0) \times (x/2,\pi-x/2)$. Hence there exists a constant $c_3 > 0$ independent of $\delta$ such that with probability at least $e^{-c_3 \alpha^2( \log \delta )^2}$ a Brownian motion starting from $w$ exits $\strip$ on $[c_2,\infty) \times \{0,\pi\}$. Therefore since $B(0,\delta^{\zeta}) \subseteq \wh \CW$ and $\wt \eta_1([C,\infty)) \cup \wt \eta_2([C,\infty)) \subseteq \varphi_2^{-1}([c_2,\infty) \times \{0,\pi\})$, and since $\tfrac{\wt{\alpha} \gamma}{p} \beta < \alpha$, we obtain that there exists a constant $C_{\xi} > 0$ depending only on $\xi$ and $\gamma$ (but which can be taken to be uniform in $\gamma \in [a,2]$ for each $a>0$) such that with probability at least $\exp(-C_{\xi} b^2( \log \delta )^2)$, a Brownian motion starting from $0$ exits $\wh{\CW}$ on the part of $\wt \eta_1 \cup \wt \eta_2$ with boundary length distance at least $C$ from either $\wt \eta_1(0)$ or $\wt \eta_2(0)$. This completes the proof.
\end{proof}

\begin{proof}[Proof of Lemma~\ref{lem:main_lemma_lbd}]
Fix $\sigma > 0,\alpha>0$ and for $r > 0$, $j=1,2$ we set 
\begin{align*}
T_{r}^j = \inf\{t \geq 0 : \eta_{j}(t) \notin B(0,r) \}
\end{align*}
and write $\eta_j^s = \eta_j([0,T_{s }^j])$ for $j=1,2$. Let $p_0, r,s,u,v,C,D$ be as in the proof of Lemma~\ref{lem:q_prob_from_below} such that $\p[ Z \,|\,E_{\epsilon,\delta}] \geq p_0$ for each $\delta \in (0,1)$ and $\epsilon > 0$ sufficiently small.  Then we let $b$ be as in Lemma~\ref{lem:lbd},  and let $E_{\epsilon,\delta}^b$ be the event defined in the same way as $E_{\epsilon,\delta}$ but with $\alpha$ replaced by $b$.  Here, $b$ is chosen with respect to the constant $C > 1$ fixed in accordance with Lemma~\ref{lem:q_prob_from_below}. Note that for all $\delta \in (0,1)$ we have that $\alpha \log(\delta_1) = b \log(\delta)$ with $\delta_1 = \delta^{b/\alpha}$ and so there exists $\delta_0 \in (0,1)$ depending only on $\alpha$ and $b$ such that $\p[Z \giv E_{\epsilon,\delta}^b] \geq p_0$ for all $\delta\in (0,\delta_0)$ and all $\epsilon>0$ sufficiently small. Since $G$ is $(A,B,h_R^1,h^2)$-measurable and $E_{\epsilon,\delta}^b$ is independent of $h^2$, we obtain that
\begin{align*}
\p[ G \,|\,A,B,h_R^1,E_{\epsilon,\delta}^b ] = \p[ G \,|\,A,B,h_R^1]
\end{align*}
and hence
\begin{align*}
	\p[ G \,|\,E_{\epsilon,\delta}^b] &= \E[ \p[ G \,|\,A,B,h_R^1] \,|\,E_{\epsilon,\delta}^b] \geq \frac{1}{4}\p[ \p[ G \,|\,A,B,h_R^1] \geq 1/4 \,|\,E_{\epsilon,\delta}^b] = \frac{1}{4}\p[ Z\,|\,E_{\epsilon,\delta}^b] \geq p_0 / 4.
\end{align*}
Fix $\xi \in (0,p_0 / 8)$. Then by Lemma~\ref{lem:lbd} there exists $\delta \in (0,1)$ such that 
\begin{align*}
\p[ G \cap H_{C,C_\xi,b,\delta} \,|\,E_{\epsilon,\delta}^b] \geq p_0 / 8
\end{align*}
for all $\epsilon > 0$ sufficiently small since $H_{C,C_\xi,b,\delta}$ is also independent of $E_{\epsilon,\delta}^b$. Therefore we obtain that
\begin{align*}
\p[ G \cap H_{C,C_\xi,b,\delta} \cap E_{\epsilon,\delta}^b ] \gtrsim \epsilon^{\sigma/2 + o(1)} \quad \text{as}\quad \epsilon \to 0.
\end{align*}
Suppose now that we work on the event $G \cap H_{C,C_\xi,b,\delta} \cap E_{\epsilon,\delta}^b$. Then the Beurling estimate implies that there exists a constant $d > 0$ depending only on $b,u,\delta, \xi$ such that $\dist(0,\partial \wh{\CW}) \geq d$. Fix $0 < \zeta < \zeta^* < 1$. Let $F$ be the event that
\begin{enumerate}[(i)]
\item $\sup_{z \in B(0,\epsilon^{\zeta^* - \zeta}) \cap U_1} \p^z[ B \, \text{hits}\,\partial B(0,s)\,\text{before}\,\eta_{1}^s \cup \eta_{2}^s \,|\, \eta_{1}^s \cup \eta_{2}^s] \leq \exp(-\epsilon^{-\sigma})$.
\item\label{it:not_return2} $\eta_1$ and $\eta_2$ do not return to $B(0,sC^{n-1})$ after leaving $B(0,sC^n)$ for $n=1,2,3$.
\item\label{it:positive_distance2} $\dist(\eta_1([T_{s / 2}^1, T_{sC^3}^1]), \eta_2([T_{s / 2}^2, T_{sC^3}^2])) \geq Ds$.
\end{enumerate}

Note that the definition of $G$ implies that~\eqref{it:not_return2} and~\eqref{it:positive_distance2} of the definition of $F$ hold. Also $B(0,\epsilon^{\zeta^* -\zeta}) \subseteq B(0,d)$ for all $\epsilon > 0$ sufficiently small and so $\varphi_1(B(0,\epsilon^{\zeta^* -\zeta}) \cap U_1) \subseteq \strip_- + u_{b,\delta}^1$. Moreover $\varphi_1^{-1}(\{z : \re(z) = 0\}) \subseteq B(0,s)$ and $\{z \in \partial \strip : \re(z) \leq 1 \} \subseteq \varphi_1( \eta_1^s \cup \eta_2^s)$ and since $|u_{b,\delta}^1| \geq \epsilon^{-\sigma}$, for each $z \in \strip_- + u_{b,\delta}^1$ the probability that a Brownian motion starting from $z$ hits $\{z : \re(z) = 0\}$ before exiting $\strip$ is at most $e^{-\mu \epsilon^{-\sigma}}$
 for some constant $\mu > 0$. Thus (possibly by varying $\zeta^*$) we obtain that 
\begin{align}\label{eq:4}
\p[ F ]  \geq \epsilon^{\sigma/2 + o(1)} \quad \text{as} \quad \epsilon \to 0.
\end{align}
Since $\epsilon^{\zeta^*} / s > \epsilon$ for all $\epsilon$ sufficiently small, by combining the scale invariance of the joint law of $(\eta_1,\eta_2)$ (by scaling with $\epsilon^{\zeta} / s$) with ~\eqref{eq:4}, the proof of~\eqref{eqn:mainlemma3} is complete.
\end{proof}

\subsection{Other versions of the main estimates}

In this subsection we state and prove versions of Lemma~\ref{lem:main_lemma_ubd} in the case of chordal $\SLE_8$ and two-sided whole-plane $\SLE_4$. The first estimate is the following which is the analog of Lemma~\ref{lem:main_lemma_ubd} but with the pair of paths $\eta_1,\eta_2$ replaced by the left and right sides of the outer boundary of an $\SLE_8$.

\begin{lemma}
\label{lem:mainlemma8}
Fix $\xi > 1$, let $\CW = (\h,h,0,\infty) \sim \qwedgeW{\sqrt{2}}{1}$ have the first exit parameterization and let $\eta'$ be an $\SLE_8$ in $\h$ from $0$ to $\infty$ sampled independently of $h$ and then parameterized by quantum area with respect to $h$. Fix $t \geq 0$ and let $x_t^L$ (resp.\ $x_t^R$) be the point on $\partial \h_t$, to the left (resp.\ right) of $\eta'(t)$ such that the boundary segment from $\eta'(t)$ to $x_t^L$ (resp.\ $x_t^R$) has quantum length $\log(\epsilon^{-1})$ and let $I_t \subseteq \partial \h_t$ denote the boundary arc from $x_t^L$ to $x_t^R$.  Fix $\sigma > 0$. Then we have that
\begin{align*}
&\p\!\left[ \inf_{ B(z,4\epsilon^{\xi}) \subseteq B(\eta'(t),\epsilon) \cap \h_t} \p^z[B \ \text{exits} \ \h_t \ \text{in} \ \partial \h_t \setminus I_t \, | \, \eta'([0,t])] \leq \exp(-\epsilon^{-\sigma}),\ \eta'([0,t]) \subseteq{\D_+} \right]\\
&=  O(\epsilon^{\sigma / 2 +o(1)}).
\end{align*}
\end{lemma}

The proof of the Lemma~\ref{lem:mainlemma8} is similar to that of Lemma~\ref{lem:main_lemma_ubd}, so we shall be rather brief. First, we will need the following lemma, which is the chordal version of \cite[Lemma~3.6]{ghm2020kpz}. This plays a role similar to Lemma~\ref{lem:fitball} in the proof of Lemma~\ref{lem:main_lemma_ubd} in the sense that with high probability, the $\SLE_8$ process, $\eta$, fills a ball of radius $\epsilon^\xi$ ($\xi > 1$) before traveling distance $\epsilon$, and hence there is a ball of radius $\epsilon^\xi$ contained in $\h_t \cap B(\eta(t),\epsilon)$.

\begin{lemma}\label{lem:SLE8fills}
Let $\eta'$ be an $\SLE_8$ process in $\h$ from $0$ to $\infty$. Let $\tau_z$ be first hitting time of the point $z$ by $\eta'$ and let $\tau_z(\epsilon)$ be the first time after $\tau_z$ that $\eta'$ leaves $B(z,\epsilon)$. Then, there exist constants $b_0,b_1>0$ such that for all $\xi>1$ and all $\epsilon > 0$ small enough,
\begin{align*}
	\p[\eta'([\tau_z,\tau_z(\epsilon)]) \, \text{does not contain a ball of radius at least} \,\epsilon^\xi] \leq b_0 \exp \!\left(-b_1 \epsilon^{(1-\xi)/5} \right).
\end{align*}
\end{lemma}
\begin{proof}
If $z$ is in the interior of the domain, at distance at least, say, $d>0$ from the boundary, then the result follows by absolute continuity from the whole-plane case in \cite[Lemma~3.6]{ghm2020kpz}. Indeed, for $z \in \h$ and $r = \im(z)$, we let $h^0$ be a zero-boundary GFF and $\wh{h}_{z,r} = h^w - h_{z,r}^w(0)$ where $h^w$ is a whole-plane GFF and $h_{z,r}^w(0)$ its average on $\partial B(z,r)$. Then, by copying the proof of \cite[Proposition~3.4]{ms2016imag1} and arguing as in the proof of \cite[Lemma~4.4]{mq2020geodesics} for $h$ (when handling its harmonic part) and using scale invariance of $h^w$ and \cite[Lemma~3.11]{ghm2020kpz}, one sees that the Radon--Nikodym derivative of the law of $h|_{B(z,r/2)}$ (resp.\ $\wh{h}_{z,r}|_{B(z,r/2)}$) with respect to the law of $h^0|_{B(z,r/2)}$ has $p$th moment bounded above and below by constants depending only on $p$, for some $p>1$ close to $1$. Hence the lemma holds for $B(z,\epsilon)$ with $\epsilon \in (0,r/2)$. For $\epsilon \in [r/2,r)$, the result follows by applying the method of the boundary case, which we treat below.

We now prove the boundary case. Fix some $a \in (1,\xi)$, and let $L = \{z:\im(z) = \epsilon^a\}$ and $\tau_L = \inf\{ t> \tau_z: \eta'([\tau_z,t]) \cap L \neq \emptyset \}$.  We shall show that there is a constant $c_0 > 0$ such that if $x \in \R$, then $\p[\tau_L \geq \tau_x(\epsilon/2)] \leq \exp(-c_0 \epsilon^{(1-a)/4})$. With this at hand, we conclude the proof as follows. Assume that $x = z \in \R$ and that $\eta'$ does intersect $L$ before exiting $B(z,\epsilon/2)$ and let $w = \eta'(\tau_L)$.  Then, upon hitting $L$, we consider the event that $\eta'$ swallows a ball of radius $\epsilon^\xi$ before exiting $\wt{B} = B(w,\epsilon^a/2)$. Since the size of the ball $\wt{B}$ is comparable to the distance to the boundary, we have that the probability of the event that $\eta'([\tau_{w},\tau_w(\epsilon^a/2)])$ does not contain a ball of radius $\epsilon^\xi$ is comparable to that of the corresponding event when $\eta'$ is replaced by a whole-plane $\SLE_8$, and that the implicit constant is independent of $\epsilon$. Thus, by \cite[Lemma~3.6]{ghm2020kpz} (and the discussion in the previous paragraph), there exist constants $a_0,a_1>0$ such that
\begin{align*}
\p[\eta'([\tau_w,\tau_w(\epsilon^a/2)]) \, \text{does not contain a ball of radius at least} \,\epsilon^\xi] \leq a_0 \exp \!\left(-a_1 \epsilon^{a-\xi} \right).
\end{align*}
Hence, letting $a = (4\xi + 1)/5$, the result follows.

We now turn to proving that the probability that $\eta'$ makes it distance $\epsilon/2$ away from a boundary point $x$ without intersecting $L$ is very small. Let $h$ be the GFF on $\h$ such that $\eta'$ is the counterflow line of $h$ starting from $0$ and targeted at $\infty$. Then $h$ has boundary values given by $\pi / \sqrt{8}$ on $\R_-$ and $-\pi / \sqrt{8}$ on $\R_+$. Assume without loss of generality that $x \in \R_+$. For positive integers $j$, let $x_j = x + (2j-1)\epsilon^{(3+a)/4}$ for $j \leq \epsilon^{(1-a)/4}/4$ and let $\eta_j$ be the flow line of angle $+\pi/2$, started from $x_j$, targeting $0$. Then, the range of $\eta_j$ will be the outer boundary of $\eta'([0,\tau_{x_j}])$. Thus, if any of the flow lines $\eta_j$ hits $L$ before going too far to the left, it follows that $\eta'$ does so without traveling too far.

Consider first $\eta_1$, stopped upon first exiting the square $S_1$ with side length $\epsilon^a$, with $x_1$ as the center of the base and let $\tau^1$ be the first exit time. With positive probability $\eta_1$ exits $S_1$ in the top boundary. Next, let $g^1:\h \setminus \eta_1([0,\tau^1]) \rightarrow \h$ be the mapping out function for $\eta_1([0,\tau^1])$. For simplicity of notation in what follows, we also write $\wt{\eta}_1 = \eta_1$. Let $\wt{\eta}_2$ be the image of $\eta_2$ under $g^1$, that is, the flow line of angle $+\pi/2$, from $g^1(x_2)$ in the field $h^1 = h \circ (g^1)^{-1} - \chi \arg(((g^1)^{-1})')$. We write $x_k^1 = g^1(x_k)$. Next, let $S_2$ be the square with side length $\epsilon^a$ and base with center $x_2^1$. Let $\tau^2$ be the first time that $\wt{\eta}_2$ exits $S_2$ and stop $\wt{\eta}_2$ when doing so. Again, there is a positive probability that $\wt{\eta}_2$ exits $S_2$ in the top boundary. We proceed iteratively, letting $g^j$ be the mapping out function of $\wt{\eta}_j([0,\tau^j])$ and define $h^j = h^{j-1} \circ (g^j)^{-1} - \chi \arg(((g^j)^{-1})')$, as well as letting $x_k^j = g^j(x_k^{j-1})$,  $\wt{\eta}_{j+1}$ be the flow line of angle $+\pi/2$ of the field $h^j$, started from $x_{j+1}^j$, $S^j$ be the square of side length $\epsilon^a$ and base centered at $x_{j+1}^j$ and $\tau^j$ the first exit time of $S_j$ for $\wt{\eta}_j$. Note that the distance between each pair of marked points is always of order $\epsilon^{(3+a)/4}$. Indeed, each map $g^j$ contracts the points $x_k^{j-1}$ and $x_{k+1}^{j-1}$ by a distance of $O(\epsilon^a)$ (since that is the order of the diameter of $\wt{\eta}_j([0,\tau^j])$). Doing this for each $j$ the distance will shrink at most by $O(\epsilon^{(1-a)/4} \cdot \epsilon^a) = O(\epsilon^{(1+3a)/4})$ and hence (since $a > 1$) the distance will always be of order $O(\epsilon^{(3+a)/4})$. Hence, the each square $S_j$ will be separated from $g^j(S_{j-1})$, with distance uniform in $j$.  In particular,  we have that $\dist(S_j,g^j(S_{j-1})) \geq \epsilon^{(3+a)/4}$ for all $\epsilon \in (0,1)$ sufficiently small and all $j$.

Next, we note that there is a deterministic constant $C >0$ such that the boundary data of $h^j$ is in $[-C,C]$ for every $j$.  Let $\CZ_j$ be the Radon-Nikodym derivative between the laws of $h|_{S_j}$ and $h^j|_{S_j}$.  Then,  since $\dist(S_j,g^j(S_{j-1})) \geq \epsilon^{(3+a)/4}$ and $S_j \subseteq B(x_{j+1}^j ,  2\epsilon^a)$,  it follows from the argument used to prove \cite[Lemma~4.15]{ms2017imag4} that there exist $p>1,C_p>0$ depending only on $C$ such that $\E[\CZ_j^p] \leq C_p$.  It then follows that the probability that $\wt{\eta}_j$ exits $S_j$ on the top boundary is uniformly bounded from below in $j$ by a constant $p_0>0$,  and so 
\begin{align*}
\p[\im(\wt{\eta}_j(\tau_j)) \neq \epsilon^a \,\,\text{for all}\,\,1 \leq j \leq \epsilon^{(1-a)/4}/4] \leq (1-p_0)^{\epsilon^{(1-a)/4}/4} = \exp(-c_0 \epsilon^{(1-a)/4}),
\end{align*}
where $c_0 = -\frac{1}{4}\log(1-p_0)$.  Suppose that $x \geq \epsilon/4$.  Then we set $y_j = x - (2j-1)\epsilon^{(3+a)/4}/8$,  for $1 \leq j \leq \epsilon^{(1-a)/4}/8$,  and let $\wh{\eta}_j$ be the flow line of $h$ of angle $\frac{\pi}{2}$ started from $y_j$ and targeted at $0$.  Then,  applying a similar argument as before and possibly taking $p_0 > 0$ to be smaller,  we obtain that
\begin{align*}
\p[\im(\wh{\eta}_j(\wh{\tau}^j)) < \epsilon^a \, \,  \text{for all} \,\,1 \leq j \leq \epsilon^{(1-a)/4}] \leq \exp(-c_0 \epsilon^{(1-a)/4}),
\end{align*}
where $\wh{\tau}^j$ is the first time that $\wh{\eta}_j$ exits the square with side length $\epsilon^a$ and base at $y_j$.  It follows that off an event with probability at most $2 \exp(-c_0 \epsilon^{(1-a)/4})$,  there exist $1 \leq j_1 \leq \epsilon^{(1-a)/4},  1 \leq j_2 \leq \epsilon^{(1-a)/4}/8$ such that $\im(\eta_{j_1}(\tau^{j_1})) > \epsilon^a$ and $\im(\wh{\eta}_{j_2}(\wh{\tau}^{j_2})) > \epsilon^a$.  But if the above occur,  we have that $\eta'$ hits $L$ before hitting $x$ for the first time,  and so $\eta'([\tau_x,\tau_x(\epsilon/2)])$ has to exit the square $[x-\epsilon/2,x+\epsilon/2] \times [0,\epsilon^a]$ either to its right side or to $L$.  Since $\im(\eta_{j_1}(\tau^{j_1})) > \epsilon^a$,  we obtain that it has to exit it via $L$.  Thus,  we deduce that
\begin{align*}
\p[\eta'([\tau_x,\tau_x(\epsilon/2)]) \cap L = \emptyset] \leq 2 \exp(-c_0 \epsilon^{(1-a)/4})
\end{align*}
if $x \geq \epsilon/4$.  Suppose that $x < \epsilon/4$.  Then we set $\wh{y}_j = -(2j-1) \epsilon^{(3+a)/4}$ for $1 \leq j \leq \epsilon^{(1-a)/4}/4$ and let $\wh{\eta}^j$ be the flow line of $h$ of angle $-\frac{\pi}{2}$ starting from $\wh{y}_j$ and targeted at $0$.  Then,  arguing as before and possibly taking $p_0>0$ to be smaller,  we obtain that off an event with probability at most $2 \exp(-c_0 \epsilon^{(1-a)/4})$,  we have that there exist $1 \leq j_1,j_2 \leq \epsilon^{(1-a)/4}$ such that $\im(\eta_{j_1}(\tau^{j_1})) > \epsilon^a$ and $\im(\wh{\eta}^{j_2}(\wh{\tau}_{j_2})) > \epsilon^a$,  where $\wh{\tau}_{j_2}$ is the first time that $\wh{\eta}^{j_2}$ exits the square with side length $\epsilon^a$ and base with center $\wh{y}_{j_2}$.  Similarly,  since $\wh{\eta}^{j_2}$ is the outer boundary of the range of $\eta'([0,\wh{\tau}^{j_2}])$,  we obtain that $\eta'([\tau_x,\tau_x(\epsilon/2)]) \cap L \neq \emptyset$ if the above occur.  Therefore,  we have that
\begin{align*}
\p[\eta'([\tau_x,\tau_x(\epsilon/2)]) \cap L = \emptyset] \leq 2\exp(-c_0 \epsilon^{(1-a)/4})
\end{align*}
in every case.  This completes the proof.

\end{proof}

We are now ready to prove Lemma~\ref{lem:mainlemma8}.

\begin{proof}[Proof of Lemma~\ref{lem:mainlemma8}]
Fix $t \geq 0$ and observe that the restriction  $\CW_t$ of $\CW$ to $\h_t$ has the law of a weight $\frac{\gamma^2}{2} = 1$ quantum wedge (Theorem~\ref{thm:sle8thm}). Let $\varphi : \h_t \rightarrow \strip$ be the conformal transformation such that $\varphi(\eta'(t)) = -\infty$, $\varphi(\infty)  = +\infty$ and $\wh h = h \circ \varphi^{-1} + Q\log(|(\varphi^{-1})'|)$ has the first exit parameterization on $\strip$. Note that Lemma~\ref{lem:QwedgeLB} implies that we can find a finite constant $\alpha > 0$ such that with probability $1 - O(\epsilon^{\sigma / 2})$ we have that $\mu_h(B(z,\epsilon^{\xi})) \geq \epsilon^{\alpha \xi}$ for every ball $B(z,4\epsilon^{\xi}) \subseteq{\h \cap B(0,2)}$, where $\xi > 1$ is fixed. Arguing as in the proof of Lemma~\ref{lem:main_lemma_ubd} but with Lemma~\ref{lem:LBcone} replaced by Lemma~\ref{lem:QwedgeLB},we can find constants $c,d>0$ independent of $t$ such that with probability $1 - O(\epsilon^{\sigma / 2})$ we have that 
\begin{align*}
\varphi(B(z,\epsilon^{\xi})) \not \subseteq (\strip_- - \epsilon^{-\sigma}) \cup \{z \in \strip : -\epsilon^{-\sigma} \leq \re(z) \leq T_{\epsilon,d}+2,\,\,\dist(z,\partial \strip) < \epsilon^c \}
\end{align*}
for every ball $B(z,4\epsilon^{\xi}) \subseteq \h_t\cap B(0,2)$ and on that event there exists $w \in \varphi(B(z,\epsilon^{\xi}))$ such that with probability at least $e^{-\epsilon^{-\sigma}}$ a Brownian motion starting from $w$ exits $\strip$ on $\partial \strip_+ + T_{\epsilon,d} + 1$. Also we note that Lemma~\ref{lem:SLE8fills} implies that with probability $1 - o_{\epsilon}^{\infty}(\epsilon)$, $\eta'([0,t])$ contains a ball of radius $\epsilon^\xi$ whenever it travels distance $\epsilon$, on the event $\eta'([0,t]) \subseteq \D_+$. Therefore in order to complete the proof, it suffices to prove that with probability $1 - O(\epsilon^{\sigma / 2})$, $\varphi(I_t) \subseteq \partial\strip_- + T_{\epsilon,d} + 1$. To prove this, we note that as in the proof of Lemma~\ref{lem:main_lemma_ubd} we have that 
\begin{align*}
\E[ \nu_{\wh{h}}([T_{\epsilon,d},T_{\epsilon,d}+1] \times \{0\})^{-1} ] \lesssim \epsilon^{\gamma d / 2}
\end{align*}
and so Markov's inequality implies that
\begin{align*}
&\p[ \nu_{\wh{h}}([T_{\epsilon,d},T_{\epsilon,d}+1] \times \{0\}) \leq \log(\epsilon^{-1})] \lesssim \log(\epsilon^{-1})\epsilon^{\gamma d / 2} \lesssim \epsilon^{\sigma / 2}
\end{align*}
for $d>0$ sufficiently large and similarly for $\nu_{\wh{h}}([T_{\epsilon,d},T_{\epsilon,d}+1] \times \{\pi\})$. Hence with probability $1 - O(\epsilon^{\sigma / 2})$, 
\begin{align*}
\min(\nu_{\wh{h}}((-\infty,T_{\epsilon,d}+1] \times \{0\}), \nu_{\wh{h}}((-\infty,T_{\epsilon,d}+1] \times \{\pi\})) \geq \log(\epsilon^{-1})
\end{align*}
and so $\varphi(I_t) \subseteq \partial \strip_- + T_{\epsilon,d} + 1$. This completes the proof.
\end{proof}

Our final result of this section will be the version of Lemma~\ref{lem:main_lemma_ubd} which will be used to deduce the continuity results related to $\SLE_4$, that is, Theorems~\ref{thm:sle4_thm} and~\ref{thm:jones_smirnov_sle4}. We shall be rather brief when proving it, as the ideas are the same as in the proof of Lemma~\ref{lem:main_lemma_ubd}, the only difference being that we will not consider the quantum length of the curves $\eta_1$ and $\eta_2$ as its tail behavior is different (since this is the critical case, $\gamma = 2$).

\begin{lemma}
\label{lem:mainlemma4}
Suppose that we have the setup of Lemma~\ref{lem:main_lemma_ubd} with $\gamma = 2$ ($\kappa = 4$). Let $E(\sigma) = E_{\epsilon}(\sigma)$ be the event that there exists $z \in B(0,2\epsilon)$ with $\dist(z,\eta_1 \cup \eta_2)\geq \epsilon$ and such that the harmonic measure of each of $\eta_1$ and $\eta_2$ as seen from $z$ is at least $\frac{1}{4}$ and that the probability that a Brownian motion starting from $z$ hits $\partial \D$ before exiting $\C \setminus (\eta_1 \cup \eta_2)$ is at most $\exp(-\epsilon^{-\sigma})$. Then 
\begin{align*}
\p[ E(\sigma) ] = O(\epsilon^{\sigma / 2 +o(1)}) \quad\text{as}\quad \epsilon \to 0.
\end{align*}
\end{lemma}
\begin{proof}
Suppose that we have the setup of the proof of Lemma~\ref{lem:main_lemma_ubd} and fix $\xi > 1$. Note that Theorem~\ref{thm:critical_wedge_welding} and~\ref{thm:critical_cone_welding} implies that $\eta_1$ and $\eta_2$ cut $(\C,h,0,\infty)$ into two independent weight-$(\tfrac{\gamma^2}{2} = 2)$ wedges parameterized by $U_1$ and $U_2$.

Arguing as in the proof of Lemma~\ref{lem:main_lemma_ubd}, we obtain that with probability $1-O(\epsilon^{\sigma / 2})$ we have that $\varphi_1(B(z,\epsilon^{\xi})) \not \subseteq \strip_- - \epsilon^{-\sigma}/4$ for every ball $B(z,\epsilon^{\xi}) \subseteq B(0,2\epsilon)$. Moreover, Lemma~\ref{lem:LBmassball} implies that there exists $p > 0$ such that $\E[ \mu_h(\D)^p ] < \infty$ and so we can find constants $c,d>0$ such that with probability $1-O(\epsilon^{\sigma / 2 + o(1)})$ we have that $\mu_h(\D) \leq \epsilon^{-c}$, $T_{\epsilon,d} \leq \epsilon^{-\sigma}/4$ and $\mu_{h^1}([T_{\epsilon,d},T_{\epsilon,d}+1] \times [0,\pi]) > \epsilon^{-c}$. Suppose that we work on the above event and let $z \in B(0,2\epsilon) \setminus B(0,\epsilon)$ be such that $\dist(z,\eta_1 \cup \eta_2) \geq \epsilon$ and such that the harmonic measure of each of $\eta_1$ and $\eta_2$ as seen from $z$ is at least $\tfrac{1}{4}$. We assume without loss of generality that $z \in U_1$. Note that for some universal constant $q > 0$, with probability at least $q$ a Brownian motion starting from $z$ makes a loop around $B(z,\epsilon^\xi)$ before exiting $\C \setminus (\eta_1 \cup \eta_2)$. Hence if $\varphi_1(z) \in \strip_- - \epsilon^{-\sigma}/2$, then with probability at least $q$ a Brownian motion starting from $\varphi_1(z)$ hits $\{-\epsilon^{-\sigma}/4\} \times (0,\pi)$ before $\partial \strip$ since $\varphi_1(B(z,\epsilon^{\xi})) \cap ( \strip_+ - \epsilon^{-\sigma}/4 ) \neq \emptyset$. But for $\epsilon$ sufficiently small this occurs with probability less than $q$ by the Beurling estimate and so we have a contradiction. Note also that $\varphi_1(U_1 \cap \partial \D) \not \subseteq \strip_+ + T_{\epsilon,d}+1$ since $\mu_{h^1}(\strip_+ + T_{\epsilon,d}) \geq \epsilon^{-c} \geq \mu_h(\D)$. Hence if $\varphi_1(z) \in \strip_+ + \epsilon^{-\sigma}/2$, then there exists a connected path $P$ of $\varphi_1(U_1 \cap \partial \D)$ connecting $\{T_{\epsilon,d}+1\} \times \{0,\pi\}$ with $\{\re(\varphi_1(z))\} \times (0,\pi)$ and with probability at least $\delta > 0$ a Brownian motion starting from $\varphi_1(z)$ hits $P$ before $\partial \strip$ since the harmonic measure of the upper and lower boundary of $\strip$ as seen from $\varphi_1(z)$ is at least $1/4$ ($\delta$ is a universal constant). By conformal invariance of Brownian motion we obtain that with probability at least $\delta$ a Brownian motion starting from $z$ hits $\partial \D$ before $\eta_1 \cup \eta_2$. But that is a contradiction for $\epsilon$ sufficiently small by applying the Beurling estimate again. Therefore $\varphi_1(z) \in [-\epsilon^{-\sigma}/2,\epsilon^{-\sigma}/2] \times (0,\pi)$ and since $\varphi_1(U_1 \cap \partial \D) \not \subseteq \strip_+ + \epsilon^{-\sigma}/2$, by arguing as in the proof of Lemma~\ref{lem:main_lemma_ubd} we obtain that a Brownian motion starting from $\varphi_1(z)$ hits $\varphi_1(U_1 \cap \partial \D)$ before $\partial \strip$ with probability $\gtrsim \exp(-3\epsilon^{-\sigma})$. This completes the proof.
\end{proof}

\section{Density lower bound for the LQG measure}
\label{sec:density}

In this section we prove a few results that we need on the density of the intensity of the LQG area measures. Recall that $\greenN{D}$ is the Green's function with Neumann boundary condition in $D$. The key bound of this section is Proposition~\ref{prop:density_bound} which is stated just below.  It will be important for the proof of the upper bound in Theorem~\ref{thm:sle8_thm}, since it tells us that we do not lose too much information when we consider the regularity assumption that $\mu_h(\D_+) \leq M$. This condition in particular, is convenient to control the behaviour and geometry of the $\SLE_8$ process, in particular how far it escapes, when parameterized by quantum area and to derive from this, the same control of the curve parameterized by capacity, for small times. In order to understand the proof of Theorem~\ref{thm:sle8_thm} which comes in Section~\ref{sec:sle8}, one need only understand the statement of Proposition~\ref{prop:density_bound} and can skip the details on a first reading.  In what follows, we assume that $\gamma = \sqrt{2}$.

\begin{proposition}
\label{prop:density_bound}
Suppose that $(\h,h,0,\infty) \sim \qwedgeW{\sqrt{2}}{1}$ has the first exit parameterization and fix $\beta > 1/8$. Then there exist constants $M,c>0$ such that if $E = \{\mu_h(\D_+) \leq M\}$ then 
\[ \E[ \one_E \mu_h(dz) ] \geq c \im(z)^{\beta - \frac{\gamma^2}{2}}|z|^{\frac{\gamma^2}{8}-\beta} \quad\text{for all}\quad z \in e^{-1}\D_+.\]
\end{proposition}

We emphasize that in the statement of Proposition~\ref{prop:density_bound}, we are using the notation $\E[ \one_E \mu_h(dz)]$ to denote the density with respect to Lebesgue measure of the measure $X \mapsto \E[ \one_E \mu_h(X)]$.  As we will see below, the reason that Proposition~\ref{prop:density_bound} takes some work to prove is because we are considering $\E[ \one_E \mu_h(dz)]$ and not simply $\E[\mu_h(dz)]$ (i.e., the density of the measure $X \mapsto \E[ \mu_h(X)]$ with respect to Lebesgue measure) and we need a lower bound on the former in Section~\ref{sec:sle8}.  Establishing this lower bound will amount to establishing a lower bound on the probability of $E$ under the measure whose formal Radon-Nikodym derivative with respect to $\p$ is given by $e^{\gamma h(z)}$.

The proof of Proposition~\ref{prop:density_bound} consists of three main steps. First, we show how to compare $\E[ \one_E \mu_h(dz)]$ to the density of the intensity of the area measure corresponding to $h$, i.e., $\E[ \mu_h(dz)]$. This is the content of the next proposition. Next, in Lemma~\ref{lem:RNint}, we derive the form of $\E[ \mu_{h^W}(dz)]$ when considering an $\alpha$-quantum wedge with $\alpha < Q$. Finally, we show how to compare the densities of the intensities of the area measures for $\qwedgeW{\sqrt{2}}{4}$ (i.e., $\alpha = 0$) and $\qwedgeW{\sqrt{2}}{1}$.
\begin{proposition}\label{prop:density_lower_bound}
Suppose that $(\h,h,0,\infty) \sim \qwedgeW{\sqrt{2}}{1}$ has the first exit parameterization and fix $\beta > 1/8$. Then we can find finite constants $M,c >0$ such that 
if $E = \{ \mu_h(\D_+) \leq M\}$ then 
\[ \E[ \one_E \mu_h(dz) ] \geq c  \im(z)^{\beta}|z|^{-\beta}f(z) \quad\text{for all}\quad z \in e^{-1}\D_+.\]
where $f(z) = \E\left[ \mu_h(dz)\right]$.
\end{proposition}

Let $(\strip,h,-\infty,+\infty) \sim \qwedgeW{\gamma}{\gamma^2/2}$ with $\gamma \in (0,2]$ (recall the definition in Section~\ref{sec:LQG}). Suppose that we have $z \in \strip_-$.  We let $\p_z$ denote the law whose (formal) Radon-Nikodym derivative with respect to $h$ is given by $e^{\gamma h(z)}$.  The law of $h$ under $\p_z$ can be sampled from using the following two steps:
\begin{itemize}
\item Taking the projection onto $H_1(\strip)$ to be given by $X_t$ where $(-X_{-t/2})_{t \geq 0}$ has the law of a $\BES^3$ weighted by the Radon-Nikodym derivative
\begin{align*}
	 \frac{e^{\gamma X_{\re(z)}}}{\E[e^{\gamma X_{\re(z)}}]}.
\end{align*}
\item Taking the projection onto $H_2(\strip)$ to  be given by the corresponding projection for a free boundary GFF plus the corresponding projection of $\gamma\greenN{\strip}(z,\cdot)$, independent of $X$.
\end{itemize}

The main step in the proof of Proposition~\ref{prop:density_lower_bound} is the following.

\begin{proposition}
\label{prop:mass_lower_bound}
Fix $\beta > 1/8$. Then we can find $M,c >0$ such that if $E = \{\mu_h(\strip_-) \leq M \}$ then 
\[ \p_z[E] \geq c \dist(z,\partial \strip)^{\beta}  \quad\text{for all}\quad z \in \strip_- \quad\text{with} \quad \re(z) \leq -1.\]
\end{proposition}

We begin by giving the asymptotics of the Laplace transform of a $\BES^3$ process.
\begin{lemma}
\label{lem:bes_rn_derivative}
Suppose that $Z \sim \BES^3$.  For any $s > 0$ there exist constants $0 < c_0 < c_1 < \infty$ such that 
\[ c_0 t^{-3/2} \leq \E[ e^{-s Z_t}] \leq c_1 t^{-3/2} \quad\text{for all}\quad t \geq 1.\] 
\end{lemma}
\begin{proof}
We begin by noting that if we let $p_t(x,y)$ be the transition density for $Z$, then
\begin{align*}
	p_t(0,y) = \sqrt{\frac{2}{\pi}} t^{-3/2} y^2 e^{-\frac{y^2}{2t}},
\end{align*}
see \cite[Chapter~XI]{ry1999book}. Then,
\begin{align*}
	\E[ e^{-s Z_t}] = \sqrt{\frac{2}{\pi}} t^{-3/2} \int_0^\infty y^2 e^{-sy-\frac{y^2}{2t}} dy \leq  \sqrt{\frac{2}{\pi}} t^{-3/2} \int_0^\infty y^2 e^{-sy} dy = \frac{2\sqrt{2}}{s^3 \sqrt{\pi}} t^{-3/2}.
\end{align*}
For the lower bound, we set $m = 2 \max(s,1)$ and note that for $t \geq 1/2$,
\begin{align*}
	\E[e^{-s Z_t}] \geq \sqrt{\frac{2}{\pi}} t^{-3/2} \!\left( \int_0^1 y^2 e^{-m y} dy + \int_1^\infty y^2 e^{-m y^2} dy \right) = C_s  t^{-3/2},
\end{align*}
since the sum of the integrals is positive and depends only on $s$. Thus, the result follows.
\end{proof}
In order to bound the moments of $\mu_h$ under the measure $\p_z$, it is convenient to bound the moments of $\mu_{\wt{h} + \gamma \greenN{\h}(z,\cdot)}$ where $\wt{h}$ is a free boundary GFF on $\h$. We will then show that the former moments are upper bounded by the latter. 
\begin{lemma}
\label{lem:mass_upper_bound1}
Let $\wt h$ be a free boundary GFF on $\h$ with the additive constant fixed so that its average on $\h \cap \partial \D$ is equal to $0$.
Then there exists $p_0 \in (0,1)$ such that for all $p \in (0,p_0)$ we can find $C_p < \infty$ depending only on $p$ such that 
\begin{align*}
	\E\!\left[\mu_{\wt h + \gamma \greenN{\h}(z,\cdot)}(\D_+)^p\right] \leq C_p \im(z)^{\gamma p (Q-2\gamma)+O(p^2)} \quad\text{for all}\quad z \in \D_+ \setminus e^{-2}\D_+ 
\end{align*}
where the implicit constants of the term $O(p^2)$ are uniform in $z$.
\end{lemma}
\begin{proof}
We write $D = \h \cap B(0,2)$ and suppose first that $\wt{h}$ is normalized so that the average on $\h \cap \partial B(0,2)$ is equal to $0$. Then $\wt h = h_D^0 + \Fh_D$ where $h_D^0$ is a zero-boundary GFF on $D$ and $\Fh_D$ is harmonic on $D$ and independent of $h_D^0$. We fix $z \in \D_+ \setminus e^{-2}\D_+$ and set $A_k = B(z,r2^{-k}) \setminus B(z,r2^{-k-1})$ for $k \in \N_0$ where $r = \im(z)$. By \cite[Proposition~3.7]{rv2010gmcrevisit} (recalling Remark~\ref{rmk:different_fields}),
\begin{align*}
	\E\!\left[\mu_{h_D^0}(B(z,r2^{-k}))^p\right] \leq C_p r^{\gamma p Q}2^{-k \zeta_p}
\end{align*}
where $C_p$ depends only on $p \in (0,4/\gamma^2)$ and $\zeta_p = (2+ \gamma^2 / 2) p - \gamma^2 p^2 / 2$, and by Lemma~\ref{lem:bound_harmonic_part}, there are constants $C,C_p>0$ such that
\begin{align*}
	\E\!\left[\exp\left(\gamma p \sup_{w \in B(z,r/2)}\Fh_D (w)\right) \right] \leq C_p r^{-C \gamma^2 p^2}.
\end{align*}
Thus we have that 
\begin{align*}
	\E\!\left[ \mu_{\wt h}(B(z,r2^{-k}))^p \right] \leq C_p  
r^{\gamma p Q - C  \gamma^2 p^2}2^{-k \zeta_p}.
\end{align*}
Moreover, since $\greenN{\h}(z,w) = -\log|z-w| - \log|z-\ol{w}| \leq (k+1)\log 2 - 2\log r$ for all $w \in A_k$ we have that if we pick $p \in (0,1)$ such that $\zeta_p  - \gamma^2 p = (2 - \gamma^2 / 2)p - \gamma^2 p^2 / 2 > 0$, then
\begin{align} \label{eqn:1}
	\E\!\left[ \mu_{\wt h + \gamma \greenN{\h}(z,\cdot)}(B(z,r/2))^p \right] \leq \sum_{k=1}^{\infty}\E\!\left[ \mu_{\wt h + \gamma \greenN{\h}(z,\cdot)}(A_k)^p \right]  \leq \wt{C}_p r^{\gamma p(Q-2\gamma) - C\gamma^2 p^2}.
\end{align}
Next, we let $\rho_{z}$ denote the uniform measure on the semicircle $\h \cap \partial B(-x,2)$ where $x = \re(z)$. Then we have that $\mu_{\wt h }(B(z,r))$ and $e^{-\gamma (\wt h, \rho_{z})} \mu_{\wt h}(B(ir,r))$ have the same law and so by H\"{o}lder's inequality,
\begin{align} \label{eqn:2}
\E\!\left[ \mu_{\wt h}(B(z,r))^p \right] \leq \E \!\left[ e^{-2\gamma p (\wt h, \rho_{z})}\right]^{1/2}  \E\!\left[\mu_{\wt h}(B(ir,r))^{2p}\right]^{1/2}.
\end{align}
Furthermore, $\var[(\wt h, \rho_z)]$ is uniformly bounded in $x$ and $\E\!\left[ \mu_{\wt h}(B(ir,r))^{2p}\right] \lesssim r^{2\gamma p Q - O(p^2)}$ where the implicit constant depends only on $p$ and the constant of the term $O(p^2)$ is independent of $r,p$ for all $p$ sufficiently small, and so we can find $\wt{C}_p < \infty$ depending only on $p$ such that $\E\!\left[ \mu_{\wt h}(B(ir,r))^{2p}\right] \leq \wt{C}_p$ for $p$ sufficiently small. Hence by~\eqref{eqn:2} we obtain that 
\begin{align} \label{eqn:3}
\E \!\left[ \mu_{\wt h + \gamma \greenN{\h}(z,\cdot)}(A_0)^p\right] \leq C_p r^{\gamma p (Q - 2\gamma)}
\end{align}
and so~\eqref{eqn:1} and~\eqref{eqn:3}  imply that 
\begin{align} \label{eqn:4}
\E\!\left[ \mu_{\wt h + \gamma \greenN{\h}(z,\cdot)}(B(z,r))^p\right] & \leq \E \!\left[ \mu_{\wt h + \gamma \greenN{\h}(z,\cdot)}(B(z,r/2))^p\right] + \E\!\left[ \mu_{\wt h + \gamma \greenN{\h}(z,\cdot)}(A_0)^p\right]\nonumber \\
& \leq C_p r^{\gamma p (Q - 2\gamma) + O(p^2)}
\end{align}
where the implicit constant in the term $O(p^2)$ is universal. 

We now consider $\D_+ \setminus B(z,r)$. For $k \in \N$ we set $\wh A_k = \h \cap B(z,r2^k) \setminus B(z,r2^{k-1})$ and $k_r = \lfloor \frac{\log(r^{-1})}{\log 2} \rfloor + 1$. We observe that $\var[(\wt h, \rho_{x, r2^{k+1}}^+)] \leq -2\log(r2^k) + O(1)$, where the $O(1)$ term is uniform in $x$ and $\rho_{x,r2^{k+1}}^+$ is the uniform measure on $\h \cap \partial B(x,r2^{k+1})$ and 
\begin{align} \label{eqn:5}
\E\!\left[ \mu_{\wt h - \left( \wt h, \rho_{x,r2^{k+1}}^+\right)}(\h \cap B(x,r2^{k+1}))^p\right] \leq C_p (r2^k)^{\gamma p Q}
\end{align}
where $C_p$ depends only on $p$. Thus since $e^{\gamma \left(\wt h, \rho_{x,r2^{k+1}}^+\right)} \mu_{\wt h - \left( \wt h , \rho_{x,r2^{k+1}}^+\right) }(\h \cap B(x,r2^{k+1}))$, and $\mu_{\wt{h}}(\h \cap B(x,r2^{k+1}))$ have the same law,~\eqref{eqn:5} and H\"{o}lder's inequality imply that 
\begin{align*}
\E\!\left[ \mu_{\wt h}(\h \cap B(x,r2^{k+1}))^p\right] \leq C_p (r2^k)^{\gamma p Q - 2\gamma^2 p^2}
\end{align*}
Since $\greenN{\h}(z,w) \leq -2\log (r2^{k-1})$ for all $w \in \wh A_k$, it follows that 
\begin{align*}
	\E\!\left[ \mu_{\wt h + \gamma \greenN{\h}(z,\cdot)}(\wh A_k)^p\right] \leq C_p (r2^k)^{a_p}
\end{align*}
for $a_p = \gamma p (Q-2\gamma) - 2\gamma^2 p^2$. By picking $p \in (0,1)$ such that $a_p < 0$ (since $Q < 2\gamma$ when $\gamma > 2/\sqrt{3}$), summing over $k$, we have (since the function $x \mapsto x^p$ is concave) 
\begin{align}\label{eqn:6}
	\E\!\left[ \mu_{\wt h + \gamma \greenN{\h}(z,\cdot)}(\h \cap B(z,1) \setminus B(z,r))^p\right] \leq C_p r^{a_p},
\end{align}
where $C_p > 0$ depends only on $p$. 

We now deal with $\D_+ \setminus B(z,1)$. Since $\greenN{\h}(z,w) \leq 0$ for all $w \in \D_+ \setminus B(z,1)$ and $\E[\mu_{\wt h}(\D_+)^p]<\infty$ for $p \in (0,1)$, we have that 
\begin{align}\label{eqn:7}
\E\!\left[ \mu_{\wt h + \gamma \greenN{\h}(z,\cdot)}(\D_+ \setminus B(z,1))^p\right] \leq C_p.
\end{align}
By combining~\eqref{eqn:4}, ~\eqref{eqn:6} and~\eqref{eqn:7} we obtain that for all $p \in (0,1)$ sufficiently small we can find $C_p > 0$ depending only on $p$ such that
\begin{align} \label{eqn:8}
\E\!\left[ \mu_{\wt h + \gamma \greenN{\h}(z,\cdot)}(\D_+)^p\right] \leq C_p r^{\gamma p (Q-2\gamma) + O(p^2)}
\end{align}
where the implicit constant in the term $O(p^2)$ depends only on $p$.

Finally we set $\wh{h} = \wt{h} - (\wt{h},\rho)$ where $\rho$ is the uniform measure on $\h \cap \partial \D$ . Then $\wh{h}$ is normalized as in the statement of the Lemma. Then, for all $p \in (0,1)$ sufficiently small, Holder's inequality together with ~\eqref{eqn:8} imply that
\begin{align*}
&\E\!\left[ \mu_{\wh{h} + \gamma \greenN{\h}(z,\cdot)}(\D_+)^p \right] \leq \E\!\left[ \mu_{\wt{h} + \gamma \greenN{\h}(z,\cdot)}(\D_+)^{2p} \right]^{1/2} \E\!\left[ e^{-2\gamma p(\wt{h},\rho_0)} \right]^{1/2}\\
&\leq \wh{C}_p (\im(z))^{\gamma p (Q - 2\gamma) + O(p^2)}
\end{align*}
for all $z \in \D_+ \setminus e^{-2} \D_+$, where $\wh{C}_p < \infty$ depends only on $p$ and the constants of the term $O(p^2)$ are uniform in $z$.
This completes the proof.
\end{proof}
We now use Lemma~\ref{lem:mass_upper_bound1} to bound the moments under $\p_z$ of the quantum area of $(\re(z) - 1, \re(z) +1) \times (0,\pi)$ with respect to $h$.
\begin{lemma}
\label{lem:area_near_z_bound}
Let $(\strip,h,-\infty,+\infty) \sim \qwedgeW{\sqrt{2}}{1}$ and consider $\p_z$ as above for $z \in \strip$ such that $\re(z) \leq -1$. Then there exists $p_0 \in (0,1)$ such that for all $p \in (0,p_0)$ there exists $C_p < \infty$ depending only on $p$ such that 
\begin{align*}
\E_z [\mu_h(A)^p] \leq C_p r^{\gamma p (Q-2\gamma) + O(p^2)}
\end{align*}
where $A = (\re(z) - 1,\re(z) +1) \times (0,\pi)$ and $r = \dist(z,\partial\strip)$.
\end{lemma}
\begin{proof}
We consider the conformal transformation $\psi : \strip \rightarrow \h$ with $\psi(w) = e^{w-\re(z) -1}$ and let $\wt h$ be a free boundary GFF on $\h$ as in the statement of Lemma~\ref{lem:mass_upper_bound1}. Then $\wt h \circ \psi$ has the law of a free boundary GFF on $\strip$ normalized so that its average on $\{\re(z) + 1\} \times [0,\pi]$ is zero. Note that $\greenN{\strip}(z,w) = \greenN{\h}(e^z,e^w)$ for $w \in \strip$. Fix $z \in \strip_-$ as in the statement of the lemma and without loss of generality we can assume that $\im(z) \leq \pi / 2$ and set $g(w) = \greenN{\strip}(z,w) - \frac{1}{\pi}\int_{0}^{\pi}\greenN{\strip}(z,\re(w) + i\theta)d\theta$. Then $g$ is the projection of $\greenN{\strip}(z,\cdot)$ onto $H_2(\strip)$. Moreover there exists a universal constant $C > 0$ such that $g(w) \leq \greenN{\strip}(z,w) + C$ for all $w \in A$. Hence it holds on $A$ that
\begin{align} \label{eqn:9}
h \leq \wt h \circ \psi + \gamma \greenN{\strip}(z,\cdot) + \gamma C - \inf_{\re(z)-1 \leq t \leq \re(z)+1}\wt h_1(t)
\end{align}
since $X$ is negative on $\strip_-$, where $\wt h_1$ is the projection of $\wt h \circ \psi$ onto $H_1(\strip)$. Furthermore, $\inf_{\re(z)-1 \leq t \leq \re(z)+1}\wt h_1(t)$ and $\inf_{0 \leq t \leq 2}B_{2t}$ have the same law where $B$ is a standard Brownian motion; note that the latter has finite exponential moments of all orders. Thus H\"{o}lder's inequality combined with~\eqref{eqn:9} imply that 
\begin{align}
\label{eqn:10}
\E_z[ \mu_h(A)^p] \leq C_p \E[\mu_{\wh h}(A)^{2p}]^{1/2}
\end{align}
where $\wh h = \wt h \circ \psi + \gamma \greenN{\strip}(z,\cdot)$ and $C_p$ depends only on $p$. Moreover,~\eqref{eq:LQG_coordinate_change} together with the observation that $e^{-2} \leq |\psi'(w)| \leq 1$ for $w \in A$, implies that
\begin{align*}
	\E\!\left[ \mu_{\wt h + \gamma \greenN{\h}(\psi(z),\cdot)}(\D_+ \setminus e^{-2}\D_+)^{2p}\right] = \E\!\left[ \mu_{\wh h + Q \log|\psi'|}(A)^{2p}\right] \geq C_p \E\!\left[ \mu_{\wh h}(A)^{2p}\right],
\end{align*} 
where $C_p$ depends only on $p$. Since $\im(\psi(z)) = e^{-1}\sin(r)$ and $\sin(y) \geq \frac{2}{\pi}y$ for $y \in (0,\pi /2)$, the result follows from~\eqref{eqn:10} and Lemma~\ref{lem:mass_upper_bound1}.
\end{proof}
Next, we bound the moments under $\p_z$ of the mass in the rest of the half-strip $\strip_-$.
\begin{lemma}
\label{lem:mass_upper-bound2}
Suppose that we have the setup of the previous lemmas. Then there exists $p_0 \in (0,1)$ such that for all $p \in (0,p_0)$ we can find a constant $C_p < \infty$ such that 
\begin{align*}
	\E_z[\mu_h(\strip_- \setminus A)^p] \leq C_p
\end{align*}
\end{lemma}
\begin{proof}
Let $\E_z$ (resp.\ $\E$) be the expectation with respect to $\p_z$ (resp.\ $\p$) and let $X = (X_t)_{t \in \R}$ be the projection of $h$ onto $H_1(\strip)$. For $k \leq 0$ we set $X_k^* = \sup_{s \in [k-1,k]}X_s$ and set $t = \re(z) \leq -1$. Note that under $\p$ we have for $s \leq 0$ that $X_s = -Y_{-2s}$ where $Y \sim \BES^3$.  Moreover, $Y$ can be coupled with a Brownian motion $B$ such that $Y_s = 2S_s - B_s$ for all $s \geq 0$ where $S_s = \sup_{0 \leq x \leq s}B_x$ (\cite[ Theorem~VI.3.3.5]{ry1999book}). Hence $X_k^* - X_k$ for $k \leq 0$ is stochastically dominated from above by $\sup_{s \in [0,1]}B_{2s}$. We recall that there exists a constant $C_p^1$ depending only on $p$ and $\gamma$ such that,
\begin{align} \label{eqn:12}
	 \E\!\left[ e^{\gamma p \sup_{s \in [0,1]} B_{2s}} \right] \leq C_p^1.
\end{align}

Set $\wt{B}_u = B_{u-2t}-B_{-2t}$ and $\wt{Y}_u = 2\sup_{0 \leq r \leq u} \wt{B}_r - \wt{B}_u$ for all $u \geq 0$.  Then,  we have that $\wt{Y}$ has the law of a $\BES^3$ process starting from $0$ which is independent of $Y_{-2t}$.  Moreover,  we have that $Y_{2s} \geq \wt{Y}_{2(s+t)} - Y_{-2t}$ for all $s \geq -t$.  It follows that
\begin{align} \label{eqn:13}
	\E\!\left[ e^{\gamma p X_k}e^{\gamma X_t}\right] &= \E\!\left[ e^{-\gamma p Y_{-2k}}e^{-\gamma Y_{-2t}}\right] \leq \E\!\left[e^{-\gamma p \wt Y_{2(-k+t)}}e^{-\gamma (1-p)Y_{-2t}}\right] \nonumber \\
	&= \E\!\left[ e^{-\gamma p \wt Y_{2(-k+t)}}\right]\! \E\!\left[ e^{-\gamma (1-p) Y_{-2t}}\right] \notag\\
	&\leq C(-k+t)^{-3/2}(-t)^{-3/2} \quad\text{(Lemma~\ref{lem:bes_rn_derivative})}
\end{align}
for some constant $C>0$ depending only on $\gamma$ and $p$. Moreover, since $\E[ e^{\gamma p (X_k^*-X_k)} | X_k = x]$ is stochastically dominated by $e^{\gamma p \sup_{s \in [0,1]} B_{2s}}$ for each $x \leq 0$, we have (with $B$ a Brownian motion, independent of $X$) by the Markov property of $X$,~\eqref{eqn:12} and~\eqref{eqn:13} that for all $k \leq t-1$,
\begin{align}\label{eqn:14}
	\E\!\left[ e^{\gamma p X_k^*}e^{\gamma X_t}\right] &= \E\!\left[ \E\!\left[ e^{\gamma p (X_k^*-X_k)} \, \middle| \,(X_u)_{u \in [k,0]} \right] e^{\gamma p X_k}e^{\gamma X_t}\right] = \E\!\left[ \E\!\left[ e^{\gamma p (X_k^*-X_k)} \, \middle| \,X_k \right] e^{\gamma p X_k}e^{\gamma X_t}\right] \nonumber \\
	&\leq \E\!\left[ e^{\gamma p \sup_{s \in [0,1]} B_{2s}} e^{\gamma p X_k}e^{\gamma X_t}\right] =  \E\!\left[ e^{\gamma p \sup_{s \in [0,1]} B_{2s}} \right] \E\!\left[ e^{\gamma p X_k}e^{\gamma X_t}\right] \nonumber \\
	&\leq C_p^1 (-k+t)^{-3/2} (-t)^{-3/2},
\end{align}
for some constant $C_p^1$ depending on $p$ and $\gamma$. Moreover we observe that the projection $g$ of $\greenN{\strip}(z,\cdot)$ onto $H_2(\strip)$ is bounded from above on $\strip_- \setminus A$ uniformly in $z$ and since the law of $h_2$ under $\p$ is invariant under horizontal translations and $\E\!\left[ \mu_{h_2}([-1,0] \times [0,\pi])^p\right]<\infty$ for $p>0$ sufficiently small, we have that 
\begin{align} \label{eqn:15}
\E_z\!\left[ \mu_{h_2}([k-1,k] \times [0,\pi])^p \right] \leq C_p^2.
\end{align}
Therefore by~\eqref{eqn:14} and~\eqref{eqn:15}, we obtain for $p \in (0,1)$ sufficiently small that 
\begin{align*}
\E_z\!\left[ \mu_h((-\infty,t-1] \times [0,\pi])^p \right] & \leq \E_z\!\left[ \left( \sum_{k=-\infty}^{t-1}e^{\gamma X_k^*}\mu_{h_2}([k-1,k]\times [0,\pi])\right)^p\right]\\
& \leq \E_z\!\left[\sum_{k=-\infty}^{t-1}e^{\gamma p X_k^*}\mu_{h_2}([k-1,k] \times [0,\pi])^p\right]\\
& = \sum_{k=-\infty}^{t-1}\E_z\!\left[ e^{\gamma p X_k^*}\right] \E_z\!\left[ \mu_{h_2}([k-1,k] \times [0,\pi])^p\right]\\
& \leq \frac{C_p^1 C_p^2 (-t)^{-3/2}}{\E\!\left[ e^{\gamma X_t}\right]}\sum_{k=1}^{\infty}k^{-3/2}.
\end{align*}
By combining the above inequality with Lemma~\ref{lem:bes_rn_derivative} we obtain that 
\begin{align*}
\E_z\!\left[ \mu_h((-\infty,t-1] \times [0,\pi])^p \right] \leq C_p 
\end{align*}
and proceeding similarly, 
\begin{align*}
\E_z\!\left[ \mu_h([t+1,0] \times [0,\pi])^p\right] \leq C_p
\end{align*}
for all $p \in (0,1)$ sufficiently small. Thus the result follows for small enough $p > 0$.
\end{proof}
We will now construct an $H_1(\C)$ approximation $\wt{\phi}_{z,r}$ to the $\log$ function centered at $z$ and equal to $\log|z-x|$ for $x \notin B(z,r)$ which is going to be useful in  what follows. Let $\psi$ be a smooth function on $\R$ which is non-decreasing and such that $\psi(x) = 0$ for all $x\leq 0$ and $\psi(x) = 1$ for all $x \geq 1$. We also consider the continuous function $f$ on $\R$ with $f(s) = \frac{1}{r}\psi(\tfrac{2}{r}(s-\tfrac{r}{2}))$ for $0 \leq s \leq r$ and $f(s) = \frac{1}{s}$ for $s \geq r$. Then we set $\wt{\phi}_{z,r}(x) = F(|x-z|)$ where $F(x) = \log r + \int_r^{|x|}f(s)ds$ for $x \in \C$. We then have that $\wt{\phi}_{z,r}(x) = \log |x-z|$ for $x \notin B(z,r)$ and $\wt{\phi}_{z,r}(x) = \log r + O(1)$ for $x \in B(z,r)$, $\int_{B(z,r)}|\nabla \wt{\phi}_{z,r}(x) |^2 dx = O(1)$ where the $O(1)$ is uniform in $x,r$.

With the above lemmas at hand, we prove the following lemma, which is the main ingredient in the proof of Proposition~\ref{prop:mass_lower_bound}. 
\begin{lemma}
\label{lem:area_bounds_box_near_z}
There exists $p_0 \in (0,1)$ such that for each $p \in (0,p_0)$ there exists a constant $c_p > 0$ depending only on $p$ so that for each $z \in \strip_-$ with $\re(z) \leq -1$ and $r = \dist(z,\partial \strip)$ the following is true.  Let $\wt{\phi}_{z,r}$  be as above and let $\phi_{z,r}$ be the projection of $\wt{\phi}_{z,r}$ onto $H_2(\strip)$ (i.e., obtained by starting with $\wt{\phi}_{z,r}$ and then subtracting its mean on vertical lines) and let $\alpha > 1/\sqrt{2}$.  Then we have that
\begin{align*} 
	\E_z[ \mu_{h + \alpha \phi_{z,r}} (\strip_-)^p] \leq c_p.
\end{align*}
\end{lemma}
\begin{proof}
First we assume that $\re(z) = -1$ and suppose that we have the same setup as in Lemmas~\ref{lem:mass_upper_bound1} and~\ref{lem:area_near_z_bound}. We let $\phi(w) = e^w$ and set $z_0 = \phi(z) \in \partial B(0,e^{-1})$ and $r_0 = \im(z_0) \in (0,1)$. We consider the field $\wh h = \wt h + \gamma \greenN{\h}(z_0,\cdot) + \alpha \log |z_0 - \cdot|$. We observe that $\alpha \log|z_0-w| \leq \alpha \log r_0-\alpha k \log 2$ for all $w \in A_k = B(z_0,r_0 2^{-k}) \setminus B(z_0,r_0 2^{-k-1})$ and $\alpha \log|z_0-w| \leq \alpha \log r_0 +\alpha k \log 2$ for all $w \in \wh A_k = \h \cap B(z_0,r_02^k) \setminus B(z_0,r_0 2^{k-1})$. Hence, by summing over $k \in \N$ as in the proof of Lemma~\ref{lem:mass_upper_bound1}, we obtain that $\E\!\left[ \mu_{\wh h}(B(z_0,r_0))^p\right] \leq C_p$ and by summing over $k = 1,\ldots,k_{r_0}$ we have that $\E\!\left[ \mu_{\wh h}(\h \cap B(z_0,1) \setminus B(z_0,r_0))^p \right] \leq C_p$. Also $\alpha \log|z_0-w| \leq \alpha\log 2$ in $\D_+ \setminus B(z_0,1)$, and thus, as in Lemma~\ref{lem:mass_upper_bound1}, we obtain that $\E\!\left[ \mu_{\wh h}(\D_+)^p\right] \leq C_p$ for $p$ sufficiently small, where $C_p$ depends only on $p$. Moreover we have that $|\log|z_0-w|-\log|\log w-z||$ is bounded for all $w \in 
\D_+ \setminus e^{-2}\D_+$, uniformly in $z_0$ and since $\log|z-\log w| = \wt{\phi}_{z,r}(\log w)$ for all $w \in \phi(A \setminus B(z,r))$,~\eqref{eq:LQG_coordinate_change} implies that 
\begin{align*}
\E\!\left[ \mu_{h^1}(A \setminus B(z,r))^p \right] 
\leq C_p.
\end{align*}
where $h^1 = \wt h \circ \phi + \gamma \greenN{\strip}(z,\cdot) + \alpha \wt{\phi}_{z,r}$.
Proceeding similarly to Lemma~\ref{lem:area_near_z_bound} and using the above, we get that 
\begin{align}\label{eq:bound_lateral+log}
\E_z\!\left[ \mu_{h_2 + \alpha \wt{\phi}_{z,r}}(A \setminus B(z,r))^p \right] \leq \wt C_p \E\!\left[ \mu_{h^1}(A \setminus B(z,r))^{2p} \right]^{1/2} \leq C_p
\end{align}
where $C_p$ depends only on $p$. It is easy to see that $\int_0^\pi \wt{\phi}_{z,r}(\re(w) + it) dt \geq -1$ and hence that $\phi_{z,r}(w) = \wt{\phi}_{z,r}(w) - \tfrac{1}{\pi} \int_0^\pi \wt{\phi}_{z,r}(\re(w) + it) dt \leq \wt{\phi}_{z,r}(w) + \frac{1}{\pi}$ for all $w \in \strip$. Thus, combining this with the fact that $h_1 \leq 0$ on $\strip_-$, we have by~\eqref{eq:bound_lateral+log} that 
\begin{align} \label{eqn:16}
\E_z\!\left[ \mu_{h + \alpha \phi_{z,r}}(A \setminus B(z,r))^p \right] \leq \E_z\!\left[ \mu_{h_2 + \alpha \phi_{z,r}}(A \setminus B(z,r))^p \right] \leq C_p.
\end{align}
We note that by invariance of the law under horizontal translation, the bound~\eqref{eqn:16} is independent of $\re(z)$.  Moreover, $\wt{\phi}_{z,r}(w) \leq \log r +O(1)$ on $B(z,r)$ and thus Lemma~\ref{lem:area_near_z_bound} and the above reasoning imply that 
\begin{align} \label{eqn:18}
	\E_z\!\left[ \mu_{h + \alpha \phi_{z,r}}(B(z,r))^p \right] \leq C_p,
\end{align}
for $p>0$ sufficiently small since $Q + \alpha -2\gamma > 0$. Thus, by~\eqref{eqn:16} and~\eqref{eqn:18}, 
\begin{align} \label{eqn:19}
	\E_z\!\left[ \mu_{h + \alpha \phi_{z,r}}(A)^p \right] \leq C_p
\end{align}
for all $p \in (0,1)$ sufficiently small and $z$ such that $\re(z) \leq -1$. Furthermore we have that $\phi_{z,r}(w) = O(1)$ uniformly for $w \in \strip_- \setminus A$ and so Lemma~\ref{lem:mass_upper-bound2} implies that 
\begin{align} \label{eqn:20}
	\E_z\!\left[ \mu_{h + \alpha \phi_{z,r}}(\strip_- \setminus A)^p \right] \leq C_p
\end{align}
Thus the result follows from~\eqref{eqn:19} and~\eqref{eqn:20}, since $p \in (0,1)$ and since $C_p$ depends only on $p$. 
\end{proof}
We now prove Proposition~\ref{prop:mass_lower_bound}.
\begin{proof}[Proof of Proposition~\ref{prop:mass_lower_bound}]
We fix $\alpha > 1/\sqrt{2}$ such that $\beta = \alpha^2/4$. Fix also $z \in (\strip_- -1)$ and set $r = \dist(z, \partial \strip_-)$ and suppose that we have the setup of Lemma~\ref{lem:area_bounds_box_near_z}.  First we observe that we can find a finite universal constant $C > 0$ such that 
\begin{align} \label{eqn:21}
\|\phi_{z,r}\|_\nabla^2 \leq C + \frac{1}{2}\log(r^{-1})
\end{align}
where the Dirichlet energy is considered on $\strip$. We note that the law of $h_2$ under $\p_z$ is given by $h_2^f + \gamma g(\cdot)$ where $h_2^f$ (resp.\ $g$) is the projection onto $H_2(\strip)$ of a free boundary GFF on $\strip$ (resp.\ $\greenN{\strip}$). Hence the law of $h$ under $\p_z$ can be obtained by weighting the law of $h + \alpha \phi_{z,r}$ under $ \p_z$ by 
\begin{align} \label{eqn:22}
e^{(h_2^f,-\alpha \phi_{z,r})_\nabla - \frac {\alpha^2}{2} \| \phi_{z,r} \|_\nabla^2}
\end{align}
Note that $(h_2^f,-\alpha \phi_{z,r})_\nabla$ is a mean zero Gaussian random variable with variance $\alpha^2 \| \phi_{z,r} \|_\nabla^2$. Moreover, by Lemma~\ref{lem:area_bounds_box_near_z}  we have for sufficiently small $p \in (0,1)$ that we can find a finite constant $C_p >0$ depending only on $p$ such that 
\begin{align*}
\E_z\!\left[ \mu_{h + \alpha \phi_{z,r}}(\strip_-)^p \right] \leq C_p,
\end{align*}
uniformly in $z$ and $r$. Hence, Markov's inequality together with the fact that $\p_z[ (h_2^f,-\alpha \phi_{z,r})_\nabla \geq 0 ] = \tfrac{1}{2}$ implies that we can find a finite $M >0$, sufficiently large, and a universal constant $q \in (0,1)$, such that 
\begin{align} \label{eqn:23}
\p_z\!\left[ \mu_{h + \alpha \phi_{z,r}}(\strip_-) \leq M,\ (h_2^f,-\alpha \phi_{z,r})_\nabla \geq 0 \right] \geq q.
\end{align}
Moreover if the event in~\eqref{eqn:23} occurs, then the Radon-Nikodym derivative in~\eqref{eqn:22} is at least $e^{-\frac{\alpha^2}{2} \| \phi_{z,r} \|_\nabla^2}$. This, together with~\eqref{eqn:21} and~\eqref{eqn:23} implies that we can find a constant $c > 0$ such that 
\begin{align*}
\p_z\!\left[ E \right] \geq c r^{\frac{\alpha^2}{4}}.
\end{align*}
This completes the proof.
\end{proof}
With Proposition~\ref{prop:mass_lower_bound} at our disposal, we are now ready to prove Proposition~\ref{prop:density_lower_bound}.
\begin{proof}[Proof of Proposition~\ref{prop:density_lower_bound}]
Suppose that $(\h,h,0,\infty) \sim \qwedgeW{\sqrt{2}}{1}$ and fix $\epsilon > 0$.  First we consider the probability measure defined on distributions given by 
\begin{align*}
\p_{z,\epsilon}[A] = \frac{\E\left[ \one_A \mu_h(B(z,\epsilon)) \right]}{\E\left[ \mu_h(B(z,\epsilon))\right]}
\end{align*}
By applying a similar method to \cite[Lemma~A.7]{dms2014mating} we obtain that a sample from $\p_{z,\epsilon}$ can be obtained as follows: 
\begin{itemize}
\item We sample $w \in B(z,\epsilon)$ from the probability measure on $B(z,\epsilon)$ whose density with respect to the Lebesgue measure is given by:
\begin{align*}
\frac{\E_{w}[\mu_{h}(dw)]}{Z}
\end{align*}
where $Z$ is a normalizing constant.
\item Next given w, we sample the law of the field from $\p_w$.
\end{itemize}
Hence $\p_{z,\epsilon}$ converges weakly to $\p_z$ as $\epsilon \to 0$. Let $M > 0$ be the constant of Proposition~\ref{prop:mass_lower_bound} and set $\wh h = h \circ \phi^{-1} + Q \log(|(\phi^{-1})'|)$, where $\phi : \h \mapsto \strip$ with $\phi(w) = \log w$. Then $(\strip,\wh{h},-\infty,+\infty) \sim \qwedgeW{\sqrt{2}}{1}$ has the first exit parameterization and thus Proposition~\ref{prop:mass_lower_bound} implies that there exist constants $M,c > 0$ such that 
\begin{align*}
\p_{z_0}[\wh E]\geq c \dist(z_0,\partial \strip)^{\beta},
\end{align*}
where $z_0 = \log z$ and $\wh E = \{\mu_{\wh h}(\strip_-) \leq M\}$. Suppose that $\im(z_0) \leq \pi/2$. Then $\dist(z_0,\partial \strip) = \arg(z)$ and so the coordinate change formula for quantum surfaces implies that 
\begin{align*}
\p_z[E] = \p_{z_0}[\wh E] \geq c \arg(z)^{\beta}
\end{align*}
Also note that $\sin(\arg(z)) = \im(z)/|z| \leq \arg(z)$ and hence 
\begin{align} \label{eqn:24}
\p_z[E] \geq c \im(z)^{\beta}|z|^{-\beta}.
\end{align}
If $\im(z_0) \geq \pi/2$ then~\eqref{eqn:24} holds by an analogous argument. Note also that 
\begin{align} \label{eqn:25}
	\p_z[E] &= \lim_{\epsilon \to 0}\frac{\E\left[ \one_E \mu_h(B(z,\epsilon))\right]}{\E\left[ \mu_h(B(z,\epsilon))\right]} = \lim_{\epsilon \to 0}\frac{\E\left[ \one_E \mu_h(B(z,\epsilon))\right]}{\epsilon^2}\frac{\epsilon^2}{\E\left[ \mu_h(B(z,\epsilon))\right]} = \frac{\E[ \one_E \mu_h(dz)]}{\E[\mu_h(dz)]}.
\end{align}
Thus, the result follows from~\eqref{eqn:24} and~\eqref{eqn:25}.
\end{proof}

We mentioned above that in order to prove Proposition~\ref{prop:density_bound}, we compare the density of the intensity measure of a $Q$-quantum wedge with that of an $\alpha$-quantum wedge conditioned on some event with high probability, where $\alpha < Q$. In order to do this, we need to examine the density of the intensity of the latter.

\begin{lemma}
\label{lem:RNint}
Fix $\gamma \in (0,2)$ and $\alpha < Q$. Let $\mathcal{W} = (\h,h,0,\infty) \sim \qwedgeA{\gamma}{\alpha}$ have the circle-average embedding and let $f_\alpha(z) \coloneqq  \E[\mu_h(dz)]/dz$ be the Radon-Nikodym derivative between the intensity of $\mu_h$ and two-dimensional Lebesgue measure. Then
\begin{align*}
	f_\alpha(z) \asymp |z|^{-\alpha\gamma} \im(z)^{-\gamma^2/2} \quad\text{for}\quad z \in \D_+.
\end{align*}
\end{lemma}

\begin{proof}

We note that when restricted to $\D_+$, $h = h^f - \alpha \log | \cdot |$, where $h^f$ is a free boundary GFF normalized so that its average on $\h \cap \partial \D$ is $0$.  Letting $\theta$ denote the uniform probability measure on $\h \cap \partial \D$, we can write $(h^f,\rho) = (\wt{h}^f,\rho - \theta)$, where $\wt{h}^f$ is a free boundary GFF without normalization.  Letting $\rho_{z,\delta}$ denote the uniform probability measure on $\partial B(z,\delta)$, we have that
\begin{align}
	\E[ \mu_{h^f}(dz)] &= \lim_{\delta \to 0} \delta^{\gamma^2/2} \E\!\left[ e^{\gamma(h_f,\rho_{z,\delta})} \right] dz = \lim_{\delta \to 0} \delta^{\gamma^2/2} \E\!\left[ e^{\gamma(\wt{h}_f,\rho_{z,\delta}-\theta)} \right] dz \nonumber \\
	&= \lim_{\delta \to 0} \delta^{\gamma^2/2} e^{\frac{\gamma^2}{2} \var[(\wt{h}_f,\rho_{z,\delta}-\theta)}] dz. \label{eq:intensity_fb_gff_1}
\end{align}
Here, we have that
\begin{align*}
	\var[(\wt{h}_f,\rho_{z,\delta}-\theta)] = \iint \greenN{\h}(w_1,w_2)(\rho_{z,\delta}-\theta)(dw_1) (\rho_{z,\delta}-\theta)(dw_2), 
\end{align*}
where $\greenN{\h}(w_1,w_2) = - \log|w_1-w_2| - \log|w_1-\ol{w_2}|$ is the Green's function with Neumann boundary data on $\h$.  We note that there exists a universal constant $C<\infty$ such that $\int |\log|x-y|| \theta(dy) \leq C$ and $\int |\log|x-\overline{y}|| \theta(dy) \leq C$,  for all $x \in B(0,5)$.  It follows that
\begin{align*}
\int \int |\greenN{\h}(w_1,w_2)| \rho_{z,\delta}(dw_1) \theta(dw_2) \leq 2C,  \quad 
\int \int |\greenN{\h}(w_1,w_2)| \theta(dw_1) \theta(dw_2) \leq 2C,
\end{align*}
whenever $|z| \leq 1-\delta$,  which implies that
\begin{align*}
-4C + \int \int \greenN{\h}(w_1,w_2) \rho_{z,\delta}(dw_1) \rho_{z,\delta}(dw_2)& \leq \var[ (\wt{h}_f ,  \rho_{z,\delta}-\theta)] \\
&\leq 4C + \int \int \greenN{\h}(w_1,w_2) \rho_{z,\delta}(dw_1) \rho_{z,\delta}(dw_2).
\end{align*}
Moreover,  whenever $0 < \delta < \tfrac{1}{2}(\im(z) \wedge (1-|z|))$,  we have that
\begin{align*}
\int \int \log |x-y| \rho_{z,\delta}(dx) \rho_{z,\delta}(dy) = \log \delta + \frac{1}{4\pi^2}\int_0^{2\pi} \int_0^{2\pi} \log |e^{i\vartheta}-e^{i\phi}| d\phi d\vartheta,
\end{align*}
and
\begin{align*}
\int \int \log |x-\overline{y}| \rho_{z,\delta}(dx) \rho_{z,\delta}(dy) &= \frac{1}{4\pi^2} \int_0^{2\pi} \int_0^{2\pi} \log |z-\overline{z} + \delta (e^{i\vartheta} - e^{i\phi})|d\phi d\vartheta \\
&=\log |z-\overline{z}| + O(1).
\end{align*}
Overall,  we obtain that there exists a universal constant $M>0$ such that
\begin{align*}
-M-\log \delta-\log |z-\overline{z}| \leq \var[(\wt{h}^f,\rho_{z,\delta}-\theta)] \leq M -\log \delta-\log |z-\overline{z}|
\end{align*}
for all $z \in \D_+,  0<\delta< \tfrac{1}{4}(\im(z) \wedge (1-|z|))$,  which implies that 
\begin{align*}
\E\big[ \mu_{h^f}(dz) \big] \asymp \lim_{\delta \to 0} \delta^{\gamma^2 / 2} e^{\tfrac{\gamma^2}{2}(-\log \delta-\log(|z-\overline{z}|))} = (2\im(z))^{-\gamma^2 / 2} dz.
\end{align*}
Finally,  since $h = h^f - \alpha \log |\cdot|$ in $\D_+$,  we have in $\D_+$,
\begin{align*}
\E\big[ \mu_h(dz) \big] = |z|^{-\alpha \gamma} \E\big[ \mu_{h^f}(dz) \big] \asymp |z|^{-\alpha \gamma} (\im(z))^{-\gamma^2 / 2}dz.
\end{align*}
This completes the proof of the lemma.
\end{proof}

We are now ready to prove Proposition~\ref{prop:density_bound}.
\begin{proof}[Proof of Proposition~
\ref{prop:density_bound}]

The main idea of the proof is to bound from below the density $\frac{\E\big[ \mu_h(dz) \big]}{dz}$ on $\D_+$ by the corresponding density for an $\alpha$-quantum wedge conditioned on a positive probability event with $\alpha < Q$.  Then,  we will use  Lemma~\ref{lem:RNint} to bound the latter density from below. Consider $(\h,h^{W,0},0,\infty) \sim \qwedgeW{\gamma}{\gamma^2 + 2}$ (i.e., $\alpha=0$) with the circle average embedding, lateral part equal to that of $h$ and radial part independent of that of $h$. Then, for $t \geq 0$, $ h_{e^{-t}}^{W,0}(0) - h_{e^{-t}}(0)= B_{2t} + Z_{2t} - Qt$, where $B$ is a Brownian motion and $Z$ is a $\BES^3$, which are chosen to be independent of each other.  Fix $\delta > 0$ small.  Hence letting $C_1 = \sup_{t\geq 0} (B_{2t} -(Q-\delta) t)$ and $C_2 =  \sup_{t \geq 0} (Z_{2t} - \delta t)$ we have that $\mu_{h^{W,0}} \leq e^{\gamma (C_1+C_2)} \mu_h$ on $\D_+$. Fix some large $M>0$ and write $E_1 = \{ C_1 \leq M\}$, $E_2 = \{ C_2 \leq M \}$ and $E_{1,2} = E_1 \cap E_2$.  Then for any $A \subseteq \D_+$,
\begin{align}\label{eq:0toQ}
	&\E\!\left[ \int \one_A d\mu_{h^{W,0}}(z) \middle| E_1 \right] = \E\!\left[ \int \one_A d\mu_{h^{W,0}}(z) \middle| E_{1,2} \right] \nonumber \\
	&\leq \E \!\left[ e^{\gamma (C_1+C_2)} \int \one_A d\mu_h(z) \middle| E_{1,2} \right] \leq e^{2M\gamma} \E\!\left[ \int \one_A d\mu_h \middle| E_2 \right] \lesssim e^{2M\gamma} \E\!\left[ \int \one_A d\mu_h\right]
\end{align}
where the implicit constants can be taken to be independent of $M$ (provided that $M$ is chosen sufficiently large). 

In order to compare the expectations of integrals with respect to $\mu_h$ and $\mu_{h^{W,0}}$, we need to lower bound the conditional expectation of an integral with respect to $\mu_{h^{W,0}}$, given $E_1$, with the corresponding unconditional expectation. We let $h^l$ denote the lateral part of the field. Then,
\begin{align*}
	\E\!\left[\mu_{h^{W,0}}(dz) \middle| E_1 \right] &= \E\!\left[ e^{\gamma B_{2 \log(1/|z|)}} \mu_{h^l}(dz) \middle| E_1 \right] = \E\!\left[ e^{\gamma B_{2 \log(1/|z|)}} \middle| E_1 \right] \E[\mu_{h^l}(dz)] \\
	&= \E\!\left[ e^{\gamma B_{2 \log(1/|z|)}} \middle| E_1 \right] \E[e^{\gamma B_{2 \log(1/|z|)}}]^{-1} \E[ \mu_{h^{W,0}}(dz)] \\
	&= \E\!\left[ e^{\gamma B_{2 \log(1/|z|)}} \middle| E_1 \right] |z|^{\gamma^2} f_0(z) dz.
\end{align*}
Since $\E\!\left[ e^{\gamma B_{2 \log(1/|z|)}} \middle| E_1 \right] = \E\!\left[ e^{\gamma B_{2 \log(1/|z|)}} \one_{E_1} \right] /\p[E_1]$ and $\p[E_1]$ can be made arbitrarily close to $1$ by choosing $M$ sufficiently large, we just have to bound $\E\!\left[ e^{\gamma B_{2 \log(1/|z|)}} \one_{E_1} \right]$ from below. We denote by $\p_z^*$ the measure obtained by weighting the measure $\p$ with the martingale $\exp(\gamma B_{2 \log(1/|z|)} - \gamma^2 \log(1/|z|))$. Sometimes we write $\p_t^* = \p_z^*$ for $t = \log(1/|z|)$ out of convenience. We have
\begin{align*}
	\E\!\left[ e^{\gamma B_{2 \log(1/|z|)}} \one_{E_1} \right] = |z|^{-\gamma^2} \E\!\left[ e^{\gamma B_{2 \log(1/|z|)} - \gamma^2 \log(1/|z|)} \one_{E_1} \right] = |z|^{-\gamma^2} \p_z^*[E_1],
\end{align*}
and thus we must find a lower bound on $\p_z^*[E_1] = \p_t^*[E_1]$. By the Girsanov theorem, we have that under $\p_t^*$, $B$ evolves like a Brownian motion with drift $\gamma$ until time $2t$ and thereafter like an ordinary Brownian motion with no drift.  That is, $B_s = W_s + \gamma \min(s,2t)$ where $W_s$ is a $\p_t^*$-Brownian motion. We note that if we define the events $\CB_1$ and $\CB_2$ by
\begin{align*}
	\CB_1 &= \!\left \{ \sup_{s \leq 2t} \left(W_s + \frac{(2\gamma - Q + \delta)}{2} s\right)  \leq \frac{M}{2} \right\}, \quad \text{and}\\ 
	\CB_2 &= \!\left\{ \sup_{s > 2t} \left(W_s - W_{2t} - \frac{(Q-\delta)}{2}(s-2t)\right) \leq \frac{M}{2} \right\},
\end{align*}
and let $\wt{E}_1 = \CB_1 \cap \CB_2$, then $\wt{E}_1 \subseteq E_1$. Moreover, $\CB_1$ and $\CB_2$ are independent, and hence $\p_t^*[E_1] \geq \p_t^*[\CB_1] \p_t^*[\CB_2]$. Since $W_s$ is a $\p_t^*$-Brownian motion, we have by \cite[Part~II, Section~2.1]{bs2002bmbook} that
\begin{align*}
	\p_t^*[\CB_2] = \p_t^*\!\left[ \sup_{s > 0} \left(W_s - \frac{(Q-\delta)}{2} s\right) \leq \frac{M}{2} \right] = 1 - \exp\!\left(-\frac{(Q-\delta)M}{2} \right).
\end{align*}
We note further that since $\gamma = \sqrt{2}$, we have that $K \coloneqq (2\gamma-Q+\delta)/2 > 0$ and hence the drift in the event $\CB_1$ is positive. By the Girsanov theorem and optional stopping (see e.g. \cite[Chapter~13.2]{bass2011stochastic}) we have that
\begin{align*}
	\p_t^*[\CB_1^c] = \int_0^{2t} \frac{M/2}{\sqrt{2\pi} s^{3/2}} e^{\frac{MK}{2} - \frac{K^2 s}{2} - \frac{M^2}{8s}} ds.
\end{align*}
Since $K$ is positive we have that $\sup_{s>0} (W_s+Ks) > M/2$ a.s., that is, the value of the above integral, with $\infty$ in place of $2t$ as the upper limit, is $1$. Consequently,
\begin{align*}
	\p_t^*[\CB_1] = \int_{2t}^\infty \frac{M/2}{\sqrt{2\pi} s^{3/2}} e^{\frac{MK}{2} - \frac{K^2 s}{2} - \frac{M^2}{8s}} ds \gtrsim \frac{M e^{\frac{M}{2} K}}{2 \sqrt{2\pi}} e^{-(K^2+o_t(1))t} \geq  \frac{M e^{\frac{M}{2} K}}{2 \sqrt{2\pi}} e^{-\frac{\gamma^2}{8} t},
\end{align*}
for large $t$, since $K^2 = 1/8 + O(\delta) = \gamma^2/16 + O(\delta) < \gamma^2/8$. Consequently, recalling that $t = \log(1/|z|)$, we have that
\begin{align*}
	\E \!\left[ e^{\gamma B_{2\log(1/|z|)}} \middle| E_1 \right] \gtrsim |z|^{-\frac{7}{8} \gamma^2},
\end{align*}
and hence
\begin{align}\label{eq:compintcond}
	\E[ \mu_{h^{W,0}}(dz) | E_1 ] \gtrsim |z|^{\frac{\gamma^2}{8}} f_0(z) dz =  |z|^{\frac{\gamma^2}{8}} \E[ \mu_{h^{W,0}}(dz) ].
\end{align}
By combining~\eqref{eq:0toQ} and~\eqref{eq:compintcond} with  Lemma~\ref{lem:RNint}, we obtain that
\[ \E[ \mu_h(dz) ] \gtrsim |z|^{\frac{\gamma^2}{8}}\im(z)^{-\frac{\gamma^{2}}{2}} \quad\text{for all}\quad z \in \D_+.\]
The claim then follows by applying Proposition~\ref{prop:density_lower_bound}.
\end{proof}

\section{Modulus of continuity of $\SLE_8$}
\label{sec:sle8}
In this section, we prove Theorem~\ref{thm:sle8_thm}. Let $\eta'$ be an $\SLE_8$ process in $\h$ from $0$ to $\infty$ parameterized by capacity.  We first explain why the event of Lemma~\ref{lem:main_lemma_ubd} locally determines the modulus of continuity.  Let $\tau_w$ denote the first hitting time of $w \in \h$ for $\eta'$.  Fix $\xi > 1$ and $\sigma > 0$ and let $\epsilon > 0$.  For $w \in \h$, assume that $z_0$ is such that $B(z_0,\epsilon^\xi) \subseteq B(w,\epsilon) \cap \h_{\tau_w}$ gets swallowed by $\wt{\eta}'(t) \coloneqq \eta'(\tau_w + t)$ before hitting $\partial B(w,\epsilon)$ and the probability that a Brownian motion $B^{z_0}$ started at $z_0$ exits $\h \setminus \eta'([0,\tau_w])$ in $\R \setminus [-M,M]$ for some $M>0$ is at least $\exp(-\epsilon^{-\sigma})$. By the conformal invariance of Brownian motion and that $g_{\tau_w}(\overline{\h} \cap (\partial \eta'([0,\tau_w]) \cup [-M,M]))$ is a.s.\ bounded (and we emphasize that this is all we need, since the constant in Theorem~\ref{thm:sle8_thm} is random and depends on the realization of $\eta'$), we have that the probability that $B_t^{z_0}$ exits the domain in $\R \setminus [-M,M]$ is comparable to the probability that a Brownian motion started at $g_{\tau_w}(z_0)$ hits the line $\{\im(z) = 1 \}$ before $\R$\footnote{Indeed, it is easy to see that if $S < T$, $-S \leq \re(w) \leq S$ and $0 < \im(w) < 1$, then $\p(B^w \ \text{exits} \ \h \ \text{in} \ \R \setminus [-T,T]) \asymp \im(w) = \p(B^w \ \text{hits} \ \{ \im(z) = 1 \} \ \text{before} \ \R)$, where the implicit constants depend only on $S$ and $T$.}. 
Since the former probability is at least $\exp(-\epsilon^{-\sigma})$, we necessarily have that
\begin{align*}
	\im(g_{\tau_w}(z_0)) \gtrsim \exp(-\epsilon^{-\sigma}).
\end{align*}
By \cite[Lemma~1]{lln2009hcap} we have that for any compact $\h$-hull $A$, $\hcap(A) \geq (\sup\{ \im(z): z \in A\})^2/2$. Thus, if $\wt{\tau}$ is the first time after $\tau_w$ that $\eta'$ hits $\partial B(w,\epsilon)$, then $\hcap (g_{\tau_w}(\eta'([\tau_w,\wt{\tau}]))) \gtrsim \exp(-2\epsilon^{-\sigma})$. Consequently, when $\eta$ travels distance $\epsilon$, it accumulates at least a constant times $\exp(-2\epsilon^{-\sigma})$ units of half-plane capacity. Or, equivalently, when $\eta'$ grows for time $\delta$, it travels a distance of at most a constant times $|\log\delta|^{-1/\sigma}$. That is, locally (i.e., for $s,t$ close to $\tau_w$) we have that 
\begin{align}\label{eq:BMmod}
	|\eta'(t)-\eta'(s)| \leq C \!\left(\log \left(1+ \frac{1}{|t-s|}\right) \right)^{-1/\sigma} \! .
\end{align}
It is clear that if this bound does not hold for a constant $C$, then the escape probability of the Brownian motion is necessarily smaller than $\exp(-\epsilon^{-\sigma})$, proving the equivalence of the escape probability and the modulus of continuity of $\eta'$.

The proof of the first part of Theorem~\ref{thm:sle8_thm} relies on the observation that if there is some ball $B(z_0,\epsilon^\xi)$ in $\h_{\tau_{w}} \cap B(\eta'(\tau_{w}),\epsilon)$ such that the escape probability of a Brownian motion is at most $\exp(-\epsilon^{-\sigma})$, then for each $z$ sufficiently close to $z_0$, the corresponding event (with a different, deterministic, constant $C_0$ as in the above remark) occurs. In order to use this, we need to know that $\eta'([\tau_{z},\tau_{z}(\epsilon)])$ contains a ball of radius $\epsilon^\xi$, \textit{simultaneously}, for all $z$ in a fixed bounded set with high probability as $\epsilon \to 0$. Lemma~\ref{lem:SLE8fills} states that for any fixed $z$, $\eta'([\tau_{z},\tau_{z}(\epsilon)])$ contains a ball of radius $\epsilon^\xi$, which is not quite enough. The next lemma shows that with high probability as $\epsilon \to 0$, it is indeed true simultaneously for $z \in \D_+$.

\begin{lemma}
\label{lem:SLE8fillsmore}
Let $\eta'$ be an $\SLE_8$ process in $\h$, from $0$ to $\infty$ and fix $\xi > 1$. Let $\mathscr{E}_\epsilon^\xi$ be the event that for all $z \in \D_+$, $\eta'([\tau_z,\tau_z(\epsilon)])$ contains a ball of radius $\epsilon^\xi$. Then $\p[\mathscr{E}_\epsilon^\xi] = 1 - o_\epsilon^\infty(\epsilon)$.
\end{lemma}
\begin{proof}
Fix some $z_0 \in \h$ and let $S$ be the square centered at $z_0$ with side length $l(S) = \im(z_0) /2$. Moreover, let $S^1$, $S^2$ and $S^3$ be the squares with center $z_0$ and side lengths $l(S^1) = \frac{3}{2} l(S)$, $l(S^2) = \frac{7}{4} l(S)$ and $l(S^3) = \frac{15}{8} l(S)$, respectively. For any set $A \subseteq \C$, let $A_\epsilon = A \cap \epsilon\Z^2$. Denote by $\eta_z^\pm$ the interior flow line emanating from $z$ with angle $\pm \frac{\pi}{2}$. For $r \in (0,1)$ and $z \in S_\epsilon^1$, let $\wt{E}_\epsilon^r(z)$ denote the event that there is some $w \neq z$ in $S_\epsilon^2$ such that $\eta_z^-$ hits $\eta_w^-$ on the left side and $\eta_z^+$ hits $\eta_w^+$ on the right side before leaving the ball $B(z,\epsilon^{1-r}/8)$ and let $\wt{\CE}_\epsilon^r = \cap_{z \in S_\epsilon^1} \wt{E}_\epsilon^r(z)$. Then by \cite[Proposition~4.14]{ms2017imag4} (see also the proof of \cite[Lemma~3.8]{ghm2020kpz}), we have that $\p[\wt{\CE}_\epsilon^r ] = 1- o_\epsilon^\infty(\epsilon)$.  Let also $\wh{\CE}_\epsilon^r$ be the event that for all $z \in S_\epsilon^2 \setminus S_\epsilon^1$ we have that if $\eta_z^\pm$ enters $S$ then there exists $w \in S_\epsilon^1 \setminus S$ so that $\eta_w^\pm$ merges with $\eta_z^\pm$ before the paths enter $S$.  Then we similarly have that $\p[\wh{\CE}_\epsilon^r ] = 1- o_\epsilon^\infty(\epsilon)$.

On $\wt{E}_\epsilon^r(z)$, the flow lines $\eta_z^\pm$ and $\eta_w^\pm$ form a pocket which consists of the set $P_z$ of those points in the component of $\h \setminus (\eta_z^\pm \cup \eta_w^\pm)$ whose boundary contains $z$, $w$ and part of the right side of $\eta_z^-$.  On the complement of $\wt{E}_\epsilon^r(z)$, we define $P_z$ to be the set of those points in the component of $\h \setminus \eta_z^\pm$ whose boundary contains $z$ and part of the right side of $\eta_z^-$.  By the proof of \cite[Lemma~3.8]{ghm2020kpz}, we have that $\diam \, P_z \leq \epsilon^{1-r}$ for each $z \in S_\epsilon^1$.  On $\wh{\CE}_\epsilon^r$, we also claim that $S$ is contained in $\cup_{z \in S_{\epsilon}^1} P_z$.  To see this, we first note that $\h = \cup_{z \in S_\epsilon^2} P_z$.  Suppose that $z \in S_\epsilon^2 \setminus S_\epsilon^1$.  For $P_z$ to intersect $S$, it must be that $\eta_z^+$ or $\eta_z^-$ enters $S$ which is ruled out in the definition of $\wh{\CE}_\epsilon^r$.

For each $z$ let $E_\epsilon^r(z)$ be the event that $\eta'$ swallows a ball of radius $\epsilon^{1+r}$ in the time interval between first hitting $z$ and the next time it leaves the ball $B(z,\epsilon)$ and let $\ol{\CE}_\epsilon^r = \cap_{z \in S_\epsilon^2} E_\epsilon^r(z)$.  Let also $\underline{\CE}_\epsilon^r = \wt{\CE}_\epsilon^r \cap \wh{\CE}_\epsilon^r \cap \ol{\CE}_\epsilon^r$.  Then by \cite[Lemma~3.6]{ghm2020kpz} we have that $\p[ \underline{\CE}_\epsilon^r] = 1 - o_\epsilon^\infty(\epsilon)$.

On the event $\underline{\CE}_\epsilon^r$ we have that if $\eta'$ travels distance $\epsilon^{1-r}$ it has to leave some pocket (since each pocket has diameter at most $\epsilon^{1-r}$) and thus whenever $\eta'$ travels distance $\epsilon^{1-r} + \epsilon$, it has to fill in a ball of radius $\epsilon^{1+r}$ (since traveling distance $\epsilon^{1-r}$ forces it to leave a pocket and hit a point in the grid, and from this point it will fill in a ball when traveling distance $\epsilon$).

The above implies that if we fix some $a \in (r,1)$ and let $\CE_\epsilon^r$ denote the event $\underline{\CE}_\epsilon^r$ with $\D \cap \{ \im(z) \geq \epsilon^{1-a} \}$ in place of $S$, then $\p[\CE_\epsilon^r] = 1-o_\epsilon^\infty(\epsilon)$, and on $\CE_\epsilon^r$, we have that whenever $\eta'$ travels distance $\epsilon^{1-r} + \epsilon$ in $\D \cap \{ \im(z) \geq \epsilon^{1-a} \}$, it fills in a ball of radius $\epsilon^{1+r}$, where the grid in question covers, say $\D \cap \{ \im(z) \geq \epsilon^{1-a}/2\}$.  Fix $b \in (0,a)$. By the proof of Lemma~\ref{lem:SLE8fills} we have that dividing $[-1,1] \times [0,\epsilon^{1-a})$ into rectangles $R_j$ of base length $\epsilon^{1-b}$ and letting flow lines of angle $\pi/2$ (resp.\ $-\pi/2$) if they are on the left (resp.\ right) side of $0$, grow at spacing $2\epsilon^{1-(a+3b)/4}$, we have that if $A_{b,\epsilon}(j)$ is the event that at least one flow line in the rectangle $R_j$ hits the line $\{ \im(z) = \epsilon^{1-a} \}$, then $\p[A_{b,\epsilon}(j)^c] = o_\epsilon^\infty(\epsilon)$. Consequently, letting $\CA_\epsilon^b = \cap_j A_{b,\epsilon}(j)$, we have that $\p[\CA_\epsilon^b] = 1 - o_\epsilon^\infty(\epsilon)$ and hence $\p[ \CE_\epsilon^r \cap \CA_\epsilon^b] = 1 - o_\epsilon^\infty(\epsilon)$. Moreover, on $\CE_\epsilon^r \cap \CA_\epsilon^b$ we have that whenever $\eta'$ travels distance $2\epsilon^{1-b} + \epsilon$, starting from any point in $\D \cap \{ \im(z) < \epsilon^{1-a} \}$, it fills a ball of radius $\epsilon^{1+r}$. This is because when traveling distance $2 \epsilon^{1-b}$, it will hit $\D \cap \{ \im(z) \geq \epsilon^{1-a} \}$ and enter a pocket at some point in the grid and from there fill a ball when traveling the further $\epsilon$ distance. Thus, the proof is done.
\end{proof}

We also state and prove the following before we prove Theorem~\ref{thm:sle8_thm}.

\begin{lemma}\label{lem:QwedgeUB}Let $\CW = (\h,h,0,\infty)\sim \qwedgeW{\sqrt{2}}{1}$. Suppose that it is embedded such that $\nu_h([-1/2,0]) = 1$. Then there exist finite constants $c_1,c_2>0$ such that 
\begin{align*}
\p[ \mu_h(\D_+)\leq \epsilon ] \leq c_1 e^{-c_2 \epsilon^{-1}} 
\end{align*}
for all $\epsilon > 0$.
\end{lemma}

\begin{proof}
Let $\eta'$ be an $\SLE_8$ in $\h$ from $0$ to $\infty$ sampled independently of $\CW$. We subsequently parameterize $\eta'$ by quantum area with respect to $h$ and set $\tau = \inf\{t \geq 0 : \eta'(t) \notin \D_+\}$ and $\sigma = \inf\{t \geq 0 : [-1/2,0] \subseteq \eta'([0,t]) \}$. Note that the event $\{\sigma \leq \tau\}$ depends only on the capacity parameterization of $\eta'$ and thus it is independent of $h$. Thus
\begin{align*}
\p[ \sigma \leq \tau ] \p[ \mu_h(\D_+) \leq \epsilon ] &= \p[ \sigma \leq \tau,\mu_h(\D_+) \leq \epsilon ] \leq \p[ \mu_h(\eta'([0,\sigma])) \leq \epsilon,\sigma \leq \tau]
\end{align*}
and so 
\begin{align*}
\p[ \mu_h(\D_+) \leq \epsilon ] \leq & \p[ \mu_h(\eta'([0,\sigma])) \leq \epsilon | \sigma \leq \tau] \leq \frac{\p[ \mu_h(\eta'([0,\sigma]))\leq \epsilon ]}{\p[ \sigma \leq \tau ]}.
\end{align*}
Since $\p[ \sigma \leq \tau ] > 0$ is independent of $\epsilon$, it suffices to give an upper bound for $\p[ \mu_h(\eta'([0,\sigma])) \leq \epsilon ]$. For $t \geq 0$, let $L_t$ be the quantum length of the arc on $\partial \h_t$ connecting $\eta'(t)$ with the leftmost point $x_t^L$ of $\R \cap \eta'([0,t])$ minus the quantum length of $[x_t^L,0]$. By Theorem~\ref{thm:sle8thm}, $(L_t)_{t \geq 0}$ evolves as a standard Brownian motion and if $T = \inf\{t \geq 0 : L_t = -1\}$, then 
\begin{align*}
\p[ \mu_h(\eta'([0,\sigma])) \leq \epsilon ]=\p[ \inf_{0 \leq t \leq \epsilon}L_t \leq -1 ] = \p[T \leq \epsilon] \lesssim e^{-1/(2\epsilon)}.
\end{align*}
This completes the proof.
\end{proof}

\begin{remark}\label{rmk:QwedgeUB}
Suppose that we have the same setup as in Lemma~\ref{lem:QwedgeUB} but $h$ has the first exit parameterization instead. For $\epsilon \in (0,1)$ we set $M_{\epsilon} = (\log(\epsilon^{-1}))^2$, $X = \nu_h([-\tfrac{1}{2},0])$ and $Y = \mu_h(\D_+)$.  Fix $k \in \Z$ and let $x > 0$ be such that $\nu_h([-x,0]) = 2^k$. Let also $c_1,c_2 >0$ be the constants of Lemma~\ref{lem:QwedgeUB} and let $Y(x) = \mu_h(\h \cap B(0,2x))$. We claim that
\begin{align}\label{eq:qlbound}
\p[ 2^k \geq \log(\epsilon^{-1})\sqrt{Y(x)} ] \leq c_1 e^{-c_2 (\log(\epsilon^{-1}))^2}
\end{align}
Indeed, consider the conformal transformation $\psi : \h \rightarrow \h$ with $\psi(z) = z/2x$ and set $\wt h = h \circ \psi^{-1} + Q\log(2x) + C$, where $C = -\frac{2k}{\gamma}\log(2)$. Then $(\h,\wt h,0,\infty) \sim \qwedgeW{\sqrt{2}}{1}$ and it is parameterized as in Lemma~\ref{lem:QwedgeUB}. Also $Y(x) = 2^{2k} \mu_{\wt h}(\D_+)$ which implies that
\begin{align*}
\p[2^k \geq \log(\epsilon^{-1})\sqrt{Y(x)}] = \p[1 \geq \log(\epsilon^{-1}) \sqrt{\mu_{\wt h}(\D_+)}]
\end{align*}
and so the claim is deduced by Lemma~\ref{lem:QwedgeUB}. 

Next we have that
\begin{align}\label{eq:qlbound2}
\p[X \geq \log(\epsilon^{-1})\sqrt{Y}]&\leq \sum_{k = -M_{\epsilon}}^{M_{\epsilon}}\p[X  \geq \log(\epsilon^{-1}) \sqrt{Y},\quad 2^k \leq X < 2^{k+1}] + \p[X \leq 2^{-M_{\epsilon}}] + \p[X > 2^{M_{\epsilon}}]\nonumber\\
&\leq \sum_{k = -M_{\epsilon}}^{M_{\epsilon}}\p[2^{k+1} \geq \log(\epsilon^{-1})\sqrt{Y(x)}] + \p[X \leq 2^{-M_{\epsilon}}] + \p[ X > 2^{M_{\epsilon}}].
\end{align}
Moreover, Lemma~\ref{lem:pthmoment} implies that $\E[X], \E[X^{-1}] < \infty$ since $\gamma = \sqrt{2}$, and hence
\begin{align}\label{eq:qlbound3}
\p[X \leq 2^{-M_{\epsilon}}] \leq \E[X^{-1}]2^{-M_{\epsilon}},\quad \p[X > 2^{M_{\epsilon}}] \leq \E[X]2^{-M_{\epsilon}}.
\end{align}
By combining~\eqref{eq:qlbound}, \eqref{eq:qlbound2} and \eqref{eq:qlbound3}, we obtain that there exist finite universal constants $\wt c_1, \wt c_2 > 0$ such that
\begin{align*}
\p[X \geq \log(\epsilon^{-1})\sqrt{Y}] \leq \wt c_1 e^{-\wt c_2 (\log(\epsilon^{-1}))^2} \quad \text{for all} \quad \epsilon \in (0,1).
\end{align*}
\end{remark}

We now turn to proving the first part of Theorem~\ref{thm:sle8_thm}. Lemma~\ref{lem:main_lemma_ubd} provides us with a bound on the probability of the bad event in the case when the $\SLE$ is away from the boundary. In the proof, we shall show that one does not need to consider the boundary case as well.  There are two main steps to the proof. The first is to bound the expected area of bad points, that is, points from which the Brownian motion is very unlikely to escape. The second step is then to prove that if there exists such a bad point, then there exists a small ball, in which each point is a bad point. This implies that if the probability that there exists a bad point is positive, then the lower bound on the expected number of bad points exceeds the upper bound provided, which causes a contradiction, proving that with probability $1$, there are no such bad points.
\begin{proof}[Proof of the first part of Theorem~\ref{thm:sle8_thm}]

\noindent{\it Step 1.  Setup.}  Let $\mathcal{W} = (\h,h,0,\infty) \sim \qwedgeW{\sqrt{2}}{1}$ (i.e. $\alpha = Q$) and let $\eta'$ be an independent $\SLE_8$ in $\h$ from $0$ to $\infty$ which is subsequently parameterized by quantum mass. For each $t\geq 0$, the quantum surface parameterized by $\h_t = \h \setminus \eta'([0,t])$ has law $\qwedgeW{\sqrt{2}}{1}$ (Theorem~\ref{thm:sle8thm}).

For $t > 0$ we let $A_t^\xi = A_t^\xi(\epsilon)$ be the event that $\eta'([0,t]) \subseteq \D_+$ and that there exists some $z \in B(\eta'(t),\epsilon)$ such that $B(z,\epsilon^\xi) \subseteq \h_t \cap B(\eta'(t),\epsilon)$ and such that the probability that a Brownian motion starting from $z$ exits $\h_t$ in $\partial \h_t \setminus I_t$ (where $I_t$ is as in Lemma~\ref{lem:mainlemma8}) is at most $e^{-\epsilon^{-4-2a}}$. Recall that by Lemma~\ref{lem:mainlemma8}, $\p[A_t^\xi] = O(\epsilon^{2+a+o(1)})$. Set $R = \log(\epsilon^{-1})$ and let $E_R$ be the event that the boundary length of each of the segments of $\partial \h_t \cap \eta'$ from $\eta'(t)$ to the leftmost and rightmost point of $\R \cap \eta'([0,t])$ is at most $R$ for all $0\leq t \leq M$, where $M$ is the constant of Proposition~\ref{prop:density_bound}.

\noindent{\it Step 2.  Bound on boundary length distance of tip of curve to real line.} We claim that there exist constants $c_1,c_2>0$ such that, for all $\epsilon > 0$, 
\begin{align*}
\p[ E_R^c ] \leq c_1 e^{-c_2 (\log \epsilon)^2}.
\end{align*}
Indeed, let $L_t$ be as in Lemma~\ref{lem:QwedgeUB} and $R_t$ be defined as $L_t$ but with the rightmost point of $\R \cap \eta'([0,t])$ instead. Then since 
\begin{align*}
E_R^c = \!\left\{\sup_{0 \leq t \leq M}\left( L_t -\inf_{0\leq s \leq t}L_s \right)\geq R \right\}\! \cup \!\left\{\sup_{0 \leq t \leq M}\left(R_t - \inf_{0 \leq s \leq t}R_s \right)\geq R \right\}
\end{align*}
and $(L_t),(R_t)$ evolve as standard Brownian motions (Theorem~\ref{thm:sle8thm}), the claim follows by the tail probabilities of the supremum and the infimum processes of a standard Brownian motion.

\noindent{\it Step 3.  Upper bound on the expected area of bad points.}
We let $\eta_c'$ denote the curve $\eta'$ parameterized by capacity, write $\h_t^c = \h \setminus \eta_c'([0,t])$ and let $\sigma_1 = \inf\{t \geq 0: \mu_h(\eta_c'([0,t])) \geq M\}$. We further fix $\delta > 0$ and let $E$ be the event that $\mu_h(\D_+) \leq M$ and $\eta_c'([0,\delta]) \subseteq e^{-1}\D_+$ and recall the definitions of $\tau_z$ and $\tau_z(\epsilon)$ of Lemma~\ref{lem:SLE8fills}. Moreover, we let $B_w^\xi$ the event that there is some $z$ such that $B(z,\epsilon^\xi) \subseteq \eta_c'([\tau_w,\tau_w(\epsilon)])$ and such that the probability that a Brownian motion starting from $z$ exits $\h_{\tau_w}^c$ on $\R \setminus [-1/2,1/2]$ is at most $e^{-\epsilon^{-4-2a}}$. Note that we are done if we prove that $\p[ B_w^\xi \, \text{occurs for some} \, w \in \eta_c'([0,\delta])]$ converges to $0$ as $\epsilon \rightarrow 0$ for some $\xi > 1$. Let also $F$ be the event that $\nu_h([-1/2,0]) \leq \log(\epsilon^{-1})$ and $\nu_h([0,1/2])\leq \log(\epsilon^{-1})$ and note that Remark~\ref{rmk:QwedgeUB} implies that $\p[ F^c \cap E ] \leq c_1 e^{-c_2 (\log \epsilon)^2}$ for some constants $c_1,c_2>0$. Note that 
\begin{align*}
\int_{0}^{M}\one_{A_t^\xi}dt \geq \int_{0}^{M}\one_{A_t^\xi}\one_{E_R}\one_F dt \geq \int \one_{B_w^\xi}\one_E\one_{\{w \in \eta_c'([0,\delta])\}}d\mu_h(w) -M\one_{E_R^c}-M\one_{E_R}\one_E\one_{F^c}
\end{align*}
and so by  taking expectations, noting that $\p[ A_t^\xi ] =  O(\epsilon^{2+a + o(1)})$ by Lemma~\ref{lem:mainlemma8} and applying Proposition~\ref{prop:density_bound} with $\beta \in (1/4,1)$, we obtain that 
\begin{align*}
O(\epsilon^{2+a + o(1)}) &= \E\!\left[\int \one_{B_w^\xi}\one_E \one_{\{w \in \eta_c'([0,\delta])\}}d\mu_h(w)\right] \\
&\gtrsim \E\!\left[ \one_{\{\eta_c'([0,\delta]) \subseteq e^{-1}\D_+\}}\int \one_{B_w^\xi} \one_{\{w \in \eta_c'([0,\delta])\}}|w|^{-\beta + \gamma^2 /8}\im(w)^{\beta - \gamma^2 /2}dw\right] \\
&\gtrsim \E\!\left[ \one_{\{\eta_c'([0,\delta]) \subseteq e^{-1}\D_+\}}\int \one_{B_w^\xi}\one_{\{w \in \eta_c'([0,\delta])\}}dw\right] \\
&=\p[ \eta_c'([0,\delta]) \subseteq e^{-1}\D_+] \E\!\left[ \int \one_{B_w^\xi}\one_{\{w \in \eta_c'([0,\delta])\}}dw \, \middle| \, \eta_c'([0,\delta]) \subseteq e^{-1}\D_+ \right].
\end{align*}
Note that $\p[ \eta_c'([0,\delta]) \subseteq e^{-1}\D_+] > 0$, and that by choosing $\delta$ sufficiently small, we can make this probability as close to $1$ as we want.  In particular, the important part of this step is the conclusion that
\begin{align}\label{eq:event_integral}
	\E\!\left[ \int \one_{B_w^\xi}\one_{\{w \in \eta_c'([0,\delta])\}}dw \, \middle| \, \eta_c'([0,\delta]) \subseteq e^{-1}\D_+ \right] = O( \epsilon^{2+a+o(1)}),
\end{align}
and the point of the next step is to show that the existence of a bad point leads to a contradiction to~\eqref{eq:event_integral}.

Finally, we fix $1 < \xi_1 < \xi < \xi_2$ and let $E_1 = E_1(\epsilon)$ be the event that for all $z \in \D_+$, $\eta_c'([\tau_z,\tau_z(\epsilon)])$ contains some ball of radius $\epsilon^{\xi_1}$. Then, by Lemma~\ref{lem:SLE8fillsmore}, $\p[E_1] = 1-o_\epsilon^\infty(\epsilon)$. Moreover, we let $E_2 = E_2(\epsilon)$ be the event that for all $z \in \D_+$, $\eta_c'([\tau_z,\tau_z(\epsilon^\xi)])$ contains some ball of radius $\epsilon^{\xi_2}$ and note that Lemma~\ref{lem:SLE8fillsmore} implies that $\p[E_2] = 1-o_\epsilon^\infty(\epsilon)$.  Moreover, we note that the order $O(\epsilon^{2+a})$ does not change if one changes $\xi>1$.

\noindent{\it Step 4.  Lower bound on the area of bad points given one exists.}
Assume that the event $B_w^{\xi_1} \cap E_1 \cap E_2$ occurs for some $w \in \eta_c'([0,\delta/2])$, and that $B(z_0,\epsilon^{\xi_1})$ is the ball in the event $B_w^{\xi_1}$. Then, for sufficiently small $\epsilon$, the event $B_{w'}^\xi$ occurs for all $w' \in B(w_0,\epsilon^{\xi_2})$, for some $w_0$ with $B(w_0,\epsilon^{\xi_2}) \subseteq \eta_c'([0,\delta])$. The reason this holds is the following. We have that $\dist(z_0,\h_{\tau_w}^c) \geq \epsilon^{\xi_1} > 2\epsilon^\xi$ for small $\epsilon$ and then $\eta_c'([\tau_w,\tau_w(\epsilon^\xi)]) \cap B(z_0,\epsilon^\xi) = \emptyset$. Moreover,  there exists a ball $B(w_0,\epsilon^{\xi_2}) \subseteq \eta_c'([\tau_w,\tau_w(\epsilon^\xi)])$ and clearly $B(z_0,\epsilon^\xi) \subseteq \eta_c'([\tau_{w'},\tau_{w'}(\epsilon)])$ for all $w' \in B(w_0,\epsilon^{\xi_2})$. Furthermore, $\h_{\tau_{w'}}^c \subseteq \h_{\tau_w}^c$ for all $w' \in B(w_0,\epsilon^{\xi_2})$ and hence the probability that a Brownian motion starting at $z_0$ exits $\h_{\tau_{w'}}^c$ in $\R \setminus [-1/2,1/2]$ is less than or equal to the probability that it exits $\h_{\tau_w}^c$ in $\R \setminus [-1/2,1/2]$, which is at most $\exp(\epsilon^{-4-2a})$. Consequently, if $\p[ B_w^{\xi_1} \, \text{occurs for some} \, w \in \eta_c'([0,\delta/2]) \, | \, \eta_c'([0,\delta]) \subseteq e^{-1} \D_+ ] \gtrsim 1$ as $\epsilon \rightarrow 0$, then
\begin{align*}
\E\!\left[ \int \one_{B_w^\xi}\one_{\{w \in \eta_c'([0,\delta])\}}dw \, \middle| \, \eta_c'([0,\delta]) \subseteq e^{-1}\D_+ \right] \gtrsim \epsilon^{2 \xi_2}.
\end{align*}
Since $\xi,\xi_1,\xi_2$ can be taken to be arbitrarily close to $1$, so that $2\xi_2 <  2+a$, this contradicts~\eqref{eq:event_integral}. Thus we must have that $\p[ B_w^{\xi_1} \, \text{occurs for some} \, w \in \eta_c'([0,\delta/2]) \, | \, \eta_c'([0,\delta]) \subseteq e^{-1}\D_+ ] \rightarrow 0$ as $\epsilon \rightarrow 0$. Thus, by the argument in the beginning of the section it follows that on the event that $\eta_c'([0,\delta]) \subseteq e^{-1} \D_+$,  we have
\begin{align}\label{eq:mod8delta}
	|\eta_c'(s)-\eta_c'(t)| \leq C \left(1+\log \frac{1}{|s-t|} \right)^{-1/4 + \zeta},
\end{align}
for all $0 \leq s<t \leq \delta$ and $\zeta \geq 8a/(1-4a)$. Hence,~\eqref{eq:mod8delta} holds for any $\zeta >0$, by choosing $a$ small enough. Finally, rescaling by $4\delta^{-2}$, we have by the scale invariance of $\eta_c'$ that, conditional on the event that $\eta_c'([0,1]) \subseteq \delta^{-2} e^{-1} \D_+$,~\eqref{eq:mod8delta} a.s.\ holds for $0 \leq s<t \leq 1$ (possibly by taking $C$ larger). By letting $\delta \rightarrow 0$ the result follows.
\end{proof}

For the second part of Theorem~\ref{thm:sle8_thm} we shall upper bound the half-plane capacity that $\eta'$ accumulates upon traveling a small distance. We have that for a compact $\h$-hull $A$,
\begin{align*}
	\hcap(A) \lesssim \diam(A) \sup_{z \in A} \im(z).
\end{align*}
Thus, if the escape probability from $B(w,\epsilon^a)$ is at most $\exp(-\epsilon^{-\sigma})$, then
\begin{align*}
	\sup_{z \in \h_{\tau_w} \cap B(w, \epsilon^a)} \im(g_{\tau_w}(z)) \lesssim \exp(-\epsilon^{-\sigma}).
\end{align*}
Since the image of $\eta'([0,\tau_w])$ under $g_{\tau_w}$ is a.s.\ bounded, we have (roughly) that $\hcap(\h_{\tau_w} \cap B(w, \epsilon^a)) \lesssim \exp(-\epsilon^{-\sigma})$.

\begin{proof}[Proof of the second part of Theorem~\ref{thm:sle8_thm}]
Fix $\zeta > 0$ and let $u > 0$ be such that $(4-2u)^{-1} < 1/4 + \zeta$. We are done if we manage to show that as $\epsilon \rightarrow 0$, the probability that at some time the curve travels distance $\epsilon$ but accumulates at most $\exp(-\epsilon^{-4+2u})$ units of half-plane capacity converges to $1$. The proof relies on the observation that we can couple the curve $\eta'$ with a GFF $h$ on $\h$ with boundary conditions given by $\lambda'$ on $\R_-$ and $-\lambda'$ on $\R_-$ where $\lambda' = \pi/\sqrt{8}$ such that if we let  two flow lines $\eta_1,\eta_2$ of angles $\pm \pi/2$, emanating from a point $z \in \h$ run until they hit $\R$ then $\eta_1 \cup \eta_2$ is equal to the outer boundary of $\eta'([0,\tau_z])$ \cite[Theorem~1.13]{ms2017imag4}.  The idea of the proof is to consider a fine grid of points $(z_k)$ and look at the local geometry of $\eta'([0,\tau_{z_k}])$ in a neighborhood of each $z_k$. This local picture can be seen by considering two flow lines of angles $\pm \pi/2$ from $z_k$ and we will show that it looks roughly the same for each $k$. The proof is done once we have shown that as the grid gets finer, the probability that there is a point $z_k$ so that the escape probability for a Brownian motion started very close to $z_k$ is small decays slower than the number of points in the grid increases.

We may consider the part of $\eta'$ that lies in a compact subset $K \subseteq \h$ at a positive distance from $\R$. Thus, for the remainder of the proof, we assume that we are working on the event that $K \subseteq \eta'([0,1])$. The advantage being that in the interior, the law of the field $h$ is absolutely continuous with respect to the law of a whole-plane GFF $h^w$.

Fix some $a \in (0,1)$ to be chosen later and let $(z_k)$ be points in  a grid in $K$, spaced at distance $2\epsilon^a$ apart, say, $\cup_k \{z_k\} = (2\epsilon^a \Z) \cap K$. Let $\CF_\epsilon^a$ be the $\sigma$-algebra generated by the values of $h$ outside of $\cup_k B(z_k,\epsilon^a)$.  By the Markov property of the GFF, we have that the restriction of $h$ to each of the balls $B(z_k,\epsilon^a)$ is conditionally independent of the others given $\CF_\epsilon^a$.  Following \cite[Section~4.1]{mq2020geodesics}, for $z \in \C$ and $r > 0$ we say that $B(z,r)$ is $M$-good for $h$ if the following is true.  Let $\Fh_{z,r}$ be the harmonic extension of the values of $h$ from $\partial B(z,r)$ to $B(z,r)$.  Then $\sup_{u \in B(z,15r/16)} |\Fh_{z,r}(u) - \Fh_{z,r}(z)| \leq M$.  For each $j \in \N$, we let $s_j = \epsilon^a 2^{-j}$ and for each $k$ we let $r_k = s_{j_0}$ where $j_0$ is the smallest $j \in \N$ so that $B(z_k,s_j)$ is $M$-good.  Fix $q > 2a$. By the proof of \cite[Proposition~4.3]{mq2020geodesics}, we have that for fixed $b \in (a,1)$, we can choose $M > 0$ large enough, so that uniformly in $\epsilon < \epsilon_0$ (for some $\epsilon_0 > 0$) the probability that we do not discover an $M$-good scale before reaching the concentric ball with radius $\tfrac{8}{7} \epsilon^b$ is at most $\epsilon^q$. Then the probability that there is a point $z_k$ in the grid around which we do not discover an $M$-good scale as above is $O(\epsilon^{q-2a})$. For each $k$, let $A_{\epsilon,k}$ be the event $r_k \in (\tfrac{8}{7} \epsilon^b,\epsilon^a)$. Then $\p[\cap_k A_{\epsilon,k}] = 1 - O(\epsilon^{q-2a}) \to 1$ as $\epsilon \to 0$.

Let $h^w$ be a whole-plane GFF independent of $h$. As $B(z_k,r_k)$ is $M$-good, we know by \cite[Lemma~4.1]{mq2020geodesics} that the laws of $h|_{B(z_k,7 r_k/8)}$ and $h^w|_{B(z_k,7r_k/8)}$ (viewed modulo $2\pi \chi$ where $\chi = 2/\sqrt{\kappa}-\sqrt{\kappa}/2$ and $\kappa=2$) are mutually absolutely continuous.  Moreover, for each $p \in \R$ the Radon-Nikodym derivative has a finite moment of order $p$ which is at most $c(p,M)$ where $c(p,M)$ is a constant which depends only on $p$ and $M$. We emphasize that this holds for each $k$ and the constants $c(p,M)$ do not depend on $k$.  For a field $\wh{h}$ and $k$, we denote by $\eta_{j,k}^{\wh{h}}$ for $j=1,2$ the flow lines of angles $\pm \pi/2$ starting from $z_k$ and stopped upon exiting $B(z_k,7 r_k/8)$.  We also let  $E_\epsilon^{\sigma,k}(\wh h)$ be the event that
\begin{align*}
	\sup_{z \in B(z_k,\epsilon)} \p^z[B \, \text{hits} \, \partial B(z_k,\epsilon^b) \, \text{before} \, \eta_{1,k}^{\wh h} \cup \eta_{2,k}^{\wh h} \, | \, \eta_{1,k}^{\wh h}, \eta_{2,k}^{\wh h}] \leq \exp(-\epsilon^{-\sigma}).
\end{align*}
Fix $p > 1$.  For each $k$, let $\CZ_k$ be the Radon-Nikodym derivative of the law of $h|_{B(z_k, 7r_k/8)}$ with respect to the law of $h^w|_{B(z_k,7r_k/8)}$ (with both fields viewed modulo $2\pi \chi$).  We have that
\begin{align*}
	\p[ E_\epsilon^{\sigma,k}(h),\ A_{\epsilon,k} \, | \, \CF_\epsilon^a]
	 &= \E[  \one_{E_\epsilon^{\sigma,k}(h^w)} \one_{A_{\epsilon,k}} \CZ_k \, | \, \CF_\epsilon^a] \\
	&\geq \E\!\left[\one_{A_{\epsilon,k}} \CZ_k^{-\frac{1}{p-1}} \, \middle| \, \CF_\epsilon^a \right]^{-(p-1)} \p[ E_\epsilon^{\sigma,k}(h^w)]^p \quad\text{(H\"older's inequality)}\\
	&\geq c(p,M)\E\!\left[\one_{A_{\epsilon,k}} \,|\,\CF_{\epsilon}^a\right] \p[E_\epsilon^{\sigma,k}(h^w)]^p \geq \E\!\left[ \one_{A_{\epsilon,k}}\,|\,\CF_{\epsilon}^a\right] \epsilon^{p \sigma/2 + o(1)} \quad\text{(Lemma~\ref{lem:main_lemma_lbd})}.
\end{align*}
Therefore we obtain that
\begin{align*}
\p\!\left[ E_{\epsilon}^{\sigma,k}(h)^c, A_{\epsilon,k}\,|\,\CF_{\epsilon}^a\right] \leq 1 - \epsilon^{p\sigma / 2 + o(1)} \E\!\left[ \one_{A_{\epsilon}}\,|\,\CF_{\epsilon}^a \right]
\end{align*}
where $A_{\epsilon} = \cap _{k} A_{\epsilon,k}$.
Let $\sigma = 4-2u$ and assume that $p>1$ is such that $p(1-u/2) < 1$ and $a \in (p(1-u/2),1)$.  Fix $c > 0$ so that $|(z_k)| \geq c \epsilon^{-2a}$ and note that
\begin{align*}
\p\!\left[ \E\!\left[ \one_{A_{\epsilon}}\,|\,\CF_{\epsilon}^a\right] \leq 1/2 \right] = O(\epsilon^{q-2a})
\end{align*}
 By the conditional independence of the restrictions of $h$ to the balls $B(z_k,\epsilon^a)$ given $\CF_\epsilon^a$, the above implies that
\begin{align*}
	\p\!\left[ \bigcap_k \big( E_\epsilon^{4-2u,k}(h)^c \cap A_{\epsilon,k} \big) \right]
	&= \E\!\left[ \p\!\left[ \bigcap_k \big(E_\epsilon^{4-2u,k}(h)^c \cap A_{\epsilon,k} \big) \, \middle| \, \CF_\epsilon^a\right] \right]\\
	&= \E\!\left[ \prod_k \p\!\left[ E_\epsilon^{4-2u,k}(h)^c \cap A_{\epsilon,k} \, \middle| \, \CF_\epsilon^a\right] \right] \\
	&\leq O(\epsilon^{q-2a})+(1-\epsilon^{p (2-u) + o(1)})^{c \epsilon^{-2a}} \to 0 \quad\text{as}\quad \epsilon \to 0.
\end{align*}
Here, we have used that $2a > p(2-u)$. Thus, recalling $(4-2u)^{-1} < 1/4 + \zeta$, we have that
\begin{align*}
	\p\!\left[ \sup_{0 \leq s < t \leq 1}\!\left\{|\eta(t)-\eta(s)|\log\!\left(\frac{1}{|t-s|}\right)^{1/4 + \zeta}\right\} < \infty  \, \middle| \, K \subseteq \eta'([0,1])\right] \lesssim \lim_{\epsilon \rightarrow 0} \p\!\left[\cap_k E_\epsilon^{4-2u,k}(h)^c \right] =0.
\end{align*}
Since this holds for any compact $K \subseteq \h$, the result follows.
\end{proof}

\section{Regularity results for $\SLE_4$}\label{sec:sle4}

This section is dedicated to proving Theorems~\ref{thm:sle4_thm} and \ref{thm:jones_smirnov_sle4}. We begin by stating some facts about two-sided whole-plane $\SLE_4$ processes. For a more detailed treatment, see \cite{zhan2019optimal,zhan2021loop}. 

A two-sided whole-plane $\SLE_4$ process $\eta$ from $\infty$ to $\infty$ through $0$ is the curve which can be sampled by first sampling a whole-plane $\SLE_4(2)$ process $\eta_1$ from $\infty$ to $0$ and then sampling a chordal $\SLE_4$ curve from $0$ to $\infty$ in $\C \setminus \eta_1$ (that is, it is the concatenation of the curves $\eta_1$ and $\eta_2$ in Lemma~\ref{lem:main_lemma_ubd}, with $\kappa = 4$). By \cite[Corollary~4.7]{zhan2021loop},  $\eta$ can be parameterized by its $\frac{3}{2}$-dimensional Minkowski content and upon doing so, it becomes a self-similar process of index $2/3$ with stationary increments. That is, if we assume that $\eta$ is parameterized by its $\frac{3}{2}$-dimensional Minkowski content and $\eta(0) = 0$, then for each $a>0$, $(\eta(at))_{t \in \R}$ and $(a^{2/3} \eta(t))_{t \in \R}$ have the same law and for each $b \in \R$, $(\eta(b + t) - \eta(b))_{t \in \R}$ has the same law as $(\eta(t))_{t \in \R}$. We call such a curve an sssi $\SLE_4$ process. By \cite[Theorems~1.2 and~1.4]{zhan2019optimal}, $\eta$ is a.s.\ locally H\"{o}lder continuous of order $\alpha$ for all $\alpha < 2/3$ but not locally H\"{o}lder continuous of order $2/3$ on any open interval.

Below, we shall consider the regularity of a conformal map $\varphi_1: \D \rightarrow \C_L$ with $\varphi_1(-i) = \eta(0) = 0$, $\varphi_1(i) = \infty$, and say, $\varphi_1(-1) = \eta(-100)$ where $\C_L$ is the component of $\C \setminus \eta$ which is on the left side of~$\eta$. This is out of convenience since this setting is the one which is the most straightforward, given the form of Lemmas~\ref{lem:main_lemma_lbd} and \ref{lem:mainlemma4}. Theorem~\ref{thm:sle4_thm} is concerned with a chordal $\SLE_4$ process from $-i$ to $i$ in $\D$, but this  raises no problem, as such a curve can be obtained as follows. Sample a two-sided whole-plane $\SLE_4$ process from $\infty$ to $\infty$ through $0$ and let $\varphi_2:\C \setminus \eta((-\infty,0]) \rightarrow \D$,  be the unique conformal transformation with $\varphi_2(0) = -i$, $\varphi_2(\infty) = i$ and which takes the prime end corresponding to $\eta(-100)$ on the left side of $\eta$ to $-1$.  Then, $\wh{\eta} = \varphi_2(\eta)$ is a chordal $\SLE_4$ from $-i$ to $i$ in $\D$ and the map $\varphi =\varphi_2 \circ \varphi_1: \D \rightarrow \D_L$ (where $\D_L$ is the left connected component of $\D \setminus \wh{\eta}$) is the uniformizing map in the statement of Theorem~\ref{thm:sle4_thm}. Moreover, the $\varphi_2$ is smooth away from $\eta((-\infty,0])$ and hence the regularity of $\varphi$ away from the points $-i$ and $i$ is determined by the regularity of $\varphi_1$.

Next, we explain the relationship between the escape probability of Brownian motion and the modulus of continuity of the uniformizing map. 
We begin with the simpler direction. Fix $d > 0$ and let $r \in (0,1)$ be such that $\dist(\varphi_1(B(0,r)),\eta) \geq d$ (we emphasize that we are allowed to choose these parameters to depend on $\eta$, as the constant in Theorem~\ref{thm:sle4_thm} is random). Assume that $z_0 \in \C_L$ is a point such that $\dist(z,\eta) \in [\epsilon,2\epsilon]$ and such that the probability that a Brownian motion escapes to distance $d$ from $\eta$ before hitting $\eta$ is at most $\exp(-\epsilon^{-\sigma})$ (as in Lemma~\ref{lem:mainlemma4}). Then, writing $w_0 = \varphi_1^{-1}(z_0)$, we have by the conformal invariance of Brownian motion that the probability that a Brownian motion started from $w_0$ hits $\partial B(0,r)$ before $\partial \D$ is upper bounded by $\exp(-\epsilon^{-\sigma})$. The former probability is given by $\tfrac{\log |w_0|}{\log r}$. This implies that 
\begin{align*}
	\dist(w_0,\partial \D) = 1 - |w_0| \lesssim \exp(-\epsilon^{-\sigma}),
\end{align*}
and since $\dist(z_0,\eta) \asymp \epsilon$, the Koebe-$1/4$ theorem implies that
\begin{align}\label{eq:bad_point}
	| \varphi_1'(w_0) | \gtrsim \epsilon e^{-\epsilon} = e^{\epsilon(1+o(1))}.
\end{align}
Thus, if for each $\epsilon > 0$ we can a.s.\ find a point $w_0 = w_0(\epsilon)$ as above, then~\eqref{eq:mod4} does not hold a.s.

The other direction follows the same principle, but is slightly more technical. Fix $\xi > 0$ and compact intervals $K_1 \subset K_2 \subset (0,\infty)$ and $d > 0$ such that if $U_d^j = \varphi_1^{-1}(\{ z \in \C_L: \dist(z,\eta(K_j)) < d \}$ for $j=1,2$, then 
\begin{align*}
&\{ z \in \partial \D: \re(z) > 0, \dist(z,\{-i,i\}) > \xi \} \subseteq U_d^1, \\ 
&\{ z \in \partial \D: \re(z) > 0, \dist(z,\{-i,i\}) < \xi/2 \} \cap U_d^1 = \emptyset, \\
&\partial \D \setminus \{ z \in \partial \D: \re(z) > 0, \dist(z,\{-i,i\}) > \xi/4 \} \subseteq U_d^2, \ \text{and} \\
&\{ z \in \partial \D: \re(z) > 0, \dist(z,\{-i,i\}) < \xi/3 \} \cap U_d^2 = \emptyset. 
\end{align*}
In particular, $\dist(\eta(K_1),\eta(\R \setminus K_2)) > 2d$.  Suppose that $z_0 \in \varphi_1(\D \setminus (B(-i,2\xi/3) \cup B(i,2\xi/3))$ is such that $\dist(z_0,\eta) = \dist(z_0,\eta(K_1)) \in [\epsilon,2\epsilon]$ and such that the probability that a Brownian motion started from $z_0$ escapes to distance $d$ from $\eta$ before hitting $\eta$ is at least $\exp(-\epsilon^{-\sigma})$. Again, write $w_0 = \varphi_1^{-1}(z-0)$. Let $r > 0$ be such that
\[ \varphi_1( \{ |z| = r \} \cap \{z \in \D: \re(z) > 0, \dist(z,\{-i,i\}) > \xi/2 \}) \subset \{ w \in \C_L: \dist(w,\eta) < d \}. \]
In particular,  escaping to $\partial B(0,r)$ is at least as easy for a Brownian motion from $z_0$ as escaping to distance $d$ from $\eta$, at least to the parts of those sets in $\D \setminus (B(-i,\xi/2) \cup B(i,\xi/2))$ and $\varphi_1(\D \setminus (B(-i,\xi/2) \cup B(i,\xi/2)))$, respectively.  Moreover, since we can decrease $r$ to be $a\xi$ for some $a \in (0,1)$ if necessary, and since $\dist(w_0,\varphi_1^{-1}(\{z \in \CL: \dist(z,\eta) = d \}) \cap (B(-i,\xi/2) \cup B(i,\xi/2)) \geq \xi/6$, it follows that the probability that a Brownian motion from $w_0$ hits $B(0,r)$ before $\partial \D$ is at least a constant (depending only on $d$, $r$ and $\xi$) times $\exp(-\epsilon^{-\sigma})$. (The importance of the last part of this argument lies in making sure that it is not only that likely to hit some part of $\varphi_1(\{z : \dist(z,\eta) = d\})$ which is of distance much smaller than $1-r$ from $\partial \D$.) Then, arguing similarly to the discussion above, using the Koebe-$1/4$ theorem, it follows that $\varphi_1'(w_0) \lesssim \exp(\epsilon^{-\sigma(1+o(1))})$. Thus, if that escape probability holds for all $w_0 \in \D \setminus (B(-i,\xi) \cup B(i,\xi))$, then~\eqref{eq:mod4} holds.

With the above discussion in mind, in order to prove that~\eqref{eq:mod4} holds, we show that a.s., as $\epsilon \rightarrow 0$ we can find no point $w_0 = w_0(\epsilon)$ within distance $C \epsilon$ from $\partial \D$, for some $C >0$, such that~\eqref{eq:bad_point} holds. This is done using a covering and union bound argument and proves~\eqref{eq:mod4} for points $z,w \in \partial \D$.  Then, we prove the interior regularity  in two steps. The first is the regularity in the part of $\D$ which is at uniformly positive distance from $\partial \D$ and this follows easily from the Koebe-$1/4$ theorem. The intermediate case requires a little more work, but boils down to controlling the regularity of the real and imaginary parts of $\varphi$, both of which are harmonic, and using the boundary regularity together with Brownian motion estimates.

In proving~\eqref{eq:mod4} on the boundary, we  pick a sequence of points $t_k \in \R$ and check  whether the escape probability of a Brownian motion started close to $\eta(t_k)$ is large enough. For this to actually cover a neighborhood of the boundary, we need to pick the sequence in such a way that for sufficiently small $\epsilon$, any point $z_0$ such that $\dist(z_0,\partial \D) \in [c_0 \epsilon,c_1 \epsilon]$ (for some constants $0<c_0<c_1$) will be sufficiently close to $\eta(t_k)$ for some $k$.  In what follows, for a domain $D$, a set $E \subseteq \partial D$ and a point $z \in D$, we let $\omega(z,E,D)$ denote the harmonic measure of the set $E$ in $D$ seen from $z$.

\subsection{Proof of Theorem~\ref{thm:sle4_thm}, upper bound}

\noindent{\it Step 1.  Regularity at the boundary.} Fix $\zeta > 0$. We first show that for sufficiently small $\epsilon > 0$, there is no point $z_0$ as in Lemma~\ref{lem:mainlemma4} with $\sigma = 3/(1-3\zeta) = 3 + \wh{\delta}$. 

Fix compact intervals $K_1 \subseteq K_2 \subseteq K_3 \subseteq (0,\infty)$.  For each $\epsilon > 0$ and $k \in \N$, let $t_k = t_k(\epsilon) = k \epsilon^{3/2+\delta}$ so that $t_{k+1} - t_k = \epsilon^{3/2+\delta}$ and let $n = \max\{ k : t_k \in K_3\}$. Since $\eta$ is locally $(2/3 - \delta/3)$-H\"{o}lder continuous, there a.s.\ exists a random constant $C$ such that
\begin{align*}
	| \eta(s) - \eta(t)| \leq C | s-t |^{2/3-\delta/3}
\end{align*}
for all $s,t \in K_3$. Fix $d \in (0,1)$ and let $F_d$ be the event that $\dist(\eta(K_1), \eta(\R \setminus K_2))\geq d$ and $C \in (d,d^{-1})$.  Since $\eta$ is a simple curve with $\lim_{|t| \to \infty} \eta(t) = \infty$ we have that $\p[F_{d}] \to 1$ as $d \to 0$.  For each $k$, write $\eta_k(t) \coloneqq \eta(t+t_k) - \eta(t_k)$.  We fix $\wh \delta > 0$ and let $\wt{E}(z,d)$ be the event that
\begin{align*}
	\p^z[ B \, \text{hits} \, \partial B(z,d/2) \, \text{before} \, \eta \, | \, \eta] \leq \exp(-\epsilon^{-3-\wh \delta}).
\end{align*}
We let $D_\epsilon = \{z \in \C_L: \dist(z,\eta) = \dist(z,\eta(K_1)) \in [\epsilon,2\epsilon] \}$ and note that for each $z \in \C$, $f_z(t) = \omega (z,\eta((-\infty,t]),\C_L)$ is a continuous and increasing function with $\lim_{t \to -\infty}f_z(t) = 0$ and $\lim_{t \to \infty}f_z(t) = 1$.  It follows that there exists $T_z > 0$ such that $f_z(T_z) = \omega (z,\eta([T_z,\infty)),\C_L) = 1/2$. Then, on the event $F_d$, the Beurling estimate implies that for sufficiently small $\epsilon$, $T_z \in K_2$ for all $z \in D_\epsilon$. Fix $C_1 > 2$ large (to be chosen independently of  $d$). Suppose that $\eta(T_z) \notin B(z,C_1 \epsilon)$ and note that there exists $t \in K_2$ such that $\eta(t) \in \overline{B(z,2\epsilon)} \setminus B(z,\epsilon)$. Without loss of generality we can assume that $t \leq T_z$ and set $S_z = \sup\{t \in K_2: t \leq T_z, \, \eta(t) \in B(z,C_1 \epsilon)\}$. Since $\dist(z,\eta([S_z,T_z])) \geq C_1 \epsilon$, the Beurling estimate implies that $\omega(z,\eta([S_z,T_z]),\C_L) = O(C_1^{-1/2})$ where the implicit constants are universal, and so for $C_1$ sufficiently large we have that $\omega(z,\eta((-\infty,S_z]),\C_L) \geq 3/8$. Also, there exists $k = k(z) \in \{1,\ldots,n\}$ such that $t_k \leq S_z \leq t_{k+1}$. Then, if we pick $\delta \in (0,1/10)$, we have that on the event $F_d$, $\diam(\eta([t_k,S_z])) \leq d^{-1} \epsilon^{1+\delta / 10}$. Thus, by applying Beurling's estimate again, we obtain that $\omega(z,\eta([t_k,S_z]),\C_L)< 1/8$ for all $\epsilon$ sufficiently small and so $f(t_k) \in [1/4,3/4]$. Similarly, if $\eta(T_z) \in B(z,C_1 \epsilon)$, we have that $f(t_k) \in [1/4,3/4]$ and $\eta(t_k) \in B(z,2C_1 \epsilon)$ for some $k = k(z) \in \{1,\dots,n\}$. The above imply that if $\wt E(z,d)$ and $F_d$ occur, then $E_{\epsilon}^{\eta_k}(3+\wh \delta)$ occurs for some $1 \leq k \leq n$, where we denote by $E_\epsilon^{\eta_k}(\sigma)$ the event in Lemma~\ref{lem:mainlemma4}, but with $\eta_k$ in place of $\eta_1 \cup \eta_2$ and $B(0,2C_1\epsilon)$ in place of $B(0,2\epsilon)$.  Fix $r>0$ and recall that we want to show that for small enough $\epsilon$ and properly chosen $\wh \delta$, there exists no point $z \in D_\epsilon$ such that $\wt E(z,2r)$ occurs. By Lemma~\ref{lem:mainlemma4}, we have that for all $\epsilon$ sufficiently small,
\begin{align*}
	&\p[  \exists z \in D_\epsilon : \wt E(z,2r) \text{ occurs}, F_d ] \\
	&\leq \p[ \exists z \in D_\epsilon \text{ such that } \widetilde{E}(z,d) \text{ occurs}, F_d] \leq \sum_{k=0}^{n-1} \p[E_\epsilon^{\eta_k}(3+\wh{\delta})] = O(\epsilon^{\wh{\delta}/2-\delta}),
\end{align*}
provided that we choose $d$ small enough. Assume that $\delta < \wh{\delta}/2$ so that the exponent of $\epsilon$ in the above expression is positive.  It therefore follows from the Borel-Cantelli lemma that there a.s.\ exists $\epsilon_0 > 0$ so that for all $\epsilon \in (0,\epsilon_0)$ there does not exist a point $z$ with $\dist(z,\eta(K_1)) \leq 2\epsilon$ so that $\wt{E}(z,2r)$ occurs.

We will now prove that~\eqref{eq:mod4} holds on the boundary.  Fix $\xi > 0$ and let $I$ be a boundary arc of the counterclockwise segment of $\partial \D$ from $-i$ to $i$ which is disjoint from $B(-i,\xi/2)$, $B(i,\xi/2)$ and has length $\Delta>0$.  Let $\wt \Delta = \diam \, \varphi(I)$. Fix $a > 1$.  By the non-self-tracing property of $\SLE$ (see \cite{mmq2018uniqueness}) we have for small enough $\Delta > 0$ that there exists a ball $B(w,\wt \Delta^{a}) \subseteq \D_L$ such that $\omega(w,\varphi(I),\D_L) \geq 1/2$. Combining this with the above, we have that
\begin{align*}
	\Delta \gtrsim \diam \, \varphi^{-1}(B(w,\wt \Delta^{a})) \gtrsim \exp(-\wt \Delta^{-a(3 + \wh{\delta})}),
\end{align*}
and rearranging, this implies that $\wt \Delta \lesssim \log(1/\Delta)^{- \frac{1}{3} + \zeta'}$, that is,
\[ |\varphi(x)-\varphi(y)| \lesssim (\log (1+|x-y|^{-1}))^{-\frac{1}{3} + \zeta'} \quad\text{for all}\quad x,y \in \partial \D \setminus (B(-i,\xi/2) \cup B(i,\xi/2))\]
for $\zeta' = \frac{1}{3} - \frac{1}{3a} + \frac{1}{a\zeta}$. By taking $a$ sufficiently close to $1$ such that $\zeta'$ is positive and close to $\zeta$, we obtain that ~\eqref{eq:mod4} holds on the boundary.
We shall now use this to deduce the regularity in the interior.

\noindent{\it Step 2.  Regularity in the interior.}
Fix $\xi>0$ and note that the bound proven above for the boundary is valid for $\partial \D \setminus (B(-i,\xi/2) \cup B(i,\xi/2))$. We will take $\xi$ to be very close to $0$. Next, write $\varphi(z) = u(z) + iv(z)$ where $u$ and $v$ are real. Since $u$ and $v$ are harmonic, we can write $u(z) = \E[ u(B_{\tau_z}^z) \giv \eta]$ (and likewise for $v$) where $B^z$ is a two-dimensional Brownian motion started from $z$, $\tau_z$ is the first exit time of $\D$, and the expectation is only over the Brownian motion while the $\SLE_4$ is fixed.  In the remainder of this step, all probabilities and expectations should be understood as having the $\SLE_4$ fixed and will only be over the relevant Brownian motion.  In order to estimate $|u(z)-u(w)|$, we couple the Brownian motions $B_t^z$ and $B_t^w$ by letting $B_t^z = z + B_t$ and $B_t^w = w + B_t$, where $B_t$ is a two-dimensional Brownian motion with $B_0 = 0$.

\noindent{\it Step 2a. $z,w$ are close to the boundary.} We first consider points $z,w\in D_\xi = \{ z \in \D \setminus ( B(-i,\xi) \cup B(i,\xi)): \dist(z,\partial \D) < \xi^2 \}$.  We begin with the case where $|z-w| \geq \dist(z,\partial \D)/2$. Assume that $2^{-k} \leq |z-w| < 2^{-k+1}$ and $2^{-j} \leq \dist(z,\partial \D) < 2^{-j+1}$, where $j>k$. Let $E_\xi$ be the event that $B^z$ and $B^w$ exit $\D$ before hitting $B(i,\xi/2) \cup B(-i,\xi/2)$. Then, $\p[E_\xi^c] = O(\xi^{-1} 2^{-j} + \xi^{-1} 2^{-k}) = O(\xi^{-1} 2^{-k})$. Let $A_m^1$ (resp.\ $A_l^2$) be the event that $2^{-m-1} \leq \dist(z,B_{\tau_z}^z) < 2^{-m}$ for $m \leq j$ (resp.\ $2^{-l-1} \leq \dist(B_{\tau_z}^z,B_{\tau_w}^w) < 2^{-l}$ for $l \leq k$). Moreover, we let $F_k$ be the event that $\dist(B_{\tau_z}^z,B_{\tau_w}^w) < 2^{-k-1}$. Then,
\begin{align*}
	&|u(z) - u(w)| \leq  \E[ |u(B_{\tau_z}^z) - u(B_{\tau_w}^w)|]\\
	&= \E[ |u(B_{\tau_z}^z) - u(B_{\tau_w}^w)| \one_{F_k^c \cap E_\xi}] +  \E[ |u(B_{\tau_z}^z) - u(B_{\tau_w}^w)| \one_{F_k \cap E_\xi}] +  \E[ |u(B_{\tau_z}^z) - u(B_{\tau_w}^w)| \one_{E_\xi^c}].
\end{align*}
Since $u$ is bounded, we have that 
\begin{align}\label{eq:bound1}
	 \E[ |u(B_{\tau_z}^z) - u(B_{\tau_w}^w)| \one_{E_\xi^c}] \lesssim \p[ E_\xi^c] = O(\xi^{-1} 2^{-k}).
\end{align}
Moreover,  on the event $F_k$, $|B_{\tau_z}^z - B_{\tau_w}^w| \leq |z-w|/2$ and thus by the boundary regularity estimate in Step 1, we have that
\begin{align}
	\E[|u(B_{\tau_z}^z) - u(B_{\tau_w}^w)| \one_{F_k \cap E_\xi}] &\lesssim \E\!\left[ \left( \log \left( 1 + \frac{1}{|B_{\tau_z}^z - B_{\tau_w}^w|} \right) \right)^{-1/3 + \zeta} \one_{F_k \cap E_\xi} \right] \notag \\
	&\leq \left( \log \left( 1 + \frac{1}{|z-w|} \right) \right)^{-1/3 + \zeta}. \label{eq:bound2}
\end{align}
Finally,  we have that
\begin{align*}
	\E[ |u(B_{\tau_z}^z) - u(B_{\tau_w}^w)| \one_{F_k^c \cap E_\xi}] &= \sum_{m=0}^j \sum_{l=0}^k \E[ |u(B_{\tau_z}^z) - u(B_{\tau_w}^w)| \, | \, E_\xi \cap A_m^1 \cap A_l^2]\p[E_\xi \cap A_m^1 \cap A_l^2].
\end{align*}
Note that the bounds~\eqref{eq:bound1} and~\eqref{eq:bound2} are sufficient. Hence, it remains to show that the double sum is at most of order $k^{-1/3 + \zeta}$.  We begin by noting on $E_\xi \cap A_m^1 \cap A_l^2$, we have that $|u(B_{\tau_z}^z) - u(B_{\tau_w}^w)| \leq |\varphi(B_{\tau_z}^z) - \varphi(B_{\tau_w}^w)| \lesssim (l+1)^{-1/3 + \zeta}$, since the bound~\eqref{eq:mod4} is proven at boundary points. Next, note that $\p[A_m^1] \asymp 2^{m-j}$ and $\p[A_l^2] \asymp 2^{l-k}$, with implicit constants independent of $m$ and $l$. Moreover, in fact, $\p[E_\xi \cap A_m^1 \cap A_l^2] \asymp 2^{m-j + l-k}$ (as is easily seen by considering the ``worst case'' where $z,w$ are roughly at distance $\xi$ from $i$ or $-i$, i.e., roughly $\xi/2$ from $B(i,\xi/2) \cup B(-i,\xi/2)$ since they are at distance at most $\xi^2$ from the boundary). Consequently,
\begin{align*}
	&\sum_{m=0}^j \sum_{l=0}^k \E[ |u(B_{\tau_z}^z) - u(B_{\tau_w}^w)| \, | \, E_\xi \cap A_m^1 \cap A_l^2]\p[E_\xi \cap A_m^1 \cap A_l^2] \lesssim \sum_{m=0}^j 2^{m-j} \sum_{l=0}^k (l+1)^{-1/3 + \zeta} 2^{l-k} \\
	&\lesssim \sum_{m=0}^j 2^{m-j} \left( \sum_{l=0}^{\lceil k/2 \rceil} 2^{-k/2} + \sum_{l = \lceil k/2 \rceil+1}^k 2^{l-k} k^{-1/3 + \zeta}\right) \lesssim \frac{k}{2} 2^{-k/2} + k^{-1/3 + \zeta} = O(k^{-1/3 + \zeta}).
\end{align*}
Doing the same for $v$, we note that $|\varphi(z)-\varphi(w)| \lesssim k^{-1/3 + \zeta}$ for $z,w \in D_\xi$ such that $2^{-k} \leq |z-w| < 2^{-k+1}$ and $\dist(z,\partial \D) \leq 2|z-w|$.

\noindent{\it Step 2b. $|z-w|$ are close relative to their distance to $\partial \D$.} Next, we consider the case where $z,w \in D_{\xi}$ are such that $|z-w| \leq \dist(z,\partial \D) / 2$. By the Koebe-$1/4$ theorem,
\begin{align*}
	|\varphi(z) - \varphi(w)| \lesssim \frac{\dist(\varphi(z),\partial \D_L)}{\dist(z,\partial \D)}|z-w|.
\end{align*}
By the discussion before the proof,  $\dist(z,\partial \D) \gtrsim \exp( - \dist(\varphi(z),\partial \D_L)^{-3 - \wh \delta})$, so rearranging gives that $\dist(\varphi(z),\partial \D_L) \lesssim  \log(\dist(z,\partial \D)^{-1}) ^{-1/3 + \zeta}$ and hence
\begin{align}\label{eq:conformal_bound}
|\varphi(z)-\varphi(w)| \lesssim \frac{\log(\dist(z,\partial \D)^{-1})^{-1/3 + \zeta}}{\dist(z,\partial \D)}|z-w|.
\end{align}
We observe that the function $f(x) = \log(x^{-1})^{-1/3+\zeta}/x$ is decreasing for $x$ sufficiently small, so by~\eqref{eq:conformal_bound},~\eqref{eq:mod4} holds in that case as well.

\noindent{\it Step 2c. General $z,w$ away from $\partial \D$.} Finally, consider points $z,w \in \wt{D}_\xi = \{ z \in \D \setminus (B(-i,\xi) \cup B(i,\xi)): \dist(z,\partial \D) \geq \xi^2 \}$.  By the Koebe-$1/4$ theorem, $|\varphi'(z)| \leq 4 \dist(\varphi(z),\partial \D_L)/\dist(z,\partial \D) \lesssim \xi^{-2}$, uniformly in $z \in \wt{D}_\xi$. Consequently, for any points $z,w \in \wt{D}_\xi$, we have that $|\varphi(z)-\varphi(w)| \lesssim \xi^{-2} |z-w|$, and hence the bound holds in $\wt{D}_\xi$.  Thus, the proof of the first part is done. 
\qed

\subsection{Proof of Theorem~\ref{thm:sle4_thm}, lower bound}

The proof of the lower bound has two main steps: first prove that the second assertion of Theorem~\ref{thm:sle4_thm} holds with positive probability, and then use this to show that it holds a.s. In proving the positive probability statement, we pick a sequence of times $t_k \in \R$ and consider the positive fraction of $k$ where $\eta(t_k)$ is sufficiently far away from $\eta((-\infty,t_{k-1}])$ and $\eta([t_{k+1},\infty))$ and let two level lines of different heights construct a pocket around a neighborhood of $\eta(t_k)$. With positive probability, these level lines will succeed in forming such a pocket.  If this occurs, the event that the conditional probability given $\eta$ that the escape probability for a Brownian motion started close to $\eta(t_k)$ from a neighborhood of $\eta(t_k)$ is conditionally independent (given the two level lines) of what happens outside of the pocket.  Since there are sufficiently many such pockets with positive probability, this will lead to establishing that the second assertion of Theorem~\ref{thm:sle4_thm} holds with positive probability. To upgrade the a.s.\ result, we map $\eta$ to a curve $\wh \eta$ in $\h$ and perform another pocket argument: upon traveling through successive annuli, there is a uniformly positive probability in each annulus that two level lines (started from the curve in the annulus) form a pocket around $\wh \eta$.  If this happens, then the law of the curve $\wh \eta$, restricted to the times between entering and exiting the pocket, is that of an $\SLE_4(\rho_L; \rho_R)$ process. By conformally mapping this pocket to $\D$, we can use the absolute continuity between $\SLE_4$-type processes to deduce that with positive probability, the second assertion of Theorem~\ref{thm:sle4_thm} holds true for the new curve in $\D$. Doing this for each annulus gives the a.s.\ statement.

We will now describe the general setup and define the events used in the proof of the lower bound of Theorem~\ref{thm:sle4_thm}.  Fix $0 < r <a < b < c < 1$ and $u > 0$.  Let $\eta$ be the sssi $\SLE_4$ curve from $\infty$ to $\infty$ so that the conditional law of $\eta|_{[0,\infty)}$ given $\eta|_{(-\infty,0]}$ is that of an $\SLE_4$ in $\C \setminus \eta((-\infty,0])$ from $0$ to $\infty$ where we view $\eta|_{[0,\infty)}$ as the level line of a GFF $h$ on $\C \setminus \eta((-\infty,0])$ with boundary conditions given by $-\lambda$ (resp.\ $\lambda$) on the left (resp.\ right) side of $\eta((-\infty,0])$ where $\lambda = \pi/2$.  Fix a compact interval $I \subseteq (0, \infty)$ and let $(t_k) = I \cap (\epsilon^{3(a-r)/2} \Z)$.  For each $k \in \N$, we let $[\tau_k^1,\tau_k^2]$ be so that $\eta(\tau_k^1)$ (resp.\ $\eta(\tau_k^2)$) is where $t \mapsto \eta(t_k - t)$ (resp.\ $t \mapsto \eta(t_k+t)$) first exits $B(\eta(t_k),\epsilon^c)$.  We let $E_k^1$ be the event that
\[ \sup_{z \in \C_L \cap B(\eta(t_k), 2\epsilon) \setminus B(\eta(t_k),\epsilon)} \p^z[ B \text{ exits } B(\eta(t_k), \epsilon^c) \text{ before hitting } \eta([\tau_k^1,\tau_k^2]) \giv \eta|_{[\tau_k^1,\tau_k^2]} ] \leq \exp( - \epsilon^{-3+2u}).\]

We let $\sigma_k$ be such that $\eta(\sigma_k)$ is equal to the first place that $t \mapsto \eta(t_k-t)$ exits $B(\eta(t_k),\epsilon^b)$.  Fix $\vartheta > 0$ and let $\eta_{k,1}$ (resp.\ $\eta_{k,2}$) be the level line of $h$ starting from $\eta(\sigma_k)$ with height $-\vartheta$ (resp.\ $\vartheta$) stopped upon exiting $B(\eta(t_k),2\epsilon^b) \setminus B(\eta(t_k),\epsilon^b/2)$.  We assume that $\vartheta > 0$ is chosen sufficiently small so that $\eta_{k,1}$ and $\eta_{k,2}$ can intersect each other.  We also let $E_k^2$ be the event that
\begin{enumerate}[(i)]
\item\label{it:ek2_p1} $B(\eta(t_k), \epsilon^c)$ is contained in a bounded component $U_k$ of $\C \setminus (\eta_{k,1} \cup \eta_{k,2})$,
\item\label{it:ek2_p2} $\eta((-\infty,t_{k-1}])$ and $\eta([t_{k+1},\infty))$ do not intersect $B(\eta(t_k), \epsilon^a)$, and
\item\label{it:ek2_p3} $\p[ E_k^1 \giv \eta_{k,1}, \eta_{k,2}, \ol{\eta}_k ] \geq \epsilon^{3/2 - u + o(1)}$ where $\ol{\eta}_k$ is the part of $\eta$ which is not contained in $U_k$.
\end{enumerate}

\begin{lemma}
\label{lem:sle4_estimate}
There exists $\epsilon_0 > 0$ and $p_2 \in (0,1)$ so that $\p[E_k^2] \geq p_2$ for all $\epsilon \in (0,\epsilon_0)$.
\end{lemma}

In order to read the proof of the lower bound of Theorem~\ref{thm:sle4_thm}, one can skip the proof of Lemma~\ref{lem:sle4_estimate}.  The first step is the following version of Lemma~\ref{lem:main_lemma_lbd} but for two-sided radial $\SLE_4$.

\begin{lemma}
\label{lem:two_sided_radial_lemma51}
There exists $\xi_0 > 0$ so that for all $\xi \in (0,\xi_0)$ the following is true.  Suppose that $x,y \in \partial \D$ are distinct with $|x-y| \geq 5 \xi$ and that $\eta$ is a two-sided radial $\SLE_4$ in $\D$ from~$x$ to~$y$ which passes through~$0$.  Let $B$ be a Brownian motion which is independent of $\eta$, fix $\epsilon > 0$, and let $F$ be the event that:
\begin{enumerate}[(i)]
\item $\sup_{z \in B(0,\epsilon)} \p^z[B \text{ hits } \partial \D \text{ before } \eta \giv \eta] \leq \exp(-\epsilon^{-\sigma})$ and
\item $\dist(\eta, \partial \D \setminus ( B(x, 2\xi) \cup B(y, 2\xi))) \geq \xi$
\end{enumerate}
Uniformly in $x,y \in \partial \D$ with $|x-y| \geq 5 \xi$, we have that $\p[F] \geq \epsilon^{\sigma/2 + o(1)}$ as $\epsilon \to 0$ (but with $\xi > 0$ fixed).
\end{lemma}

The strategy of the proof is to consider a whole-plane $\SLE_4$ process $\eta$ from~$\infty$ to~$\infty$ through~$0$, where we know from Lemma~\ref{lem:main_lemma_lbd} that the exponent of the probability of part (i) of the event $F$ of the above lemma is the right one. Then, we condition on the part of $\eta$ before first hitting and after last exiting $\D$. The law of the remaining curve is a two-sided radial $\SLE_4$ in the complementary domain, passing through~$0$. Then, since the exponent of the probability of part (i) of $F$ is the right one for $\eta$, we know that there has to be some configuration for the two-sided radial $\SLE_4$ such that the same holds true. We then conclude by proving that with positive probability, the  configuration of the two-sided radial $\SLE_4$ process is sufficiently close to that configuration.

\begin{proof}
We are going to deduce the result from the conformal Markov property of two-sided radial $\SLE_4$ and Lemma~\ref{lem:main_lemma_lbd}.  To begin with, we assume that we have a two-sided whole-plane $\SLE_4$ process~$\eta$ in~$\C$ from~$\infty$ to~$\infty$ through $0$ as above.  Let $\ol{\eta}$ be the time-reversal of $\eta$.  Let $\tau$ (resp.\ $\ol{\tau}$) be the first time that $\eta$ (resp.\ $\ol{\eta}$) hits $\partial \D$.  Then we know that the conditional law of $\eta$ given $\eta|_{(-\infty,\tau]}$ and $\ol{\eta}|_{(-\infty,\ol{\tau}]}$ is that of a radial two-sided $\SLE_4$ in $\C \setminus (\eta((-\infty,\tau]) \cup \ol{\eta}((-\infty,\ol{\tau}]))$ from $\eta(\tau)$ to $\ol{\eta}(\ol{\tau})$ which passes through $0$.  Let $\varphi$ be the unique conformal map from $\C \setminus (\eta((-\infty,\tau]) \cup \ol{\eta}((-\infty,\ol{\tau}]))$ to $\D$ which fixes $0$ and sends $\eta(\tau)$ to $1$.  Let $\wh{\eta}$ be the image of the remainder of $\eta$ under $\varphi$.  Then $\wh{\eta}$ is a two-sided radial $\SLE_4$ in $\D$ from $x = 1$ to $y = \varphi(\ol{\eta}(\ol{\tau}))$ which passes through $0$.  Let $z_L$ (resp.\ $z_R$) be the image under $\varphi$ of the prime end corresponding to $\infty$ in the component of $\C \setminus \eta$ which is to the left (resp.\ right) of $\eta$.  Then we note that $z_L$ (resp.\ $z_R$) is on the clockwise (resp.\ counterclockwise) arc of $\partial \D$ from $x$ to $y$.

We now suppose that we have fixed $C >1$ and $D > 0$ so that ~\eqref{eqn:mainlemma3} holds.   Let $\eta_1 = \ol{\eta}|_{[0,\infty)}$ and $\eta_2 = \eta|_{[0,\infty)}$ so that the joint law of $(\eta_1,\eta_2)$ is the same as the pair of paths in Lemma~\ref{lem:main_lemma_ubd}.  Suppose that the event $E$ of Lemma~\ref{lem:main_lemma_lbd} holds with $\zeta = 0$.  Then we have that
\[ |\eta(\tau)-\ol{\eta}(\ol{\tau})| \geq \dist(\eta_1([\tau_1,\tau_1^C]),\eta_2([\tau_2,\tau_2^C])) \geq D.\]
Consequently, there exists $\xi_0 > 0$ depending only on $C,D$ so we have that $|x - y| \geq 5 \xi_0$.  By possibly decreasing the value of $\xi_0 > 0$ further (still only depending on $C,D$) we can also assume that the pairwise distances of $x$, $y$, $z_L$, and $z_R$ are all in fact at least $5\xi_0$.We can also assume that $\dist(\wh{\eta},z_q) \geq 5 \xi_0$ for $q \in \{L,R\}$. Indeed, let $\gamma_L$ and $\gamma_R$ be the arcs of the boundary of $B(0,C)$ which are on the boundaries of the two unbounded components of $(\C \setminus B(0,C)) \setminus (\eta((-\infty,\tau]) \cup \ol{\eta}((-\infty,\ol{\tau}]))$. Then $\varphi(\gamma_q)$ is a curve in $\D$ which disconnects $z_q$ for $q \in \{L,R\}$ from $0$ which $\wh{\eta}$ does not cross. Moreover the distance of $\varphi(\gamma_q)$ to $z_q$ is bounded from below. Indeed, to prove this, first we find a point $z$ between $\gamma_q$ and $B(0,C^2)$ whose distance to $\eta_1,\eta_2$ is at least $D/2$ and the harmonic measure of both of $\eta_1$ and $\eta_2$ as seen from $z$ is bounded from below. More precisely, we assume that $D < C$ is sufficiently small (but fixed) and divide the annulus $B(0,C^2 - 2D) \setminus B(0,C^2 - (2 + 1/10)D)$ using lines from the origin with equally spaced angles so that each region has diameter at most $D/2$ and the distance of the center to the boundary is at least $c_1 D$ for a constant $c_1 \in (0,1/2)$. Let $R_1,\ldots,R_n$ be those regions. Then each $R_j$ can be intersected by at most one of $\eta_1$ and $\eta_2$, but not both by the definition of $E$. Moreover, if, say $\eta_1$ intersects $R_j$ then $\eta_2$ cannot intersect the adjacent regions $R_{j-1}$ or $R_j$. Suppose that $\eta_1$ intersects $R_i$ and let $j$ be the smallest integer larger than $i$ so that $\eta_2$ intersects $R_j$, where we take the convention that $R_{n+l} = R_l$. Then $\eta_1$ cannot intersect $R_{j-1}$ and by definition $\eta_2$ does not intersect it either. Let $z$ be the center of $R_{j-1}$. Then $z$ has distance at least $c_1 D$ from $\eta_1$ and $\eta_2$. Moreover, if we start with a Brownian motion from $z$, then it has positive probability (depending only on $C,D$) of hitting either $\eta_1$ or $\eta_2$. This is because the Brownian motion has a positive probability of making either a clockwise or counterclockwise loop around the origin while staying inside a tube of width $D/2$. Next, we claim that a Brownian motion starting from $z$ has a positive probability (depending only on $C,D$) of hitting $B(0,1/2)$ before hitting either $\eta_1$ or $\eta_2$. To prove this, we divide $\C$ into squares of side length $D/10$. Then $z$ will be contained inside such a square $S$ and each of the squares which are adjacent to $S$ do not intersect $\eta_1$ or $\eta_2$, otherwise $z$ would be too close to $\eta_1$ or $\eta_2$. Also, there has to be a path of squares $S_1,\ldots,S_k$ where $S_1 = S$ and $S_k$ is contained inside of $B(0,1/2)$ none of which intersect $\eta_1 \cup \eta_2$, otherwise there has to exist a path of squares which connect $\eta_1$ and $\eta_2$ which disconnect $z$ from $B(0,1/2)$ and such that each one intersects either $\eta_1$ or $\eta_2$. But the only way that the latter can happen is if there are adjacent squares, one of which intersects $\eta_1$ and the other which intersects $\eta_2$. This cannot happen because the distance from the relevant parts of $\eta_1$ and $\eta_2$ inside $B(0,C^2) \setminus B(0,1/2)$ is at least $D$. Note also that $k = O(C^2 / D^2)$ where the implicit constants are universal. Therefore, since a Brownian motion starting from any point inside of the square $T_j$ with the same center as $S_j$ but half the width has a positive probability of hitting $T_{j+1}$ before leaving the union of $S_j$ and $S_{j+1}$, the proof of the claim is complete. Combining with the above, we obtain that there exists $p \in (0,1)$ depending only on $C,D$ such that with probability at least $p$, a Brownian motion starting from $z$ exits $\partial \D$ on the clockwise (resp.\ counterclockwise) arc of $\partial \D$ starting from $z_q$ before hitting $\varphi(\gamma_q)$ for $q \in \{L,R\}$, and also that $\dist(\varphi(z),\partial \D) \geq p$. Then, the Beurling estimate implies that there exists $\xi_0 > 0$ (depending only on $C,D$) such that $\dist(\wh{\eta},z_q) \geq \dist(\varphi(\gamma_q),z_q) \geq 5 \xi_0$ for $q \in \{L,R\}$ if $E$ occurs.
Since
\[ \p[E] = \E[ \p[E \giv y,z_L,z_R]] \geq \epsilon^{\sigma/2+o(1)} \quad\text{as}\quad \epsilon \to 0,\]
we have that the following is true.  Fix $v > 0$.  Then there exists $\epsilon_0 > 0$ so that for every $\epsilon \in (0,\epsilon_0)$ there exist $y_0$, $z_{L,0}$, $z_{R,0}$ in $\partial \D$ so that 
\[ \p[E \giv y=y_0,z_L=z_{L,0},z_R=z_{R,0}] \geq \epsilon^{\sigma/2 + v}.\]
In particular, since $\diam(\varphi(B(0,\epsilon))) \asymp \epsilon$, we have that $\p[ G ] \gtrsim \epsilon^{\sigma / 2 + v}$, where $G$ is the event that if $\eta$ is a two-sided radial $\SLE_4$ in $\D$ from $1$ to $y_0$ passing through $0$, then $\dist(\eta,z_{q,0}) \geq 5\xi_0$ for $q \in \{L,R\}$ and 
  $\sup_{z \in B(0,\epsilon)} \p^z[B \text{ hits } \partial \D \text{ before } \eta \giv \eta] \leq \exp(-\epsilon^{-\sigma})$ and the implicit constant is independent of $\epsilon, v$.

Our goal now is to upgrade this assertion to the assertion in the statement of the lemma.  Fix $x$, $y$ in $\partial \D$ with $|x-y| \geq 5 \xi$. Here, $\xi > 0$ is small (but fixed) to be chosen later in the proof. Let $\ol{\eta}$ be the time-reversal of $\eta$.  Consider the line segments $\gamma = [x,0]$, $\ol{\gamma} = [y,0]$ which respectively connect $x,y$ to $0$ and such that both of them have the unit speed parameterization.  Let $H$ (resp.\ $\ol{H}$) be $(2\pi)^{-1}$ times the length of the clockwise (resp.\ counterclockwise) arc of $\partial \D$ from $z_{L,0}$ to $z_{R,0}$.  In other words, $H$ (resp.\ $\ol{H}$) is the harmonic measure of the clockwise (resp.\ counterclockwise) arc of $\partial \D$ from $z_{L,0}$ to $z_{R,0}$ as seen from $0$.  Assume that $H > \ol{H}$ (the other case is analogous).  Fix $\delta,\zeta > 0$ small (depending only on $\xi$) and fix $1-\delta \leq t, \ol{t} \leq 1$ so the harmonic measure of $\gamma([0,t])$ (resp.\ $\ol{\gamma}([0,\ol{t}]))$ in $\D \setminus (\gamma([0,t]) \cup \ol{\gamma}([0,\ol{t}]))$ as seen from $0$ is in $[H-\zeta / 2,H+\zeta / 2]$ (resp.\ $[\ol{H}-\zeta / 2,\ol{H}+\zeta / 2]$).  We note that this is always possible because for any $\ol{s} \in [1-\delta,1)$ fixed, the harmonic measure of $\ol{\gamma}([0,\ol{s}])$ in $\D \setminus (\gamma([0,s]) \cup \ol{\gamma}([0,\ol{s}]))$ as seen from $0$ tends to $0$ as $s$ increases to $1$ and likewise when the roles of $(\gamma,s)$ and $(\ol{\gamma},\ol{s})$ are swapped.

Let $\tau$ (resp.\ $\ol{\tau}$) be the first time that $\eta$ (resp.\ $\ol{\eta}$) gets within distance $\delta$ of $\gamma(t)$ (resp.\ $\ol{\gamma}(\ol{t})$), where $\eta$ (resp.\ $\ol{\eta}$) is parameterized according to log-conformal radius as seen from $0$.  Let also $A_\delta$ (resp.\ $\ol{A}_\delta$) be the $\delta$-neighborhood of $\gamma([0,t])$ (resp.\ $\ol{\gamma}([0,\ol{t}])$).  Let $E_1$ be the event that $\tau$ and $\ol{\tau}$ both occur before $\eta$ and $\ol{\eta}$, respectively, leave $A_\delta$ and $\ol{A}_\delta$.  Then \cite[Lemma~2.3]{mw2017slepaths} implies that there exists $p_1 \in (0,1)$ depending only on $\delta, \xi$ so that $\p[E_1] \geq p_1$.  We now assume that we are working on $E_1$.  Let $\varphi$ be the unique conformal map from $\D \setminus (\eta([0,\tau]) \cup \ol{\eta}([0,\ol{\tau}]))$ to $\D$ which fixes $0$ and takes $\eta(\tau)$ to $1$.  Let $I = \varphi(\eta([0,\tau]))$ and $\ol{I} = \varphi(\ol{\eta}([0,\ol{\tau}]))$.  By choosing $\delta > 0$ sufficiently small, the Beurling estimate implies that $|I|/2 \pi \in [H-\zeta,H+\zeta]$ and $|\ol{I}|/2\pi \in [\ol{H} - \zeta, \ol{H}+\zeta]$ where $|\cdot|$ denotes Lebesgue measure on $\partial \D$.  Let $z_L$ (resp.\ $z_R$) be the left (resp.\ right) endpoint of $I$. We claim that $\varphi(K) \subseteq B(z_L,\xi) \cup B(z_R,\xi)$ for $\delta > 0$ sufficiently small, where $K = \{z \in \D \setminus (\eta([0,\tau]) \cup \ol{\eta}([0,\ol{\tau}])) : \dist(z,\partial \D) \leq \xi \}$. Indeed, the Beurling estimate implies that the probability that a Brownian motion starting from $0$ hits $K$ before exiting $\D \setminus (\eta([0,\tau]) \cup \ol{\eta}([0,\ol{\tau}]))$ is at most $O(\delta^{1/2})$ where the implicit constants depend only on $\xi$. Then the conformal invariance of the Brownian motion and \cite[Exercise~2.7]{law2005slebook} together imply that $\diam(K) = O(\delta^{1/2})$ where the implicit constants depend only on $\xi$ and so $\varphi(K) \subseteq B(z_L, \xi / 2) \cup B(z_R, \xi / 2)$ for $\delta$ sufficiently small.

Let $x_1 \in I$ be such that the length of the clockwise arc in $\partial \D$ from $x_1$ to $z_L$ is equal to the length of the clockwise arc in $\partial \D$ from $1$ to $z_{L,0}$.  We also let $y_1 \in \ol{I}$ be such that the length of the clockwise arc in $\partial \D$ from $y_1$ to $z_R$ is equal to the length of the clockwise arc in $\partial \D$ from $y_0$ to $z_{R,0}$. Note that $x_1$ and $y_1$ are well-defined if $\zeta \in (0,\xi)$. Set $\wh{\eta} = \varphi(\eta([\tau,\infty)) \cup \ol{\eta}([\ol{\tau},\infty)))$ and note that conditional on $\eta|_{[0,\tau]}$ and $\ol{\eta}|_{[0,\ol{\tau}]}$, the curve $\wh{\eta}$ has the law of a two-sided radial $\SLE_4$ in $\D$, from $1$ to $\varphi(\ol \eta(\ol \tau))$, passing through $0$.  Let $\sigma$ be the first time that $\wh{\eta}$ gets within distance $\delta$ of $x_1$ (where again $\wh{\eta}$ is parameterized according to log-conformal radius as seen from $0$).  Let $E_2$ be the event that $\sigma$ occurs before the first time $\wh{\eta}$ leaves the $\delta$-neighborhood of the arc of $I$ from $1$ to $x_1$. We note that on $E_1$, the harmonic measure of each side of $\eta([0,\tau])$, seen from $0$, is bounded below by a positive constant, depending only on $\xi$ and $\delta$, and hence the same is true for the distances from $z_L$ and $z_R$ to $1$. Then \cite[Lemma~2.3]{mw2017slepaths} implies that there exists $p_2 > 0$ depending only on $\xi$ and $\delta$ so that $\p[E_2 \giv E_1] \geq p_2$.  We assume that we are working on $E_2 \cap E_1$.  For each $s \geq 0$, we let $\psi_s$ be the unique conformal map from $\D \setminus (\wh{\eta}([0,\sigma]) \cup \wt{\eta}([0,s]))$ to $\D$ with $\psi_s(0) = 0$ and $\psi_s'(0) > 0$, where $\wt{\eta} = \varphi(\ol{\eta}([\ol{\tau},\infty)))$ and $\wt{\eta}$ is parameterized according to log-conformal radius as seen from $0$.  Let $\wt{\sigma}$ be the first time $s \geq 0$ that the harmonic measure of the clockwise arc of $\partial \D$ from $\psi_s(\wh{\eta}(\sigma))$ to $\psi_s(\wt{\eta}(s))$ is equal to the harmonic measure of the clockwise arc of $\partial \D$ from $1$ to $y_0$, both as seen from $0$.  Let $E_3$ be the event that $\wt{\sigma}$ occurs before $\wt{\eta}$ leaves the $\delta$-neighborhood of $\ol{I}$.  Then \cite[Lemma~2.3]{mw2017slepaths} (considering the event that $\wt \eta$ stays close to $I$ and travels along the clockwise arc from $\varphi(\ol \eta(\ol \tau))$ to $z_R$ if $y_1$ lies to the right of $\varphi(\ol \eta(\ol \tau))$ and otherwise the counterclockwise arc) implies that there exists $p_3 > 0$ depending only on $\delta$ so that $\p[E_3 \giv E_1, E_2] \geq p_3$.

On $E_1 \cap E_2 \cap E_3$, let $x_2 = \psi_{\wt{\sigma}}(\wh{\eta}(\sigma))$, $y_2 = \psi_{\wt{\sigma}}(\wt{\eta}(\wt{\sigma}))$, $z_{L,2} = \psi_{\wt{\sigma}}(z_L)$, $z_{R,2} = \psi_{\wt{\sigma}}(z_R)$.  Then we have that the harmonic measure as seen from $0$ of the clockwise arc of $\partial \D$ from $x_2$ to $y_2$ is equal to that of the clockwise arc of $\partial \D$ from $1$ to $y_0$.  Moreover, we can find a function $p(\delta)$ of $\delta$ such that $p(\delta) \to 0$ as $\delta \to 0$ and the harmonic measure of the clockwise (resp.\ counterclockwise) arc of $\partial \D$ from $x_2$ to $z_{L,2}$ (resp.\ $z_{R,2}$) is equal to the harmonic measure of the corresponding arcs with $1, z_{L,0}, z_{R,0}$ in place of $x_2, z_{L,2}, z_{R,2}$ up to an error of which is at most $p(\delta)$. Therefore, there exists $\theta \in [0,2\pi)$ such that $e^{i\theta} x_2 = 1$, $e^{i\theta} y_2 = y_0$, $|e^{i\theta}z_{L,2} - z_L | = O(p(\delta))$ and $|e^{i\theta} z_{R,2} - z_R | = O(p(\delta))$. Note also that for $\xi \in (0,1)$ sufficiently small we have that (by again applying \cite[Exercise~2.7]{law2005slebook}) $\psi_{\wt{\sigma}}(B(z_L, \xi) \cup B(z_R, \xi)) \subseteq B(z_{L,0}, \xi_0 / 2) \cup B(z_{R,0} , \xi_0 / 2)$ and so $\psi_{\wt{\sigma}} \circ \varphi (K) \subseteq B(e^{-i\theta}z_{L,0},\xi_0) \cup B(e^{-i\theta}z_{R,0}, \xi_0)$ on $E_1 \cap E_2 \cap E_3$ for $\xi$ sufficiently small and $\delta, \zeta$ sufficiently small depending only on $\xi$. To conclude the proof of the lemma, set $\gamma = \psi_{\wt{\sigma}}(\wh{\eta}([\sigma,\infty)) \cup \wt{\eta}([\wt{\sigma},\infty)))$ and let $G$ be the event corresponding to $\gamma$ as before. Then since $\p[ E_1 \cap E_2 \cap E_3] > 0$ and on $E_1 \cap E_2 \cap E_3$ we have that $\psi_{\wt{\sigma}}\circ \varphi(K) \subseteq B(e^{-i\theta}z_{L,0},\xi_0) \cup B(e^{-i\theta}z_{R,0}, \xi_0)$ and $\diam (\psi_{\wt{\sigma}} \circ \varphi (B(0,\epsilon))) \asymp \epsilon$ (with the implicit constants depending only on $\xi, \delta , \zeta$), we obtain that if $G$ holds, then $F$ holds as well for $\epsilon$ sufficiently small. The proof is then complete since the law of a two-sided radial $\SLE_4$ in $\D$ passing through $0$ is invariant under rotation and so $\p[ G ] \gtrsim \epsilon^{\sigma / 2 + v}$ for all $\epsilon$ sufficiently small, and the implicit constant is independent of $\epsilon,v$.
\end{proof}

\begin{proof}[Proof of Lemma~\ref{lem:sle4_estimate}]
We suppose that we have the setup which is described just before the statement of Lemma~\ref{lem:sle4_estimate}.  Let $\ol{\eta}$ be the time-reversal of $\eta$.

\noindent{\it Part~\eqref{it:ek2_p1} of $E_k^2$.} Let $A_k$ denote the event of part (i) of the definition of $E_k^2$. Then there exists a $p>0$ such that $\p[A_k] \geq p$ for all $k$. This follows since by scaling and translation invariance of sssi $\SLE_4$ curves, we may assume that $\epsilon = 1$ and $t_k = 0$ and thus, by \cite[Lemmas~2.3 and~2.5]{mw2017slepaths} we have (after conformally mapping $\C \setminus \eta((-\infty,0])$ to $\h$) that $\p[ A_{k} \,\giv \, \eta] > 0$ a.s. The claim then follows by integrating over the randomness of $\eta$.

\noindent{\it Part~\eqref{it:ek2_p2} of $E_k^2$.} Note that by the translation and scaling invariance of sssi $\SLE_4$ curves, the probability that $\eta((-\infty,t_{k-1}])$ or $\eta([t_{k+1},\infty))$ intersects $B(\eta(t_k), \epsilon^a)$ is equal to the probability that $\eta((-\infty,-\epsilon^{-3r / 2}])$ or $\eta([\epsilon^{-3r / 2},\infty))$ intersects $\D$, and this tends to~$0$ as $\epsilon \to 0$.  In particular, for small enough $\epsilon > 0$, we can make the probability of this part of the event as close to~$1$ as we want.

\noindent{\it Part~\eqref{it:ek2_p3}.}  Let $\tau$ (resp.\ $\ol{\tau}$) be the first time that $\eta$ (resp.\ $\ol{\eta}$) hits $\partial B(0,\epsilon^c)$.  Then the conditional law of the remainder of $\eta$ given $\CG = \sigma(\eta|_{(-\infty,\tau]}, \ol{\eta}|_{(-\infty,\ol{\tau}]})$ is that of a two-sided radial $\SLE_4$ in $\C \setminus (\eta((-\infty,\tau]) \cup \ol{\eta}((-\infty,\ol{\tau}]))$ from $\eta(\tau)$ to $\ol{\eta}(\ol{\tau})$ passing through $0$.

Let $\varphi$ be the unique conformal map from $\C \setminus (\eta((-\infty,\tau]) \cup \ol{\eta}((-\infty,\ol{\tau}]))$ to $\D$ which fixes $0$ and sends $\eta(\tau)$ to $1$.  We assume that we are working on the event that $x = 1$ and $y = \varphi(\ol{\eta}(\ol{\tau}))$ have distance at least $\xi > 0$. Note that the probability of this event tends to $1$ as $\xi$ tends to $0$. On this event, Lemma~\ref{lem:two_sided_radial_lemma51} and distortion estimates for conformal maps imply that the conditional probability given $\CG$ of the intersection $\wt{E}_k$ of $E_k^1$ and the event that the intermediate part of $\eta$ does not leave $B(\eta(t_k),\epsilon^b/4)$ is at least $\epsilon^{\sigma/2 + o(1)}$.

Suppose that $F$ is any event for $\eta_{1,k}$, $\eta_{2,k}$ with $\p[F \giv \CG] > 0$ and such that on $F$ the intermediate part of $\eta$ does not leave $B(\eta(t_k),\epsilon^b / 4)$.  Then we have that
\begin{align*}
\p[ \eta_{1,k}, \eta_{2,k} \in F \giv \CG]
&\asymp \p[ \eta_{1,k}, \eta_{2,k} \in F\giv \CG, \wt{E}_k ] \quad\text{(\cite[Lemma~2.8]{mw2017slepaths})}\\
&= \frac{\p[ \wt{E}_k \giv \eta_{1,k}, \eta_{2,k} \in F, \CG]}{\p[ \wt{E}_k \giv \CG]} \p[ \eta_{1,k}, \eta_{2,k} \in F \giv \CG],
\end{align*}
using Bayes' rule for the equality. Note that the implicit constants depend only on $\xi$. Rearranging gives that
\[ \p[ \wt{E}_k \giv \eta_{1,k}, \eta_{2,k} \in F, \CG] \asymp \p[ \wt{E}_k \giv \CG] \geq \epsilon^{3/2 - u + o(1)}\]
uniformly in the choice of $F$. We therefore have that
\[ \p[ \wt{E}_k \giv \eta_{1,k}, \eta_{2,k}, \CG] \geq \epsilon^{3/2 - u + o(1)}.\]
By decreasing the value of $\xi > 0$ if necessary, we can thus make the probability that this part holds as close to $1$ as we want.

Altogether, we see that we can make each of the events in part~\eqref{it:ek2_p2} and part~\eqref{it:ek2_p3} hold with probability at least $1 - p / 4$ where $p$ is the constant in part~\eqref{it:ek2_p1}. Combining we obtain that for $\xi > 0$ sufficiently small, we have that the events in all three parts hold simultaneously with probability at least $p / 2$, which completes the proof.
\end{proof}

\begin{proof}[Proof of Theorem~\ref{thm:sle4_thm}, lower bound]
\noindent{\it Step 1. Setup.}  Suppose that we have the same setup as just before the statement of Lemma~\ref{lem:sle4_estimate}. Fix $\zeta>0$ and pick $\xi>1$ and $u > 0$ such that $1 < \xi < (1/3 + \zeta)(3 - 2u)$.  
Let $F_{k,\epsilon}$ be the event that there is a ball $B(z,\epsilon^\xi) \subseteq B(\eta(t_k),2\epsilon) \cap \C_L$. Then Lemma~\ref{lem:sle4_estimate}, Remark~\ref{rem:fitball} and the scaling and translation invariance properties of the joint law of $(\eta_1,\eta_2)$ imply that there exists $p_2 > 0$ so that $\p[E_k^3] \geq p_2 > 0$ for all $\epsilon > 0$ sufficiently small, where $E_{k}^3 = E_{k}^2 \cap F_{k,\epsilon}$.

\noindent{\it Step 2 The event $E_k^3$ occurs for a positive fraction of times.}  Let $n(\epsilon) = |(t_k)|$.  Then the fraction $F(\epsilon) = \frac{1}{n(\epsilon)}\sum_{k=1}^{n(\epsilon)} \one_{E_k^3}$ of $t_k$'s for which the event $E_k^3$ occurs is a non-negative random variable with $\E[F(\epsilon)] \geq p_2 > 0$. Hence, by the Paley-Zygmund inequality, we have for $\theta \in (0,1)$ that
\[ \p[ F(\epsilon) > \theta p_2] \geq \frac{(1-\theta)^2 p_2^2}{\E[F(\epsilon)^2]} \geq (1-\theta)^2 p_2^2.\]
That is, with positive probability, the events $E_k^3$ occur at a positive fraction of times $t_k$ with $1 \leq k \leq n(\epsilon)$.  Let $\CI = \{1 \leq k \leq n(\epsilon): E_k^3 \, \text{occurs}\}$.  Taking $\theta = 1/2$, we have shown that with probability at least $p_2^2/4$, we have $|\CI| \gtrsim p_2 \epsilon^{-\frac{3}{2}(a-r)}$.

\noindent{\it Step 3. Lower bound is attained with positive probability.}
By definition, we have on $E_k^3$ that $\p[E_k^1  \giv \eta_{k,1}, \eta_{k,2}, \ol{\eta}_k] \geq \epsilon^{\frac{3}{2}-u +o(1)}$.  Conditionally on $\CI$ and $\eta_{k,1}$, $\eta_{k,2}$ for $k \in \CI$ we thus have that the number of $k \in \CI$ for which $E_k^1$ occurs is stochastically dominated from below by a binomial random variable with parameters $|\CI|$ and $\epsilon^{\frac{3}{2}-u +o(1)}$.  Thus choosing $a$ close to $1$ and $r$ close to $0$ so that $\frac{3}{2} - u + \frac{3}{2}(r-a) < 0$ it follows by binomial concentration that on the event $|\CI| \gtrsim p_2 \epsilon^{-\frac{3}{2}( a-r)}$ we have off an event which decays to $0$ as $\epsilon \to 0$ faster than any power of $\epsilon$ that there exists $1 \leq k \leq n(\epsilon)$ so that $E_k^1$ occurs. Let $C_{\epsilon}$ be the event that $E_{k}^1$ occurs for some $1 \leq k \leq n(\epsilon)$ and for $d \in (0,1)$ let $A_d$ be the event that the following hold
\begin{enumerate}[(i)]
\item $\dist(\eta(I), \eta((-\infty,0])) \geq d$,
\item $d \leq |\varphi_2'(z)| \leq d^{-1}$ for all $z \in \C_L$ such that $\dist(z,\eta(I)) \leq d / 2$,
\item $\dist(\varphi^{-1}(\varphi_2(z)), J) \geq d$ for all $z \in \C_L$ such that $\dist(z,\eta(I)) \leq d / 2$,
\end{enumerate}
where $J = \{e^{i\theta} : \pi / 2 \leq \theta \leq 3\pi / 2\}$. Since $\p[ A_d ] \to 1$ as $\epsilon \to 0$, we can find constants $p,d \in (0,1)$ such that $\p[ A_d \cap C_{\epsilon} ] \geq p$ for all $\epsilon$ sufficiently small. Suppose that we are working on the event $A_d \cap C_{\epsilon_n}$ for $\epsilon_n = 2^{-n}$ ($n$ sufficiently large) and let $B(z,\epsilon_{n}^{\xi})  \subseteq B(\eta(t_k),2\epsilon_n) \cap \C_L$ be the ball considered in the event $F_{k,\epsilon_{n}}$. Set $z_0 = \varphi_2(z)$. Then the Koebe-$1/4$ theorem implies that $\dist(z_0, \partial \D_L) \geq \tfrac{d}{4} \epsilon_n^\xi$. Moreover, the probability that a Brownian motion starting from $z_0$ exits $\D$ on $J$ is at most $e^{-\epsilon_n^{-3+2u}}$ and so conformal invariance implies that the probability that a Brownian motion starting from $\varphi^{-1}(z_0)$ exits $\D$ on $J$ is at most $e^{-\epsilon_{n}^{-3+2u}}$. This implies that $\dist(\varphi^{-1}(z_0), \partial \D) \lesssim e^{-\epsilon_{n}^{-3+2u}}$ and so  we can find  $w_n$ in $\D$ such that $|\varphi^{-1}(z_0) - w_n| \lesssim e^{-\epsilon_{n}^{-3+2u}}$  and if $z_n = \varphi(w_n)$, we have that 
\begin{align*}
|\varphi(\varphi^{-1}(z_0)) - \varphi(w_n)|\left( \log \frac{1}{|\varphi^{-1}(z_0) - w_n|} \right)^{1/3 + \zeta} \gtrsim \epsilon_{n}^{\xi + (2u-3)(1/3 + \zeta)}
\end{align*}
where the implicit constant depends only on $d$. Therefore letting $n \to \infty$, we obtain that with probability at least $p$ we have that
\begin{equation}
\label{eq:non_regularity}
\sup_{z, w \in \D,\ z \neq w} |\varphi(z)-\varphi(w)| \!\left( \log \frac{1}{|z-w|} \right)^{1/3+\zeta} = \infty.
\end{equation}

We note that by absolute continuity, the same holds for $\SLE_4(\rho_L;\rho_R)$ processes with $\rho_L,\rho_R > -2$ (here we mean $\SLE_4(\rho_L;\rho_R)$ processes from $-i$ to $i$ in $\D$ place of the $\SLE_4$ process in $\D$ in the statement of the theorem). Indeed, if we map out two small segments of the beginning and at the end of the $\SLE_4(\rho_L;\rho_R)$ process $\eta^*$, i.e., map them to the boundary $\partial\D$ with some conformal map, say $g$, and denote the resulting curve by $\eta^\circ$, then on the positive probability event that $\eta^\circ$ does not intersect $\partial \D$ before hitting $i$,  $\eta^\circ$ is absolutely continuous with respect to an $\SLE_4$ process everywhere. Hence clearly, conditional on the event that $\eta^\circ$ does not hit $\partial \D$ before $i$,~\eqref{eq:non_regularity} holds with positive probability for this curve. Since $g$ is smooth on $\eta^\circ$, we have that if~\eqref{eq:non_regularity} holds somewhere on $\eta^\circ$ (for the map $\varphi^\circ$, taking $\D$ to the left side of $\eta^\circ$ as in Theorem~\ref{thm:sle4_thm}), then the same holds somewhere on $\eta^*$, away from the boundary (for the map $\varphi^*$, defined analogously).

\noindent{\it Step 3.  From positive probability to probability $1$.}  We shall be rather brief, as a similar strategy is explained in the proof of Lemma~\ref{lem:fitball}. Let $\wt \eta$ be the $\SLE_4$ in the statement of the theorem and consider a conformal map $\psi: \D \to \h$, with $\psi(-i) = 0$, $\psi(i) = \infty$ and, say, $\psi(-1) = -1$. The image of $\wt \eta$ under $\psi$ is that of an $\SLE_4$ process $\wh \eta$ in $\h$ from $0$ to $\infty$.  We assume that $\wh{\eta}$ is the $0$-height level line of a GFF on $\h$ with boundary conditions given by $-\lambda$ on $\R_-$ and $\lambda$ on $\R_+$.  For $j \in \N$, let $A_j = B(0,2^j) \setminus B(0,2^{j-1})$ and let $\tau_j$ be the first hitting time of $\partial B(0, (1+\frac{1}{10})2^{j-1})$ for $\wh \eta$. We fix some $\vartheta > 0$ and let $\wh \eta_j^1$ (resp.\ $\wh \eta_j^2$) be the level line of height $-\vartheta$ (resp.\ $+\vartheta$) started from $\wh \eta(\tau_j)$, stopped upon exiting $A_j$. Then with positive probability, uniform in the realization of $\wh \eta([0,\tau_j])$, $\wh \eta_j^1$ and $\wh \eta_j^2$ form a pocket $P_j$, containing a ball of diameter $2^{j-1}$ before exiting $A_j$. We let $\phi_j:\D \to P_j$ be the conformal map taking $-i$ (resp.\ $i$) to the opening (resp.\ closing) point of the pocket and such that $\phi_j'(0) > 0$. Let $P_j^L$ denote part of $P_j$ which is to the left of $\wh \eta$. Let $\wh \eta^j$ be the image under $\phi_j^{-1}$ of the part of $\wh \eta$ which is in $P_j$. Then the law of $\wh \eta^j$ is that of an $\SLE_4(\rho_L;\rho_R)$ process from $-i$ to $i$ in $\D$, where $\rho_L,\rho_R$ depend on $\vartheta$. We set $\D_L^j = \phi_j^{-1}(P_j^L)$ and let $\varphi^j: \D \to \D_L^j$ be the conformal map fixing $-i$, $i$ and $-1$ (so that it is the analogue of $\varphi$, but for the $\SLE_4(\rho_L;\rho_R)$ process $\wh \eta^j$). Then since the laws of $\wt \eta$ and $\wh \eta^j$ are mutually absolutely continuous away from $\partial \D$, we have that with positive probability~\eqref{eq:non_regularity}, independent of $j$, for $\varphi^j$ as explained above. Moreover, if we let $O_j = (\psi \circ \varphi)^{-1}(P_j^L)$ and $f_j : O_j \rightarrow \D$ be a conformal map taking $\partial O_j \cap \partial \D$ to the $\{z \in \partial \D: \re(z) \geq 0 \}$. This leaves one degree of freedom, which can be chosen so that $\varphi^j \circ f_j = (\phi_j^{-1} \circ \psi \circ \varphi)|_{O_j}$. Since $\psi$, $f_j$ and $\phi_j^{-1}$ are smooth on the images and preimages of the curve $\wt \eta$ under the various maps, we have that if $\varphi^j$ satisfies~\eqref{eq:non_regularity} for some $j$, then so does $\varphi$. Consequently, since this occurs with positive probability, uniform in $j$, for each $\varphi^j$ it occurs a.s.\ for $\varphi$.  This concludes the proof.
\end{proof}

\subsection{Proof of Theorem~\ref{thm:jones_smirnov_sle4}}

The strategy will be to consider a Whitney square decomposition of $\D_L$ and note that the quasihyperbolic distance between two points is roughly minus $\log$ of the probability that a Brownian motion started from one point reaches the Whitney square containing the other point before exiting $\D_L$. With this in mind, we can consider the escape probabilities of points close to the boundary and use them to lower bound the quasihyperbolic distance between them and a fixed interior point.

Let $(Q_j)$ be a Whitney square decomposition of $\D_L$ and denote by $x_j$ the center of $Q_j$. We choose it so that the side length of $Q_j$ is $2^{-n_j}$ for some $n_j \in \Z$ and $\diam(Q_j) \leq \dist(Q_j,\partial \D_L) < 4 \diam(Q_j)$. We let $G = (V,E)$ be the graph with vertices $V = (x_j)$ and edge set $E$ such that $\{x_i,x_j\} \in E$ if $Q_i$ and $Q_j$ are share a boundary segment. Let $d_G$ denote the graph distance in $G$, that is, $d_G(x_i,x_j)$ is the minimal number of Whitney squares one can cross when traveling from $x_i$ to $x_j$ in $G$. Then
\begin{align*}
	d_G(x_i,x_j) \asymp \distqh(x_i,x_j),
\end{align*}
where the implicit constants depend only on the Whitney square decomposition. Next, we note that if $z,w \in Q_j$, then $\distqh(z,w) \leq 1$. Indeed,
\begin{align*}
	\distqh(z,w) = \inf_{\gamma: \gamma(0) = z, \, \gamma(1) = w} \int_\gamma \frac{1}{\dist(\xi,\partial \D_L)} |d\xi| \leq \frac{|z - w|}{\dist(Q_j,\partial \D_L)} \leq \frac{\diam(Q_j)}{\dist(Q_j,\partial \D_L)} \leq 1.
\end{align*}
In the same way, $\distqh(z,x_j) \leq 1/2$ whenever $z \in Q_j$.
Let $\wh{Q}_j$ denote the square with center $x_j$, but with half the side length of $Q_j$ and define $x(w) = x_j$ if and only if $w \in Q_j$. Then,
\begin{align*}
	\distqh(w,z) \asymp d_G(x(w),x(z))
\end{align*}
whenever $w \in \wh{Q}_j$ and $z \notin Q_j$, where the implicit constant depends only on the Whitney square decomposition.

We now relate the quantities to the hitting probability of Brownian motion. Let $B$ be a Brownian motion with $B_0 = z_0 \in \wh{Q}_j$. Assume that $Q_k$ shares a boundary segment with $Q_j$. Then $\frac{1}{8} \diam(Q_j) \leq \diam(Q_k) \leq 8 \diam(Q_j)$ and hence there is a $p>0$ such that, uniformly in $z_0 \in \wh{Q}_j$, the probability that $B$ hits $\wh{Q}_k$ before exiting the domain is at least $p$. Moreover, there is a $q>0$ such that uniformly for $w_0 \in Q_j \setminus \wh{Q}_j$, the probability that $B$ hits $\wh{Q}_j$ before exiting $\D_L$ is at least $q$. Clearly, $q$ is independent of the size of $Q_j$. Consequently,
\begin{align*}
	\p[B \, \text{hits} \, Q_k \, \text{before exiting} \,  \D_L] \geq q p^{d_G(x_j,x_k)},
\end{align*}
that is,
\begin{align*}
	d_G(x_j,x_k) \gtrsim -\log \p[B \, \text{hits} \, Q_k \, \text{before exiting} \, \D_L],
\end{align*}
where the implicit constant depends only on the Whitney square decomposition. Next, we recall that for sufficiently small $\epsilon > 0$, there exist points $z_l$ with $\dist(z_l,\D_L) \in [c_1 \epsilon,c_2\epsilon]$, $0< c_1 < c_2 < \infty$, such that the probability that a Brownian motion started at $z_l$ travels macroscopically far away with probability at most $\exp(-\epsilon^{-3})$. This holds by the conformal invariance of Brownian motion, since the map $\varphi_2$ (defined in the beginning of this section) is smooth away from $\pm i$, and since there exist points $z \in \C_L$ of distance $[\epsilon,2\epsilon]$ from the two-sided whole-plane $\SLE_4$ considered above, at which the escape probability is at most $\exp(-\epsilon^{-3})$. Thus, if we fix some square $Q_k$, the probability that a Brownian motion started at such a point $z_l$ has probability at most $\exp(-\epsilon^{-3})$ of hitting $Q_k$ before exiting the domain. Moreover, assuming that $\dist(z_l,Q_k) > c$ for some $c>0$ we have that
\begin{align*}
	\sup_{z \in B(z_l,\epsilon/2)} \p^z[B \, \text{hits} \, Q_k \, \text{before exiting} \,  \D_L] \asymp \p^{z_l}[B \, \text{hits} \, Q_k \, \text{before exiting} \,  \D_L]
\end{align*}
for some universal implicit constant. Thus,
\begin{align*}
	&\int_{B(z_l,\epsilon/2)} \distqh(w,x_k) dw \gtrsim \int_{B(z_l,\epsilon/2)} d_G(x(w),x_k) dw \\
&\gtrsim \int_{B(z_l,\epsilon/2)} -\log \p^w[B \, \text{hits} \, Q_k \, \text{before exiting} \,  \D_L] dw \gtrsim \int_{B(z_l,\epsilon/2)} \epsilon^{-3+o(1)} dw \gtrsim \epsilon^{-1 +o(1)},
\end{align*}
with implicit constants depending only on the Whitney square decomposition. Since the event that there exists $\epsilon_0 > 0$ so that for all $\epsilon \in (0,\epsilon_0)$ there is such a point $z_l$ has positive probability, it follows that $\p[\distqh(\cdot,x_k) \notin L^1(\D_L)] > 0$. The a.s.\ statement follows from the same argument as in Step 3 of the proof of the lower bound of Theorem~\ref{thm:sle4_thm}.
\qed

\appendix

\section{Bounds and moments for circle-average embedded surfaces}
\subsection{Moment bounds of a quantum wedge}
In this section we our goal is to prove the following.
\begin{lemma}\label{lem:pthmoment}
Let $(\strip,h,-\infty,+\infty) \sim \qwedgeW{\gamma}{\gamma^2/2}$ have the first exit parameterization. Then, for every $p \in (-\infty,\min(\frac{2}{\gamma^2},\frac{3}{2}))$,
\begin{align}
	c_p \coloneqq \E[\mu_h(\strip_-)^p] < \infty.
\end{align}
Moreover, if $u_{\alpha,\epsilon} \coloneqq \inf \{t \in \R: X_t = \alpha\log\epsilon\}$, where $X_t$ is the average on vertical lines process of $h$, then
\begin{align}\label{eq:pthmomenteps}
	\E[\mu_h(\strip_-+u_{\alpha,\epsilon})^p] = c_p \epsilon^{\alpha p \gamma}.
\end{align}
Furthermore, if $p \in (-\infty, \min( \frac{4}{\gamma^2},3))$ we have that
\begin{align}
	\wt{c}_p = \E[ \nu_h((-\infty,0]\times\{0\})^p] = \E[ \nu_h((-\infty,0]\times\{\pi\})^p] < \infty,
\end{align}
and
\begin{align}\label{eq:pthmomentepsbdy}
	\E[ \nu_h((-\infty,u_{\alpha,\epsilon}]\times\{0\})^p] = \E[ \nu_h((-\infty,u_{\alpha,\epsilon}]\times\{\pi\})^p]  = \wt{c}_p \epsilon^{\alpha p \gamma/2}.
\end{align}
\end{lemma}
In proving Lemma~\ref{lem:pthmoment}, it is convenient to decompose the strip $\strip$ into the rectangles $A_k = (k -1 , k] \times (0,\pi)$, for $k \in \Z$ and use the decomposition $h = X + h_2$, where $(-X_{-t/2})_{t \geq 0}$ is a $\BES^3$ and $h_2$ is the lateral part of a free boundary GFF, independent of $X$. We then recall that the Laplace transform of $X$ decays at a sufficient rate (Lemma~\ref{lem:bes_rn_derivative}) which together with the translation invariance of $h_2$ helps us bound the moments of the quantum area and boundary length of each square $A_k$ for $k \leq 0$. With this it is straightforward to bound the moments of $\mu_h(\strip_-)$ and $\nu_h((-\infty,0] \times \{0\})$. The natural route is to first prove that the quantum area (resp.\ boundary length) of $A_0$ with respect to a free boundary GFF on $\strip$ has finite moments of orders between $0$ and $2/\gamma^2$ (resp.\ $4/\gamma^2$). This is then used to bound the moments of the lateral part of the free boundary GFF. Before doing that, however, we shall prove a result on the exponential moments of the supremum of the harmonic part of a GFF in some sets. This is  used in many places throughout the paper, among other things it will aid us in proving that the negative moments of the quantum measure with respect to a free boundary GFF are finite (when not too small sets are considered).

\begin{lemma}\label{lem:bound_harmonic_part}

Let $D \subseteq \h$ be a simply connected domain and $U \subseteq D$ be a bounded subdomain such that $\partial U \cap \partial D \subset \R$ be connected and $\dist( \partial U \cap \partial D, \R \setminus \partial D) > 0$. 
Furthermore, we let $h$ be a free boundary GFF on $D$ and $h = h^0 + \Fh$ be the Markovian decomposition into a zero-boundary GFF $h^0$ on $U$ and a random distribution $\Fh$ which is harmonic on $U$. Assume that the additive constant for $h$ has been chosen either of the following items hold:
\begin{itemize}
	\item $(h,\rho) = 0$, where $\rho$ is a fixed finite Radon measure on $\ol{D}$, with bounded support contained in $D \cup (\partial D \cap \R)$ and such that $\sup_{x \in K}\left|  \int \greenN{D}(x,y) \rho(dy) \right| < \infty$ for each compact set $K \subset D \cup (\partial D \cap \R)$.
	\item $\Fh(z_0) = 0$, for some fixed $z_0 \in U$.
\end{itemize}
Fix $b\geq 1$ and $c \in (1,2b)$. Then, for each $z \in U$ such that $\dist(z,\partial U) > 2b\epsilon$,
\begin{align*}
	\E\!\left[ \exp\!\left( a \sup_{w \in B(z,\epsilon)} \Fh(w)\right) \right] = O(\epsilon^{-\frac{a^2 c^4}{(c-1)^4}} ),
\end{align*}
where the implicit constant depends only on $b$, $c$, $D$, $U$ and the way that we have fixed the additive constant.
\end{lemma}
\begin{proof}

We begin by showing that that $\var[\Fh(w)] \leq -2 \log \epsilon + O(1)$, uniformly in $\epsilon$ and $w \in U$ with $\dist(w,\partial U) \geq (2b-c)\epsilon$. Fix $z_0 \in U$ and let $\phi: D \to \h$ be the conformal map such that $\phi(z_0) = i$ and $\phi'(z_0) > 0$. Let $I$ be the interior in $\R$ of the set $\partial D \cap \R$ Then, by Schwarz reflection, $\phi$ extends to a conformal map from $D \cup D^* \cup I$ to $\h \cup \h^* \cup \phi(I)$ where $D^*$ (resp.\ $\h^*$) is the reflection of $D$ (resp.\ $\h$) across $\R$. We note that, since $\dist(U,\partial D \setminus I) > 0$, there exists a constant $M > 1$, depending only on $D$ and $U$, such that $M^{-1}  \leq |\phi'(w)| \leq M$ for all $w \in U$ and hence we have by the Koebe-$1/4$ theorem that
\begin{align*}
	\dist(\phi(w),\partial \h) \geq \frac{2b-c}{4M} \epsilon, \quad \text{for all} \ w \in U \ \text{such that} \ \dist(w,\partial U) \geq (2b-c)\epsilon.
\end{align*}
We set $\wt{h} = h \circ \phi^{-1}$, $\wt{h}^0 = h^0 \circ \phi^{-1}$, $\wt{\Fh} = \Fh \circ \phi^{-1}$ and $\wt{U} = \phi(U)$ and note that it is enough to prove that $\var[\wt{\Fh}(w)] \leq -2 \log \epsilon + O(1)$ uniformly in $\epsilon$ and $\{ \wt{w} \in \h: \im(\wt{w}) \geq (4M)^{-1} (2b-c) \epsilon \}$.

Assume, for now, that $h$ is normalized so that $(h,\rho) = 0$, for some finite Radon measure $\rho$ as in the statement of the lemma. Then, setting $\wt{\rho} = |(\phi^{-1})'|^2 ( \rho \circ \phi^{-1})$, we have that $\sup_{\wt{w} \in K} \left| \int \greenN{\h}(\wt{w},z) \wt{\rho}(dz) \right| < \infty$ for all $K \subseteq \h \cup \phi(I)$ compact. Letting $\wt{h}_\delta(\wt{w})$ (resp.\ $\wt{h}_\delta^0(\wt{w})$) denote the average value of $\wt{h}$ (resp.\ $\wt{h}^0$) on $\partial B(\wt{w},\delta)$, we have that $\wt{h}_\delta(\wt{w}) = \wt{h}_\delta^0(\wt{w}) + \wt{\Fh}(\wt{w})$, since $\wt{\Fh}$ is harmonic. We shall use this identity to deduce the variance of $\wt{\Fh}(\wt{w})$, however, we must a bit careful, as $\wt{h}^0$ and $\wt{\Fh}$ are not independent when we have fixed the additive constant in this way. Indeed, if $h_u$ denotes the field without normalization, then $(h,\rho_1) = (h_u,\rho_1-\rho)$ and we let $\wt{h}_u = h_u \circ \phi^{-1}$.Then, $(\wt{h},\rho_1) = (\wt{h}_u,\rho_1-\wt{\rho})$ and the Markov decomposition of $\wt{h}_u$ is given by $\wt{h}^0 + \wt{\Fh}_u$ where $\wt{h}^0$ and $\wt{\Fh}_u$ are exactly independent and where the relationship between $\wt{\Fh}_u$ and $\wt{\Fh}$ is that $(\wt{\Fh},\rho_1) = (\wt{\Fh}_u,\rho_1-\wt{\rho})-(\wt{h}^0,\wt{\rho})$. In particular, $\cov((\wt{h}^0,\rho_1),(\wt{\Fh},\rho_2)) = -\int \greenD{\h}(x,y)\rho_1(dx)\wt{\rho}(dy)$. Thus,
\begin{align}
	\var[\wt{\Fh}(\wt{w})] &= \var[ \wt{h}_\delta(\wt{w})] - \var[ \wt{h}_\delta^0(\wt{w})] -2\cov(\wt{h}_\delta^0(\wt{w}),\wt{\Fh}(\wt{w})) \nonumber \\
	&= \var[ \wt{h}_\delta(\wt{w})] - \var[ \wt{h}_\delta^0(\wt{w})] + 2 \int \greenD{\h}(x,y)\rho_{\wt{w},\delta}(dx)\wt{\rho}(dy), \label{eq:variance}
\end{align}
where $\rho_{\wt{w},\delta}$ denotes the uniform measure on $\partial B(\wt{w},\delta)$. Since $\sup_{\wt{w} \in K} \left| \int \greenN{\h}(\wt{w},z) \wt{\rho}(dz) \right| < \infty$ for all compact $K \subseteq \h \cup \phi(I)$, we have that the third term is bounded by a constant $C > 0$, uniformly in $\wt{w} \in \h$ with $\im(\wt{w}) \geq (4M)^{-1} (2b-c)\epsilon$ and all $\delta \in (0,(8M)^{-1} (2b-c)\epsilon)$. Moreover, $\var[\wt{h}_\delta^0(\wt{w})] = - \log \delta + \log \confrad(\wt{w},\h) \geq -\log \delta + \log \im(\wt{w})$. Finally, arguing as in Lemma~\ref{lem:RNint}, we get that $\var[ \wt{h}_\delta(\wt{w})] = -\log \delta - \log \im(\wt{w}) + O(1)$ for all $\delta$ small enough, where the constants in the $O(1)$-term depend only on $D$, $U$ and $\rho$. Consequently, plugging this into~\eqref{eq:variance}, we have that
\begin{align*}
	\var[\wt{\Fh}(\wt{w})] = -2\log \im(\wt{w}) + O(1) \leq -2 \log \epsilon + O(1),
\end{align*}
where the constant in $O(1)$ depends only on $b$, $c$, $D$, $U$ and $\rho$.

Now, instead assume that $h$ is normalized so that $\Fh(z_0) = 0$, or equivalently, $\wt{h}$ is normalized so that $\wt{\Fh}(i) = 0$. By conformal invariance, we have that $\wt \Fh \overset{d}{=} \Fh^{\D} \circ F$ where $\Fh^{\D}$ is the harmonic part of a free boundary GFF in $\D$ normalized so that $\Fh^{\D}(0) = 0$ and $F(z) = -(z-i)/(z+i)$ is a conformal map taking $\h$ to $\D$ with $F(i) = 0$. By \cite[Lemma~2.9]{gms2018multifractal}, we have that
\begin{align*}
	\var[\Fh^{\D}(z)] = -2\log(1-|z|^2).
\end{align*}
Thus
\begin{align*}
	\var[\wt \Fh(\wt{w})] = \var[\Fh^{\D}(F(\wt{w}))] = -2 \log \frac{4 \im(\wt{w})}{|\wt{w}+i|^2} \leq -2 \log \epsilon + O(1),
\end{align*}
uniformly in $\epsilon > 0$ and $\wt{w} \geq (4M)^{-1} (2b-c)\epsilon$.

We now turn to deducing the result. Let $z \in U$ be such that $\dist(z,\partial U ) \geq 2 b \epsilon$. Since $\Fh$ is harmonic on $B(z,2b\epsilon)$, we have for each $w \in B(z,\epsilon)$ that
\begin{align*}
	\Fh(w) = \frac{1}{2\pi} \int_0^{2\pi} \Fh(z+c\epsilon e^{i\theta}) \frac{c^2\epsilon^2 - |w-z|^2}{|w-(z+c\epsilon e^{i\theta})|^2} d\theta \leq \frac{c^2}{2\pi (c-1)^2} \int_0^{2\pi} |\Fh(z + c\epsilon e^{i\theta})| d\theta.
\end{align*}
Then, by Jensen's inequality and the fact that for $\wh w \in U$ with $\dist(\wh{w},\partial U) \geq (2b-c)\epsilon$, $\var[\Fh(\wh w)] \leq -2 \log \epsilon + O(1)$, uniformly in $\epsilon> 0$, we have that
\begin{align*}
	\E\!\left[ \exp\!\left( a \sup_{w \in B(z,\epsilon)} \Fh(w) \right) \right] \leq \frac{1}{2\pi} \int_0^{2\pi} \E\!\left[ \exp\!\left( \frac{a c^2}{(c-1)^2}\Fh(z+c\epsilon e^{i\theta}) \right) \right] d\theta \leq \wt{C} \epsilon^{-\frac{a^2 c^4}{(c-1)^4}},
\end{align*}
as was to be proven.

\end{proof}

\begin{remark}
We note that the assumptions in the statement of Lemma~\ref{lem:bound_harmonic_part} are satisfied when $D = \h$ (resp.\ $D=\strip$) and $\rho$ is the uniform probability measure on $\h \cap \partial B(0,R)$ (resp.\ $\{0\} \times [0,\pi]$) for $R>0$ fixed.
\end{remark}

\begin{lemma}\label{lem:fbpmom}
Let $h$ be a free boundary GFF on $\strip$ with additive constant fixed so that the average of $h$ on the line $\{0\}\times (0,\pi)$ is equal to zero. Fix $\gamma \in (0,2]$. Then $\E\!\left[\mu_{h}(A_{0})^p \right] < \infty$ for all $p \in (-\infty,2/ \gamma^2)$ and $\E\!\left[\nu_{h}([-1,0] \times \{0,\pi\}))^p\right] < \infty$ for all $p \in (-\infty,4/ \gamma^2)$.
\end{lemma}
\begin{proof}
We begin by proving the analogous statement for a free boundary GFF $h'$ on $\h$ with zero average on $\h \cap \partial \D$ and with $A_0$ replaced by $\D_+ = \h \cap \D$. 

Let $\wt{h}$ be a free boundary GFF on $\h$ and write $\wt{h} = h^0 + \Fh$, where $h^0$ is a zero-boundary GFF on $\h$ and $\Fh$ is harmonic.  Assume that $\wt{h}$ is normalized so that $\Fh(i) = 0$. We consider the conformal transformation $F : \h \rightarrow \D$ with $F(z) = \frac{z-i}{z+i}$ and set $\wh h = \wt{h} \circ F^{-1} + Q\log(|(F^{-1})'|)$. Then $\wt{h} \circ F^{-1}$ is a free boundary GFF on $\D$ having zero average on $\partial \D$ since $\Fh \circ F^{-1}$ is the harmonic part of $\wt{h} \circ F^{-1}$ and $\Fh \circ F^{-1} (0) = 0$. By \cite[Theorem~1]{h2023moments} ($\gamma < 2$) and \cite[Theorem~4.2]{hrv2018lqgdisk} ($\gamma = 2$), we have that 
\begin{align}\label{eq:momentbound}
\E\!\left[ \mu_{\wt{h} \circ F^{-1}}(\D)^p \right] < \infty
\end{align}
for all $p \in (0,2/\gamma^2)$ (in fact, for all $p \in (-\infty,1)$ if $\gamma = 2$). By~\eqref{eq:LQG_coordinate_change}, we have that $\E[ \mu_{\wt h}(\D_+)^p] = \E[ \mu_{\wh h}(F(\D_+))^p ]$, so~\eqref{eq:momentbound} together with the fact that $|(F^{-1})'|$ is bounded on $F(\D_+)$, implies that
\begin{align*}
\E\left[ \mu_{\wt h}(\D_+)^p \right] < \infty,
\end{align*}
for all $p \in (0,2/\gamma^2)$. Let $\rho_{0,1}^+$ denote the uniform measure on $\h \cap \partial \D$. Since $h' = \wt h - (\wt h, \rho_{0,1}^+)$ is a free boundary GFF on $\h$ with zero average on $\h \cap \D_+$ and $(\wt h, \rho_{0,1}^+)$ is a zero mean Gaussian with finite variance, we have by H\"{o}lder's inequality that
\begin{align*}
\E\left[ \mu_{h'}(\D_+)^p \right] < \infty
\end{align*}
for all $p \in (0,2/\gamma^2)$. Then the claim follows by conformally mapping to $\strip$ and applying~\eqref{eq:LQG_coordinate_change}.

We now turn to the negative moments. Again, let $\wt h = h^0 + \Fh$ and consider the set $S = [-1,0] \times [1,2]$. By \cite[Proposition~3.6]{rv2010gmcrevisit} (and the comments in Remark~\ref{rmk:different_fields}), $\E[\mu_{h^0}(S)^{-p}]$ is finite for all $p>0$. Moreover, by the symmetry of $\Fh$ and Lemma~\ref{lem:bound_harmonic_part}, we have that
\begin{align*}
	\E[ \mu_{\wt h}(S)^{-p}] \leq \E[ e^{-\gamma p \inf_{z \in S} \Fh(z)}] \E[\mu_{h^0}(S)^{-p}] < \infty.
\end{align*}
Since $\mu_{\wt h}(A_0) \geq \mu_{\wt h}(S)$ a.s., we have that $\E[\mu_{\wt h}(A_0)^{-p}]$ is finite for all $p>0$.

Finally, by applying \cite[Proposition~3.5]{rv2010gmcrevisit} (keeping Remark~\ref{rmk:different_fields} in mind) instead of \cite[Theorem~1]{h2023moments} when $\gamma \in (0,2)$ and \cite[Theorem~4.1]{hrv2018lqgdisk} when $\gamma = 2$, we obtain in the same way that
\begin{align*}
\E\left[\nu_{h}([-1,0] \times \{0,\pi\})^p\right] < \infty
\end{align*}
for all $p \in (-\infty,4/ \gamma^2)$.
\end{proof}

With this lemma at hand, it is easy to see that the mass of each rectangle $A_k$, with respect to the lateral part of the field, has finite moments of orders $p \in (-\infty,2/\gamma^2)$ as well.
\begin{lemma}\label{lem:fblateral}
Fix $\gamma \in (0,2]$, let $h$ be as in Lemma~\ref{lem:fbpmom} and let $h_2$ be the projection of $h$ onto $H_2(\strip)$ (recall Remark~\ref{rmk:stripdecomposition}). Then it holds that $\E[\mu_{h_2}(A_k)^p] < \infty$ for all $k \in \Z$ and $p \in (-\infty,2/\gamma^2)$ and $\E[\nu_{h_2}([k-1,k] \times \{0,\pi\})^p] < \infty$ for all $k \in \Z$ and $p \in (-\infty,4/\gamma^2)$ .
\end{lemma}
\begin{proof}
We note that it suffices to prove the claim for $k = 0$ since the law of $h_2$ is invariant under horizontal translations, so that $\mu_{h_2}(A_0) \overset{d}{=}\mu_{h_2}(A_k)$ for all $k \in \Z$. We also know that $h_1(t) = B_{2t}$ for all $t \in \R$, where $(B_{t})_{t \in \R}$ is a two-sided Brownian motion with $B_0 = 0$.

Let $p' > 1$ be such that $pp' \in (0,\frac{2}{\gamma^2})$, let $q' > 1$ be such that $\frac{1}{p'} + \frac{1}{q'} = 1$ and set $h_1^* = \inf_{x \in [-1,0]} h_1(x)$. Then $h_2 \leq h - h_1^*$ and thus H\"{o}lder's inequality implies that 
\begin{align*}
	\E\!\left[\mu_{h_2}(A_0)^p \right] \leq \E\!\left[\mu_h(A_0)^{pp'}\right]^{\frac{1}{p'}} \E\!\left[\exp(-\gamma q' h_1^*)\right]^{\frac{1}{q'}}.
\end{align*}
Lemma~\ref{lem:fbpmom} implies that $\E[\mu_{h}(A_0)^{pp'}] < \infty$ and standard facts about one-dimensional Brownian motion imply that $\E[\exp(-\gamma q' h_1^*)] < \infty$. Note that a similar method gives the part of the claim involving quantum lengths and negative moments and hence the proof is complete.
\end{proof}

We now turn to the proof of Lemma~\ref{lem:pthmoment}.
\begin{proof}[Proof of Lemma~\ref{lem:pthmoment}]
We write $h$ as $h = h_1 + h_2$ where $h_j$, $j=1,2$, is the projection of $h$ on $H_j(\strip)$ (recall Remark~\ref{rmk:stripdecomposition}) and let $X_{\re(z)} = h_1(z)$. Then $(-X_{-t/2})_{t \geq 0}$ has the law of a $\BES^3$ with $X_0 = 0$ and $(X_t)_{t \geq 0}$ has the same law as $(B_{2t})_{t \geq 0}$ where $B$ is a Brownian motion with $B_0 = 0$. Note that for $k \leq 0$ we have that 
\begin{align*}
	\mu_h(A_k) \leq \mu_{h_2}(A_k) \exp \!\Bigg(\gamma p \sup_{t \in [k - 1,k]} X_t \!\Bigg),
\end{align*}
and so
\begin{align*}
\E[\mu_h(A_k)^p] \leq \E[\mu_{h_2}(A_0)^p] \E\!\left[\exp \!\Bigg(\gamma p \sup_{t \in [k - 1,k]} X_t \!\Bigg)\right]\!.
\end{align*}
By the Markov property of Bessel processes, it holds that for each $k \leq 0$, $\sup_{t \in [k - 1,k]} (X_t - X_k)$ is stochastically dominated by $\sup_{t \in [0,2]} \wt{B}_t$, where $\wt{B}_t$ is a Brownian motion with $\wt{B}_0 = 0$. This implies that 
\begin{align*}
\E\!\left[\exp \!\Bigg(\gamma p \sup_{t \in [k - 1,k]}X_t \!\Bigg)\right] \leq \E\!\left[\exp \!\Bigg(\gamma p \sup_{t \in [0,2]} \wt{B}_t \!\Bigg)\right]\! \E[\exp (\gamma p X_k)].
\end{align*}
Note that 
\begin{align}\label{eq:brownian_supremum_exponential_moments}
	\E\!\left[\exp \!\left(\gamma p \sup_{t \in [0,2]} 
\wt{B}_t \right)\right] < \infty,
\end{align}
and by the proof of Lemma~\ref{lem:bes_rn_derivative} we can find a positive constant $C_{p,\gamma} $ depending only on $p$ and $\gamma$ such that 
\begin{align*}
	\E[\exp (\gamma p X_k)] \leq C_{p,\gamma} (-2k)^{-\frac{3}{2}},
\end{align*}
for all $k \leq 0$. 
If $p \in (0,1] \cap (0,\frac{2}{\gamma^2})$, then it follows by Lemma~\ref{lem:fblateral} that
\begin{align*}
\E[\mu_h(\strip_-)^p] \leq \sum_{k \leq 0} \E[\mu_h(A_k)^p] < \infty.
\end{align*}
Now suppose that $p \in (1,\frac{3}{2}) \cap (0,\frac{2}{\gamma^2})$ and for $k \leq 0$, we set
\begin{align*}
	X_k^* = \sup_{t \in [k-1,k]} X_t, \quad
	p_k = \frac{e^{\gamma X_k^*}}{\sum_{n \leq 0} e^{\gamma X_n^*}}.
\end{align*}
We note that $\E\!\left[ \sum_{k \leq 0} e^{\gamma X_k^*} \right] < \infty$, so that a.s., $\sum_{k \leq 0} e^{\gamma X_k^*} < \infty$ and 
\begin{align*}
    \mu_h(\strip_-) \leq \!\left( \sum_{j \leq 0} e^{\gamma X_j^*}\right) \sum_{k \leq 0} p_k \mu_{h_2}(A_k).
\end{align*}
Taking expectation we have by the independence of $X$ and $h_2$ that
\begin{align}\label{eq:stripbound1}
	\E[\mu_h(\strip_-)^p] \leq \sum_{k \leq 0} \E[\mu_{h_2}(A_k)^p] \E\!\Bigg[ p_k \!\Big( \sum_{j \leq 0} e^{\gamma X_j^*}\!\Big)^p \!\Bigg] = \E[\mu_{h_2}(A_0)^p] \E\!\Bigg[ \!\Big(\sum_{j \leq 0} e^{\gamma X_j^*}\!\Big)^p\!\Bigg].
\end{align}
Fix $\alpha > 1$ such that $\frac{3}{2} - \alpha (p-1) > 1$ and set $N = \sum_{k \in \mathbb{Z}_{-}}(-k)^{-\alpha} < \infty$. Then by Jensen's inequality
\begin{align}\label{eq:stripbound2}
	\E \!\Bigg[ \!\Big( \sum_{j \leq -1} e^{\gamma X_j^*} \!\Big)^p \!\Bigg] &= N^p \E\!\left[ \!\Bigg( \sum_{k \leq -1} e^{\gamma X_k^*} (-k)^{\alpha}\frac{(-k)^{-\alpha}}{N} \!\Bigg)^p\right] \\
	&\leq N^{p-1} \sum_{k \leq  -1} \E\!\left[ e^{\gamma p X_k^*}\right] (-k)^{\alpha(p-1)} \nonumber \\
	&\leq C_{p,\gamma}N^{p-1}\sum_{k \geq  1} k^{-\frac{3}{2}+\alpha(p-1)} < \infty. \nonumber
\end{align}
Thus, by~\eqref{eq:stripbound1},~\eqref{eq:stripbound2} and~\eqref{eq:brownian_supremum_exponential_moments}, we have that $\E[ \mu_h(\strip_-)^p]$ is finite when $p \in (1,\frac{3}{2}) \cap (0,\frac{2}{\gamma^2})$.

Next, we handle the negative moments. For $p>0$ we have by Lemma~\ref{lem:fblateral} that
\begin{align*}
	\E[\mu_h(\strip_-)^{-p}] \leq \E[\mu_h(A_0)^{-p}] \leq \E\!\left[ \exp\!\left(-\gamma p \inf_{-1 \leq t \leq 0} X_t \right) \right] \E[\mu_{h_2}(A_0)^{-p}] < \infty,
\end{align*}
since the supremum of a $\BES^3$ on the interval $[0,2]$ has finite exponential moments.

Finally, we turn to proving~\eqref{eq:pthmomenteps}. We note that, conditionally on $u_{\alpha,\epsilon}$, the law of $(X_{-t})_{t \geq -u_{\alpha,\epsilon}}$ is that of $(\wt{X}_{-t}+\alpha \log \epsilon)_{t \geq 0}$, where $\wt{X}$ is an independent copy of $X$. Consequently, by the translation invariance of $h_2$, we have that for $z \in \strip_-$, conditionally on $u_{\alpha,\epsilon}$, 
\begin{align*}
X_{\re(z) + u_{\alpha,\epsilon}} + h_2(z+u_{\alpha,\epsilon}) = \wt{X}_{\re(z)} + \alpha \log \epsilon + h_2(z+u_{\alpha,\epsilon}) \overset{d}{=} h(z) + \alpha \log \epsilon,
\end{align*}
and thus
\begin{align*}
	\E[ \mu_h(\strip_- +u_{\alpha,\epsilon})^p] = \E[ \mu_h(\strip_-)^p] e^{\gamma p \alpha \log \epsilon} = c_p \epsilon^{\alpha p \gamma}.
\end{align*}
The bounds for the boundary measure are proven similarly.
\end{proof}
In proving Lemma~\ref{lem:main_lemma_ubd} it is important that the image of a ball close to the curves does not get squeezed too close to the boundary, making it hard for a Brownian motion starting from it, to travel far. For this very reason, we need to prove that the mass of a set of points close to the boundary is small. In order to do this, we need to bound moments of the mass close to the boundary, for which it will be convenient to consider a Whitney square decomposition and bound the quantum area of each square. The first we need is to be able to bound the moments of mass of balls of radius $\epsilon$, at distance of order $\epsilon$ from the boundary, which will give us a way to bound the quantum area of Whitney squares. This is achieved by combining Lemma~\ref{lem:pthmom0} below, which handles the zero-boundary GFF part and Lemma~\ref{lem:bound_harmonic_part} which takes care of the harmonic part of a free boundary GFF.
\begin{lemma}\label{lem:pthmom0}
Let $D$ be a simply connected domain and let $h$ be a zero-boundary GFF on $D$. Fix $\gamma \in (0,2)$ and $p \in (0,\frac{4}{\gamma^{2}})$. Then there exists a finite constant $C_{p,\gamma}$, depending only on $p$ and $\gamma$, such that
\begin{align*}
	\E\!\left[\mu_{h}(B(z,\epsilon))^p \right] \leq C_{p,\gamma} \epsilon^{(2 + \gamma^{2}/2)p}
\end{align*}
for all $\epsilon > 0$ and $z \in D$ such that $2\epsilon \leq \dist(z,\partial D) \leq 9\epsilon$.
\end{lemma}
\begin{proof}
By translation invariance, we can assume that $z = 0$. Fix $\beta_{0},\beta_{1}$ such that $3+2\sqrt{2} < \beta_{0} < \beta_{1}$ and set $r_{0} = \tfrac{1}{\beta_{0}} \in (0,1)$. Suppose first that $\beta_{0}\epsilon \leq \dist(0,\partial D) \leq \beta_{1}\epsilon$ and let $\phi :  D \rightarrow \h$ be a conformal transformation which maps $0$ to $i$. Then, \cite[Corollary~3.25]{law2005slebook} implies that $\phi(B(0,\epsilon)) \subseteq B(i,r_1)$ where $r_1 = \tfrac{4 r_0}{(1 - r_0)^2}$. The coordinate change formula implies that $\mu_{h}(B(0,\epsilon)) = \mu_{\wt{h}}(\phi(B(0,\epsilon)))$ with $\wt{h} = h \circ \phi^{-1} + Q\log(|(\phi^{-1})'|)$ and $Q = \tfrac{2}{\gamma} + \tfrac{\gamma}{2}$. 

We note that each $w \in \phi(B(0,\epsilon))$ satisfies $\dist(w,\partial \h) \geq 1-r_1$ and $\dist(\phi^{-1}(w),\partial D) \leq (\beta_1+1)\epsilon$ and hence the Koebe-$1/4$ theorem implies that $|(\phi^{-1})'(w)| \leq 4(\beta_1+1)\epsilon/(1-r_1)$. Consequently,
\begin{align*}
	\mu_{h}(B(0,\epsilon)) \leq \mu_{h \circ \phi^{-1}}(B(i,r_1))\epsilon^{\gamma Q}\!\left(\frac{4(\beta_1+1)}{1-r_1}\right)^{\gamma Q}.
\end{align*}
(It is here the bound $\beta_0 > 3 + 2\sqrt{2}$ is used, as this guarantees that $4(\beta_1+1)/(1-r_1)$ is positive.) Since $h \circ \phi^{-1}$ is a zero-boundary GFF on $\h$, \cite[Proposition~3.5]{rv2010gmcrevisit} (see Remark~\ref{rmk:different_fields}) implies that $\E\!\left[\mu_{h \circ \phi^{-1}}(B(i,r_1))^p\right]$ is finite.
Thus there exists a positive constant $C_{p,\gamma}^1$, depending only on $\gamma$ and $p$, such that 
\begin{align}\label{eq:0pmom}
	\E\!\left[\mu_{h}(B(0,\epsilon))^p\right] \leq C^{1}_{p,\gamma} \epsilon^{\gamma Q p}
\end{align}
Now suppose that $2\epsilon \leq \dist(0,\partial{D}) \leq 9\epsilon$ and choose $\beta_{0},\beta_{1}$ such that $2\beta_1/\beta_0 > 9$, and consider the conformal transformation $\psi(z) = \frac{\beta_{0}}{2}z$. Then we have that $\beta_{0}\epsilon \leq \dist(0,\partial \psi(D)) \leq \beta_{1}\epsilon$ and the result follows from the coordinate change formula together with~\eqref{eq:0pmom}.
\end{proof}

We are now ready to bound the moments of the quantum area of a set close to the boundary.

\begin{lemma}
\label{lem:mass_close_to_boundary}
Fix $\gamma \in (0,2)$ and let $(\strip,h,-\infty,+\infty) \sim\qwedgeW{\gamma}{\gamma^2 / 2}$ have the first exit parameterization. Fix $p \in (0,1 \wedge \frac{2}{\gamma^2})$. If $c>0$, then for $s>1$,
\begin{align}\label{eq:mass_close_to_boundary}
\E[\mu_h(\{z \in \strip_-: \re(z) \geq -s, \dist(z,\partial \strip) < \epsilon^c \})^p] \leq C_{p,\gamma} s\epsilon^{M c},
\end{align}
where $M = M(p,\gamma) > 0$.
\end{lemma}
\begin{proof}
Let $N$ be such that $\epsilon^c \in [2^{-N-1},2^{-N}]$ and for $r>0$, let $A_k(r) = \{ z \in \strip : \re(z) \in [-k,-k+1],  \dist(z,\partial \strip) < r \}$.  Moreover, we let $h^f$ be a free boundary GFF in $\strip$ and consider $A_1(\epsilon^c)$. Note that if $\psi(z) = e^z$ (so that $\psi: \strip \rightarrow \h$ is conformal), then for small $\epsilon$, $\psi(A_1(\epsilon^c)) \subseteq R_\epsilon \coloneqq ([-1,-e^{-1}/2] \cup [e^{-1}/2,1]) \times [0,2\epsilon^c]$.  In order to prove~\eqref{eq:mass_close_to_boundary}, we shall bound the quantum mass of $R_\epsilon$ with respect to the field $\wt{h}^f$ on $\h$, related to $h^f$ as in~\eqref{eq:LQG_coordinate_change}, with $\psi$.

Let $(S_j)$ be a Whitney square decomposition of a large domain containing $R_\epsilon$, say, $R = [-100,100]\times[0,100]$. Then, letting $D_k$ be the collection of squares of side length $2^{-k}$, intersecting $R_\epsilon$, we have that $|D_k| \asymp 2^k$, with implicit constant independent of $k$, $c$ and $\epsilon$. 
Denote by $z_j$ and $2^{-n_j}$ the center and side length of $S_j$, respectively, and let $\epsilon_j = 2^{-(n_j+\frac{1}{2})}$. Then $S_j \subseteq B(z_j,\epsilon_j)$ and $2 \epsilon_j \leq \dist(z_j,\partial R) \leq 9\epsilon_j$.

By the domain Markov property we can write $\wt{h}^f = h_R^0 + \Fh_R$, where $h_R^0$ is a zero-boundary GFF on $R$, $\Fh_R$ is a distribution on $\h$ which is harmonic on $R$ and independent of $h_R^0$. Lemma~\ref{lem:pthmom0} implies that there exists a positive constant $C_{p_1,\gamma}$ depending only on $p_1$ and $\gamma$ such that 
\begin{align*}
	\E\!\left[\mu_{h_R^0}(B(z_j,\epsilon_j))^{p_1}\right] \leq C_{p_1,\gamma}\epsilon_j^{(2 + \frac{\gamma^{2}}{2})p_1}
\end{align*}
for each $j \in \Z$, where $p_1 \in (\frac{1}{2},\frac{2}{\gamma^2})$.

Let $f(t) = (2 + \frac{\gamma^2}{2})t - t^2 \gamma^2$ and note that $f(t) > 1$ when $t \in (\frac{1}{2},\frac{2}{\gamma^{2}})$. Hence, we can find $\beta > 1$ (sufficiently close to $1$) such that $(2 + \frac{\gamma^2}{2})p_1 - \beta^2 p_1^2 \gamma^2 > 1$. We let $\alpha \in (0,1)$ be such $\alpha^{-1} \in \N$ and such that $(2 + \frac{\gamma^2}{2})p_1 - \beta^2 p_1^2 \gamma^2 > 1$ is satisfied with $\beta = (1 - \alpha)^{-2}$. We divide $S_j$ into squares $S_{j,1},\dots,S_{j,k_j}$ having pairwise disjoint interiors and sides parallel to the axes, and each one having side length $r_j = \alpha \sqrt{2} \epsilon_j$. Then $k_j = 1/\alpha^2$. Let $z_{j,n}$ be the center of $S_{j,n}$ for $n \in \{1,\dots,k_j\}$. Since $S_{j,n} \subseteq B(z_{j,n},\alpha\epsilon_j)$ and $B(z_{j,n},(2+\alpha/\sqrt{2})\epsilon_j) \subseteq R$, we have that by Lemma~\ref{lem:bound_harmonic_part} with $c = 1/\alpha$ that
\begin{align*}
\E \!\left[ \exp\!\left(p_1\gamma \sup_{z \in S_j} \Fh_R(z) \right)\right] &\leq \sum_{n=1}^{k_j} \E\!\left[\exp\!\left(p_1\gamma \sup_{z \in S_{j,n}} \Fh_R(z) \right)\right] \leq \wt{C} k_j \epsilon_j^{-p_1^2 \gamma^2 \beta^2} \leq \frac{C}{\alpha^2}\epsilon_j^{-p_1^2\gamma^2 \beta^2}
\end{align*}
for some universal constant $C$. The independence of $h_R^0$ and $\Fh_R$ implies that 
\begin{align}\label{eq:pmomsq}
	\E[\mu_{\wt{h}^f}(S_j)^{p_1}] \leq \E[\mu_{h_R^0}(S_j)^{p_1}]\E\!\left[\exp\!\left(p_1\gamma \sup_{z \in S_j}\Fh_R(z)\right)\right] \leq \frac{C_{p_1,\gamma}}{\alpha^2} \epsilon_j^{(2 + \frac{\gamma^2}{2})p_1 - \beta^2 p_1^2 \gamma^2}.
\end{align}
Moreover, by~\eqref{eq:pmomsq}, if $p_1 < 1$, $S \in D_k$, then $\E[\mu_{\wt{h}^f}(S)^{p_1}] \leq C_{p_1,\gamma} 2^{-k((2+\frac{\gamma^2}{2})p_1-\beta^2 p_1^2 \gamma^2)}$. Thus, letting $D = \cup_k D_k$, we have
\begin{align*}
	&\E[ \mu_{\wt{h}^f}(R_\epsilon)^{p_1}] \leq \sum_{S \in D} \E[ \mu_{\wt{h}^f}(S)^{p_1}] = \sum_{k \geq N} \sum_{S \in D_k} \E[ \mu_{\wt{h}^f}(S)^{p_1}] \lesssim \sum_{k \geq N} 2^{-k((2+\frac{\gamma^2}{2})p_1-\beta^2 p_1^2 \gamma^2-1)} \\
	&= 2^{-N((2+\frac{\gamma^2}{2})p_1-\beta^2 p_1^2 \gamma^2-1)} \sum_{k \geq 0} 2^{-k((2+\frac{\gamma^2}{2})p_1-\beta^2 p_1^2 \gamma^2-1)} \lesssim 2^{-N((2+\frac{\gamma^2}{2})p_1-\beta^2 p_1^2 \gamma^2-1)},
\end{align*}
since $(2+\frac{\gamma^2}{2})p_1-\beta^2 p_1^2 \gamma^2-1 > 0$, and hence the sum on the second line is finite. It is clear that the implicit constant depends only on $p_1$ and $\gamma$.  Thus, 
\begin{align*}
	\E[\mu_{h^f}(A_1(\epsilon^c))^{p_1}] \leq \E[ \mu_{\wt{h}^f}(R_\epsilon)^{p_1}] \lesssim \epsilon^{c((2+\frac{\gamma^2}{2})p_1-\beta^2 p_1^2 \gamma^2-1)}.
\end{align*}
Choosing $p_* > 1$ such that $p p_* \in (1/2,1) \cap (0,\frac{2}{\gamma^2})$, we have by the same argument as in Lemma~\ref{lem:fbpmom} that
\begin{align*}
	\E[\mu_{h_2}(A_1(\epsilon^c))^p] \lesssim \E[\mu_{h^f}(A_1(\epsilon^c))^{p p_*} ]^{1/p_*} \lesssim \epsilon^{c ((2+\frac{\gamma^2}{2})p p_*-\beta^2 p^2 p_*^2 \gamma^2-1)/p_*} = \epsilon^{M c}
\end{align*}
where $M = (2+\frac{\gamma^2}{2})p-\beta^2 p^2 p_* \gamma^2-1/p_*$, and the implicit constants are independent of $\epsilon$ and $c$. Consequently,
\begin{align*}
\E[ \mu_h(\{ z \in \strip_-: \re(z) \geq-s, \dist(z,\partial \strip) < \epsilon^c\})^p] &\leq \E[ \mu_h(\{ z \in \strip_-: \re(z) \geq -\lceil s \rceil, \dist(z,\partial \strip) < \epsilon^c\})^p] \\
&\leq \sum_{k=1}^{\lceil s \rceil} \E[ \mu_h(A_k(\epsilon^c))^p] \leq \sum_{k=1}^{\lceil s \rceil} \E[ \mu_{h_2}(A_k(\epsilon^c))^p] \lesssim s \epsilon^{M c},
\end{align*}
where the implicit constant depends only on $p$ and $\gamma$. Thus, the proof is done.
\end{proof}

We end this section by stating and proving an analogue of Lemma~\ref{lem:pthmoment} for quantum cones.

\begin{lemma}\label{lem:pthmoment_cone}
Suppose that  $(\cyl,h,-\infty,+\infty) \sim \qconeW{2}{4}$ has the first exit parameterization and let $\cyl_- = \{ z \in \cyl: \re(z) \leq 0 \}$. Then for every $p \in (0,1)$,
\begin{align*}
\E[ \mu_h(\cyl_-)^p ] < \infty.
\end{align*}
\end{lemma}
\begin{proof}
We follow the notation and the method of the proof of Lemma~\ref{lem:pthmoment}. Let $(X_t)$ be the average of $h$ on the vertical line $\{t+iy:y \in [0,2\pi]\}$ and set $h_2 = h - X_{\re(\cdot)}$. Then $(X_t)$ can be sampled by setting $X_t = Y_{-t}$ for $t \leq 0$ and $X_t = \wh{B}_t + t$ for $t \geq 0$, where $Y_t = B_t - t$ conditioned so that $Y_t < 0$ for all $t > 0$ and $B,\wh{B}$ are independent standard Brownian motions with $B_0 = \wh{B}_0 = 0$. Moreover, $X$ and $h_2$ are independent. Note that $(X_{-t})_{t \geq 0}$ can be sampled by setting $X_{-t} = \wt{B}_{t+T}-(t+T)$ for $t \geq 0$, where $\wt{B}_t$ is a standard Brownian motion, independent of $\wt{B}$, and $T = \sup\{t \geq 0 : \wt{B}_t - t =0\}$. We let $\wh{h}$ be the random distribution on $\cyl$ given by $\wh{h} = \wt{B}_{\re(\cdot)+T} - (\re(\cdot)+T) + h_2$. Then $\mu_h(\cyl_-) \overset{d}{=} \mu_{\wh{h}}([T,+\infty) \times [0,2\pi])$ by the invariance of $h_2$ under horizontal translations. Therefore, writing $\cyl_+ = \{z \in \cyl: \re(z) \geq 0 \}$, we have that
\begin{align*}
\E[ \mu_h(\cyl_-)^p ] &\leq \E[ \mu_{\wh{h}}(\cyl_+)^p ] \leq \sum_{k=1}^{\infty}\E[ \mu_{\wh{h}}(A_k)^p ] \leq C_p \sum_{k=1}^{\infty}\E\left[ \exp\left(\gamma p \sup_{k-1\leq t \leq k} \wt{B}_t - t\right)\right],
\end{align*}
where $C_p = \E[ \mu_{h_2}(A_1)^p]$, which is finite for $p < 1 = \frac{4}{\gamma^2}$.  Hence 
\begin{align*}
\E[ \mu_h(\cyl_-)^p ] \lesssim \sum_{k=1}^{\infty}e^{-\gamma p (1-\gamma p /2)k} < \infty
\end{align*}
for $p \in (0,1)$. This completes the proof.
\end{proof}

\subsection{Mass bounds}
In this section, we prove a few bounds on the quantum mass when the surface in question has the circle-average embedding. These are used in Sections~\ref{sec:exit_times} and \ref{sec:main_lemma} to bound exit times for $\SLE$ parameterized by quantum mass, as well as the mass of sets with respect to quantum surfaces with different parameterizations.

\begin{lemma}\label{lem:LBmassball}
Fix $\gamma \in (0,2]$ and let $\CC = (\C,h,0,\infty) \sim \qconeW{\gamma}{W}$ have the circle-average embedding. Then there exists $p > 0$ such that for every fixed $M > 0$, we can find $C_{M,\gamma} > 0$ such that 
\begin{align*}
	\p[\mu_h(B(0,M)) \geq R] \leq C_{M,\gamma}R^{-p},\quad \text{for all} \quad R > 0.
\end{align*}
\end{lemma}
\begin{proof}
This proof is similar to that of Lemma~\ref{lem:LBcone}, so we will be rather brief. We recall that $h|_{\D} = \wt{h}|_{\D} - \alpha \log|\cdot|$, where $\wt{h}$ is a whole-plane GFF normalized so that $\wt{h}_1(0) = 0$ (where $\wt h_s(0)$ denotes the average value of $\wt h$ on $\partial B(0,s)$). Fix $M >0$, let $\varphi^{M}(z) = 2Mz$ and set $\wt{h}^M = \wt{h} \circ \varphi^M$. Then
\begin{align*}
	\mu_{\wt h}(B(0,M)) = \mu_{\wt{h}^M + Q\log(2M)}\!\left( B(0,1/2) \right) = (2M)^{2 + \frac{\gamma^2}{2}}\mu_{\wt{h}^M}\!\left (B(0,1/2) \right).
\end{align*}
By translation and scale invariance, $\wt{h}^M$ has the law of a whole-plane GFF modulo additive constant and the circle-average of $\wt h^{M}$ on $\partial \D$ is equal to $\wt h_{2M}(0)$. Hence the field $\wh{h}^M = \wt{h}^M - \wt{h}_{2M}(0)$ agrees in law with $\wt{h}$ and 
\begin{align}
	\mu_{\wt h}(B(0,M)) = e^{\gamma \wt h_{2M}(0)}(2M)^{2 + \frac{\gamma^{2}}{2}}\mu_{\wh h^M}\!\left (B(0,1/2) \right )
\end{align}
The Markov property implies that $\wh h^M = h^0 + \Fh$ where $h^0$ is a zero-boundary GFF on $\D$ and $\Fh$ harmonic on $\D$ and independent of $h^0$. Moreover, $\E[ \mu_{h^0}(B(0,1/2) )^p ]$ is finite for all $p \in (0,4/\gamma^2)$.
As in the proof of Lemma~\ref{lem:LBcone}, the probability that $\D$ is not $R$-good for $\wh{h}^M$ is at most $c_1 e^{-c_2 R^2}$, for some universal constants $c_1,c_2$. In particular, $\E[ \exp( \gamma p \sup_{z \in B(0,1/2)} |\Fh(z)|)]$ is finite for all $p \in (0,4/\gamma^2)$, and consequently so is $\E[ \mu_{\wh{h}^M}(B(0,1/2))^p]$. Moreover, since Brownian motion has finite exponential moments, $\E[e^{\gamma p \wt h_{2M}(0)}]$ is finite for all $p > 0$ and hence H\"{o}lder's inequality implies that 
\begin{align}
	\E \!\left [ \mu_{\wt{h}}(B(0,M))^p \right] < \infty,\quad \text{for all} \quad p \in (0,4/\gamma^2).
\end{align}
We write $A^M = B(0,M) \setminus B(0,1/M)$ and note that \cite[Lemma~3.10]{ghm2020kpz} implies that 
\begin{align}\label{eq:expmomsup}
	\E \!\Bigg[ \exp \!\left\{ q\sup_{z \in A^M} |h(z) - \wt{h}(z)| \right\} \!\Bigg] < \infty,
\end{align}
for all $q > 0$ sufficiently small.
Moreover, by H\"{o}lder's inequality and~\eqref{eq:expmomsup} and since 
\begin{align*}
	\mu_h \!\left( A^M \right) \leq \mu_{\wt h}\!\left (B(0,M)\right)\exp \!\Big\{ \gamma \sup_{z \in A^M}|h(z)-\wt{h}(z)| \!\Big\} 
\end{align*}
we have that
\begin{align*}
\E\!\left [ \mu_h(A^M)^p \right] < \infty
\end{align*}
for all $p \in (0,4 / \gamma^2)$ sufficiently small.
We now set $\beta_{p,\gamma} = (2 + \gamma^{2} /2)p - \gamma^{2} p^{2} /2$ and $A_k = B(0,M^{-1}2^{-k}) \setminus B(0,M^{-1}2^{-k-1})$ for $k \in \N_0$. Note that since $\a < Q$, there is a $p \in (0,1)$ such that $\beta_{p,\gamma} - \gamma \alpha p>0$ and hence
\begin{align*}
	\E \!\left[\mu_{\wt h}(A_k)^{p}\right] \leq C2^{-k \beta_{p,\gamma}},
\end{align*}
for some constant $C$ independent of $k$ by \cite[Lemma~5.2]{ghm2020kpz}. Therefore we obtain that $\E[\mu_h(A_k)^p] \leq C2^{-k (\beta_{p,\gamma}-\alpha \gamma p)}$ for some (possibly different) constant $C$, independent of $k$, and by summing over $k$ we obtain that $\E[\mu_h(B(0,1/M))^p]$ is finite. Combining this with the above, we have that $\E [\mu_h(B(0,M))^p]$ is finite and so the claim follows by Markov's inequality.
\end{proof}

The next result is instrumental in proving the main estimate needed for the proof of the regularity in the case of chordal $\SLE_8$.

\begin{lemma}\label{lem:QwedgeLB}
Fix $\gamma \in (0,2]$ and suppose that $\CW = (\h,h,0,\infty) \sim \qwedgeW{\gamma}{\gamma^2/2}$ has the first exit parameterization. Fix $R,\sigma > 0$. Then, there is a constant $a > 0$ such that for $\epsilon>0$,
\begin{align*}
	\p[\mu_h(B(z,\epsilon)) \geq \epsilon^a,  \text{for all} \ z \ \text{such that} \ B(z,4\epsilon) \subseteq B(0,R) \cap \h] = 1 - O(\epsilon^{\sigma/2}).
\end{align*}
\end{lemma}
\begin{proof}
Let $D_R = B(0,R) \cap \h$ and let $h_1$ (resp.\ $h_2$) be the projection of $h$ onto $H_1(\h)$ (resp.\ $H_2(\h)$). Then $\left( h_{e^{-s}}(0)\right)_{s \geq 0}$ has the law of $(-Z_{2s} - Qs)_{s \geq 0}$ and $\left( h_{e^s}(0)\right)_{s \geq 0}$ has the law of $(B_{2s}+Qs)_{s \geq 0}$ where $B$ is  a standard Brownian motion starting from zero and $Z$ is a $\BES^3$ with $Z_0 = 0$ and $B$ and $Z$ are independent. 

We begin by proving the result for a free boundary GFF in place of the quantum cone. Let $\wt h$ be a free boundary GFF on $\h$ with the additive constant fixed so that its average on $\h \cap \partial \D$ is equal to zero. We let $\wt h = h_{D_R}^0 + \Fh_{D_R}$ where $h_{D_R}^0$ is a zero-boundary GFF on $D_R$ and $\Fh_{D_R}$ is harmonic on $D_R$. Fix $z \in D_R,\epsilon \in (0,1)$ such that $B(z,2\epsilon) \subseteq D_R$. Then by applying the coordinate change formula together with \cite[Proposition~3.7]{rv2010gmcrevisit} (recall Remark~\ref{rmk:different_fields}) if $\gamma \in (0,2)$ and \cite[Corollary~14]{drsv2014criticalgmckpz} if $\gamma = 2$, we have that
\begin{align*}
\E\!\left[ \mu_{h_{D_R}^0}(B(z,\epsilon))^{-2}\right] = O(\epsilon^{-4-3\gamma^2}).
\end{align*}
Moreover, Lemma~\ref{lem:bound_harmonic_part} implies that 
\begin{align*}
\E\!\left[  \exp\!\left( 2\gamma \sup_{w \in B(z,\epsilon)}\Fh_{D_R}(w)\right) \right] = O(\epsilon^{-C\gamma^2}),
\end{align*}
for some $C>0$, where the implicit constants in both of the inequalities depend only on $R,\gamma$. Let $r = 1-1/\sqrt{2}$ and note that 
\begin{align*}
\mu_{\wt h}(B(z,r\epsilon)) \geq \mu_{h_{D_R}^0}(B(z,r\epsilon))\exp\left( \gamma \inf_{w \in B(z,r\epsilon)}\Fh_{D_R}(w) \right).
\end{align*}
By the Cauchy-Schwarz inequality,  we have that
\begin{align*}
&\E\left[\mu_{h_{D_R}^0}(B(z,r\epsilon))^{-1} \exp\left(-\gamma \inf_{w \in B(z,r\epsilon)} \Fh_{D_R}(w)\right)\right] \\
&\leq \E\left[\mu_{h_{D_R}^0}(B(z,r\epsilon))^{-2}\right]^{1/2} \E\left[\exp\left(2\gamma \sup_{w \in B(z,r\epsilon)} \Fh_{D_R}(w) \right)\right]^{1/2}\\
&= O(\epsilon^{-2-(3+C)\gamma^2/2}).
\end{align*}
Hence by the Markov inequality we obtain that for $b>0$
\begin{align*}
\p[\mu_{\wt h}(B(z,r\epsilon))\leq \epsilon^{b}] = O( \epsilon^{b - 2 - (C+3)\gamma^2/2}).
\end{align*}
Moreover, for all $w \in D_R$ such that $B(w,4\epsilon) \subseteq D_R$ we can find $z \in (\epsilon \Z)^2 \cap D_R$ such that $B(z,r\epsilon) \subseteq B(w,\epsilon)$ and $B(z,2\epsilon) \subseteq B(w,4\epsilon)$ and since $|(\epsilon \Z)^2 \cap D_R| \asymp \epsilon^{-2}$ we obtain that
\begin{align}\label{eq:QwedgeLB1}
	&\p[\mu_{\wt h}(B(w,\epsilon))\geq \epsilon^{b}, \text{for all} \, w \, \text{such that} \, B(w,4\epsilon) \subseteq D_R ] \nonumber \\
	&\geq \p[\mu_{\wt h}(B(z,r\epsilon))\geq \epsilon^{b} , \text{for all} \, z \in (\epsilon \Z)^2 \cap D_R \, \text{such that} \, B(z,2\epsilon) \subseteq D_R] \nonumber \\
	&\geq 1 - O(\epsilon^{b -4-(C+3)\gamma^2/2}) \geq 1 - O(\epsilon^{\sigma/2})
\end{align}
for $b$ sufficiently large. Note that all the implicit constants depend only on $R,\gamma$.  Thus, we have proven the result with a free boundary GFF in place of the quantum cone.

The rest of the proof consists of transferring the result from the free boundary GFF to the quantum cone. For now, we assume that $R=1$, that is, we consider points in $\D_+$. By \cite[Theorem~VI.3.3.5]{ry1999book}, a $\BES^3$ $Z$, can be coupled with a one-dimensional Brownian motion $B$, such that $Z_t = 2S_t - B_t$, where $S_t = \max_{0 \leq s \leq t} B_s$. Consequently, since $\wt{h}_{e^{-t}}(0) = B_{2t}$, where $B_t$ is a two-sided Brownian motion, we can couple $h$ and $\wt h$ such that $h_{e^t}(0) = \wt{h}_{e^t}(0) + Qt$ for $t \geq 0$ and $h_{e^{-t}}(0) = -2S_{2t} + \wt{h}_{e^{-t}}(0) - Qt$ (where the maximum $S$ is over the Brownian motion $\wt{h}_{e^{-t/2}}(0)$) and such that their lateral parts are the same. Thus we have that $h_{e^{-t}}(0) \geq \wt{h}_{e^{-t}}(0) + A_{\epsilon}$ for $t \in [0,\log(\epsilon^{-1})]$ where 
\begin{align*}
	A_{\epsilon} = -\sup_{0\leq t \leq \log(\epsilon^{-1})} 2S_{2t}-Qt,
\end{align*}
and so $\mu_h(B(z,\epsilon)) \geq e^{\gamma A_\epsilon} \mu_{\wt h}(B(z,\epsilon))$ where $z$ is as in the statement of the lemma. Note that
\begin{align*}
	\p\!\left[ e^{\gamma A_{\epsilon}}\leq \epsilon^{b} \right]&\leq \p\!\left[ \sup_{0\leq t \leq 2\log(\epsilon^{-1})} B_t \geq b\log(\epsilon^{-1}) / (2\gamma ) \right] \leq \p\!\left[ T_{\epsilon}\leq 2\log(\epsilon^{-1})\right]
\end{align*}
where $T_{\epsilon} = \inf\{t \geq 0 : B_t = b\log(\epsilon^{-1})/(2\gamma ) \}$. Note that the explicit form of the probability density function of $T_{\epsilon}$ given in the proof of Lemma~\ref{lem:main_lemma_ubd} implies that $\p[T_{\epsilon} \leq 2\log(\epsilon^{-1})] = O(\epsilon^b)$ for $b$ sufficiently large. Hence since 
\begin{align*}
\p[\mu_h(B(z,\epsilon)) \leq \epsilon^{2b}] \leq \p[e^{\gamma A_{\epsilon}}\leq \epsilon^{b}] + \p[\mu_{\wt h}(B(z,\epsilon))\leq \epsilon^{b}]
\end{align*}
by applying a method similar to the above, we obtain that
\eqref{eq:QwedgeLB1} still holds with $\wt h$ replaced by $h$.

Finally, we turn to the case $R>1$. It is easy to see that we can find a constant $C_R > 0$ depending only on $R $ such that 
\begin{align*}
\p\!\left[ \sup_{z \in D_R \setminus D_{1/R}}|(h-\wt{h})(z)| > M\right] \leq e^{-C_R M} \quad \text{for all} \quad M >0,
\end{align*}
and so 
\begin{align*}
\p\!\left[ \sup_{z \in D_R \setminus D_{1/R}}\left|\frac{1}{\gamma}\log\left( \frac{d\mu_{\wt h}}{d\mu_h}(z)\right) \right|  > M \right] \leq e^{-C_R M} \quad \text{for all} \quad M > 0.
\end{align*}
Note here that while $h$ and $\wt{h}$ are not pointwise defined, their difference is, as they have the same lateral part and the radial parts are continuous functions. Thus by setting $M = b\log(\epsilon^{-1}) / \gamma$ and noting that if $\sup_{z \in D_R \setminus D_{1/R}}\left| \frac{1}{\gamma}\log\left( \frac{d\mu_{\wt h}}{d\mu_h}(z)\right) \right| \leq M$ and $\mu_h(B(z,\epsilon)) \leq \epsilon^{2b}$ then $\mu_{\wt h}(B(z,\epsilon))\leq \epsilon^{b}$, we obtain that 
\begin{align*}
\p\!\left[ \mu_h(B(z,\epsilon))\geq \epsilon^{2b} ,\text{for all} \, z  \in D_R \setminus D_{1/R} \, \text{such that} \, B(z,4\epsilon) \subseteq D_R \right] = 1 - O(\epsilon^{\sigma/2}),
\end{align*}
for $b$ sufficiently large. Then the proof is complete by combining this with the $R =1 $ case.
\end{proof}

\section{Estimates for whole-plane $\SLE$}
\label{app:proofs}

In this part of the appendix, we shall prove that one can fit small balls between two flow lines of critical angle difference (that is, the angle difference is at the critical value for whether they intersect or not) as well as bound the probability that a whole-plane $\SLE$ revisits a small ball around the origin, after exiting $\D$. The proofs are to a large extent modifications of the proof of \cite[Proposition~4.14]{ms2017imag4}. We begin with the former.

\begin{lemma}\label{lem:fitball}
Fix $\gamma \in (0,2)$, $\kappa = \gamma^{2} \in (0,4)$ and let $h$ be a field in $\C$ such that $h = \wt h - \alpha \arg( \cdot )$, modulo $2\pi(\chi + \alpha)$, where $\alpha > -\chi$ and $\wt h$ is a whole-plane GFF on $\C$. Let $\eta_1$ (resp.\ $\eta_2$) be the flow line of $h$ with angle $0$ (resp.\ $\theta_c = \frac{\pi \kappa}{4-\kappa}$) starting from zero and suppose that 
\begin{align*}
    \frac{(2\pi - \theta_c)\chi +  2\pi \alpha}{\lambda} = \frac{\kappa}{2} 
\end{align*}
Note that the marginal  law of $\eta_1$  is that  of a whole-plane  $\SLE_{\kappa}(\kappa - 2)$ starting from zero and conditional on $\eta_1$, $\eta_2$ has the law of a chordal $\SLE_{\kappa}(\frac{\kappa}{2}-2;\frac{\kappa}{2}-2)$ in $\C \setminus \eta_1$. Let $U_1,U_2$ be the two connected components of $\C \setminus (\eta_1 \cup \eta_2)$. Then with probability $1-o_\epsilon^\infty(\epsilon)$, we have that for $j = 1,2$ we can find a ball $B(z,\epsilon) \subseteq{U_j} \cap B(0,1)$.
\end{lemma}

For the proof of the above lemma, we need the following, the proof of which is essentially the same as that of \cite[Lemma~4.15]{ms2017imag4} and is omitted.
\begin{lemma}
\label{lem:cutout}
Fix $x^L \leq 0 \leq x^R$ and $\rho^L,\rho^R > -2$. Suppose that $h$ is a GFF on $\h$ with boundary data such that its flow line $\eta$ starting from $0$ is an $\SLE_\kappa(\rho^L;\rho^R)$ process with $\kappa \in (0,4)$ and force points at $x^L,x^R$. Fix $\theta > 0$ such that the flow line $\eta_\theta$ of $h$ starting from $0$ with angle $\theta$ a.s.\ does not hit the continuation threshold and a.s.\ intersects $\eta$. Let $T$ (resp.\ $T_\theta$) be the first time that $\eta$ (resp.\ $\eta_\theta$) leaves $B(0,2)$.  For $c_1, p \in (0,1)$,  we let $E = E_{c_1,p}$ be the event that $A = \eta([0,T]) \cup \eta_\theta([0,T_\theta])$ separates $i$ from $\infty$, the harmonic measure of the left side of $\eta$ as seen from $z \in B(i,c_1)$ in $\h \setminus A$ is at least $p \in (0,1)$ and $\dist(i,A) \geq c_1$. Then there exist $p_0,c_1,p \in (0,1)$, depending only on $\kappa, \rho^{L}, \rho^{R}$ and $\theta$, such that $\p[E] \geq p_0$.
\end{lemma}
\begin{proof}[Proof of Lemma~\ref{lem:fitball}]
This proof is more or less the same as the proof of \cite[Proposition~4.14]{ms2017imag4}, but we write it out for clarity. We fix $K \geq 5$ sufficiently large (to be chosen later). For $\epsilon \in (0,1)$, $1 \leq n \leq (K\epsilon)^{-1}$, let $T^{n,\epsilon}$ be the first time that $\eta_1$ exits $B(0,Kn\epsilon)$ when it is parameterized by $\log$-conformal radius. Let also $w_{n,\epsilon} \in  \partial B(0,(Kn+2)\epsilon)$ be such that $|\eta_1(T^{n,\epsilon})-w_{n,\epsilon}| = 2\epsilon$. Let $\theta_0 \in (0, \theta_c/2)$, where $\theta_c = \pi \kappa/(4-\kappa)$ is the critical angle of flow line interaction.

The flow line interaction rules imply that a flow line with angle $\theta_0$ a.s. hits a flow line of angle $0$ started at some point on its left side. We also let $\gamma_{n,\epsilon}$ be the flow line of $h$ starting from $\eta_1(T^{n,\epsilon})$ with angle $\theta_0$ and let $\sigma^{n,\epsilon}$ be the first time that $\gamma_{n,\epsilon}$ leaves $B(0,(Kn+4)\epsilon)$. Let $\mathcal F_n$ be the $\sigma$-algebra generated by $\eta_1|_{[0,T^{n,\epsilon}]}$ as well as $\gamma_{j,\epsilon} |_{[0,\sigma^{j,\epsilon}]}$ for $j = 1,\dots,n-1$. We fix $c,p>0$ (to be determined later) and define $E_{n,\epsilon}$ to be the event that $A_{n,\epsilon} = \eta_1 ([0,T^{n+1,\epsilon}]) \cup \gamma_{n,\epsilon}([0,\sigma^{n,\epsilon}])$ separates $w_{n,\epsilon}$ from $\infty$, the harmonic measure of the left side of $\eta_1 ([0,T^{n+1,\epsilon}])$ as seen from $z \in B(w_{n,\epsilon},c\epsilon)$ in the connected component $P_{n,\epsilon}$ of $\C \setminus A_{n,\epsilon}$ which contains $w_{n,\epsilon}$ is at least $p > 0$ and $\dist(w_{n,\epsilon},\partial P_{n,\epsilon}) \geq c\epsilon$.

Next, let $\psi_{n,\epsilon}$ be the conformal transformation which takes the unbounded component $U_{n,\epsilon}$ of $\C \cup \{ \infty \} \setminus \!\left (\eta_1([0,T^{n,\epsilon}]) \cup \bigcup_{j=1}^{n-1}\gamma_{j,\epsilon}([0,\sigma^{j,\epsilon}]) \right )$ to $\h$ with $\psi_{n,\epsilon}(\eta_1(T^{n,\epsilon})) = 0$ and $\psi_{n,\epsilon}(w_{n,\epsilon}) = i$. The Beurling estimate and the conformal invariance of the Brownian motion imply that for $K$ sufficiently large the images of $\cup_{j=1}^{n-1}\gamma_{j,\epsilon}([0,\sigma^{j,\epsilon}])$ and $\infty$ under $\psi_{n,\epsilon}$ lie outside of $B(0,100)$. Indeed, by picking $K$ large enough, we can make the harmonic measure of $\cup_{j=1}^{n-1}\gamma_{j,\epsilon}([0,\sigma^{j,\epsilon}])$, in $U_{n,\epsilon}$, as seen from $w_{n,\epsilon}$, be arbitrarily small. (More precisely, by the mapping $\varphi_{n,\epsilon}(z) = (z-w_{n,\epsilon})/((K-4)\epsilon)$, we have that the image of $\cup_{j=1}^{n-1}\gamma_{j,\epsilon}([0,\sigma^{j,\epsilon}])$ lies outside $\D$, whereas the distance from $\varphi_{n,\epsilon}(w_{n,\epsilon})=0$ to $\eta_1([0,T^{n,\epsilon}])$ is $2/(K-4)$, so by the Beurling estimate, the probability that a Brownian motion exits $\D$ before hitting $\eta_1([0,T^{n,\epsilon}])$ is at most a universal constant times $K^{-1/2}$, as $K \to \infty$.) Consequently, the law of the field $\wh h_{n,\epsilon} = h \circ \psi_{n,\epsilon}^{-1} - \chi \arg((\psi_{n,\epsilon}^{-1})')$ restricted to $B(0,50)$ is mutually absolutely continuous with respect to the law of a GFF on $\h$ restricted to $B(0,50)$ whose boundary data is chosen so that its flow line starting from 0 and targeted at $\infty$ is a chordal $\SLE_{\kappa}$, where the branch cut for the argument is taken to be the vertical line passing through $\psi_{n,\epsilon}(\infty)$. Note that since $\kappa - 2 \geq \frac{\kappa}{2} - 2$, there is no intersection of $\eta_1$ with itself.

The arguments of the proof of \cite[Lemma~4.15]{ms2017imag4} imply that there exists a universal constant $p_0 \in (0,1)$ which is uniform in the location of the images $\psi_{n,\epsilon}(\gamma_{j,\epsilon})$ such that if $\wt E_{n,\epsilon}$ is the event of Lemma~\ref{lem:cutout} corresponding to the flow lines of the field $\wh h_{n,\epsilon}$, then $\p[ \wt E_{n,\epsilon} | \mathcal F_n] \geq p_0$. We fix the constant $c \in (0,1)$ such that $\frac{4r}{(1-r)^2} < c_1$, where $r =  \frac{c}{2}$ and $c_1$ is the constant in Lemma~\ref{lem:cutout}. Then \cite[Corollary~3.25]{law2005slebook} implies that $\psi_{n,\epsilon}(B(w_{n,\epsilon},c\epsilon)) \subseteq{B(i,c_1)}$ and the conformal invariance of the harmonic measure implies that if $\wt E_{n,\epsilon}$ occurs, the harmonic measure of $\eta_1([0,T^{n+1,\epsilon}])$ as seen from $z$ is at least $p > 0$, for all $z \in B(w_{n,\epsilon},c\epsilon)$. By combining everything, we obtain that there exists a universal constant $\xi > 1$ such that with probability at least $1 - e^{-\epsilon^{-\xi}}$, we have more than $\frac{1}{2}p_0 n_{\epsilon}$ of the events $E_{j,\epsilon}$ occur, where $n_{\epsilon} = \lfloor K\epsilon \rfloor$. We observe that if the event $E_{n,\epsilon}$ occurs, then a pocket is formed between $\eta_1$ and $\gamma_{n,\epsilon}$ containing a ball of radius $c\epsilon$. The flow line interaction rules imply that $\eta_2$ cannot cross $\gamma_{n,\epsilon}$ and so the above ball must be contained in one of the connected components of $\C \setminus (\eta_1 \cup \eta_2)$. This completes the proof.
\end{proof}

\begin{remark}\label{rem:fitball}
Let $(\eta_1,\eta_2)$ be a pair of random curves in $\C$ starting from $0$ and targeted at $\infty$ sampled in the following way: First we sample $\eta_1$ to be a whole-plane $\SLE_4(2)$ starting from $0$ and conditional on $\eta_1$ we sample $\eta_2$ to be a chordal $\SLE_{4}$ in $\C \setminus \eta_1$ from $0$ to $\infty$. We note that the constants in Lemma~\ref{lem:cutout} can be chosen to be uniform as $\kappa \to 4$ if $\rho^L$ and $\rho^R$ are bounded from below away from $-2$ as $\kappa \to 4$. Also if we set $\theta_0 = \frac{\theta_c}{4}$ then $\chi \theta_0$ is bounded as $\kappa \to 4$ and so the Radon-Nikodym derivatives used in the proof are bounded uniformly on $\kappa$ and the weights of the processes that Lemma~\ref{lem:cutout} is applied, are bounded away from $-2$ as $\kappa \to 4$. Observe also that $\alpha$ is bounded uniformly on $\kappa$. Therefore the constant $\xi$ in Lemma~\ref{lem:fitball} can be made uniform in $\kappa$ as $\kappa \to 4$ and note that the law of the pair $(\eta_1,\eta_2)$ in Lemma~\ref{lem:fitball} converges to that of $(\eta_1,\eta_2)$ when $\kappa = 4$ as $\kappa \to 4$ in the Carath\'{e}odory sense. Therefore the claim of Lemma~\ref{lem:fitball} still holds for $\kappa = 4$.
\end{remark}

Finally, we turn our attention to the return probability of whole-plane $\SLE$.
\begin{lemma}\label{lem:Breturn_time}
 Fix $\kappa \in (0,4]$ and let $\eta$ be a whole-plane $\SLE_{\kappa}(\kappa - 2)$ in $\C$ from $0$ to $\infty$. Let $A(\epsilon)$ be the event that $\eta$ intersects $B(0,\epsilon)$ after it exits $\D$ for the first time. Then we claim that we can find a constant $b > 0$ uniform on $\kappa $ when $\kappa$ is away from $0$, such that $\p[A(\epsilon)] = O(\epsilon^b)$ as $\epsilon \to 0$.
\end{lemma}
\begin{proof}
First we assume that $\kappa \in (0,4)$. In order to prove the above claim, we follow a method similar to the one of Lemma~\ref{lem:fitball}. Fix $\theta \in (0,\theta_c /3)$ such that $-2 + \frac{\theta}{\pi}(\kappa /2 - 2) < \kappa /2 - 2$ and let $\eta_1$ (resp.\ $\eta_2$) be the flow line of angle $0$ (resp.\ $3\theta$) of $h$ emanating from $0$, where $h = \wt h -\alpha \arg$ for a whole-plane GFF $\wt{h}$ and $\alpha$ chosen to satisfy $\kappa - 2 = 2 - \kappa + 2\pi \alpha/\lambda$. Note that then $\eta_1$ and $\eta$ have the same law.  For $\epsilon \in (0,1)$ let $T^1_{n,\epsilon}$ (resp.\ $T^2_{n,\epsilon}$) be the first time that $\eta_1$ (resp.\ $\eta_2$) exits $B(0,2^n\epsilon)$ for $1 \leq n \leq \frac{\log(1/\epsilon)}{\log 2}$.  Let $\sigma_{n,\epsilon}^1$ (resp.\  $\sigma_{n,\epsilon}^2$) be the first time that $\eta_1$ (resp.\  $\eta_2$) exits $B(0,2^{n-1/4}\epsilon)$,  and we let $\gamma_{n,\epsilon}^{1,1}$ (resp.\  $\gamma_{n,\epsilon}^{1,2}$) be the flow line of $h$ with angle $-\theta$ (resp.\  $\theta$) starting from $\eta_1(\sigma_{n,\epsilon}^1)$ when the latter is seen as a prime end on the left (resp.\  right) side of $\eta_1$.  Similarly,  we let $\gamma_{n,\epsilon}^{2,1}$ (resp.\  $\gamma_{n,\epsilon}^{2,2}$) be the flow line of $h$ of angle $2\theta$ (resp.\ $4\theta$) starting from $\eta_2(\sigma_{n,\epsilon}^2)$ when the latter is seen as a prime end on the left (resp.\ right) side of $\eta_2$.  We stop $\gamma_{n,\epsilon}^{j,1}$ and $\gamma_{n,\epsilon}^{j,2}$ for $j=1,2$ at the first time that they exit $B(0,2^{n+1/2} \epsilon)$ and let $\wt{\sigma}_{n,\epsilon}^{j,i}$ be the first time that $\gamma_{n,\epsilon}^{j,i}$ exits $B(0,2^{n+1/2}\epsilon)$ for $i,j=1,2$.  Then,  for fixed constants $c ,  p \in (0,1)$ independent of $n$ and $\epsilon$,  we let $E_{n,\epsilon}^1$ (resp.\ $E_{n,\epsilon}^2$) be the event that $A_{n,\epsilon}^1 = \gamma_{n,\epsilon}^{1,1}([0,\wt{\sigma}_{n,\epsilon}^{1,1}]) \cup \gamma_{n,\epsilon}^{1,2}([0,\wt{\sigma}_{n,\epsilon}^{1,2}])$ (resp.\ $A_{n,\epsilon}^2 =\gamma_{n,\epsilon}^{2,1}([0,\wt{\sigma}_{n,\epsilon}^{2,1}]) \cup \gamma_{n,\epsilon}^{2,2}([0,\wt{\sigma}_{n,\epsilon}^{2,2}])$) separates $\eta_1(T_{n,\epsilon}^1)$ (resp.\ $\eta_2(T_{n,\epsilon}^2)$) from $\infty$,  the harmonic measure of the left and right sides of $\gamma_{n,\epsilon}^{1,1}$ and $\gamma_{n,\epsilon}^{1,2}$ (resp.\ $\gamma_{n,\epsilon}^{2,1}$ and $\gamma_{n,\epsilon}^{2,2}$) as seen from $z \in B(\eta_1(T_{n,\epsilon}^1),c\epsilon)$ (resp.\ $z \in B(\eta_2(T_{n,\epsilon}^2),c\epsilon)$) in the connected component $P_{n,\epsilon}^1$ (resp.\ $P_{n,\epsilon}^2$) of $\C \setminus A_{n,\epsilon}^1$ (resp.\ $\C \setminus A_{n,\epsilon}^2$) which contains $\eta_1(T_{n,\epsilon}^1)$ (resp.\ $\eta_2(T_{n,\epsilon}^2)$) is at least $p$ and $\dist(\eta_1(T_{n,\epsilon}^1),\partial P_{n,\epsilon}^1) \geq c 2^n \epsilon$ (resp.\ $\dist(\eta_2(T_{n,\epsilon}^2),\partial P_{n,\epsilon}^2) \geq c 2^n \epsilon$).  We set $E_{n,\epsilon} = E_{n,\epsilon}^1 \cap E_{n,\epsilon}^2$.

Let $\mathcal{F}_{n,\epsilon}$ be the $\sigma$-algebra generated by $\eta_1$ and $\eta_2$ stopped at times $T^1_{n,\epsilon}$ and $T^2_{n,\epsilon}$ respectively and the first $2n$ auxiliary flow lines for each of $\eta_1$ and $\eta_2$.  Then as in the proof of Lemma~\ref{lem:fitball},  we can find constants $c,p$ and $p_0>0$ independent of $n,\epsilon$ such that $\p[E_{n+1,\epsilon} \giv \CF_{n,\epsilon}] \geq p_0$.  Let $U_n$ be the connected component of $\C \cup \{\infty\} \setminus \left( \eta_1([0,T_{n,\epsilon}^1]) \cup \eta_2([0,T_{n,\epsilon}^2])\right)$ containing $\infty$ and $g_n : U_n \rightarrow \h$ be the conformal map such that $g_n(\eta_1(T^1_{n,\epsilon})) = 1$, $g_n(\eta_2(T^2_{n,\epsilon})) = -1$ and $g_n(0) = \infty$. We also consider the field $h_n = h \circ g_n^{-1} - \chi \arg((g_n^{-1})')$ where the branch cut for the argument is taken to be the vertical line passing through $g_n(\infty)$.

Note that if $E_{n,\epsilon}$ occurs, then there exists a universal constant $q \in (0,1)$ such that for all $z \in \partial{B(0,2^{n+1/2}\epsilon)}$, the sets  $[-1,0],[0,1]$ and $\R \setminus [-1,1]$ all have harmonic measure as seen from $g_n(z)$ at least $q$. Standard estimates for Brownian motion imply that in that case, $\dist([-1,1],g_n(\partial{B(0,2^{n+1/2}\epsilon)})) \geq d >0$ for some constant $d > 0$. Note that $E_{n,\epsilon}$ occurs with positive probability uniformly in $n,\epsilon$ and let $D \subseteq{\h}$ be a fixed simply connected domain such that $[-1,1] \subseteq \partial D \cap \partial \h$ and $\dist ([-1,1], \h \cap \partial D)<d/2$. Again the law of $h_n$ restricted to $D$ is mutually absolutely continuous with respect to the law of a GFF $\wh{h}$ on $\h$ restricted to $D$ with boundary values given by $-\lambda -3\chi \theta$ on $(-\infty,-1)$, $\lambda -3\chi \theta$ on $(-1,g_n(x_n))$, $-\lambda$ on $(g_n(x_n),1)$ and $\lambda$ on $(1,\infty)$, where $x_n$ is the last intersection of $\eta_1$ and $\eta_2$ before they exit $B(0,2^n \epsilon)$. Hence the law of $(g_n(\eta_1),g_n(\eta_2))$ stopped at the first time they exit $D$ is mutually absolutely continuous with respect to the law of $(\wh{\eta}_1,\wh{\eta}_2)$, where $\wh{\eta}_1$ (resp.\ $\wh{\eta}_2$) is the flow line of angle 0 (resp.\ $3\theta$) of $\wh{h}$, started from $1$ (resp.\ $-1$) and stopped at the first time they exit $D$. Then we have that with positive probability, uniform in the location of $g_n(x_n) \in (-1,1)$, $\wh{\eta}_1,\wh{\eta}_2$ intersect before they exit $D$ and so the same holds for $g_n(\eta_1)$ and $g_n(\eta_2)$. If that occurs, then $\eta_1$ and $\eta_2$ intersect before they exit $B(0,2^{n+1/2}\epsilon)$. Therefore as in the proof of Lemma~\ref{lem:fitball} and Lemma~\ref{lem:cutout} we obtain that for some constant $p > 0$ uniform in $n$, $\epsilon$ and $\kappa$ as $\kappa \to 4$, we have that conditional on $\CF_{n,\epsilon}$, with probability at least $p$, $\eta_1$ and $\eta_2$ intersect before they exit $B(0,2^{n+1/2}\epsilon)$. This shows that we can find a constant $c > 0$ (uniform as  $\kappa \to 4$) such that off an event with probability $O(\epsilon^b)$, $\eta_1$ and $\eta_2$ intersect before they exit $\D$ for the first time. Note that the same is true for $\eta_j$ and $\eta_{j+1}$ for all $j = 1,\ldots,m$, where $\eta_j$ is the flow line of $h$ of angle $3(j-1)\theta$ starting from $0$ and $m = \lfloor 2\pi(1+\alpha /\chi)/(3\theta) \rfloor$. We can therefore assume that off an event with probability $O(\epsilon^b)$, $\eta_j$ intersects both of $\eta_{j-1}$ and $\eta_{j+1}$ for all $j$ where $b$ is uniform in $\kappa$ as $\kappa \to 4$ since $m$ is bounded uniformly in $\kappa$ as $\kappa \to 4$ if we pick $\theta = \theta_c/4$. But if the latter occurs, then the flow line interaction rules imply that $\eta_1$ cannot enter $B(0,\epsilon)$ after it exits $\D$. This completes the proof of the claim since by taking $\kappa \to 4$ and using the Carath\'{e}odory convergence of the pair $(\eta_1,\eta_2)$ to the corresponding pair when $\kappa = 4$,  the claim is also true when $\kappa = 4$.
\end{proof}

\bibliographystyle{abbrv}
\bibliography{biblio}

\end{document}